\documentclass{imsart}

\usepackage{geometry}
\usepackage{amsthm,amsmath,amsfonts,amssymb}
\RequirePackage[numbers,sort&compress]{natbib}
\usepackage[colorlinks,citecolor=blue,urlcolor=blue]{hyperref}
\usepackage{graphicx}

\usepackage{adjustbox}
\usepackage{booktabs} % for professional table
\usepackage{multirow}
\newcommand{\ra}[1]{\renewcommand{\arraystretch}{#1}}
\usepackage{bbm}
\usepackage{stmaryrd}
\usepackage{enumitem}
\usepackage{etoc}
\usepackage{float}

\usepackage{comment}

\allowdisplaybreaks % pagebreak in equation allows
\RequirePackage{dsfont}% pour le 1 avec double barre
\usepackage{color}
\usepackage{xcolor}         % define colors in text

\usepackage[disable,textsize=tiny]{todonotes}

\startlocaldefs
\newtheorem{theorem}{Theorem}
\newtheorem{lemma}[theorem]{Lemma} 
\newtheorem{proposition}[theorem]{Proposition} 
\newtheorem{remark}[theorem]{Remark}
\newtheorem{corollary}[theorem]{Corollary}
\newtheorem{definition}[theorem]{Definition}

\newtheorem{assumption}{Assumption}

\newcommand\Qh{{\widehat{Q}}}
\newcommand{\calD}{{\cal D}}
\newcommand{\calE}{{\cal E}}
\newcommand\calX{{\cal X}}
\newcommand\calY{{\cal Y}}

\newcommand\Chat{{\widehat{\mathcal{C}}}}

\newcommand\fh{{\widehat{f}}}
\newcommand\IE{{\mathbb{E}}}
\newcommand\E{{\mathbb{E}}}
\newcommand\IR{{\mathbb{R}}}
\newcommand\IN{{\mathbb{N}}}
\newcommand\R{{\mathbb{R}}}
\newcommand\IP{{\mathbb{P}}}

\newcommand\calS{{\cal S}}
\renewcommand\P{{\mathbb{P}}}

\renewcommand{\leq}{\leqslant}
\renewcommand{\geq}{\geqslant}
\DeclareMathOperator{\Var}{Var}
\newcommand{\egalloi}{\stackrel{d}{=}}
\DeclareMathOperator{\crit}{crit}

\newcommand{\methodCentral}{\text{CentralM}}
\newcommand{\methodCentralC}{\text{CentralC}}
\newcommand{\method}{\text{QQM}}
\newcommand{\methodAvg}{\text{FedCP-Avg}}
\newcommand{\methodlow}{\text{QQM-Fast}}
\newcommand{\methodcondFL}{\text{QQC}}
\newcommand{\methodcondFLk}{\text{QQC-Fast}}

\newcommand{\1}{\mathds{1}}

\newcommand{\intset}[1]{\llbracket #1 \rrbracket}%% ensemble des entiers de 1 a #1
\DeclareMathOperator{\Beta}{Beta}%% loi Beta
\newtheorem{algorithm}{Algorithm}
\DeclareMathOperator*{\argmin}{argmin}
\DeclareMathOperator*{\argmax}{argmax}
\newcommand{\egaldef}{:=}%% le membre de gauche est defini par le membre de droite de l'egalite
\newcommand{\defegal}{=:}%% le membre de droite est defini par le membre de gauche de l'egalite
\newcommand{\Xu}{\widetilde{Z}}
\renewcommand{\Im}{\ensuremath{\mathrm{Im}}}
\newcommand{\tU}{\widetilde{U}}

\endlocaldefs

\begin{document}

\begin{frontmatter}
	\title{Marginal and training-conditional guarantees in one-shot federated conformal prediction}
	\runtitle{Marginal and training-conditional validity in one-shot FCP}
	%\thankstext{T1}{A sample of additional note to the title.}
	
	\begin{aug}
\author[A]{\fnms{Pierre}~\snm{Humbert}\ead[label=e1]{pierre.humbert@universite-paris-saclay.fr}},
\author[B]{\fnms{Batiste}~\snm{Le Bars}\ead[label=e2]{batiste.le-bars@inria.fr}},
\author[C]{\fnms{Aur\'{e}lien}~\snm{Bellet}\ead[label=e3]{aurelien.bellet@inria.fr}}
\and
\author[A,D]{\fnms{Sylvain}~\snm{Arlot}\ead[label=e4]{sylvain.arlot@universite-paris-saclay.fr}}
%%%%%%%%%%%%%%%%%%%%%%%%%%%%%%%%%%%%%%%%%%%%%%
%% Addresses                                %%
%%%%%%%%%%%%%%%%%%%%%%%%%%%%%%%%%%%%%%%%%%%%%%
\address[A]{
	Universit\'{e} Paris-Saclay, CNRS, Inria, Laboratoire de math\'{e}matiques d'Orsay, 91405, Orsay, France\printead[presep={,\ }]{e1,e4}}

\address[B]{
Inria Paris - Ecole Normale Sup\'{e}rieure, PSL Research University\printead[presep={,\\ }]{e2}}

\address[C]{
	PreMeDICaL project team, Inria, Univ Montpellier, France\printead[presep={,\\ }]{e3}}

\address[D]{
	Institut Universitaire de France (IUF)}

\end{aug}

\begin{abstract}
We study conformal prediction in the one-shot federated learning setting.
The main goal is to compute marginally and training-conditionally valid prediction sets, at the server-level, in only one round of communication between the agents and the server. 
Using the quantile-of-quantiles family of estimators 
and split conformal prediction, we 
introduce a collection of computationally-efficient and distribution-free algorithms that 
satisfy the aforementioned requirements.
Our approaches come from theoretical results related to order statistics and the analysis of the Beta-Beta distribution.
We also prove upper bounds on the coverage of all proposed algorithms 
when the nonconformity scores are almost surely distinct.
For algorithms with training-conditional guarantees, 
these bounds are of the same order of magnitude as those of the centralized case. Remarkably, this implies that the one-shot federated learning setting entails no significant loss compared to the centralized case.
Our experiments confirm that our algorithms return prediction sets with coverage and length similar to those obtained in a centralized setting.
\end{abstract}

	
\end{frontmatter}

\section{Introduction} \label{sec:intro}
% !TEX root = ../main.tex

\subsection{Problem statement and motivation}
We consider the one-shot federated learning (FL) set-prediction problem, where 
a set of agents connected to a central server try to compute a valid prediction
set in only one round of communication, and without sharing their raw data \citep{humbert2023one}.

Formally, assume that a data set $\calD$ is distributed over 
$m \in \IN^*$ agents connected to a central server. 
In addition, suppose that the server has an independent 
test point $(X, Y) \in \calX \times \calY$ following 
the same distribution as the elements 
$(X_{i, j}, Y_{i, j})$ of $\calD$ 
---which are assumed to be identically distributed--- 
but whose outcome $Y$ is unobserved.  
The goal of the server is to compute a distribution-free \emph{valid} prediction set for $Y$. 
One can aim at marginal validity, that is, 
for a given miscoverage level $\alpha \in (0, 1)$, constructing 
from $\calD$ a set $\Chat(X)$ such that 
\begin{equation}\label{eq:marg_cov}
	\IP\left(Y \in \Chat(X) \right) \geq 1-\alpha \; ,
\end{equation}
whatever the data distribution.
In addition, the set $\Chat(\cdot)$ must be
computed in a single round of communication between the agents and the server to satisfy the one-shot constraint ---a condition motivated by the fact that the number of communication rounds is often the main bottleneck in FL \citep{kairouz2021advances}.
Note also that we want the set to be small (for instance, with respect to the counting or the Lebesgue measure), a property called \textit{efficiency}.

The probability in \eqref{eq:marg_cov} is taken 
with respect to $(X, Y)$ and the data set~$\calD$.
However, in practice, we only have access to one particular data set. Another quantity of interest is therefore the training-conditional miscoverage rate defined by 
\begin{equation}\label{eq:misc}
\alpha(\calD) 
:= \IP\bigl(Y \notin \Chat(X) \,\big\vert\, \calD \bigr) 
\, ,
\end{equation}
where the probability is now only taken with respect to the test point $(X, Y)$. 
Notice that $\alpha(\calD)$ is a $\sigma(\calD)$-measurable random variable and that the marginal guarantee \eqref{eq:marg_cov} corresponds to a bound on the miscoverage rate \eqref{eq:misc} on average over all possible datasets. 
In other words, Eq.~\eqref{eq:marg_cov} is equivalent to a control on the expectation of the training-conditional miscoverage rate: $\IP\bigl(Y \notin \Chat(X)  \bigr)  = \IE [ \alpha(\mathcal{D})] \leq \alpha$. 
However, the random variable $\alpha(\mathcal{D})$ can have high variance, and 
it is important to also control its deviation from the desired upper-bound $\alpha$ 
with high probability.
For any $\alpha, \beta \in (0, 1)$, we are therefore also interested in constructing prediction sets 
with training-conditional coverage guarantees, that is, of the form
\begin{equation}\label{eq:cond_cov}
\IP\bigl(
1-\alpha(\calD) \geq 1-\alpha 
\bigr) \geq 1-\beta 
\, .
\end{equation}
Prediction sets satisfying \eqref{eq:cond_cov} are also known in the statistical literature as $(\alpha, \beta)$-tolerance regions or ``Probably Approximately Correct'' (PAC) predictive sets.
This type of guarantee dates back to \cite{wilks1941determination} and an overview can be found in \cite{krishnamoorthy2009statistical}. 
More recent works on this subject include \cite{vovk2012conditional, park2019pac, kivaranovic2020adaptive, park2021pac, bian2023training}.

\subsection{Related works} 
Our contribution takes place in the FL framework, a rather recent paradigm that allows training from decentralized data sets stored locally by multiple agents \citep{kairouz2021advances}. 
In this framework, the learning is made without exchanging raw data, making FL advantageous when data are highly sensitive and cannot be centralized for privacy or security reasons.
So far, the design of FL algorithms has mainly focused on the learning step of statistical machine learning,
the goal being to fit a (pointwise) predictor to decentralized data sets while minimizing the amount of communication \citep[see e.g.][]{mcmahan2017communication, li2020federated,scaffold}. 
However, quantifying the uncertainty in the prediction of these FL algorithms has not been widely studied yet.

Conformal prediction (CP) methods have become the state-of-the-art to construct marginally and conditionally valid distribution-free prediction sets \citep{papadopoulos2002inductive, vovk2005algorithmic, vovk2012conditional, romano2019conformalized}. 
Unfortunately, one of the key steps of 
CP methods is the ordering of some computed scores, 
which is not possible in FL settings without sharing the full local data sets or performing many agent-server communication rounds. 
These methods are thus not well-suited to the constraints of FL in which agents process their data locally and only interact with a central server by sharing some aggregate statistics. 
Constructing a valid prediction set is even more difficult in the one-shot FL setting 
\cite{zhang2012communication,guha2019one,Bayesian-oneshot,practical-oneshot,clustering-oneshot, one-shot-jmlr} 
considered in this work, where only one round of communication between the agents and the server is allowed. 

To our knowledge, \cite{lu2021distribution} is the first paper considering the one-shot federated set-prediction problem with conformal prediction.
Its idea is to locally calculate some quantiles of computed scores for each agent and to average them in the central server. 
Unfortunately, \cite{lu2021distribution} does not prove that 
the corresponding prediction sets are valid and its method 
is non-robust, especially when the size of local data sets 
is small (see Appendix~\ref{app.avg_qq} for more details).
To address these issues, \citep{humbert2023one} has recently proposed a family of estimators called \emph{quantile-of-quantiles}.  
The idea is that each agent sends to the server a local empirical quantile of its scores and the server aggregates them by computing a quantile of these quantiles. 
Interestingly, these estimators can therefore be seen as a clever way of aggregating several quantiles calculated locally by $m$ agents, 
who each locally use the (centralized) split CP method (see Section~\ref{subsec:splitCP} for details). 
Last but not least, when $m=1$ this approach exactly recovers split CP, 
emphasizing that methods based on quantile-of-quantiles estimators generalize centralized split CP to the FL setting.
However, an important limitation of \cite{humbert2023one} is that in order to determine which order of quantile to select for marginal validity \eqref{eq:marg_cov}, 
the proposed methodology can be computationally intensive at the server level. 
Moreover, the training-conditional guarantee \eqref{eq:cond_cov} given by \cite{humbert2023one} is obtained with a conservative procedure, leading to large prediction sets in practice. 
The present paper shows how to solve these two issues. 

Outside the one-shot FL setting considered here, 
we can also mention \citep{lu2023federated} and \citep{pmlr-v202-plassier23a}, 
which focus on data-heterogeneous settings but require many communication rounds between the agents and the server. 
Finally, we can also mention recent works on federated evaluation of classifiers \citep{cormode2023federated}, 
federated quantile computation \cite{andrew2021differentially, pillutla2022differentially}, 
and on uncertainty quantification with Bayesian FL \cite{el2021federated,kotelevskii2022fedpop} which, 
although related to our work, do not study CP and do not obtain formal coverage guarantees.

\subsection{Contributions}
In this work, we consider the quantile-of-quantiles family of estimators 
proposed in \cite{humbert2023one} and introduce several
algorithms to find the appropriate order of quantiles so that the marginal condition \eqref{eq:marg_cov} or the training-conditional condition \eqref{eq:cond_cov} are satisfied. 
Each of these algorithms is computationally-efficient, 
distribution-free (depending only on the number of agents and the size of their local data sets) 
and specially tailored to satisfy the aforementioned conditions. 
Importantly, they come from novel theoretical results, 
which take their roots in the theory of order statistics. 
For clarity and simplicity of reading, all the contributions 
are first presented in the case where agents have 
the same number of data points $n \geq 1$ 
(Assumption~\ref{ass:same_n} in Section~\ref{sec:QQ}).\\

\noindent
For distribution-free marginal guarantees \eqref{eq:marg_cov}, in Section~\ref{subsec:marg_valid}: 
\begin{itemize}
	\item We prove that when $\Chat$ is obtained using our method, its probability of coverage 
	$\IE[1-\alpha(\calD)] = \IP(Y \in \Chat(X) )$ 
	is lower bounded by the expectation of a random variable following a particular Beta-Beta distribution \citep{cordeiro2013simple, makgai2019beta} 
	(when the scores are almost surely distinct, these two quantities are equal). 
	We also derive a closed-form expression for this expectation (Theorem~\ref{them:main}), 
	improving the one obtained in \citep{humbert2023one} and leading to Algorithm \ref{alg:FedCPQQ}.
	\item This closed-form expression remains difficult to compute 
	for large values of $n$ and $m$, 
	leading to a quite computationally demanding algorithm. 
	To tackle this problem, we show that the expectation of this Beta-Beta distribution is lower and 
	upper bounded by the quantile function of a standard Beta distribution evaluated at particular values (Proposition~\ref{cor:bound_Ekl}). 
	These bounds are tight and fast to compute, which makes them interesting for practical use: they lead to Algorithm~\ref{algo.FCP-QQ-marg.2}, 
	which is more computationally efficient.
\end{itemize}
\noindent
In Section~\ref{sec.multi-order.algos}, we build the first 
(to the best of our knowledge) 
one-shot FL algorithms with distribution-free \emph{training-conditional} guarantees~\eqref{eq:cond_cov}. 
More precisely: 
\begin{itemize}
\item We prove that $1-\alpha(\calD)$ is stochastically larger 
(equal when the scores are a.s. distinct) than a Beta-Beta random variable (Theorem~\ref{them:cond_main}). 
We also show that the quantiles of the Beta-Beta are fast to compute, 
leading to an efficient algorithm that constructs training-conditionally valid prediction sets in one-shot FL 
(Algorithm~\ref{algo.FCP-QQ-cond.1}). 
\item In order to obtain an even faster algorithm, 
we provide a tight bound on the cumulative distribution
function (cdf) of the Beta-Beta distribution (Proposition~\ref{prop:low_bound_Fl})
which allows the automatic selection of the empirical-quantiles order that the agents should send to the server  
(Algorithm~\ref{algo.FCP-QQ-cond.2}). 

\end{itemize}
Importantly, our results allow to trade-off the tightness of the bounds for computational efficiency. 
More generally, our contributions go beyond the setting of FL, 
in the sense that they investigate the open question of 
how several split CP estimators obtained over independent 
data sets should be aggregated to obtain valid prediction sets. \\

In Section~\ref{sec.uppbound_cov}, we give
several upper-bounds on
the probability of coverage of our prediction sets.
We first derive an upper-bound of order $1-\alpha + \mathcal{O}(m^{-1}n^{-1/2})$ 
for the marginal coverage of our methods from Section~\ref{subsec:marg_valid} 
(Algorithms~\ref{alg:FedCPQQ}--\ref{algo.FCP-QQ-marg.2}). 
This result shows that our coverage is not too much above $1-\alpha$ and strictly 
improves upon the work of \cite{humbert2023one}, 
which did not provide such type of results. 
In the same vein, we also prove high-probability upper bounds 
on the training-conditional miscoverage rate obtained by 
the methods of Section~\ref{sec.multi-order.algos} 
(Algorithms~\ref{algo.FCP-QQ-cond.1}--\ref{algo.FCP-QQ-cond.2}). 
Remarkably, these bounds are of order $1-\alpha + \mathcal{O}((mn)^{-1/2})$, 
the same order of magnitude as those of the centralized case. 
Hence, up to constant factors, the one-shot federated learning 
setting does not incur any loss. 

In Section~\ref{sec:diff_n}, we extend our theoretical results and associated methods 
to the more complicated setting where the agents can have 
different data set sizes
(Theorems~\ref{thm:main_nj}--\ref{them:cond_main_nj} 
and Algorithms~\ref{algo.FCP-QQ.1_nj}--\ref{algo.FCP-QQ-cond.1_nj}), 
improving again upon the results presented in \cite{humbert2023one}. 

Finally, in Section~\ref{sec:xps}, we empirically evaluate the performance of our algorithms on standard CP benchmarks,  
validating that they produce prediction sets that are computationally efficient and close to those obtained when data are centralized.

\section{Preliminaries} \label{sec:prelim}
% !TEX root = ../main.tex

\subsection{Split Conformal Prediction} \label{subsec:splitCP}
Conformal Prediction (CP) is a framework to construct prediction sets satisfying \eqref{eq:marg_cov} 
without relying on any distributional assumption on the data \citep{vovk2005algorithmic}. 
In a centralized setting, one of the most popular methods to perform CP is the split conformal method (split CP) \citep{papadopoulos2002inductive}. 
Since the quantile-of-quantiles procedure 
studied in this paper generalizes split CP, 
let us detail here how it is defined and 
recall some of the key results of previous literature. 
Simple and full proofs of these results are 
provided in Appendix~\ref{app.pr.split-CP}. 

Assume that we have access to a centralized 
dataset $\calD$, that we split into 
a learning set $\calD^{lrn}$ and a calibration set 
$\calD^{cal}=(X_i,Y_i)_{1 \leq i \leq n_c}$, 
where $n_c \geq 1$ and the calibration data 
$(X_i,Y_i)$, $1 \leq i \leq n_c$, are i.i.d. 
and follow the same distribution as the independent 
test point $(X,Y)$.

First, a predictor $\fh$ is built from $\calD^{lrn}$ only, 
and it is used to define a nonconformity score function 
$s = s_{\fh}: \calX \times \calY \rightarrow \IR$, 
such that for every $(x,y) \in \calX \times \calY$, 
$s_{\fh}(x,y)$ measures how far the prediction $\fh(x)$ is from the true output~$y$. 
Whether we are in the regression or classification setting, many different score functions exist in the literature (see e.g. \cite{angelopoulos2023conformal}). 
In regression, for instance, a common choice is the fitted absolute residual
$s_{\fh} : (x,y) \mapsto \lvert y - \fh(x) \rvert$. 
In the sequel, we often write $s$ instead of $s_{\fh}$ for simplicity. 
Furthermore, note that split CP does not assume a particular choice of score function, so throughout the paper, we keep the function $s$ abstract. 

Second, we calculate the values of $s_{\fh}$ taken on the calibration set $\calD^{cal}$, 
called the nonconformity scores 
$S_i := s_{\fh}(X_i,Y_i)$, $i=1, \ldots, n_c$. 

Third, we compute the $r$-th smallest nonconformity score 
$S_{(r)} := \Qh_{(r)}(\calS^{cal}_{n_c})$ for some $r \in \intset{n_c} := \{ 1 , \ldots , n_c \}$, 
where 
$\calS^{cal}_{n_c} := ( S_{1},\ldots, S_{n_c} )$
and $\Qh_{(\cdot)}$ is the sample quantile function defined by 
\begin{equation}
\label{def:Qk}
\forall r , N \geq 1 \, , \, 
\forall \calS' \in \R^N \, ,
\qquad 
\Qh_{(r)}(\calS') := \begin{cases}
\calS'_{(r)} & \text{if } r \leq N  \\
+ \infty & \text{otherwise} 
\, ,
\end{cases}
\end{equation}
with $\calS'_{(1)} \leq \ldots \leq \calS'_{(N)}$ the ordered 
values of $\calS' = (\calS'_1, \ldots, \calS'_N)$. 
Finally, the split CP prediction set is defined, for any $x\in \calX$, by
\begin{align}
\label{set_conf}
\Chat_{r}(x) 
:= \Bigl\{ y \in \calY \,:\, s_{\fh}(x, y) \leq S_{(r)} \Bigr\} 
\, . 
\end{align} 

Following Eq.~\eqref{eq:misc} in Section~\ref{sec:intro}, 
we define the training-conditional miscoverage rate of $\Chat_{r}$ by 
\[
\alpha_r (\calD) := \IP\bigl(Y \notin \Chat_{r}(X) \,\big\vert\, \calD \bigr)
\, , 
\]
where $(X,Y)$ is a test point, independent of $\calD$ and with the same distribution as the $(X_i,Y_i)$. 
We then have
\[ 
1 - \alpha_r(\calD) 
= \P\bigl( s_{\fh}(X, Y) \leq S_{(r)} \bigr) 
= F_S ( S_{(r)} )
\]
where $F_S$ is the common cdf of the scores~$S_i$. 
It is well known that for every $r \in \intset{n_c}$, 
\begin{equation} 
\notag 
\IP\bigl(Y \in \Chat_{r}(X) \bigr) =
\E\bigl[ 1 - \alpha_r (\calD) \bigr] 
\geq \frac{r}{n_c+1} 
\, , 
\end{equation}
with equality if the scores $S_i$ are almost surely distinct \citep{vovk2005algorithmic, lei2018distribution}. 
As a consequence, for any $\alpha \in (0,1)$, 
if $(1-\alpha) (n_c+1) \leq n_c$, 
taking $r = \lceil (1-\alpha) (n_c+1)\rceil$ in Eq.~\eqref{set_conf} 
yields a prediction set satisfying Eq.~\eqref{eq:marg_cov}, 
and if the $S_i$ are almost surely distinct, 
its expected miscoverage rate is 
\begin{equation}
\label{eq.splitCP.marg-cov-upp}
\E\bigl[ 1 - \alpha_{\lceil (1-\alpha) (n_c+1)\rceil} (\calD) \bigr] 
= \frac{ \lceil (1-\alpha) (n_c+1)\rceil }{n_c+1} 
\leq 1 - \alpha + \frac{1}{n_c+1} 
\, . 
\end{equation}

Regarding training-conditional guarantees,
a straightforward consequence of \cite[Proposition~2b]{vovk2012conditional}
 is that 
\begin{equation}
\label{eq:miscov_vovk}
\forall r \geq 1 \, , \, 
\forall \beta \in (0,1) \, , \qquad 
\IP\bigl( 1-\alpha_{r}(\calD) \geq F^{-1}_{U_{(r:n_c)}}(\beta) \bigr) \geq 1-\beta 
\, ,
\end{equation}
where $F^{-1}_{U_{(r:n_c)}}$ denotes the quantile function of the
Beta$(r, n_c-r+1)$ distribution.
In other words, $1-\alpha_{r}(\calD)$ is \emph{stochastically larger} than the Beta$(r, n_c-r+1)$ distribution. 
Furthermore, Eq.~\eqref{eq:miscov_vovk} becomes an equality if the scores $S_i$ are almost surely distinct 
---that is, in such a case, $\alpha_{r}(\calD)$ exactly follows a Beta$(r, n_c-r+1)$ distribution. 
Finally, taking $r \in \intset{n_c}$ such that 
$F^{-1}_{U_{(r:n_c)}}(\beta) \geq 1-\alpha$, 
Eq.~\eqref{eq:miscov_vovk} implies that the split CP prediction set $\Chat_{r}$ 
satisfies Eq.~\eqref{eq:cond_cov}, that is, 
$\Chat_{r}$ is a $(\alpha, \beta)$-tolerance region 
---see also Eq.~\eqref{eq.splitCP.algo-cond} in Appendix~\ref{app.pr.split-CP.cond}.

\begin{remark}
\label{rk.splitCP.asympt}
When $n_c$ tends to infinity, the optimal 
asymptotically training-conditionally valid~$r$ 
---given by Eq.~\eqref{eq.app.pr.split-CP.cond.algo-asympt} 
in Appendix~\ref{app.pr.split-CP.cond}--- 
yields a prediction set $\Chat_r$ with coverage 
between $1-\alpha$ and $1-\alpha + \mathcal{O}(1/\sqrt{n_c})$ 
with high probability 
---see Eq.~\eqref{eq.app.pr.split-CP.cond.algo-asympt.coverage} 
in Appendix~\ref{app.pr.split-CP.cond} 
for a precise statement. 
\end{remark}

The problem of split CP in a federated setting 
is that computing the quantile $S_{(r)}$ requires in general 
several (and often many) communications between 
the central server and the agents, 
hence it cannot be used in one-shot FL. 
Therefore, we consider in this paper another family of procedures (quantile-of-quantiles estimators), 
that we define in the next section.

Note that, although the first CP methods were the \emph{split} 
and the related \emph{full} methods \citep{papadopoulos2002inductive, vovk2005algorithmic}, 
many extensions based upon them and with similar guarantees have been proposed in the literature. 
Their principal novelty lies in a clever choice of the non-conformity score function~$s$. 
In regression, \cite{lei2018distribution} presents a method called locally weighted CP and provides theoretical insights for conformal inference. 
More recently, \cite{romano2019conformalized} has developed a variant of the split CP 
called Conformal Quantile Regression (CQR). 
Other recent alternatives have been proposed \citep{kivaranovic2020adaptive, sesia2021conformal, gupta2022nested, ndiaye2022stable, han2022split, guan2023localized}. 
We refer to \cite{vovk2005algorithmic}, \cite{angelopoulos2023conformal}, and \cite{fontana2023conformal} for in-depth presentations of CP and 
to \cite{manokhin_2022_6467205} for a curated list of papers related to CP.

\subsection{Federated conformal prediction with the quantile-of-quantiles} \label{sec:QQ}
We now present the quantile-of-quantiles family of estimators, first introduced in \citep{humbert2023one}, 
and how it can be used to obtain valid prediction sets in a one-shot FL setting, that is, 
in a setting where only one round of communication between the agents and the server is allowed \citep{guha2019one, zhang2012communication} 
and where only aggregated statistics computed locally by the agents can be sent to the server. 
From now on, we assume that the decentralized data set 
$\calD$ is divided into a learning set $\calD^{lrn}$ 
and a calibration set 
$\calD^{cal} = (X_{i,j}, Y_{i,j})_{1 \leq j \leq m, 1 \leq i \leq n_j}$ 
where $n_j \geq 1$ for every $j \in \intset{m}$. 
Among calibration data, agent $j \in \intset{m}$ 
has only access to $(X_{i,j}, Y_{i,j})_{1 \leq i \leq n_j}$. 
We assume that $\calD^{lrn}$ is independent from $\calD^{cal}$, 
and that the calibration data $(X_{i, j}, Y_{i, j})$, $j \in \intset{m}$, $i \in \intset{n_j}$ are i.i.d. 
with the same distribution as the test point $(X,Y)$ 
(which is independent from $\calD$). 

We assume that a (pointwise) predictor $\fh$ is learned on $\calD^{lrn}$ only, 
using for instance standard FL algorithms such as FedAvg \citep{mcmahan2017communication}. 
Therefore, $\fh$ is independent from $\calD^{cal}$. 
As in the centralized setting, $\fh$ is used to define a nonconformity score function 
$s = s_{\fh}: \calX \times \calY \rightarrow \IR$ 
such that for every $(x,y) \in \calX \times \calY$, 
$s_{\fh}(x,y)$ measures how far the prediction $\fh(x)$ is from the true output~$y$. 
In the sequel, we only focus on the calibration of the prediction set and not on the learning part.

\begin{remark} \label{rk.fh-general}
	Note that $\fh$ does not have to be a point-wise predictor, 
	that is, a function $\calX \to \calY$. 
	For instance, like in CQR 
	\cite{romano2019conformalized}, we can rely on 
	the use of a score function $s$ depending on $\fh = (\fh_-,\fh_+)$
	a pair of quantile functions $\calX \to \calY$.
\end{remark}

\begin{remark}
In the following, all probabilistic statements are 
valid conditionally to $\calD^{lrn}$. 
This amounts to acting as if $\fh$ were deterministic, since $\calD^{lrn}$ 
appears only through $\fh$ and $\calD^{cal}$ is independent from $\calD^{lrn}$.
\end{remark}

For simplicity, from now on and until the end of Section~\ref{sec.uppbound_cov}, 
we also make the following assumption. 
\begin{assumption}
\label{ass:same_n}
Each agent $j \in \intset{m}$ has exactly $n_j=n \geq 1$ calibration data points. 
\end{assumption}
Under Assumption~\ref{ass:same_n}, the calibration data set size is equal to $nm$. 
We refer to Section~\ref{sec:diff_n} for the more general case where agents 
have data sets of calibration of different sizes $(n_j)_{1\leq j \leq m}$.

For calibration, the first step 
is to ask each agent
$j\in \intset{m}$ to compute its $n$ i.i.d. local calibration scores $\mathcal{S}_j := (S_{1, j}, \ldots, S_{n, j})$, 
where $S_{i, j} = s_{\fh}(X_{i, j}, Y_{i, j})$ is the score associated to the $i$-th calibration data point of agent $j$. 
Then, the server and the agents jointly compute the quantile-of-quantiles (QQ) estimator, defined as follows.
\begin{definition} (Quantile-of-quantiles \cite{humbert2023one}) \label{def:QQ}
	For any $(\ell,k) \in \intset{n} \times \intset{m}$, the QQ estimator of order $(\ell,k)$ calculated on the sets of scores $(\mathcal{S}_j)_{1 \leq j \leq m}$ is
	\begin{equation} \label{eq:CP-QQ}
	S_{(\ell, k)} := \Qh_{(k)}\left(\Qh_{(\ell)}(\mathcal{S}_1), \ldots, \Qh_{(\ell)}(\mathcal{S}_m)\right)\; ,
	\end{equation}
	where 
	$\Qh_{(\cdot)}(\cdot)$ is the sample quantile function defined in Eq.~\eqref{def:Qk}.
\end{definition}
In other words, each agent sends to the server its $\ell$-th smallest local score, denoted by $\Qh_{(\ell)}(\mathcal{S}_j)$, and the server then computes the $k$-th smallest value of these scores, denoted by $S_{(\ell, k)}$. 
This strategy requires a single round of communication and thus fits the constraints of one-shot FL. 
Finally, for any $x \in \calX$, we define the prediction set 
\begin{equation}\label{CPQQ_set}
\Chat_{\ell, k}(x) 
= \Bigl\{y \in \calY : s(x, y) \leq S_{(\ell, k)} \Bigr\} 
\, ,
\end{equation}
similarly to the centralized split CP method ---see Eq. \eqref{set_conf}---, 
but with $S_{(r)}$ replaced by $S_{(\ell, k)}$. 

The decentralized QQ prediction set can be seen as 
a generalization of (centralized) split CP 
since they coincide when $m=1$. 
The crucial remaining component is a computationally efficient approach to identify a pair $(\ell, k)$ so that 
$\Chat_{\ell, k}$ defined in Eq.~\eqref{CPQQ_set} satisfies the marginal condition \eqref{eq:marg_cov} or the training-conditional condition \eqref{eq:cond_cov} while being as small as possible. 
We investigate these points in the next sections.

\subsection{Prediction set performance measure} \label{sec.prelim.perf}
A natural way to evaluate the performance of a valid prediction set 
$x \mapsto \Chat(x)$ is to measure its size 
$\mu(\Chat(x))$, where $\mu$ is some measure on $\calY$, 
for instance, the counting measure when $\calY$ is finite, 
or the Lebesgue measure when $\calY \subset \R^p$ for some $p \geq 1$.
This size should be minimized, 
either at a given $x \in \calX$ or on average over $x = X$. 
For a prediction set of the form $\Chat_{\ell,k}$, 
as defined by Eq.~\eqref{CPQQ_set}, 
its size depends on $\mu$, 
on the score function $s$, on the predictor~$\fh$, and on the pair $(\ell , k)$.
In order to build general-purpose algorithms for choosing $(\ell,k)$, 
a natural strategy is thus to select, 
among all the pairs such that $\Chat_{\ell,k}$ 
is marginally or conditionally valid, the one which also 
minimizes the size $\mu(\Chat_{\ell,k}(x))$.
By Eq.~\eqref{CPQQ_set}, this size
is a nondecreasing function of 
$S_{(\ell,k)}$. Furthermore, we know that
\begin{align} \label{eq.lien-score-coverage.QQ}
S_{(\ell,k)} 
&\stackrel{a.s.}{=} F_S^{-1} \circ F_S ( S_{(\ell,k)} ) 
= F_S^{-1} \bigl( 1 - \alpha_{\ell, k}(\calD) \bigr)
\, ,
\\
\notag 
\text{where} \qquad 
1 - \alpha_{\ell, k}(\calD)
&:= \P \bigl( Y \in \Chat_{\ell,k}(X) \,\vert\, \calD \bigr)
= F_S ( S_{(\ell,k)} )
\end{align}
is the coverage of $\Chat_{\ell,k}$, 
$F_S$ is the cdf of the scores~$S_{i,j}$, 
and $F_S^{-1}$ its generalized inverse 
---see Eq.~\eqref{eq.def-gal-inverse} in Appendix~\ref{app.order-stat.order}. 
Hence, the size of $\Chat_{\ell,k}(x)$ 
is also a nondecreasing function of the coverage. 

Therefore, the quantiles or the expectation of the coverage are good ways to measure the performance of prediction 
sets of the form $\Chat_{\ell,k}$. 
In the following, we will use these quantities as criteria
(to be minimized) for choosing among pairs $(\ell,k)$ 
such that $\Chat_{\ell,k}$ is (marginally or training-conditionally) valid. 
More detailed arguments about this strategy 
can be found in Appendix~\ref{app.prelim.perf-compl}. 

\section{Choice of the quantiles for coverage guarantees} \label{sec:FCP-QQ}
% !TEX root = ../main.tex

We now present theoretical results together with one-shot FL algorithms 
which return prediction sets with marginal (Section~\ref{subsec:marg_valid}) 
or training-conditional (Section~\ref{sec.multi-order.algos}) guarantees.

\subsection{Marginal guarantees} \label{subsec:marg_valid}
In this section, we present two strategies for choosing a pair $(\ell, k)$ 
ensuring that $\Chat_{\ell, k}(x)$, 
defined by Eq.~\eqref{CPQQ_set}, 
is a marginally valid prediction set, 
that is, satisfies~\eqref{eq:marg_cov} for a given $\alpha \in (0,1)$. 
Following Section~\ref{sec.prelim.perf}, 
in order to get the best possible prediction set, 
our strategy is to take $(\ell,k)$ such that
the marginal coverage 
\[
\P\bigl( Y \in \Chat_{\ell, k}(X) \bigr) 
= \P\bigl( s_{\fh}(X,Y) \leq S_{(\ell,k)} \bigr) 
\]
is above $1-\alpha$ while being as small as possible.

Our first algorithm is based on the following theorem, which simplifies the formula given in \cite[Theorem 3.2]{humbert2023one}.
\begin{theorem} \label{them:main}
In the setting of Section~\ref{sec:QQ}, with Assumption~\ref{ass:same_n}, 
for any $(\ell,k) \in \intset{n} \times \intset{m}$, 
the set $\Chat_{\ell, k}$ defined by Eq.~\eqref{CPQQ_set} satisfies 
\begin{equation}
\label{eq:main_equa}
\IP\bigl( Y \in \Chat_{\ell, k}(X) \bigr) \geq M_{\ell, k}
\end{equation}
where
\begin{align*}
M_{\ell, k} 
:= \dfrac{ 
	\displaystyle k \binom{m}{k} \sum^{n}_{i_1=\ell} \ldots \sum^{n}_{i_{k-1}=\ell} \sum^{n}_{i_{k+1}=0} \ldots \sum^{n}_{i_{m}=0} 
	\dfrac{\binom{n}{i_1} \cdots \binom{n}{i_{k-1}} \binom{n}{i_{k+1}} \cdots \binom{n}{i_{m}}}{\binom{m n}{i_1 + \cdots + i_{k-1} + \ell + i_{k+1} + \ldots + i_{m}}}
	}{(m n + 1)\mathrm{B}(\ell, n-\ell+1)} 
\end{align*}
and 
\[ 
\mathrm{B} : (a,b) \in (0,+\infty)^2 \mapsto \int_0^1 t^{a-1} (1-t)^{b-1} \mathrm{d}t
\]
denotes the Beta function \citep{temme1996special}. 
Moreover, when the associated scores $(S_{i, j})_{1 \leq j \leq m, 1 \leq i \leq n}$ and $S:= s(X,Y)$ are almost surely distinct, 
Eq.~\eqref{eq:main_equa} is an equality.
\end{theorem}
Theorem~\ref{them:main} is proved in Appendix~\ref{thm:main_proof}.
It shows that we can lower bound the probability of coverage of the quantile-of-quantiles prediction set by a distribution-free quantity $M_{\ell, k}$, 
which depends only on $m$, $n$, $\ell$ and~$k$. 
Furthermore, the lower bound is sharp as it becomes an equality when the scores have a continuous cdf. 
This is for instance the case with the fitted absolute residual if 
the noise distribution given $X$ is almost surely atomless.

Theorem~\ref{them:main} suggests the following algorithm for the selection of $\ell$ and $k$, 
which we call \method~(QQ stands for \emph{Quantile-of-Quantiles} and M for \emph{Marginal}).\footnote{
\method~is a slight modification of the algorithm FedCP-QQ proposed in \citep{humbert2023one}. 
More specifically, Eq.~\eqref{eq:main_equa} is simpler and easier to compute than the corresponding formula for $M_{\ell,k}$ in \citep{humbert2023one}.
}

\begin{algorithm}[\method]\label{alg:FedCPQQ}
	Given $\alpha \in (0, 1)$,
	\begin{align*}
	&\text{compute} \quad (\ell^*, k^*) 
	= {\argmin_{(\ell, k)\in \intset{n} \times \intset{m}}} \left\{M_{\ell, k} : M_{\ell, k} \geq 1-\alpha\right\} \; ,\\
	&\text{and output} \quad \Chat_{\ell^*, k^*}(x) = \left\{y \in \calY : s(x, y) \leq S_{(\ell^*, k^*)} \right\} 
\, .
	\end{align*}
\end{algorithm}
\begin{remark}
\label{rk.algos.argmin-vide}
By convention, when the $\argmin$ defining $(\ell^*, k^*)$ 
in Algorithm~\ref{alg:FedCPQQ} is empty, 
we define $\Chat_{\ell^*, k^*}(x) = \calY$. 
A similar convention is used in 
all of our algorithms.
\end{remark}
The minimization step of Algorithm~\ref{alg:FedCPQQ} 
comes from the fact that all pairs $(\ell,k)$ 
such that $M_{\ell, k} \geq 1-\alpha$ ensure 
the marginal coverage, by Theorem~\ref{them:main}, 
but we need to select one specific pair. 
For the prediction set to be as small as possible, 
following Section~\ref{sec.prelim.perf}, 
we should minimize the coverage. 
In Algorithm~\ref{alg:FedCPQQ}, we minimize $M_{\ell,k}$ 
since it is equal to the expected coverage 
when the scores cdf is continuous. 

By Theorem~\ref{them:main}, the set $\Chat_{\ell^*, k^*}$ is marginally valid. 
It is also nontrivial when $mn$ is large enough, 
according to the following lemma. 
\begin{lemma}
\label{le.alg:FedCPQQ.non-trivial}
The $\argmin$ defining $(\ell^*, k^*)$ 
in Algorithm~\ref{alg:FedCPQQ} is non-empty 
---hence $\Chat_{\ell^*, k^*}$ is nontrivial--- 
if and only if $mn \geq \alpha^{-1} - 1$. 
\end{lemma}
This lemma is proved in Appendix~\ref{le.alg:FedCPQQ.non-trivial_proof}. 
Remarkably, the necessary and sufficient condition in Lemma~\ref{le.alg:FedCPQQ.non-trivial} 
depends on the calibration samples sizes only through 
the total number~$mn$ of calibration data points, 
hence it is the same for the centralized case 
(with one agent holding $mn$ calibration data points) 
and for the one-shot FL case (with $m \geq 2$ agents,  
each having access to $n$ calibration data points). 

A critical limitation of \method~is that computing $(\ell^*,k^*)$, 
and even a single $M_{\ell, k}$, is costly when $m \times n$ is large, 
preventing the approach to scale to a large number of data points or agents. 
For example, based on techniques from \citep{lebrun2013efficient}, 
\citep{humbert2023one} describes an algorithm that can be used to compute a single $M_{\ell, k}$ 
with a worst-case complexity of $\mathcal{O}(m^4 n \log(n))$. 
Although this complexity could be slightly improved using more advanced ideas from \citep{lebrun2013efficient}, 
the overall complexity of Algorithm~\ref{alg:FedCPQQ} 
would still remain too high for very large data sets (e.g., when the number of agents $m$ is large). 
Our second strategy enables us to find a valid pair $(\ell, k)$ much faster. 
It is based on sharp upper and lower bounds over $M_{\ell, k}$ that can be computed efficiently.
\begin{proposition} \label{cor:bound_Ekl}
Let $n, m \geq 1$ be two integers, $(\ell, k) \in \intset{n} \times \intset{m}$, and $M_{\ell, k}$ be defined in Theorem~\ref{them:main}. 
Then,
\begin{align}\label{eq:bound_marg}
F_{U_{(\ell:n)}}^{-1} \left( \frac{k-1/2}{m+1/2} \right) 
< M_{\ell, k} 
< F_{U_{(\ell:n)}}^{-1} \left( \frac{k}{m+1/2} \right) 
\, ,
\end{align}
where $F_{U_{(\ell:n)}}^{-1}$ is the quantile function of the $\Beta(\ell, n-\ell+1)$ distribution \citep{temme1996special}.
\end{proposition}
Proposition~\ref{cor:bound_Ekl} is proved in Appendix~\ref{cor:bound_Ekl_proof}.
By combining Proposition~\ref{cor:bound_Ekl} 
and Theorem~\ref{them:main}, 
if a pair $(\ell, k)$ is such that the left-hand side of 
Eq.~\eqref{eq:bound_marg} is greater or equal to $1-\alpha$, 
then the associated prediction set $\Chat_{\ell,k}$ is marginally valid. 
As a consequence, 
for any $\ell \in \intset{n}$ and $\alpha \in (0, 1)$, if we set
\begin{gather} 
\label{eq:k_low_bound}
k = \tilde{k}_{m, n}(\ell, \alpha) 
:= \left\lceil (m+1/2) \cdot F_{U_{(\ell:n)}}(1-\alpha) + 1/2 \right\rceil 
\, ,
\\
\notag 
\text{then} \qquad 
1-\alpha 
\leq F_{U_{(\ell:n)}}^{-1} \left(\dfrac{\tilde{k}_{m, n}(\ell, \alpha)-1/2}{m+1/2}\right)  
\leq M_{\ell, \tilde{k}_{m, n}(\ell, \alpha)} 
\leq \IP\left(Y \in \Chat_{\ell, \tilde{k}_{m, n}(\ell, \alpha)}(X)\right) 
\, ,
\end{gather}
provided that 
$\tilde{k}_{m, n}(\ell, \alpha) \in \intset{m}$. 
Therefore, choosing the associated pair 
$(\ell, \tilde{k}_{m, n}(\ell, \alpha))$ leads to 
a marginally-valid prediction set. 
Following our idea to minimize $M_{\ell,k}$ 
among marginally-valid pairs $(\ell,k)$ 
in Algorithm~\ref{alg:FedCPQQ}, 
it is here natural to choose $\ell$ by minimizing the upper-bound 
$F_{U_{(\ell:n)}}^{-1} \left(\frac{\tilde{k}_{m, n}(\ell, \alpha)}{m+1/2}\right)$ provided by Proposition~\ref{cor:bound_Ekl}, 
which leads to Algorithm~\ref{algo.FCP-QQ-marg.2} below. 
\begin{algorithm}[\methodlow]
	\label{algo.FCP-QQ-marg.2}
	Given $\alpha \in (0, 1)$,
	\begin{align*}
	&\text{compute} \quad \tilde{\ell} = \argmin_{\ell \in \intset{n} \text{ s.t. } \tilde{k}_{m, n}(\ell, \alpha) \in \intset{m} } 
	\left\{F_{U_{(\ell:n)}}^{-1} \left(\dfrac{\tilde{k}_{m, n}(\ell, \alpha)}{m+1/2}\right)\right\} 
	\quad \text{and} \quad \tilde{k}_{m, n}(\tilde{\ell}, \alpha) 
\, , 
\\
	& \qquad \qquad \text{where $\tilde{k}_{m, n}(\ell, \alpha)$ is defined by Eq.~\eqref{eq:k_low_bound}} 
\, , 
\\
	&\text{and output} \quad \Chat_{\tilde{\ell}, \tilde{k}_{m, n}(\tilde{\ell}, \alpha)}(x) = \left\{y \in \calY : s(x, y) \leq S_{(\tilde{\ell}, \tilde{k}_{m, n}(\tilde{\ell}, \alpha))} \right\} \; .
	\end{align*}
\end{algorithm}
By construction, 
following the arguments detailed after Proposition~\ref{cor:bound_Ekl}, 
the set 
$\Chat_{\tilde{\ell}, \tilde{k}_{m, n}(\tilde{\ell}, \alpha)}$ 
is marginally valid. 
It is also nontrivial when $mn$ is large enough, 
according to the next lemma. 
\begin{lemma} \label{le.algo.FCP-QQ-marg.2.non-trivial}
The $\argmin$ defining $\tilde{\ell}$ 
in Algorithm~\ref{algo.FCP-QQ-marg.2} is non-empty 
---hence $\Chat_{\tilde{\ell}, \tilde{k}_{m, n}(\tilde{\ell}, \alpha)}$ is nontrivial--- 
if and only if 
\begin{equation}
\label{eq.le.algo.FCP-QQ-marg.2.non-trivial.CNS} 
(1-\alpha)^n \leq \frac{m-1/2}{m+1/2} 
\, . 
\end{equation} 
In particular, condition 
\eqref{eq.le.algo.FCP-QQ-marg.2.non-trivial.CNS} 
holds true if 
$n (m-1/2) \geq \alpha^{-1} - 1$. 
\end{lemma}
Lemma~\ref{le.algo.FCP-QQ-marg.2.non-trivial} 
is proved in Appendix \ref{le.algo.FCP-QQ-marg.2.non-trivial_proof}. 
In addition, \methodlow~is computationally efficient. 
Indeed, the functions $F_{U_{(\ell:n)}}$ and 
$F^{-1}_{U_{(\ell:n)}}$ are fast to evaluate 
even when $n$ is large 
(it takes a few milliseconds for $n = 10^6$). 
Furthermore, 
the minimization step in Algorithm~\ref{algo.FCP-QQ-marg.2} 
requires at most $n$ evaluations, 
and the complexity of Algorithm~\ref{algo.FCP-QQ-marg.2} is almost not impacted by the value of~$m$.
For example, on a standard personal machine 
(Intel i5 with 4 CPU at 2.50GHz), with $m=n=10^6$, 
\methodlow{} takes a few seconds to return the valid pair 
$(\tilde{\ell}, \tilde{k}_{m, n}(\tilde{\ell}, \alpha))$ 
using the SciPy implementation \citep{2020SciPy-NMeth} of the cdf $F_{U_{(\ell:n)}}$ and the quantile function $F_{U_{(\ell:n)}}^{-1}$ of the Beta distribution,
whereas \method~can take several hours.

While the two methods presented above ensure that 
the marginal coverage at level $1-\alpha$ 
is satisfied whatever the data distribution, 
one may wonder how much above $1-\alpha$ it can be. 
This question is answered in detail in Section~\ref{sec.cov-upper.marginal}.

\subsection{Training-conditional guarantees} \label{sec.multi-order.algos}
In this section, we present two algorithms for choosing $(\ell, k)$ such that 
the quantile-of-quantiles prediction set $\Chat_{\ell, k}(x)$ 
defined by Eq.~\eqref{CPQQ_set} is a 
distribution-free training-conditionally valid prediction set, 
that is, satisfies Eq.~\eqref{eq:cond_cov}. 
In other words, the goal is to select $\ell$ and $k$ such that the miscoverage random variable
\begin{align} \label{eq:Mkl_cond}
&\alpha_{\ell, k}(\calD) 
= \IP\bigl(Y \notin \Chat_{\ell, k}(X) \,\big\vert\, \calD \bigr)
\end{align}
is smaller than $\alpha \in (0, 1)$ with probability at least $1-\beta  \in (0, 1)$, that is,
\begin{equation} \label{eq.def.cond-cov}
\P \left( 
	\P\bigl( Y \in \Chat_{\ell, k}(X) \,\big\vert\,\calD \bigr) 
	\geq 1-\alpha \right) 
= \P \bigl(1- \alpha_{\ell, k}(\calD) \geq 1-\alpha \bigr) 
\geq 1-\beta
\, .
\end{equation}
Our first algorithm is based on the following theorem. 
\begin{theorem} \label{them:cond_main}
\looseness=-1 In the setting of Section~\ref{sec:QQ}, with Assumption~\ref{ass:same_n}, 
for any $(\ell,k) \in \intset{n} \times \intset{m}$
and any $\beta \in (0, 1)$, the miscoverage random variable $\alpha_{\ell, k}(\calD)$ defined 
by  Eq.~\eqref{eq:Mkl_cond} satisfies
\begin{align}\label{eq:cond_main_equa}
\P \bigl(1-\alpha_{\ell, k}(\calD) \geq F_{U_{(\ell:n, k:m)}}^{-1} ( \beta ) \bigr) 
&\geq 1-\beta \; ,
\end{align}
where $F_{U_{(\ell:n, k:m)}}^{-1} := F_{U_{(\ell:n)}}^{-1} \circ F_{U_{(k:m)}}^{-1}$ 
and for every $1 \leq r \leq N$, 
$F_{U_{(r:N)}}^{-1}$ is the quantile function of the 
Beta$(r,N-r+1)$ distribution.

Moreover, when the associated scores $(S_{i, j})_{1 \leq i \leq n , 1 \leq j \leq m}$ and $S:= s(X,Y)$ are almost surely distinct, 
Eq.~\eqref{eq:cond_main_equa} is an equality and, for any $\beta' \in (0, 1)$ such that $\beta \leq 1-\beta'$, 
\begin{equation}\label{eq:cov_two_side}
\P \bigl(F_{ U_{(\ell:n, k:m)} }^{-1} (\beta) \leq 1 - \alpha_{\ell, k}(\calD) \leq F_{ U_{(\ell:n, k:m)} }^{-1} (1-\beta') \bigr) = 1 - \beta - \beta' 
\, .
\end{equation}
The interval $[F_{ U_{(\ell:n, k:m)} }^{-1}(\beta), F_{ U_{(\ell:n, k:m)} }^{-1}(1-\beta')]$ is thus 
a two-sided fluctuation interval for 
the coverage random variable $1-\alpha_{\ell, k}(\calD)$. 
\end{theorem}
Theorem~\ref{them:cond_main} is proved in Appendix~\ref{app:proof-them:cond_main}. 
Eq.~\eqref{eq:cond_main_equa} is sharp since 
it becomes an equality when the scores are a.s. distinct. 
It can be seen as a generalization of Eq.~\eqref{eq:miscov_vovk} to the case of the quantile-of-quantiles estimator. 
Theorem~\ref{them:cond_main} 
is based on the fact that 
if $U_{(\ell:n, k:m)}$ is a random variable 
with cdf $F_{ U_{(\ell:n, k:m)} }$, then 
$1 - \alpha_{\ell, k}(\calD)$ stochastically dominates 
$U_{(\ell:n, k:m)}$ in general,  
and they have the same distribution 
when the scores are a.s. distinct 
by Lemma~\ref{lemma:centra_unif_kl} 
in Appendix~\ref{app.order-stat.order-of-order} 
(which also implies that 
Theorem~\ref{them:main} holds true with $M_{\ell,k} = \E[ U_{(\ell:n, k:m)} ]$).

\begin{remark}
The distribution associated to $F_{U_{(\ell:n, k:m)}}$ is a particular case of the Beta-Beta distribution \citep{cordeiro2013simple, makgai2019beta}. 
It can also be seen as the cdf of the $k$-th order statistics of 
a sample of $m$ independent Beta$(\ell, n-\ell+1)$ 
random variables \cite{castellares2021note, jones2004families}.
\end{remark}

Theorem~\ref{them:cond_main} suggests the following 
algorithm to select $\ell$ and $k$ 
for training-conditional validity,
which we call \methodcondFL~(QQ stands for \emph{Quantile-of-Quantiles} and C for \emph{Conditional}).
\begin{algorithm}[\methodcondFL]
	\label{algo.FCP-QQ-cond.1}
	Given $\alpha \in (0, 1)$ and $\beta \in (0, 1)$,
	\begin{align} \label{eq.def.algo.FCP-QQ-cond.1}
	&\text{compute} \quad
	(\ell_c^*, k_c^*) 
	= \argmin_{(\ell, k) \in \intset{n} \times \intset{m}} 
	\left\{ F_{U_{(\ell:n, k:m)}}^{-1} (1-\beta) 
		\, : \, 
		F_{U_{(\ell:n, k:m)}}^{-1} (\beta) \geq 1 - \alpha 
	\right\}  
	\\
	&\text{and output} \quad 
	\Chat_{\ell_c^*, k_c^*}(x) 
	= \left\{y \in \calY : s(x, y) \leq S_{(\ell_c^*, k_c^*)} \right\} 
	\; . \notag
	\end{align}
\end{algorithm}
By construction and Theorem~\ref{them:cond_main}, 
the set $\Chat_{\ell_c^*, k_c^*}$ satisfies the training-conditional condition~\eqref{eq.def.cond-cov}. 
It is also nontrivial when $mn$ is large enough, 
according to the following lemma. 
\begin{lemma}
\label{le.algo.FCP-QQ-cond.1.non-trivial}
The $\argmin$ defining $(\ell^*_c,k^*_c)$ 
in Algorithm~\ref{algo.FCP-QQ-cond.1} is non-empty 
---hence $\Chat_{\ell_c^*, k_c^*}$ is nontrivial--- 
if and only if 
\begin{equation}
\label{eq.le.FCP-QQ-cond.1.non-trivial.CNS} 
mn \geq \frac{\log(\beta)}{\log(1-\alpha)}
\, . 
\end{equation} 
\end{lemma}
Lemma~\ref{le.algo.FCP-QQ-cond.1.non-trivial} 
is proved in Appendix~\ref{app.pr.le.algo.FCP-QQ-cond.1.non-trivial}.
Furthermore, as long as $m$ and $n$ are not too large, Algorithm~\ref{algo.FCP-QQ-cond.1} 
is computationally efficient since $F^{-1}_{U_{(\ell:n)}}$ and $F^{-1}_{U_{(k:m)}}$ can be computed quickly 
(see the comments below Algorithm~\ref{algo.FCP-QQ-marg.2} 
in Section~\ref{subsec:marg_valid}). 

Note that any pair $(\ell, k)$ satisfying the condition 
$F_{U_{(\ell:n, k:m)}}^{-1} (\beta) \geq 1 - \alpha$ leads to a conditionally valid prediction set. 
As a selection criterion 
among these pairs $(\ell,k)$, following Section~\ref{sec.prelim.perf}, 
we minimize $F_{U_{(\ell:n, k:m)}}^{-1}(1-\beta)$ which, 
according to Eq.~\eqref{eq:cov_two_side} in Theorem~\ref{them:cond_main}, 
is equal to the $(1-\beta)$-quantile of the distribution
of $1-\alpha_{\ell,k} (\calD)$
when the scores are a.s. distinct.

\methodcondFL~may require to evaluate $F_{U_{(\ell:n, k:m)}}^{-1}$ 
for many pairs $(\ell, k) \in \intset{n} \times \intset{m}$, 
which can be too costly when $n$ and $m$ are large. 
In the sequel, we propose a simpler sufficient condition 
for Eq.~\eqref{eq.def.cond-cov} to be satisfied, leading to a faster algorithm. 
It is based on the following lower bound 
on the quantile function $F_{U_{(\ell:n, k:m)}}^{-1}$.
\begin{proposition} \label{prop:low_bound_Fl}
Let $n, m \geq 1$ be two integers, $k \in \intset{m}$, $\ell \in \intset{n}$, and $F_{U_{(\ell:n, k:m)}}^{-1}$ be defined as in Theorem~\ref{them:cond_main}. For any $\beta \in (0, 1)$, we have 
\begin{equation}
\label{eq.cond-cov-Xlk.lower.minor.1}
F_{U_{(\ell:n, k:m)}}^{-1} (\beta) 
\geq F_{U_{(\ell:n)}}^{-1} \left(\frac{k}{m+1} - \sqrt{\dfrac{\log(1/\beta)}{2(m+2)}} \, \right) 
\, .
\end{equation}
\end{proposition}
Proposition~\ref{prop:low_bound_Fl} is proved in Appendix~\ref{prop:low_bound_Fl_proof}.
From this result, 
we see that for any $\ell \in \intset{n}$, taking (if possible)
\begin{equation}
\label{eq.algo.FCP-QQ-cond.k}
k = \tilde{k}^{\text{cond}}_{m,n} (\ell, \alpha, \beta) 
\egaldef \left\lceil (m+1) \left(F_{ U_{(\ell:n)} } (1-\alpha) + \sqrt{\dfrac{\log(1/\beta)}{2(m+2)}} \, \right) \right \rceil
\end{equation}
implies that $F_{ U_{(\ell:n, \tilde{k}^{\text{cond}}_{m,n} (\ell, \alpha, \beta):m)} }^{-1} (\beta) \geq 1 - \alpha$ 
and by Theorem~\ref{them:cond_main},
$\Chat_{\ell, \tilde{k}^{\text{cond}}_{m,n} (\ell, \alpha, \beta)}$ 
satisfies the training-conditional condition~\eqref{eq.def.cond-cov}. 
Choosing $\ell$ similarly to Algorithm~\ref{algo.FCP-QQ-cond.1} leads to the following algorithm.
\begin{algorithm}[\methodcondFLk]
	\label{algo.FCP-QQ-cond.2}
	Given $\alpha \in (0, 1)$ and $\beta \in (0, 1)$,
	\begin{align*}
	&\text{compute} \qquad 
	\tilde{\ell}^{\text{cond}} 
	= \argmin_{\ell \in \intset{n} \text{ s.t. } \tilde{k}^{\text{cond}}_{m,n} (\ell, \alpha, \beta) \leq m} 
	\left\{F^{-1}_{U_{(\ell:n, \tilde{k}^{\text{cond}}_{m,n} (\ell, \alpha, \beta):m)}} (1-\beta)\right\} 
	\\
	&\text{and} \quad \tilde{k}^{\text{cond}}_{m,n} (\tilde{\ell}^{\text{cond}}, \alpha, \beta) 
	\, ,  
	\qquad \text{where $\tilde{k}^{\text{cond}}_{m,n} (\ell, \alpha, \beta)$ is defined by Eq.~\eqref{eq.algo.FCP-QQ-cond.k}} \; . \\
	& \text{Output} \qquad 
	\Chat_{\tilde{\ell}^{\text{cond}}, \tilde{k}^{\text{cond}}_{m,n} 
	(\tilde{\ell}^{\text{cond}}, \alpha, \beta)}(x) 
= \left\{y \in \calY : s(x, y) 
	\leq S_{(\tilde{\ell}^{\text{cond}}, \tilde{k}^{\text{cond}}_{m,n} 
	(\tilde{\ell}^{\text{cond}}, \alpha, \beta))} \right\} 
\; .
	\end{align*}
\end{algorithm}
By construction, following the arguments detailed 
after Proposition~\ref{prop:low_bound_Fl},  
the set \\$\Chat_{\tilde{\ell}^{\text{cond}}, \tilde{k}^{\text{cond}}_{m,n} 
	(\tilde{\ell}^{\text{cond}}, \alpha, \beta)}$ satisfies the training-conditional condition~\eqref{eq.def.cond-cov}. 
It is also nontrivial when $mn$ is large enough, 
as shown by the following lemma. 
\begin{lemma}
\label{le.algo.FCP-QQ-cond.2.non-trivial}
The $\argmin$ defining $\tilde{\ell}^{\text{cond}}$ 
in Algorithm~\ref{algo.FCP-QQ-cond.2} is non-empty 
---hence the set $\Chat_{\tilde{\ell}^{\text{cond}}, \tilde{k}^{\text{cond}}_{m,n} 
	(\tilde{\ell}^{\text{cond}}, \alpha, \beta)}$ is nontrivial--- 
if and only if 
\begin{equation}
\label{eq.le.FCP-QQ-cond.2.non-trivial.CNS} 
\frac{m}{m+1} - \sqrt{\frac{\log(1/\beta)}{2 (m+2)}} 
\geq (1-\alpha)^n 
\, . 
\end{equation} 
For instance, it holds true when 
\begin{equation}
\label{eq.le.FCP-QQ-cond.2.non-trivial.CS} 
n \geq \frac{\log(1/3)}{\log(1-\alpha)}
\qquad \text{and} \qquad 
m \geq \max \left\{ 2 \,,\, \frac{9}{2}  \log ( 1 / \beta) - 2 \right\} 
\, . 
\end{equation} 
\end{lemma}
Lemma~\ref{le.algo.FCP-QQ-cond.2.non-trivial} 
is proved in Appendix~\ref{app.pr.le.algo.FCP-QQ-cond.2.non-trivial}. 
In addition, Algorithm~\ref{algo.FCP-QQ-cond.2}~is computationally efficient. 
Indeed, 
$F_{U_{(\ell:n, k:m)}}^{-1} := F_{U_{(\ell:n)}}^{-1} \circ F_{U_{(k:m)}}^{-1}$ 
by Theorem~\ref{them:cond_main}, 
so Algorithm~\ref{algo.FCP-QQ-cond.2} 
requires at most $n$ evaluations of 
the functions $F_{U_{(\ell:n)}}$, 
$F^{-1}_{U_{(\ell:n)}}$ and $F^{-1}_{U_{(k:m)}}$, 
which are fast to evaluate even when $n,m$ are large 
as explained below Algorithm~\ref{algo.FCP-QQ-marg.2}. 

\begin{remark} \label{rk.complexity-k-ell}
Note that the computational complexities 
of all proposed algorithms have two parts: 
(i) the choice of a valid pair $(\ell,k)$, 
and (ii) the computation of $S_{(\ell,k)}$ 
given the data and $(\ell,k)$. 
The second part is not an issue since $\fh$ 
is supposed to be given and the score 
functions usually can be computed fastly.  
So, all the complexities discussed throughout 
Section~\ref{sec:FCP-QQ} are about the 
first part, which 
can be reused across multiple data sets, 
score functions, or predictors $\fh$, 
as long as $m$ (the number of agents) and 
$n$ (the size of local data sets) remain fixed. 
\end{remark}

\section{Upper bounds on the coverage} \label{sec.uppbound_cov}
% !TEX root = ../main.tex

In the previous section, we focused on the problem of choosing $\ell$ and $k$ such that 
$1-\alpha_{\ell, k}(\calD)$ is above $1-\alpha$, either in expectation (Section~\ref{subsec:marg_valid}), 
or with high probability (Section~\ref{sec.multi-order.algos}). 
In the present section, we raise the natural and important question 
of \emph{upper}-bounding $1-\alpha_{\ell, k}(\calD)$ in expectation or with high-probability, 
as a function of $k,\ell,m$, and~$n$. 
In particular, when the coverage is guaranteed to be larger than $1 - \alpha$, 
it is interesting to evaluate how far from $1 - \alpha$ it can be, 
in order to quantify the quality of the corresponding algorithms.

\subsection{Marginal coverage upper bounds}\label{sec.cov-upper.marginal}
In this section, we provide an upper bound on the expectation of the coverage 
$\E[ 1 - \alpha_{\ell,k}(\calD)]$ 
obtained by Algorithms~\ref{alg:FedCPQQ}--\ref{algo.FCP-QQ-marg.2}  
of Section~\ref{subsec:marg_valid}. 
Recall that, by Theorem~\ref{them:main}, when the scores are a.s. distinct, 
$\E[1-\alpha_{\ell,k}(\calD) ] 
= \IP(Y \in \Chat_{\ell, k}(X)) 
= M_{\ell, k}$. 
Hence, in that case it is sufficient to upper-bound $M_{\ell, k}$, 
as done by the following result.
\begin{theorem}\label{thm:cov_up_margin}
Given $\alpha \in (0,1)$, let $(\ell^{*},k^{*})$ and $(\tilde{\ell},\tilde{k}_{m, n}(\tilde{\ell}, \alpha))$ 
be the pairs respectively returned by Algorithms~\ref{alg:FedCPQQ} and~\ref{algo.FCP-QQ-marg.2} of Section~\ref{subsec:marg_valid}, 
where $\tilde{k}_{m, n}(\ell, \alpha)$ is defined by Eq.~\eqref{eq:k_low_bound}. 
Then, some $n_0(\alpha), m_0 \geq 1$ exist such that, 
if $n \geq n_0(\alpha)$ and $m \geq m_0$, we have 
\begin{align}
	\label{eq:marg-upper}
1 - \alpha 
\leq M_{\ell^{*}, k^{*}} 
\leq M_{\tilde{\ell}, \tilde{k}_{m, n}(\tilde{\ell}, \alpha)}
\leq 1 - \alpha + \frac{C}{(2m+1)\sqrt{n+2}} 
\, ,
\end{align}
where $C$ denotes a numerical constant 
\textup{(}for instance, the result holds true with 
$m_0=18$ and $C=27$\textup{)}.
\end{theorem}
Theorem~\ref{thm:cov_up_margin} is proved in Appendix~\ref{thm:cov_up_margin_proof}.
Combined with Theorem~\ref{them:main}, it shows that, 
if the scores are a.s. distinct, 
and when $m,n$ are sufficiently large, 
Algorithms~\ref{alg:FedCPQQ} and~\ref{algo.FCP-QQ-marg.2} return prediction sets with marginal coverage 
between $1-\alpha$ and $1-\alpha + \mathcal{O}(n^{-1/2} m^{-1})$. 
In particular, their marginal coverages tend to $1-\alpha$ 
when either $n$ or $m$ tends to infinity, 
a result that was not obtained previously for 
any one-shot FL marginally-valid prediction set~\cite{humbert2023one}. 
 
Note that we do not exactly recover the upper bound 
$1-\alpha + 1/(mn)$ obtained in the centralized case 
when there are $mn$ calibration points 
---Eq.~\eqref{eq.splitCP.marg-cov-upp} 
in Section~\ref{subsec:splitCP}. 
For Algorithm~\ref{alg:FedCPQQ}, 
following \cite{humbert2023one} we conjecture the stronger result 
$M_{\ell^{*}, k^{*}} \leq 1-\alpha + \mathcal{O}(1/mn)$.  
While our experiments in Section \ref{sec:synth_data} corroborate this, proving it would require a different analysis that we leave for future work.
For Algorithm~\ref{algo.FCP-QQ-marg.2}, 
the results of Section~\ref{sec:synth_data.equalnj} 
suggest that Eq.~\eqref{eq:marg-upper} 
is tight (up to the value of~$C$). 

\subsection{Training-conditional coverage upper bounds}\label{sec.cov-upper.conditional}
We now provide high-probability upper bounds
on the coverage $1 - \alpha_{\ell,k}(\calD)$ 
obtained by Algorithms~\ref{algo.FCP-QQ-cond.1}--\ref{algo.FCP-QQ-cond.2}
of Section~\ref{sec.multi-order.algos}. 
\begin{theorem} \label{thm:cov_up_cond}
Let $\alpha \in (0,1)$ and $\beta \in (0,1)$ be given.
Recall that the pairs $(\ell^*_{c}, k^*_{c})$ and 
$(\tilde{\ell}^{\text{cond}}, \tilde{k}^{\text{cond}}_{m,n} (\tilde{\ell}^{\text{cond}}, \alpha, \beta))$ 
respectively denote the pairs $(\ell,k)$ chosen by 
Algorithms~\ref{algo.FCP-QQ-cond.1} and~\ref{algo.FCP-QQ-cond.2}.
In the setting of Section~\ref{sec:QQ}, with Assumption~\ref{ass:same_n}, 
assuming in addition that the scores $S_{i,j},S$ are a.s. distinct, 
some $n'_0(\alpha), m'_0 (\beta) \geq 1$ exist  
such that, for any $n \geq n'_0(\alpha)$ and $m \geq m'_0 (\beta)$, 
\begin{align}
\label{eq.thm:cov_up_cond.algo-cond-1}
\P \left(
1 - \alpha_{\ell^*_{c}, k^*_{c}} (\calD) 
\leq 1-\alpha 
+ \dfrac{\Delta(\beta)}{\sqrt{(m+2)(n+2)}}
\right) 
&\geq 1 - \beta 
\\
\label{eq.thm:cov_up_cond.algo-cond-2}
\text{and} \qquad 
\P \left(
1 - \alpha_{\tilde{\ell}^{\text{cond}}, \tilde{k}^{\text{cond}}_{m,n} (\tilde{\ell}^{\text{cond}}, \alpha, \beta)} (\calD) 
\leq 1-\alpha 
+ \dfrac{\Delta(\beta)}{\sqrt{(m+2)(n+2)}}
\right) 
&\geq 1 - \beta 
\, ,
\end{align}
with $\Delta(\beta) = 12 \max\{2 \sqrt{\log(1/\beta)} , 1 \}$.
\end{theorem}
Theorem~\ref{thm:cov_up_cond} is proved in Appendix~\ref{thm:cov_up_cond_proof}.
Remarkably, the bounds 
are of the form $1 - \alpha + \mathcal{O}(1/\sqrt{mn})$, 
exactly as for split CP in the centralized case 
(see Remark~\ref{rk.splitCP.asympt} in Section~\ref{subsec:splitCP}). 
Therefore, up to constants, using 
one-shot FL training-conditionally valid prediction sets 
---also known as $(\alpha,\beta)$-tolerance regions---
such as Algorithms~\ref{algo.FCP-QQ-cond.1} and~\ref{algo.FCP-QQ-cond.2} incurs no loss
compared to a centralized algorithm such as split CP \cite{vovk2012conditional}. 

A consequence of Theorem~\ref{thm:cov_up_cond} is that, 
under the same assumptions, 
for any $n \geq n'_0(\alpha)$ and $m \geq m'_0 (\beta)$, 
\begin{align}
	\label{eq:bilateral-confidence}
\P \left( 
1-\alpha \leq 
1 - \alpha_{\ell^*_{c}, k^*_{c}} (\calD) 
\leq 1-\alpha 
+ \dfrac{\Delta(\beta)}{\sqrt{(m+2)(n+2)}}
\right) 
\geq 1 - 2 \beta 
\end{align} 
(by Theorem~\ref{them:cond_main} and the union bound), 
and a similar results holds true for the output of 
Algorithm~\ref{algo.FCP-QQ-cond.2}. 
This shows that the deviations of the coverage are (at most) of order $(mn)^{-1/2}$, 
as for split CP in the centralized case \cite[Proposition~2a]{vovk2012conditional}.
The proof of this theorem, 
and by extension Eq.~\eqref{eq:bilateral-confidence}, 
rely on the fact that  
\[ 
\bigl[ F_{ U_{(\ell:n, k:m)} }^{-1}(\beta) 
\, , \, 
F_{ U_{(\ell:n, k:m)} }^{-1}(1-\beta) \bigr] 
\subseteq 
\left[ 1-\alpha 
\, , \, 
1-\alpha + \frac{\Delta(\beta)}{\sqrt{(m+2)(n+2)}}
\right] 
\] 
when the pair $(\ell, k)$ is chosen by our algorithms. 
Then, thanks to Eq.~\eqref{eq:cov_two_side} 
in Theorem~\ref{them:cond_main}, we know that, 
with probability $1-2\beta$, 
the coverage random variable $1-\alpha_{\ell,k}(\calD)$ 
belongs to the first interval, which concludes the proof.
This means that, when $(\ell, k)$ is the pair 
chosen by either Algorithm~\ref{algo.FCP-QQ-cond.1} 
or Algorithm~\ref{algo.FCP-QQ-cond.2}, 
we have that $[F_{ U_{(\ell:n, k:m)} }^{-1}(\beta), F_{ U_{(\ell:n, k:m)} }^{-1}(1-\beta)]$ 
contains $1-\alpha_{\ell, k}(\calD)$ with probability $1-2\beta$, 
and that both $F_{ U_{(\ell:n, k:m)} }^{-1}(\beta)$ 
and $F_{ U_{(\ell:n, k:m)} }^{-1}(1-\beta)$ 
tend to $1-\alpha$ at a rate of at least $(mn)^{-1/2}$. 
According to the numerical experiments of 
Section~\ref{sec:synth_data}, this rate $(mn)^{-1/2}$ 
seems exact (that is, not improvable in worst-case) 
when $(\ell, k)$ is chosen by Algorithm~\ref{algo.FCP-QQ-cond.2}.

\begin{remark} \label{rk.algo-cond.ell-fixe}
The proof of Theorem~\ref{thm:cov_up_cond} also shows that  
\[
\P \left( 
1 - \alpha_{ \ell_n , \tilde{k}^{\text{cond}}_{m,n} ( \ell_n , \alpha, \beta)} (\calD) 
\leq 1-\alpha 
+ \dfrac{\Delta(\beta)}{\sqrt{(m+2)(n+2)}}
\right) 
\geq 1 - \beta 
\]
for $\ell_n = \lceil n(1-\alpha) \rceil$. 
This inequality could be used to build prediction sets 
with a smaller computational complexity 
than Algorithm~\ref{algo.FCP-QQ-cond.2}, 
at the price of being slightly more conservative. 
By Theorem~\ref{thm.asympt-unif-order-stat} 
in Appendix~\ref{app.more-details-order-stat.asymp}, 
the proof can also be generalized to 
any sequence $(\ell_n)_{n \geq 1}$ such that 
$\ell_n \in \intset{n}$ for every $n$ 
and $\ell_n = n (1-\alpha) + \mathrm{o}(\sqrt{n})$ 
as $n \to +\infty$. 
\end{remark}

\section{Generalization to different $n_j$} \label{sec:diff_n}
% !TEX root = ../main.tex

In the previous sections, for simplicity, 
we assume that all agents have the same number $n$ 
of (calibration) data points. 
Let us now consider the general case, that is, 
the setting described in Section~\ref{sec:QQ} 
\emph{without} Assumption~\ref{ass:same_n}. 
Since the $m$ agents have data sets of possibly different sizes $n_1, \ldots, n_m$, 
one needs to adapt the order of the local quantile to each agent. 
Given $(\ell_1, \ldots, \ell_m) \in \intset{n_1} \times \cdots \times \intset{n_m}$, 
the QQ estimator (Definition~\ref{def:QQ}) is now defined by
\begin{equation} \label{eq:CP-QQ_nj}
	S_{(\ell_1, \dots, \ell_m, k)} 
	:= \Qh_{(k:m)}\left(\Qh_{(\ell_1:n_1)}(\mathcal{S}_1), \ldots, \Qh_{(\ell_m:n_m)}(\mathcal{S}_m)\right) 
	\, ,
\end{equation}
where $\mathcal{S}_j = (S_{i, j})_{1 \leq i \leq n_j}$ denotes 
the set of local scores computed by the $j$-th agent. 
In words, each agent $j \in \intset{m}$ now sends to the server 
its $\ell_j$-th smallest score and the server computes 
the $k$-th smallest value of these values. 
The associated prediction set is, for any $x\in \calX$,
\begin{equation} \label{eq:set_nj}
\Chat_{\ell_1, \dots, \ell_m, k}(x) 
:= \left\{ y \in \calY : s(x, y) \leq S_{(\ell_1, \dots, \ell_m, k)} \right\} 
\, ,
\end{equation}
and its miscoverage rate is denoted by 
\begin{equation} \label{eq:Mkl_cond_nj}
\alpha_{\ell_1, \dots, \ell_m, k}(\calD) 
:= \IP\bigl( Y \notin \Chat_{\ell_1, \dots, \ell_m, k}(X) \,\big\vert\, \calD \bigr)
\, .
\end{equation}
When the $n_j$ are different, the difficulty is that the random variables 
$(\Qh_{(\ell_j:n_j)}(\mathcal{S}_j))_{1 \leq j \leq m}$ 
in the right-hand side of Eq.~\eqref{eq:CP-QQ_nj} 
are not identically distributed as they are computed on data sets of different sizes. 
Hence, in this setting, the counterpart of Theorem~\ref{them:main} 
of Section~\ref{subsec:marg_valid} 
(marginal guarantee) is as follows.
\begin{theorem} \label{thm:main_nj}
In the setting of Section~\ref{sec:QQ}, \emph{without} Assumption~\ref{ass:same_n}, 
for any $(\ell_1, \ldots, \ell_m, k) \in \intset{n_1} \times \cdots \times \intset{n_m} \times \llbracket m \rrbracket$, 
the set $\Chat_{\ell_1, \dots, \ell_m, k}$ defined by Eq.~\eqref{eq:set_nj} satisfies 
\begin{equation} \label{eq:main_equa_nj}
\IP\bigl( Y \in \Chat_{\ell_1, \dots, \ell_m, k}(X) \bigr) 
\geq M_{\ell_1, \ldots, \ell_m, k}
\end{equation}
with 
\begin{align*}
M_{\ell_1, \ldots, \ell_m, k} 
&= 1 - \dfrac{1}{1+\sum_{j=1}^{m}n_j} 
\sum^{m}_{j=k} \sum_{a \in \mathcal{P}_j} \sum^{n_{a_1}}_{i_1=\ell_{a_1}} \cdots \sum^{n_{a_{j}}}_{i_{j}=\ell_{a_{j}}} \sum^{\ell_{a_{j+1}} - 1}_{i_{j+1}=0} \cdots \sum^{\ell_{a_m}-1}_{i_m=0} 
\dfrac{\binom{n_{a_1}}{i_1} \cdots \binom{n_{a_m}}{i_m}}{\binom{n_{1} + \cdots + n_{m}}{i_1 + \cdots + i_m}} 
\, ,
\end{align*}
where 
$\mathcal{P}_{j}$ denotes the set of permutations 
$a=(a_1, \ldots, a_m)$ of $\{1, \ldots, m\}$ 
such that $a_1 < a_2 < \ldots < a_{j}$ 
and $a_{j + 1 } < a_{j + 2 } < \ldots < a_m$. 
Moreover, when the associated scores $(S_{i, j})_{1 \leq j \leq m, 1 \leq i \leq n_j}$ and $S:= s(X,Y)$ are almost surely distinct, 
Eq.~\eqref{eq:main_equa_nj} is an equality.
\end{theorem}
Theorem~\ref{thm:main_nj} is proved in Appendix~\ref{thm:main_nj_proof}.
Furthermore, the miscoverage rate defined 
by Eq.~\eqref{eq:Mkl_cond_nj} can be controlled 
with high probability thanks to the following result, 
which generalizes Theorem~\ref{them:cond_main} 
in Section~\ref{sec.multi-order.algos}.
\begin{theorem} \label{them:cond_main_nj}
In the setting of Section~\ref{sec:QQ}, \emph{without} Assumption~\ref{ass:same_n}, 
for any $(\ell_1, \ldots, \ell_m, k) \in \intset{n_1} \times \cdots \times \intset{n_m} \times \llbracket m \rrbracket$, 
and any $\alpha \in (0, 1)$, 
the miscoverage random variable 
$\alpha_{\ell_1, \dots, \ell_m, k}(\calD)$ defined 
by Eq.~\eqref{eq:Mkl_cond_nj} satisfies 
\begin{align}\label{eq:cond_main_equa_nj}
\P \bigl(1-\alpha_{\ell_1, \dots, \ell_m, k}(\calD) \geq 1-\alpha\bigr) 
&\geq \mathrm{PB}\left(k-1 ; 
	\bigl( F_{U_{(\ell_j:n_j)}}(1-\alpha) \bigr)_{1 \leq j \leq m} 
	\right) 
\, ,
\end{align}
where $F_{U_{(\ell_j:n_j)}}$ denotes the cdf of 
the $\Beta(\ell_j, n_j-\ell_j+1)$ distribution 
and for every $t \in \R$ and $u \in [0,1]^m$, 
$\mathrm{PB}(\cdot;u)$ denotes the cdf of a Poisson-Binomial (PB) 
random variable with parameter $u$\footnote{Recall that the PB distribution 
with parameter $u \in [0,1]^m$ is the distribution 
of $Z_1 + \cdots + Z_m$ where the random variables 
$Z_j$ are independent and respectively follow 
the Bernoulli$(u_j)$ distribution. 
It therefore generalizes the Binomial$(m,p)$ 
distribution, which corresponds to the case 
where $u_1 = \cdots = u_m = p$. 
More details can be found in \citep{tang2023poisson}. }.
Moreover, when the associated scores 
$(S_{i, j})_{1 \leq j \leq m, 1 \leq i \leq n_j}$ 
and $S:= s(X,Y)$ are almost surely distinct, 
Eq.~\eqref{eq:cond_main_equa_nj} is an equality.
\end{theorem}
Theorem~\ref{them:cond_main_nj} is proved 
in Appendix~\ref{them:cond_main_nj_proof}.
Theorem~\ref{thm:main_nj} and Theorem~\ref{them:cond_main_nj} 
imply that 
(i) if $M_{\ell_1, \ldots, \ell_m, k} \geq 1-\alpha$, 
then  $\Chat_{\ell_1, \dots, \ell_m, k}$ 
is a marginally valid prediction set, and 
(ii) given $\beta \in (0, 1)$ if 
$\mathrm{PB}\left(k-1 ; (F_{U_{(\ell_j:n_j)}}(1-\alpha))_{1 \leq j \leq m}\right) \geq 1-\beta$, 
then $\Chat_{\ell_1, \dots, \ell_m, k}$ is an $(\alpha, \beta)$-tolerance region, 
that is, a training-conditionally valid prediction set. 
It remains to select $( \ell_1, \dots, \ell_m, k )$, 
among those satisfying the desired condition. 
Applying directly the strategy of 
Algorithms \ref{alg:FedCPQQ} or~\ref{algo.FCP-QQ-cond.1}, 
although theoretically possible, 
can be computationally untractable since the $\ell_j$ 
can be different, hence a parameter set 
of cardinality $m \cdot n^m$. 
Therefore, we propose to fix 
$\ell_j=\lceil (1-\alpha)(n_j+1)\rceil$ 
for all $j \in \intset{m}$, 
similarly to the classical (centralized) split CP methodology. 
Then, we only need to select $k \in \intset{m}$, 
thereby reducing significantly 
the computational complexity.\footnote{This strategy 
can also be used when $n_j = n$ (Section~\ref{sec:FCP-QQ}) 
at the cost of being less accurate than searching for 
all values of $\ell$ and $k$. 
See also Section~\ref{sec:synth_data.diffnj} and 
Remark~\ref{rk.algo-cond.ell-fixe} in Appendix~\ref{thm:cov_up_cond_proof}.} 
Following the principles described in  Section~\ref{sec.prelim.perf} 
for choosing $k$, 
Theorems~\ref{thm:main_nj}--\ref{them:cond_main_nj} 
lead to the two following algorithms. 
\begin{algorithm}[\method-$n_j$~\citep{humbert2023one}]
	\label{algo.FCP-QQ.1_nj}
	Given $\alpha \in (0, 1)$,
	\begin{align*}
	&\text{compute} \quad 
	\ell^*_j = \lceil (1-\alpha)(n_j+1)\rceil \quad
	\text{ for every } j \in \intset{m} 
	\\
	&\text{and} \quad 	
	k^*(\ell^*_{1 \ldots m}) = {\arg\min_{k \in \intset{m}}} \left\{M_{\ell^*_1, \ldots, \ell^*_m, k} : M_{\ell^*_1, \ldots, \ell^*_m, k} \geq 1-\alpha\right\} 
	\, .
	\\
	&\text{Output} \quad 
	\Chat_{\ell^*_1, \ldots, \ell^*_m, k^*(\ell^*_{1 \ldots m})}(x) 
	= \left\{ y \in \calY : s(x, y) \leq S_{(\ell^*_1, \ldots, \ell^*_m, k^*(\ell^*_{1 \ldots m}))} \right\} 
	\, .
	\end{align*}
\end{algorithm}
\begin{algorithm}[\methodcondFL-$n_j$]
	\label{algo.FCP-QQ-cond.1_nj}
	Given $\alpha \in (0, 1)$ and $\beta \in (0, 1)$,
	\begin{align*}
	&\text{compute} \quad 
	\ell^*_j = \lceil (1-\alpha)(n_j+1)\rceil \quad
	\text{ for every } j \in \intset{m} 
	\\
	&\text{and} \quad 
	k_c^*(\ell^*_{1 \ldots m}) 
	= \argmin_{k \in \mathcal{K}} \left\{ 
		\mathrm{PB}\left(k-1 ; (F_{U_{(\ell^*_j:n_j)}}(1-\alpha))_{1 \leq j \leq m}\right) 
		\right\} 
	\, , 
	\\
	&\text{where} \quad
	\mathcal{K} := \left\{ 
		k \in \intset{m} \,:\, 
		\mathrm{PB}\left(k-1 ; (F_{U_{(\ell^*_j:n_j)}}(1-\alpha))_{1 \leq j \leq m}\right) \geq 1-\beta 
	\right\}
	\, . 
	\\
	&\text{Output} \qquad 
	\Chat_{\ell^*_1, \ldots, \ell^*_m, k_c^*(\ell^*_{1 \ldots m})}(x) 
	= \left\{ y \in \calY : s(x, y) \leq S_{(\ell^*_1, \ldots, \ell^*_m, k_c^*(\ell^*_{1 \ldots m}))} \right\} 
	\, . \notag
	\end{align*}
\end{algorithm}
For reasons detailed below Theorem \ref{them:cond_main_nj}, 
Algorithm~\ref{algo.FCP-QQ.1_nj} yields 
a distribution-free marginally valid prediction set, 
and Algorithm~\ref{algo.FCP-QQ-cond.1_nj} yields 
a distribution-free training-conditionally valid prediction set. 
The exact computation of $M_{\ell_1, \ldots, \ell_m, k}$ 
can be done in the same way as that of $M_{\ell, k}$ 
(see details in \cite[Appendix A.2.]{humbert2023one}). 
Furthermore, because $M_{\ell_1, \ldots, \ell_m, k} = \IE[Z]$ 
where $Z$ has cdf 
$\mathrm{PB}(k-1 ; \bigl( F_{U_{(\ell_j:n_j)}}(1-\alpha) \bigr)_{1 \leq j \leq m})$, 
another possibility is to numerically integrate this 
Poisson-Binomial cdf to approximate the expectation.
There exist many algorithms to compute the Poisson-Binomial cdf of 
Eq.~\eqref{eq:cond_main_equa_nj} efficiently, see for instance \citep{hong2013computing}. 
Moreover, precise approximations of it can be obtained when $m$ is large. 
We refer to the recent review of \citep{tang2023poisson} on the Poisson-Binomial distribution for more details. \\
\begin{remark}[Other variants of
Algorithms~\ref{algo.FCP-QQ.1_nj}--\ref{algo.FCP-QQ-cond.1_nj}]
\label{rk.diff_nj.variants}
Following the comments below Theorem~\ref{them:cond_main_nj}, 
Algorithm~\ref{algo.FCP-QQ.1_nj} (resp. Algorithm~\ref{algo.FCP-QQ-cond.1_nj}) 
remains marginally (resp. training-conditionally) valid 
with any other choice for the $\ell^*_j$, 
leading to many possible variants. 
Let us detail three of them. 
First, in Algorithm~\ref{algo.FCP-QQ-cond.1_nj}, 
one can replace $\ell^*_j$ by the value of $r_c$ defined 
by Eq.~\eqref{def.Cond-Central.rc} when 
$\lvert \calD^{cal} \rvert = n_j$, 
which corresponds to training-conditionally valid splitCP 
applied to the data of agent~$j$ alone. 
Second, when the $\argmin$ defining $k^*$ or $k^*_c$ is empty, 
it is possible to increase the $\ell^*_j$ 
(for instance, by adding $1$ to each of them) 
until the $\argmin$ becomes nonempty 
(which happens for large enough $\ell^*_j$ 
if and only if Eq.~\eqref{eq.le.FCP-QQ-cond.1.non-trivial.CNS} 
in Lemma~\ref{le.algo.FCP-QQ-cond.1.non-trivial}) holds true. 
Third, when $m$ is small, one can consider small sets 
$\mathcal{L}^*_j$ of candidate values for $\ell^*_j$, 
and replace the $\argmin$ in Algorithms~\ref{algo.FCP-QQ.1_nj}--\ref{algo.FCP-QQ-cond.1_nj} 
by an $\argmin$ over $(k , \ell^*_1, \ldots , \ell^*_m) 
\in \mathcal{K} \times \mathcal{L}^*_1 \times \cdots 
\times \mathcal{L}^*_m$, 
which yields a tractable algorithm provided this set remains small. 
While our Theorems \ref{thm:main_nj}--\ref{them:cond_main_nj} 
prove coverage lower bounds for each of these variants, 
we leave their detailed study for future work. \\
\end{remark}

\begin{remark}
If we set $n_1 = \cdots = n_m = n$ 
and $\ell_1 = \cdots = \ell_m = \ell \in \intset{n}$ 
in Eq.~\eqref{eq:main_equa_nj}, 
respectively in Eq.~\eqref{eq:cond_main_equa_nj}, 
we exactly recover the results obtained in Theorem~\ref{them:main} 
of Section~\ref{subsec:marg_valid} (marginal guarantee), 
resp. in Theorem~\ref{them:cond_main} of 
Section~\ref{sec.multi-order.algos} (conditional guarantee). 
Indeed, the summation over permutations $a \in \mathcal{P}_j$ 
in Theorem~\ref{thm:main_nj} then becomes 
the $\binom{m}{k}$ of Theorem~\ref{them:main}, after rearranging the terms. 
For Theorem~\ref{them:cond_main_nj}, the parameters 
of the Poisson-Binomal $(F_{U_{(\ell_j:n_j)}}(1-\alpha))_{1 \leq j \leq m}$ 
are then equal, hence we get the 
Binomial$(m; F_{U_{(\ell:n)}}(1-\alpha))$ distribution. 
This leads to Theorem~\ref{them:cond_main} 
with $\beta = F_{U_{(\ell:n, k:m)}}(1-\alpha)) 
= F_{U_{(k:m)}} \circ F_{U_{(\ell:n)}}(1-\alpha))$ 
since $F_{\mathrm{Binomial}(m,p)} (k-1) 
= 1 - F_{U_{(k:m)}}(p)$, 
here used with $p = F_{U_{(\ell:n)}}(1-\alpha)$. 
\end{remark}

\section{Numerical experiments} \label{sec:xps}
% !TEX root = ../main.tex

In this section, we first study the behavior of our algorithms 
in a generic setting and then evaluate them on real datasets. 
The code of our experiments is publicly available.\footnote{\url{https://github.com/pierreHmbt/One-shot-FCP}}
In the experiments, we compare the performance 
of the following conformal-based prediction-set algorithms: 
\begin{itemize}
\item Algorithms~\ref{alg:FedCPQQ}--\ref{algo.FCP-QQ-cond.2},  
respectively called \method, \methodlow, \methodcondFL{} and \methodcondFLk{} 
(when the $n_j$ are equal);

\item Algorithms~\ref{algo.FCP-QQ.1_nj}--\ref{algo.FCP-QQ-cond.1_nj}, 
respectively called \method-$n_j$ and \methodcondFL-$n_j$ 
(when the $n_j$ may be different);

\item \methodCentral: 
split conformal prediction $\Chat_{r}$ 
(as defined in Section~\ref{subsec:splitCP}) 
on the full (centralized) calibration data set $\calD^{cal}$, 
with 
$r= \lceil (1-\alpha) (\lvert \calD^{cal} \rvert + 1)\rceil$ 
so that it is marginally valid;
in other words, \methodCentral~is the centralized version of 
\method{} and they coincide when $m=1$;

\item \methodCentralC: 
split conformal prediction $\Chat_{r}$ 
on the full (centralized) calibration data set $\calD^{cal}$ 
with 
\begin{equation} 
\label{def.Cond-Central.rc}
r = r_c \egaldef \min
\bigl\{ 
\widetilde{r} \in \intset{ \lvert \calD^{cal} \rvert}
: F^{-1}_{U_{(\widetilde{r}: \lvert \calD^{cal} \rvert)}} ( \beta ) \geq 1-\alpha
\bigr\}
\, , 
\end{equation}
so that it is conditionally valid, 
by Eq.~\eqref{eq:miscov_vovk} in Section~\ref{subsec:splitCP}; 
in other words, \methodCentralC~is the centralized version of 
\methodcondFL{} and they coincide when $m=1$; 

\item \methodAvg{}: federated approach proposed by~\citep{lu2021distribution}, 
which averages the $m$ quantiles of order $\lceil (n_j+1)(1-\alpha) \rceil$ sent by the agents, 
hence the prediction set
\[
\Chat_{\methodAvg} (x) := 
\bigg\{y \in \calY : s(x, y) \leq 
\frac{1}{m} \sum_{j=1}^m \Qh_{(\lceil (n_j+1)(1-\alpha) \rceil)}(\mathcal{S}_j)
\bigg\}
\, . 
\]
\end{itemize}

All our experiments are done with 
$\alpha = 0.1$ 
and $\beta = 0.2$. 

As explained in Section~\ref{sec.prelim.perf}, 
we compare the performance of these methods either 
through the length of the prediction sets they produce 
(in Section~\ref{sec:xps:real}, where these are intervals) 
or their coverage, 
which should both be as small as possible while keeping the associated set valid. 
Note that it is fair to compare split CP (\methodCentral~and \methodCentralC), 
\methodAvg{} and our quantile-of-quantiles prediction sets 
(Algorithms~\ref{alg:FedCPQQ}--\ref{algo.FCP-QQ-cond.1_nj})
through their coverages since they all are of the form 
$
\{y \in \calY \,:\, s(x,y) \leq \overline{S}\}
$
for some random variable $\overline{S}$.

\subsection{Generic comparison.}
\label{sec:synth_data}
In this section, we compare numerically the performance (measured by their coverages) 
of our algorithms and their centralized counterparts (\methodCentral{} and \methodCentralC). 
This comparison is made under the mild assumption that $F_S$ is continuous, 
which implies that the cdf of the training-conditional coverage distribution 
of each algorithm is known and, more importantly, universal --
see Eq.~\eqref{eq:miscov_vovk} and the discussion below for centralized algorithms, 
Theorem~\ref{them:cond_main} for federated algorithms with equal $n_j$ 
and Theorem~\ref{them:cond_main_nj} for federated algorithms with general~$n_j$.
We also have closed-form formulas for the expected 
coverage when $F_S$ is continuous, 
by Eq.~\eqref{eq.splitCP.marg-cov-upp} 
for centralized algorithms, 
and by Theorems \ref{them:main} and~\ref{thm:main_nj} 
for our federated algorithms. 
Under this \emph{generic setting}, we thus consider simultaneously 
all learning problems, data distributions, predictors $\fh$ 
and score functions $s$ such that the scores cdf $F_S$ 
is continuous.

Throughout this subsection, for each algorithm considered, 
denoting by $1-\alpha(\calD)$ its (random) coverage, 
we are interested in $F_{1-\alpha(\calD)}$ the coverage cdf, 
$\Delta\E \egaldef \E[ 1-\alpha(\calD) ] - (1-\alpha)$ 
the difference between the expected coverage 
and the nominal coverage $(1-\alpha)$, 
and, for several values of $\zeta \in (0,1)$, 
$\Delta q_{\zeta} \egaldef F_{1-\alpha(\calD)}^{-1} (\zeta) - (1-\alpha)$ 
the difference between the $\zeta$-quantile of the coverage 
and the nominal coverage. 

\subsubsection{Equal $n_j$}
\label{sec:synth_data.equalnj}
We first compare \method, \methodlow, \methodcondFL, \methodcondFLk, 
\methodCentral{} and \methodCentralC{} under Assumption~\ref{ass:same_n}, 
that is, when each agent $j \in \intset{m}$ 
has the same number $n_j = n$ of calibration data points. 
In Section~\ref{sec.uppbound_cov}, we prove that 
$\Delta\E = \mathcal{O} ( m^{-1} n^{-1/2} )$ 
for our marginally valid federated prediction sets 
(Theorem~\ref{thm:cov_up_margin}) 
and that $\Delta q_{1-\beta} = \mathcal{O}  ( m^{-1/2} n^{-1/2} )$ 
for our training-conditionally valid federated prediction sets 
(Theorem~\ref{thm:cov_up_cond}). 
But these only are \emph{upper bounds}, 
and we would like to know whether they match 
the true order of magnitude of their coverages 
---at least in worst case, since one can for instance 
get $\Delta\E = 0$ for \method{} 
by choosing $\alpha = 1 - M_{\ell,k}$ for some 
$(\ell,k) \in \intset{n} \times \intset{m}$, 
or for \methodCentral{} by choosing 
$\alpha \in \{ \frac{r}{n_c+1} \,:\, r \in \intset{n_c+1} \}$. 
Furthermore, knowing the true (worst-case) order 
of magnitude of the coverage of each algorithm 
is crucial to determine when one of the federated 
prediction sets we propose should be preferred to another, 
and for knowing precisely what may be lost by 
considering the one-shot FL setting instead of 
the centralized setting.

\paragraph*{Methods.} 
In order to answer these questions, 
for each algorithm considered,
we evaluate below the rates of convergence 
towards zero of $\Delta \E$, $\Delta q_{\beta}$ 
and $\Delta q_{1-\beta}$
as $m$ and $n$ increase.
The choice of the quantile orders $\beta$ and $(1-\beta)$ 
comes from the fact that 
(i) the conditional validity is equivalent to 
$\Delta q_{\beta} \geq 0$, 
so for prediction sets satisfying this condition, 
knowing how tight is this inequality in worst case 
is meaningful, 
and (ii) since $\beta = 0.2$, 
the $\beta$ and $(1-\beta)$-quantiles of 
the coverage are the bounds of an interval containing 
``typical values'' of the coverage. 
In particular, by definition, the coverage $1-\alpha(\calD)$ 
is below $(1-\alpha) + \Delta q_{1-\beta}$ 
with probability $1 - \beta = 80\%$. 
 
We consider values of $(m,n)$ 
in the grid $\{ \lfloor 10^{i/3} \rfloor \,:\, i \in \intset{9} \}^2$. 
For each pair $(m, n)$ and each algorithm, 
we compute $\Delta \E$, $\Delta q_{\beta}$ and $\Delta q_{1-\beta}$. 
Then, for each algorithm, 
using empirical risk minimization with the Huber loss \citep{huber1992robust}, 
we robustly fit the log-linear regression model 
$\log{y} = \log(c) - \gamma \log(m) - \delta \log(n) + \varepsilon$, 
where $\varepsilon$ is some residual term, 
$c>0$ and $\gamma, \delta \in \R$ are the model parameters, 
and $y$ is either $\Delta \E$, $\Delta q_{\beta}$ or $\Delta q_{1-\beta}$.  
Note that by construction, 
$\Delta \E \geq 0$ for marginally-valid algorithms, 
and $\Delta q_{1-\beta} \geq \Delta q_{\beta} \geq 0$ 
for conditionally-valid algorithms. 
The estimated values of $c,\gamma,\delta$ 
for each algorithm and each quantity of interest 
are reported in Table~\ref{tab:res_coeff_lin_all}. 
Plots showing the values of $\Delta \E$, $\Delta q_{\beta}$, $\Delta q_{1-\beta}$ 
(together with their log-linear model approximation) 
are provided in Figures \ref{fig.DeltaE-Algos-margin} 
and~\ref{fig.qbeta-Algos-cond},
and in Appendix~\ref{app:rate.bound}.

In order to compare the performance of the algorithms more closely, 
we also plot in Figures \ref{fig.cdf-Algos-margin} and~\ref{fig.cdf-Algos-cond} 
the full cdfs of the coverage distribution 
for two specific pairs $(m,n)$, $(200,20)$ and $(20,200)$, 
and report the corresponding expectations, standard-deviations 
(computed by numerically integrating the cdfs), 
and $\zeta$-quantiles for $\zeta \in \{ \beta , 1- \beta\}$ 
in Table~\ref{tab:es_std}.  

\begin{table}[t]
	\footnotesize
	{\centering
	\ra{1.}
	\begin{adjustbox}{max width=\textwidth}
		\begin{tabular}{@{}llll|lll|lll@{}}
			\toprule
& 
\multicolumn{3}{c}{
$\boldsymbol{\Delta\E \approx c_1 m^{-\gamma_1} n^{-\delta_1}}$
}
& 
\multicolumn{3}{c}{
$\boldsymbol{\Delta q_{\beta} \approx c_2 m^{-\gamma_2} n^{-\delta_2}}$
}
& 
\multicolumn{3}{c}{
$\boldsymbol{\Delta q_{1-\beta} \approx c_3 m^{-\gamma_3} n^{-\delta_3}}$
}
\\ 
\textbf{Method} & $\boldsymbol{c_1}$ & $\boldsymbol{\gamma_1}$ & $\boldsymbol{\delta_1}$ & $\boldsymbol{c_2}$ & $\boldsymbol{\gamma_2}$ & $\boldsymbol{\delta_2}$ & $\boldsymbol{c_3}$ & $\boldsymbol{\gamma_3}$ & $\boldsymbol{\delta_3}$ \\
\midrule
\methodCentral & $0.502$ & $0.999$ &  $0.999$ & $0.235$ & $0.498$ & $0.498$ & $0.269$ & $0.501$ & $0.501$ 
\\
\method & $0.924$ & $1.026$ & $1.019$ & $0.463$ & $0.492$ & $0.487$ & $0.522$ & $0.501$ & $0.495$
\\
\methodlow & $0.336$ & $1.001$ & $0.582$ & $0.227$ & $0.483$ & $0.499$ & $0.400$ & $0.512$  & $0.504$
\\
\midrule
\methodCentralC & $0.257$ & $0.501$ & $0.501$ & $0.590$ & $1.014$ & $1.014$ & $0.504$ & $0.499$ & $0.499$ 
\\
\methodcondFL & $0.341$ & $0.507$ & $0.502$ & $0.508$ & $0.910$ & $0.847$ & $0.637$ & $0.502$  & $0.497$
\\
\methodcondFLk & $0.676$ & $0.498$ & $0.501$ & $0.400$ & $0.503$ & $0.509$ & $0.958$ & $0.496$  & $0.498$
\\
\bottomrule
\end{tabular}
\end{adjustbox}
}
\caption{%
Estimated parameters of the log-linear model 
$\log{y} = \log(c_i) - \gamma_i \log(m) - \delta_i \log(n)$ 
where $y$ is either $\Delta\E$, $\Delta q_{\beta}$ 
or $\Delta q_{1-\beta}$, 
for the six algorithms compared in Section~\ref{sec:synth_data.equalnj}\textup{;} 
see text for details.
\label{tab:res_coeff_lin_all}
}
%\medskip
\end{table}
\paragraph*{Marginally-valid algorithms.}
\begin{figure}[t]
	\centering
	\includegraphics[width=0.47\linewidth]{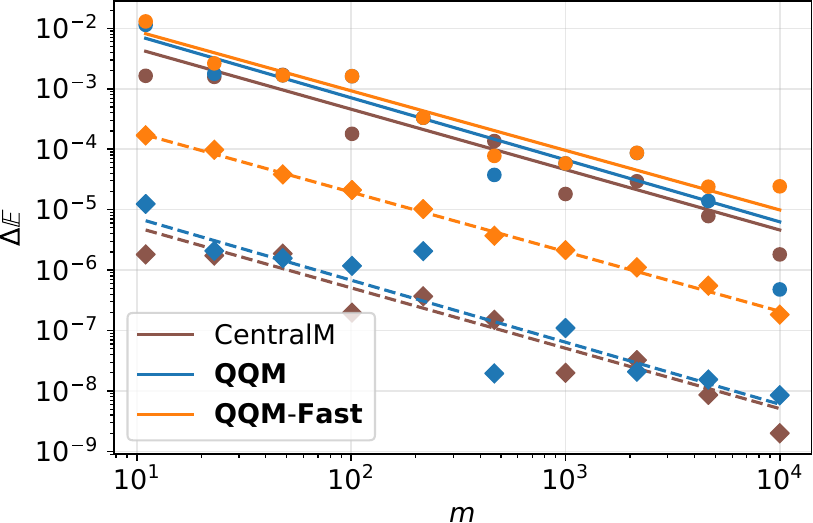}\hspace{2em}
	\includegraphics[width=0.47\linewidth]{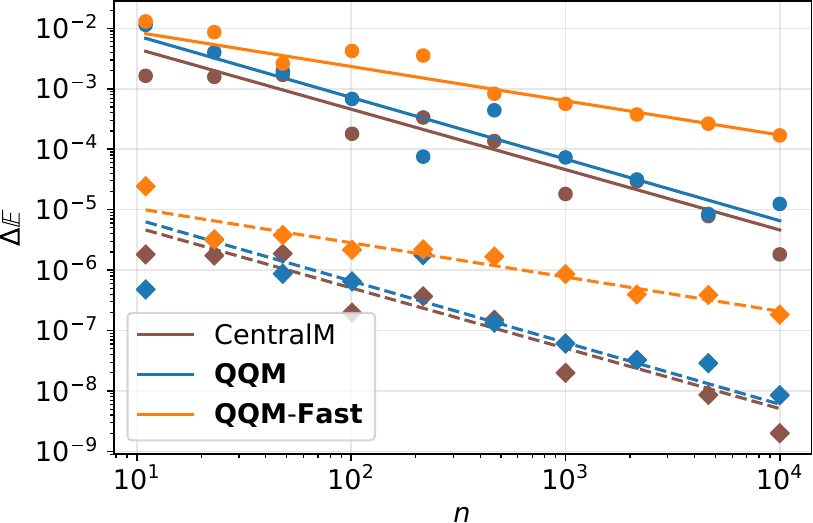}
\caption{%
Marginally-valid algorithms\textup{:}  
log-log plot of $\Delta\E$ as a function of 
$m$ \textup{(}left\textup{)} or $n$ \textup{(}right\textup{)}. 
Lines show the approximation 
$\log{\Delta \E} \approx \log(c_1) - \gamma_1 \log(m) - \delta_1 \log(n)$ 
with $c_1,\gamma_1,\delta_1$ given by Table~\ref{tab:res_coeff_lin_all}. 
Plain lines and dots correspond to 
$n=10$ \textup{(}left\textup{)} or 
$m=10$ \textup{(}right\textup{)}. 
Dashed lines and diamonds correspond to 
$n=10^4$ \textup{(}left\textup{)} or 
$m=10^4$ \textup{(}right\textup{)}. 
\label{fig.DeltaE-Algos-margin}
}
\end{figure}
\begin{figure}[t]
	\centering
	\includegraphics[width=0.49\linewidth]{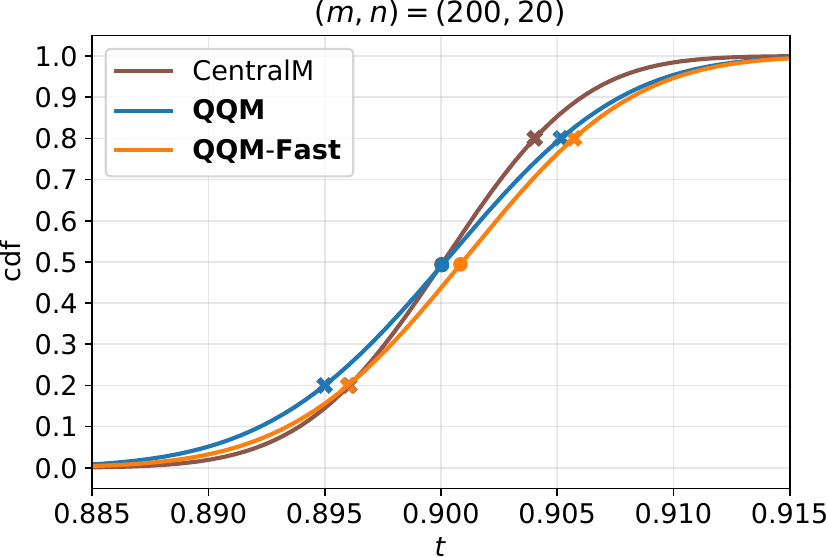}
	\includegraphics[width=0.49\linewidth]{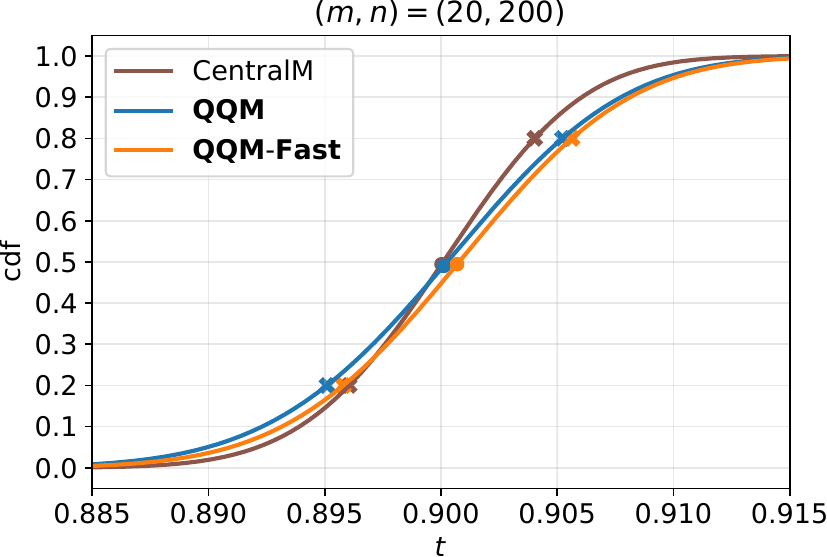}
\caption{%
Marginally-valid algorithms\textup{:}  
cdf of the coverage $1-\alpha(\calD)$.  
The mean of each distribution is shown by a dot 
and the quantiles of order $\beta$ and $1-\beta$ by crosses.
Left\textup{:} $(m, n) = (200,20)$. 
Right\textup{:} $(m, n) = (20, 200)$.  
See Table~\ref{tab:es_std} for additional information. 
\label{fig.cdf-Algos-margin}
}
\end{figure}
For \method{} and \methodCentral, 
according to Table~\ref{tab:res_coeff_lin_all}, 
$\Delta \E$ decrease to zero 
at the same rate, close to $m^{-1}n^{-1}$, 
with only a multiplicative constant $\approx 2$ 
in favour of \methodCentral. 
Overall, the difference between expected coverages 
of \method{} and \methodCentral{} 
is small, as shown by Figure~\ref{fig.DeltaE-Algos-margin}. 
In comparison, \methodlow{}  
shows a worse performance with $\Delta \E$ 
decreasing to zero at a rate close to $m^{-1}n^{-1/2}$ 
(Table~\ref{tab:res_coeff_lin_all}). 
This leads to a significant gap between the 
performances of \methodlow{} 
and the two other ones 
as shown by Figure~\ref{fig.DeltaE-Algos-margin}. 
For $(m,n) \in \{ (200,20) , (20,200)\}$, 
the coverage cdfs of the three algorithms have similar shapes 
(Figure~\ref{fig.cdf-Algos-margin}) 
apart from a shift corresponding to already mentioned differences between 
expectations, 
and a slightly larger standard-deviation for federated algorithms 
compared to \methodCentral{} (see also Table~\ref{tab:es_std}). 
The $(1-\beta)$-quantile of the coverage is also slightly 
larger for the federated algorithms compared to \methodCentral, 
even if they all are of the same order of magnitude 
$(1-\alpha) + \mathcal{O}(m^{-1/2} n^{-1/2})$ 
at first order (see Tables~\ref{tab:res_coeff_lin_all} and~\ref{tab:es_std}). 

Note also that Table~\ref{tab:res_coeff_lin_all} 
suggests that our theoretical expected coverage upper bound 
(Theorem~\ref{thm:cov_up_margin}) 
$(1-\alpha) + \mathcal{O}(m^{-1} n^{-1/2})$ 
provides the correct worst-case order of magnitude 
for \methodlow{} 
(even if the estimated exponent for $n$ is $\delta_1 = 0.58$ instead of exactly $1/2$), 
and supports our conjecture that it could be improved to 
$(1-\alpha) + \mathcal{O}(m^{-1} n^{-1})$ 
for \method{}. 

Overall, these numerical results show that among 
one-shot FL prediction sets, 
\method{} should be preferred to 
\methodlow{} as long as its 
computational complexity remains tractable 
(keeping in mind Remark~\ref{rk.complexity-k-ell} 
about the computational complexity of our federated algorithms). 
Then, using \method{} 
yields a mild loss compared to \methodCentral, 
with a one-shot FL algorithm.
Note however that the expected coverage of 
\methodlow{} also converges to $1-\alpha$ 
when $mn$ tends to infinity, 
as proved by Theorem~\ref{thm:cov_up_margin} 
and illustrated in this section, 
so when $m$ or $n$ is too large so that 
one must use \methodlow{}, 
the loss in terms of expected coverage remain very small; 
for instance, when $m=200$, the expected coverage of 
\methodlow{} is $0.90084$ 
while the one of \method{} 
is $0.90004$ (Table~\ref{tab:es_std}).

\begin{table}[t]
\footnotesize
\centering
\ra{1.}
\begin{adjustbox}{max width=\textwidth}
\begin{tabular}{@{}lllll|llll@{}}
\toprule
& 
\multicolumn{4}{c}{
$\boldsymbol{(m,n) = (200,20)}$
}
&
\multicolumn{4}{c}{
$\boldsymbol{(m,n) = (20,200)}$
}
\\
\textbf{Method} & $\boldsymbol{\IE[\cdot]}$ & \textbf{Std} & $\boldsymbol{q_{\beta}}$ & $\boldsymbol{q_{1-\beta}}$ & $\boldsymbol{\IE[\cdot]}$ & \textbf{Std} & $\boldsymbol{q_{\beta}}$ & $\boldsymbol{q_{1-\beta}}$\\ \midrule
\methodCentral &  $0.90002$ & $0.00474$ & $0.89605$ & $0.90403$ & $0.90002$ & $0.00474$ & $0.89605$ & $0.90403$\\
\method &  $0.90004$ & $0.00604$ & $0.89500$ & $0.90515$ & $0.90012$ & $0.00603$ & $0.89510$ & $0.90522$\\
\methodlow & $0.90084$ & $0.00577$ & $0.89601$ & $0.90572$ & $0.90070$ & $0.00585$ & $0.89580$ & $0.90563$\\
\midrule
\methodCentralC & $0.90402$ & $0.00466$ & $0.90013$ & $0.90796$ & $0.90402$ & $0.00466$ & $0.90013$ & $0.90796$ \\
\methodcondFL & $0.90524$ & $0.00569$ & $0.90048$ & $0.91004$ & $0.90526$ & $0.00589$ & $0.90036$ & $0.91025$ \\
\methodcondFLk & $0.91084$ & $0.00558$ & $0.90618$ & $0.91556$ & $0.91046$ & $0.00560$ & $0.90579$ & $0.91519$\\
\bottomrule
\end{tabular}
\end{adjustbox}
\caption{%
Expectation, standard-deviation, 
$\beta$-quantile and $(1-\beta)$-quantile 
of the coverage $1-\alpha(\calD)$ 
of the algorithms compared in Section~\ref{sec:synth_data.equalnj} 
\textup{(}equal $n_j$\textup{)}.
\label{tab:es_std}
}
\medskip
\end{table}
\paragraph*{Conditionally-valid algorithms.} 
\begin{figure}[t]
	\centering
	\includegraphics[width=0.47\linewidth]{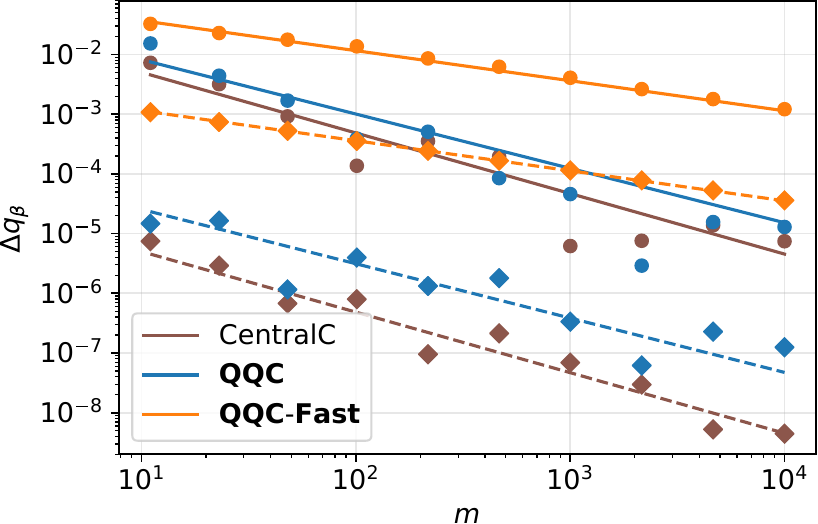}\hspace{2em}
	\includegraphics[width=0.47\linewidth]{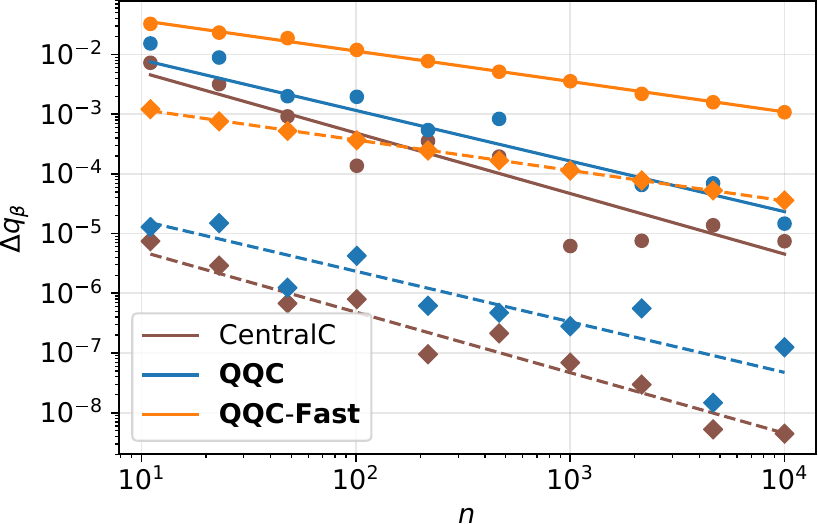}
	\caption{%
Conditionally-valid algorithms\textup{:}  
log-log plot of $\Delta q_{\beta}$ as a function of 
$m$ \textup{(}left\textup{)} or $n$ \textup{(}right\textup{)}. 
Lines show the approximation 
$\log{\Delta q_{\beta}} \approx \log(c_2) - \gamma_2 \log(m) - \delta_2 \log(n)$ 
with $c_2,\gamma_2,\delta_2$ given by Table~\ref{tab:res_coeff_lin_all}. 
Plain lines and dots correspond to 
$n=10$ \textup{(}left\textup{)} or 
$m=10$ \textup{(}right\textup{)}. 
Dashed lines and diamonds correspond to 
$n=10^4$ \textup{(}left\textup{)} or 
$m=10^4$ \textup{(}right\textup{)}. 
\label{fig.qbeta-Algos-cond}
	}
\end{figure}
\begin{figure}[t]
	\centering
	\includegraphics[width=0.49\linewidth]{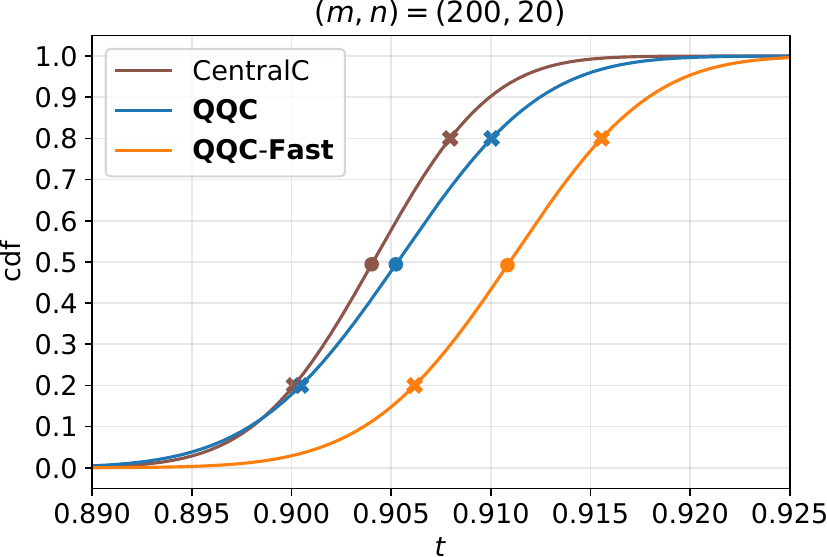}
	\includegraphics[width=0.49\linewidth]{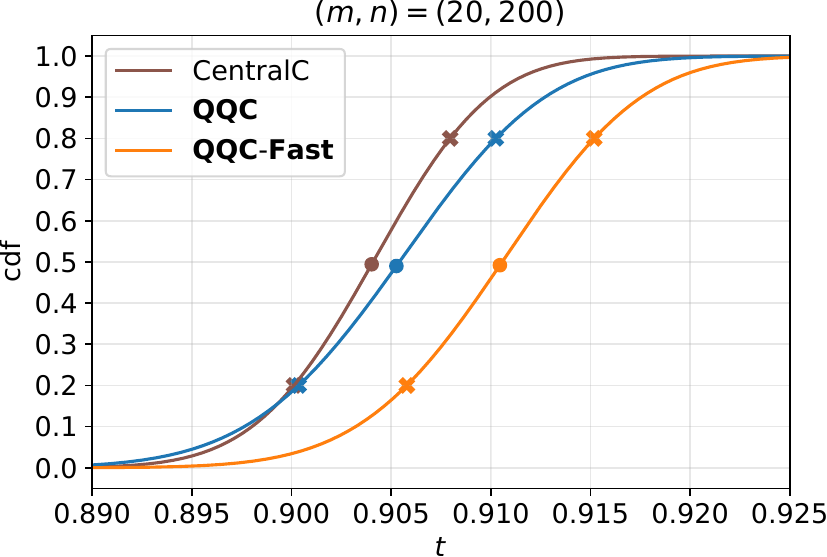}
\caption{%
Conditionally-valid algorithms\textup{:}  
cdf of the coverage $1-\alpha(\calD)$.  
The mean of each distribution is shown by a dot 
and the quantiles of order $\beta$ and $1-\beta$ by crosses.
Left\textup{:} $(m, n) = (200,20)$. 
Right\textup{:} $(m, n) = (20, 200)$.  
See Table~\ref{tab:es_std} for additional information. 
\label{fig.cdf-Algos-cond}
}
\end{figure}
The $(1-\beta)$-quantiles of the coverages of 
the three conditionally-valid algorithms 
are of the same order of magnitude 
$1-\alpha + \mathcal{O}(m^{-1/2} n^{-1/2})$, 
with constants within a factor $2$ in front of the 
remainder term (Table~\ref{tab:res_coeff_lin_all}, 
where the estimated $\gamma_3,\delta_3$ all are very 
close to $1/2$; see also 
Figure~\ref{fig.q1_beta-Algos-cond} in Appendix~\ref{app:rate.bound}). 
This matches our theoretical upper bound 
(Theorem~\ref{them:cond_main}), suggesting its tightness.  
Bigger differences appear when considering the $\beta$-quantiles 
of the coverage, 
which seem to be of order $1-\alpha + \mathcal{O}(m^{-1} n^{-1})$ 
for \methodCentralC{} and \methodcondFL{} 
(with $\gamma_2,\delta_2$ slightly smaller than $1$ for 
\methodcondFL{}), 
while it is clearly larger, of order $1-\alpha + \mathcal{O}(m^{-1/2} n^{-1/2})$,  
for \methodcondFLk{},   
according to Table~\ref{tab:res_coeff_lin_all} and Figure~\ref{fig.qbeta-Algos-cond}. 
A good summary of what happens here may be found 
in Figure~\ref{fig.cdf-Algos-cond}: 
\methodCentralC{} and \methodcondFL{} tightly adjust the 
$\beta$-quantile of the coverage 
(which is the one that must be above $1-\alpha$), 
better than \methodcondFLk{}.

Overall, as expected, the computationally more efficient method 
\methodcondFLk{} is also a bit 
more conservative than the other two methods, 
which would suggest to use \methodcondFL{} 
(in a one-shot FL context) 
if one has enough computational power, 
leading to only a small loss compared to the centralized case. 
Note finally that all three training-conditionally valid algorithms 
also are marginally valid in the settings considered by 
Figure~\ref{fig.cdf-Algos-cond}, since the expectations of their coverages 
are larger than their $\beta$-quantiles 
(which is not surprising since $\beta = 0.2 < 1/2$ 
and the medians here are close to the expectations). 

\subsubsection{Different $n_j$}
\label{sec:synth_data.diffnj}
We now study the case where the $n_j$ are different, 
so that only \method-$n_j$ and \methodcondFL-$n_j$  
are available in the one-shot FL case. 

\paragraph*{Methods.} 
We consider several settings with a total of $N=4000$ 
calibration points distributed 
across $m$ agents, where $m$ divides~$N$.  
For each value of $m$ considered, 
the values of $(n_1, \ldots, n_m)$ 
are chosen (once for all) randomly, 
according to a multinomial distribution with 
parameters $N$ and $(1/m, \ldots, 1/m)$; 
in other words, each of the $N$ calibration points 
is uniformly assigned among the $m$ agents, independently. 
Figure~\ref{fig.val_nj.25-160} in Appendix~\ref{app:rate.bound_diffnj}
provides the exact values of the $n_j$ for $m=4$ 
and $m=25$.

In addition to \method-$n_j$ and \methodcondFL-$n_j$,  
which are the only ones able to deal with 
agents having different number of data points 
in a one shot FL setting, 
we also consider several other algorithms for comparison:
\begin{itemize}
\item Centralized algorithms (\methodCentral{} and \methodCentralC{}); 

\item Algorithms~\ref{algo.FCP-QQ.1_nj}--\ref{algo.FCP-QQ-cond.1_nj} 
with data split equally into $m$ agents 
(hence having each $n_j=N/m$ points), 
that we call \method-$(N/m)$ and \methodcondFL-$(N/m)$, respectively; 

\item \method{} and \methodcondFL{}, 
with data split equally into $m$ agents, 
as in Section~\ref{sec:synth_data.equalnj}. 

\end{itemize}
For each of these algorithms, 
we report the values of the expectation, standard-deviation 
and $\zeta$-quantile for $\zeta \in \{ \beta , 1-\beta \}$ 
of their coverages when $m \in \{ 4 , 25 \}$
in Table~\ref{tab:es_std_nj_160}, 
and we plot the cdfs of their coverages when $m=25$  
in Figure~\ref{fig.diffnj.cdf-m=25}. 
Since specific values of $m$ can be misleading 
(as explained at the beginning of Section~\ref{sec:synth_data.equalnj}), 
we report $\Delta \E$ (respectively, $\Delta q_{\beta}$) 
as a function of $m$ 
for marginally-valid (respectively, conditionally-valid) algorithms 
in Figure~\ref{fig.diffnj.m-variable.E-qb}; 
for completeness, the values of $\Delta q_{1-\beta}$ 
for conditionally-valid algorithms are plotted 
in Figure~\ref{fig.diffnj.m-variable.q1-b} of Appendix~\ref{app:rate.bound_diffnj}. 
Note that the values $m \in \{4, 25\}$ in Table~\ref{tab:es_std_nj_160} 
and $m=25$ in Figure~\ref{fig.diffnj.cdf-m=25} 
have been chosen because they are typical of the worst-case 
performance of most algorithms, 
according to Figure~\ref{fig.diffnj.m-variable.E-qb}. 

\begin{table}[H]
	\footnotesize
	\centering
	\ra{1.}
	\begin{adjustbox}{max width=\textwidth}
		\begin{tabular}{@{}lllll|llll@{}}
			\toprule
			& 
			\multicolumn{4}{c}{
				$\boldsymbol{(m, N) = (4, 4000)}$
			}
			&
			\multicolumn{4}{c}{
				$\boldsymbol{(m, N) = (25, 4000)}$
			}
			\\
			\textbf{Method} 
			& $\boldsymbol{\IE[\cdot]}$ & \textbf{Std} & $\boldsymbol{q_{\beta}}$ & $\boldsymbol{q_{1-\beta}}$
			& $\boldsymbol{\IE[\cdot]}$ & \textbf{Std} & $\boldsymbol{q_{\beta}}$ & $\boldsymbol{q_{1-\beta}}$\\ 
			\midrule
			\methodCentral 
			& $0.90002$ & $0.00474$ & $0.89605$ & $0.90403$
			& $0.90002$ & $0.00474$ & $0.89605$ & $0.90403$
			\\
			\method{}% \ ($n_j=\frac{N}{m}$) 
			& $0.90009$ & $0.005659$ & $0.89535$ & $0.90485$
			& $0.90017$ & $0.00706$ & $0.89435$ & $0.90617$
			\\
			\method-$(N/m)$ 			
			& $0.90305$ & $0.00558$ & $0.89838$ & $0.90775$
			& $0.90217$ & $0.00581$ & $0.89731$ & $0.90709$
			\\
			\method-$n_j$ 
			& $0.90335$ & $0.00558$ & $0.89869$ & $0.90804$
			& $0.90238$ & $0.00583$ & $0.89755$ & $0.90733$
			\\
			\midrule
			\methodCentralC 
			& $0.90402$ & $0.00465$ & $0.90012$ & $0.90796$
			& $0.90402$ & $0.00465$ & $0.90012$ & $0.90796$
			\\
			\methodcondFL{}% \ ($n_j=\frac{N}{m}$) 
			& $0.90503$ & $0.00553$ & $0.90040$ & $0.90968$
			& $0.90549$ & $0.00581$ & $0.90062$ & $0.91039$
			\\
			\methodcondFL-$(N/m)$ 
			& $0.90969$ & $0.00622$ & $0.90444$ & $0.91487$
			& $0.90679$ & $0.00569$ & $0.90203$ & $0.91160$
			\\
			\methodcondFL-$n_j$ 
			& $0.90997$ & $0.00620$ & $0.90474$ & $0.91512$
			& $0.90698$ & $0.00567$ & $0.90222$ & $0.91177$  
			\\
			\bottomrule
		\end{tabular}
	\end{adjustbox}
	\caption{%
		Different $n_j$\textup{:}  
		Expectation, standard-deviation, 
		$\beta$-quantile and $(1-\beta)$-quantile 
		of the coverage $1-\alpha(\calD)$ 
		of the eight algorithms listed at the beginning 
		of Section~\ref{sec:synth_data.diffnj}. 
		\label{tab:es_std_nj_160}
	}
	\medskip
\end{table}

\begin{figure}[H]
	\centering 
	\includegraphics[width=0.49\linewidth]{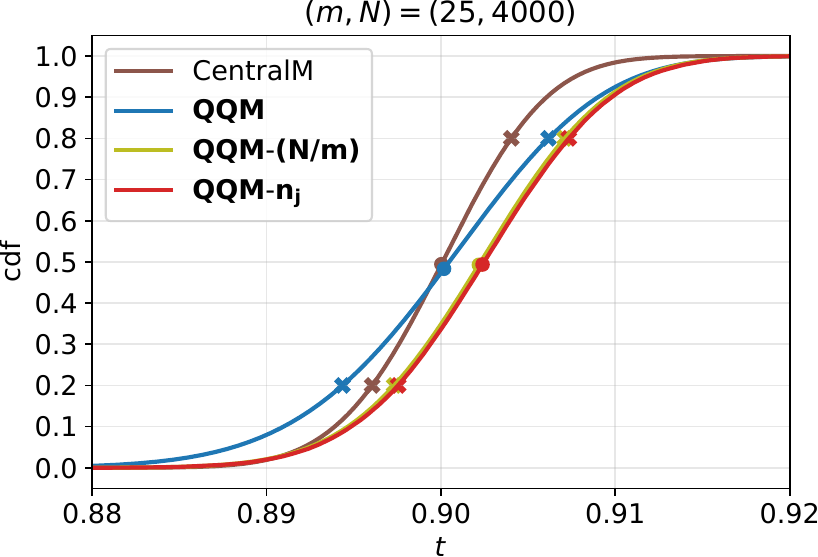}
	\includegraphics[width=0.49\linewidth]{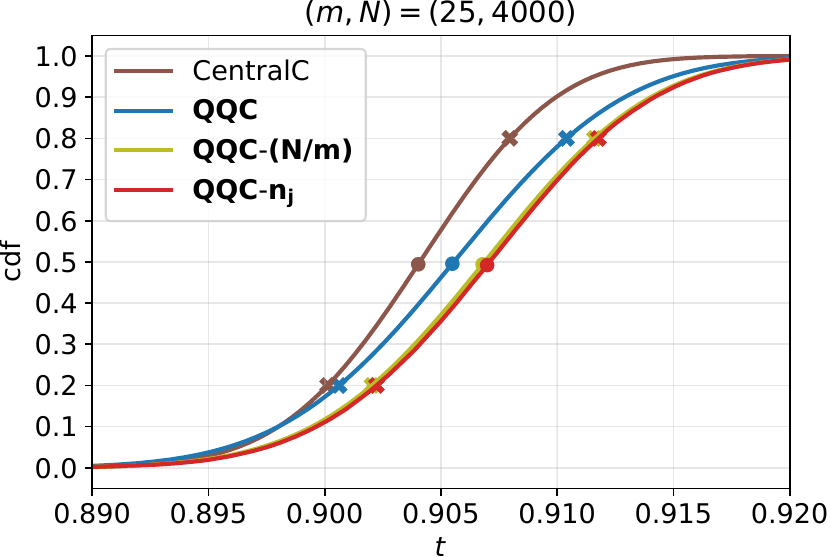}
	\caption{%
Different $n_j$\textup{:}  
cdf of the coverage $1-\alpha(\calD)$ 
of the algorithms 
considered in Section~\ref{sec:synth_data.diffnj} 
with $N=\sum_{j=1}^{m} n_j = 4000$ data points 
distributed among $m=25$ agents. 
The mean of each distribution is shown by a dot 
and the quantiles of order $\beta$ and $1-\beta$ by crosses.
See Table~\ref{tab:es_std_nj_160} for additional information. 
Left\textup{:} marginally-valid algorithms\textup{:}  
\methodCentral, 
\method{} with $n_j = N/m$ for all $j$, 
\method{}-$(N/m)$ \textup{(}that is, 
Algorithm~\ref{algo.FCP-QQ.1_nj} with $n_j = N/m$ for all $j$\textup{)} 
and \method{}-$n_j$. 
Right\textup{:} conditionally-valid algorithms\textup{:}  
\methodCentralC, 
\methodcondFL{} with $n_j = N/m$ for all $j$, 
\methodcondFL{}-$(N/m)$ \textup{(}that is, 
Algorithm~\ref{algo.FCP-QQ-cond.1_nj} with $n_j = N/m$ for all $j$\textup{)} 
and \methodcondFL{}-$n_j$. 
}
\label{fig.diffnj.cdf-m=25}
\end{figure}

\begin{figure}[H]
	\centering
	\includegraphics[width=0.49\linewidth]{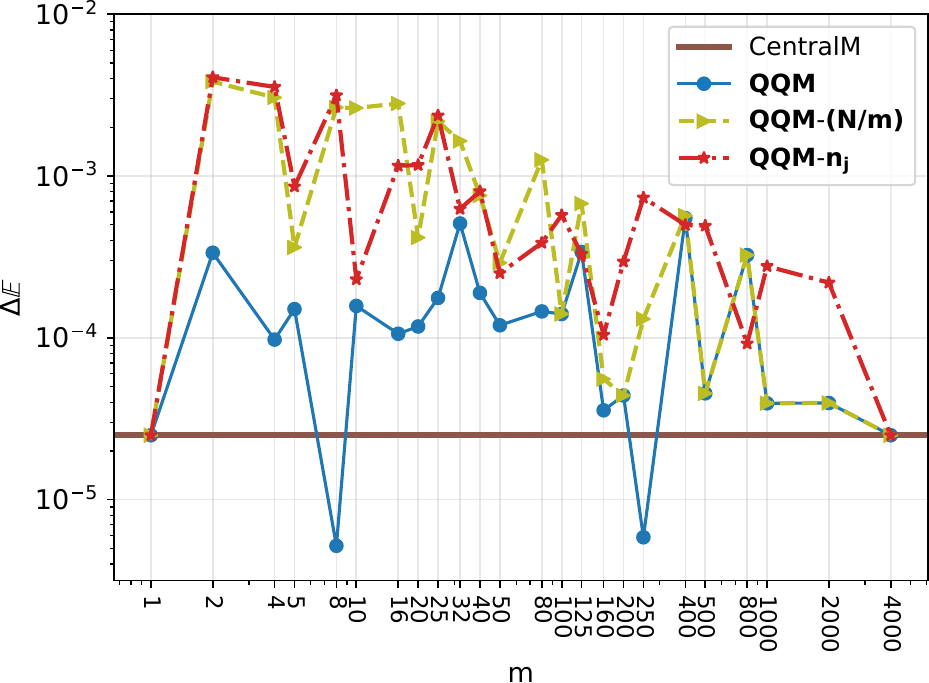}
	\includegraphics[width=0.49\linewidth]{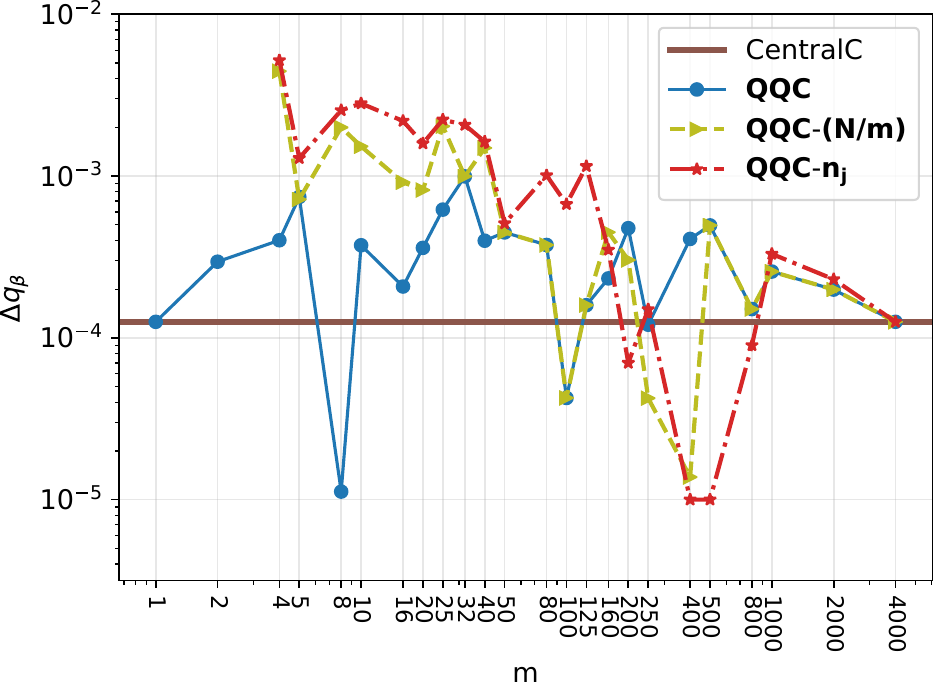}
	\caption{% 
Different $n_j$\textup{:}  
log-log plot of the performance 
\textup{(}left\textup{:} $\Delta\E$; right\textup{:} $\Delta q_{\beta}$\textup{)} 
of the algorithms 
considered in Section~\ref{sec:synth_data.diffnj} 
\textup{(}left\textup{:} marginally-valid algorithms\textup{;}  
right\textup{:} conditionally-valid algorithms\textup{)} 
as a function of the number of agents~$m$.  
The total number of data points $N=\sum_{j=1}^{m} n_j = 4000$ 
is fixed and $m$ varies within the set of divisors of~$N$. The values when $m=1, 2$ for \methodcondFL-$(N/m)$ and \methodcondFL-$n_j$ are missing because there is no $k$ such that the resulting sets are conditionally valid.
}
\label{fig.diffnj.m-variable.E-qb}
\end{figure}
\paragraph*{Comparison to centralized algorithms and to federated algorithms with $n_j=n$.} 
Let us start with marginally-valid algorithms, 
that is, the ones appearing in the top half 
of Table~\ref{tab:es_std_nj_160} 
and on the left of Figures~\ref{fig.diffnj.cdf-m=25}--\ref{fig.diffnj.m-variable.E-qb}. 
The typical ordering of their performance 
(for instance measured by their expected coverages) 
is not surprising: first \methodCentral, 
then \method{} (with a small loss compared to \methodCentral, 
as already studied in Section~\ref{sec:synth_data.equalnj}), 
and finally \method{}-$(N/m)$ and \method{}-$n_j$, 
with similar performances (and a larger loss compared to \methodCentral). 
In addition to the increase of expected coverage, 
the quantiles also increases slightly, 
with the same ordering. 
Nevertheless, \method-$n_j$   
has a reasonably good performance, 
with an expected coverage below $1 - \alpha + 0.005$. 

One possible reason explaining the fact 
that \method{}-$n_j$ is more conservative 
than the federated algorithms we propose for 
the case of equal $n_j$ (with the same overall sample size) can be 
the suboptimality of the choice $\ell^*_j = \lceil (1-\alpha) (n_j+1) \rceil$ 
made for computational reasons.
Indeed, we see on Table~\ref{tab:es_std_nj_160} and Figures~\ref{fig.diffnj.cdf-m=25}--\ref{fig.diffnj.m-variable.E-qb} that \method{}-$(N/m)$ and \method{}-$n_j$ have similar 
performances and are not as good as \method{} or \methodCentral.

Similar comments can be made about 
conditionally-valid algorithms, 
that is, the ones appearing in the bottom half 
of Table~\ref{tab:es_std_nj_160} 
and on the right of Figures~\ref{fig.diffnj.cdf-m=25}--\ref{fig.diffnj.m-variable.E-qb},
whose performance should primarily be measured 
by the $\beta$ and $(1-\beta)$-quantiles 
of their coverages. 
\paragraph*{Impact of the number of agents~$m$.}
Let us now consider when the number of agents $m$ vary 
while the total number of samples $N$ is fixed. 
Once again, let us start with marginally-valid algorithms. 
First, notice that when $m=1$ and $m=N$, 
the four algorithms considered choose the same pair $(\ell,k)$ 
by definition, hence they have the same performance. 
Then, putting aside the value $m=1$, 
Figure~\ref{fig.diffnj.m-variable.E-qb} shows that 
the performance of both \method{}-$(N/m)$ and \method{}-$n_j$ 
generally improve when $m$ increases, 
while the one of \method{} is approximately constant. 
The fact that distributing the data across more agents 
can yield a better performance may seem surprising, 
but it can easily be explained. 
Indeed, as shown by Table~\ref{tab:res_coeff_lin_N_M} and Figure~\ref{fig.DeltaE-Algos-margin_lfix} 
in Appendix~\ref{app:rate.bound_lfix}, 
for \method{}-$(N/m)$, when both $m$ and $n \egaldef N/m$ vary, 
$\Delta \E$ is approximately proportional 
to $m^{-1} (N/m)^{-1/2} = m^{-1/2} N^{-1/2}$. 
So, in the experiment of Figure~\ref{fig.diffnj.m-variable.E-qb},  
where $N$ is fixed, $\Delta \E$ is approximately proportional to $m^{-1/2}$, 
hence it decreases when $m$ increases. 
In addition, we propose the following intuition 
for explaining the behaviors of both \method{}-$n_j$ and \method{}-$(N/m)$: 
when $m$ increases, the algorithm has more options for $k \in \intset{m}$  
(while each $\ell_j$ remains fixed equal to $\ell^*_j$), 
allowing to adjust more precisely the coverage (in worst case), 
hence a better performance. 
The fact that \method{} behaves differently 
(with $\Delta \E$ roughly constant when $m$ varies) 
should not be surprising, since the number of candidate values 
for $(\ell,k)$ is equal to $N=nm$ (hence it remains constant), 
and Table~\ref{tab:res_coeff_lin_all} shows that for \method{}, 
$\Delta \E$ is approximately proportional to $(mn)^{-1} = N^{-1}$.

For conditionally-valid algorithms, 
Figure~\ref{fig.diffnj.m-variable.E-qb} shows a similar behavior 
for $\Delta q_{\beta}$ 
(decreasing function of $m$ for \methodcondFL-$n_j$ and 
\methodcondFL-$(N/m)$, roughly constant for \methodcondFL),
for which we propose the same intuitive explanation 
(see also Table~\ref{tab:res_coeff_lin_N_M} and Figure~\ref{fig.DeltaE-Algos-cond_lfix} 
in Appendix~\ref{app:rate.bound_lfix}). 
Note however that \methodcondFL-$n_j$ and 
\methodcondFL-$(N/m)$ do not coincide with \methodcondFL{} 
and \methodCentralC~when $m=1$ because of the choice 
made for $\ell_j$ in Algorithm~\ref{algo.FCP-QQ-cond.1_nj}, which prevents to find a value of $k$ satisfying the training-conditional guarantee.
Other choices would be possible 
(see Remark~\ref{rk.diff_nj.variants} in Section~\ref{sec:diff_n}). 

\paragraph*{Conclusion.} 
Our one-shot FL algorithms dealing with 
non-equal $n_j$ provide prediction sets 
only slightly more conservative than 
their centralized counterparts, 
while being computationally tractable, 
hence they seem effective for performing CP 
in one-shot FL with different~$n_j$.
Note that the values of $m$ for which 
\method-$n_j$ and \methodcondFL-$n_j$ are the most conservative 
are the small values $m \geq 2$, 
which precisely are the ones for which 
\method-$n_j$ and \methodcondFL-$n_j$ could be improved 
(while remaining tractable) 
by choosing among a few values of $\ell_j$ 
for every $j \in \intset{m}$ (see Remark~\ref{rk.diff_nj.variants}).

% !TEX root = ../main.tex

\subsection{Real data} \label{sec:xps:real}

\looseness=-1 
In this section, we evaluate the performance of each algorithm 
(in terms of coverage and length of the returned prediction sets) 
on five public-domain regression data sets also considered in 
\citep{romano2019conformalized} and \citep{sesia2021conformal}: 
physicochemical properties of protein tertiary structure 
(bio) \citep{rana2013physicochemical}, 
bike sharing (bike) \cite{bikeshare}, 
communities and crimes (community) \citep{2011crime}, Tennessee's student teacher achievement ratio (star) \citep{achilles2008tennessee}, 
and concrete compressive strength (concrete)~\citep{yeh1998modeling}.

\subsubsection{Setup}
\label{sec:xps:real:setup}
For each experiment, we split the full data set 
into three parts: 
a learning set ($40\%$), a calibration set ($40\%$), and a test set ($20\%$). 
To simulate a FL scenario, we also split 
the calibration set in $m$ disjoint subsets of equal size~$n$. 
We consider scenarios where either $m > n$ or $m < n$, 
with $\max\{m/n , n/m\} \in \{4,8\}$ 
depending on the data set; 
the exact values of $(m,n)$ for each data set are given 
in Appendix~\ref{app:add-xp}. 
All features are then standardized to have zero mean and unit variance. 
For each algorithm, we compute the empirical coverage 
obtained on the test set and the average length of the prediction set over the test set.
These two metrics are collected over $50$ different 
learning-calibration-test random splits. 

Prediction sets 
are constructed using the Conformalized Quantile Regression method 
(CQR) \citep{romano2019conformalized}, 
a popular variant of split CP directly compatible with our approach, 
following Remark \ref{rk.fh-general} in Section~\ref{sec:QQ}. 
In CQR, $\fh$ is $(\fh_{\alpha/2}, \fh_{1-\alpha/2})$ 
where $\fh_{\delta}$ is a quantile regressor 
of order $\delta$ \citep{koenker1978regression} 
and $s(x, y)= \max(\fh_{\alpha/2}(x)-y, y - \fh_{1-\alpha/2}(x))$, 
so that the prediction set 
$\{y \in \IR : s(x, y) \leq \hat{q}\} 
= [\fh_{\alpha/2}(x) - \hat{q}, \fh_{1-\alpha/2}(x) + \hat{q}]$ 
has a size adaptive to heteroscedasticity. 
In the following experiments $\fh_{\alpha/2}$ 
and $\fh_{1-\alpha/2}$ are quantile regression forests 
\citep{meinshausen2006quantile} built from the learning 
set only. 
The number of trees in the forest is set to $1000$, 
the two parameters controlling the coverage rate 
on the learning data are tuned using cross-validation 
(within the learning set), 
and the remaining hyperparameters are set as done in \citep{romano2019conformalized}.

\subsubsection{Results}
\label{sec:xps:real:results}
Figure~\ref{fig:all_dataset_m} displays the boxes 
of the empirical coverages obtained by each method 
over all the data sets and all the $50$  
different data splits (one point represents 
the empirical coverage obtained on one random split of one data set). 
Figures~\ref{fig:bike} and~\ref{fig:bike_low_m} show 
the empirical coverages as well as the lengths 
of the intervals obtained on the bike data set. 
All the results on other individual data sets 
are provided in Appendix~\ref{app:add-xp}. 
For each box, the white circle indicates the mean, 
the left-end of the box corresponds to 
the empirical quantile of order $\beta = 0.2$ and 
the right-end to the empirical quantile of order $1-\beta = 0.8$.

\begin{figure}[t]
\centering
\includegraphics[width=0.49\linewidth]{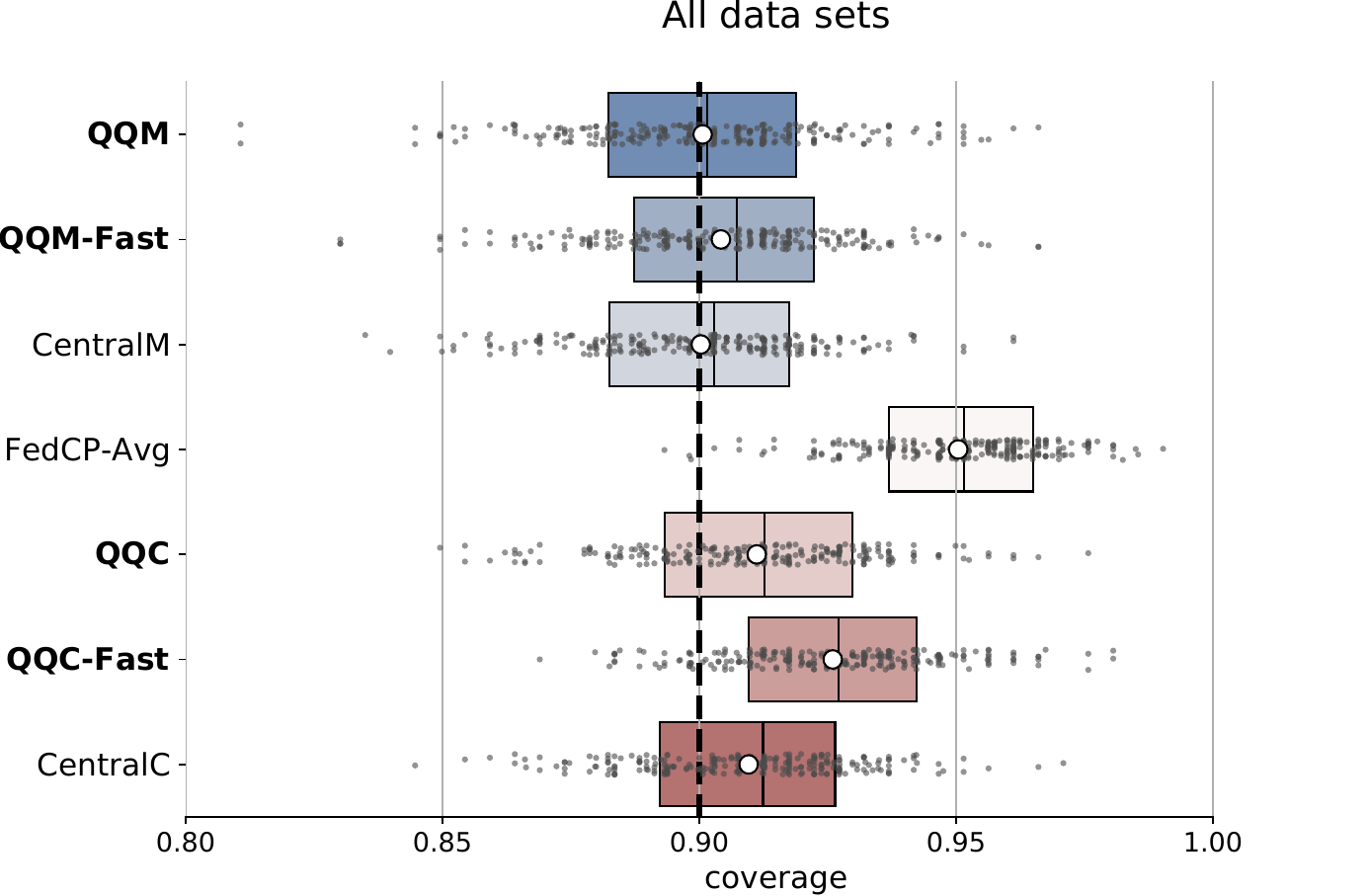}
\includegraphics[width=0.49\linewidth]{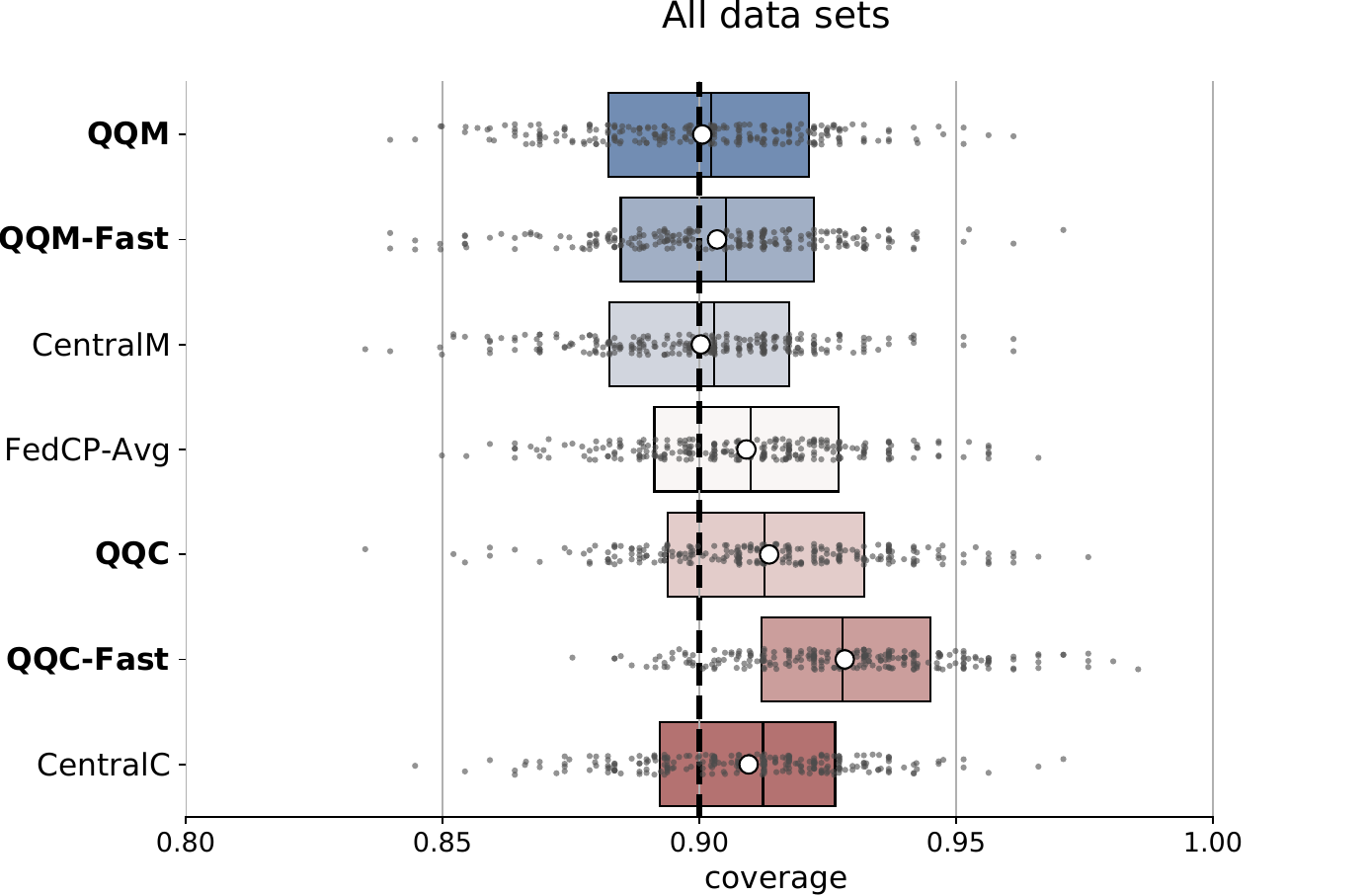}
\caption{Empirical coverages of prediction intervals 
($\alpha= 0.1$) constructed by various methods, 
aggregated across all data sets. 
Our methods are shown in bold font. 
See the beginning of Section~\ref{sec:xps:real:results} for details. 
Left: $m>n$. Right: $m<n$.} 
\label{fig:all_dataset_m}
\end{figure}

\begin{figure}[t]
	\centering
	\includegraphics[width=0.49\linewidth]{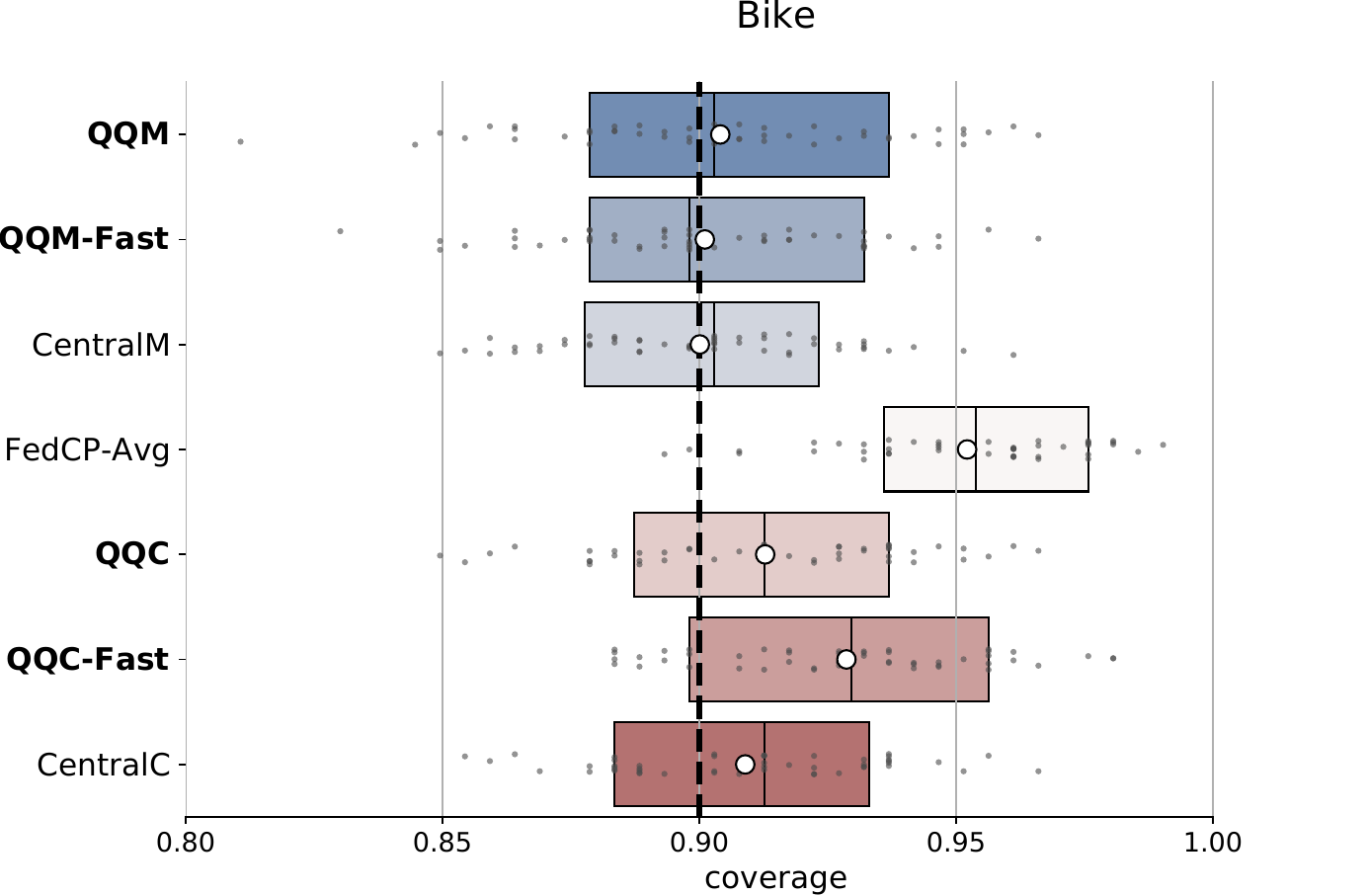}
	\includegraphics[width=0.49\linewidth]{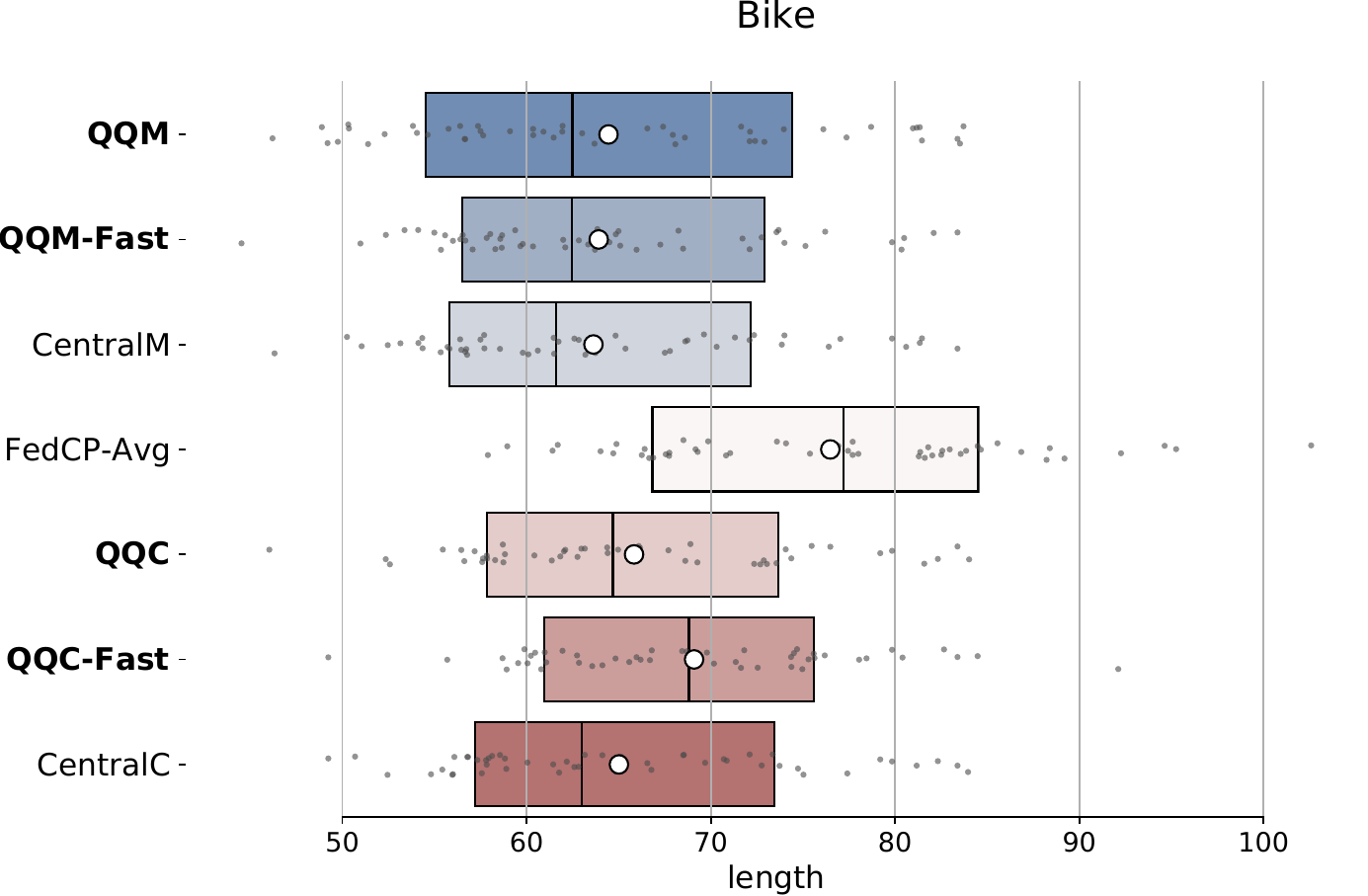}
\caption{Coverage (left) and average length (right) 
of prediction intervals for $50$ random learning-calibration-test 
splits of the bike data set. 
The miscoverage is $\alpha=0.1$, $\beta=0.2$, 
and the calibration set is split into $m=40$ disjoint subsets of equal size $n=10$. 
See the beginning of Section~\ref{sec:xps:real:results} for details. 
	}
	\label{fig:bike}
\end{figure}

\begin{figure}[t]
	\centering
	\includegraphics[width=0.49\linewidth]{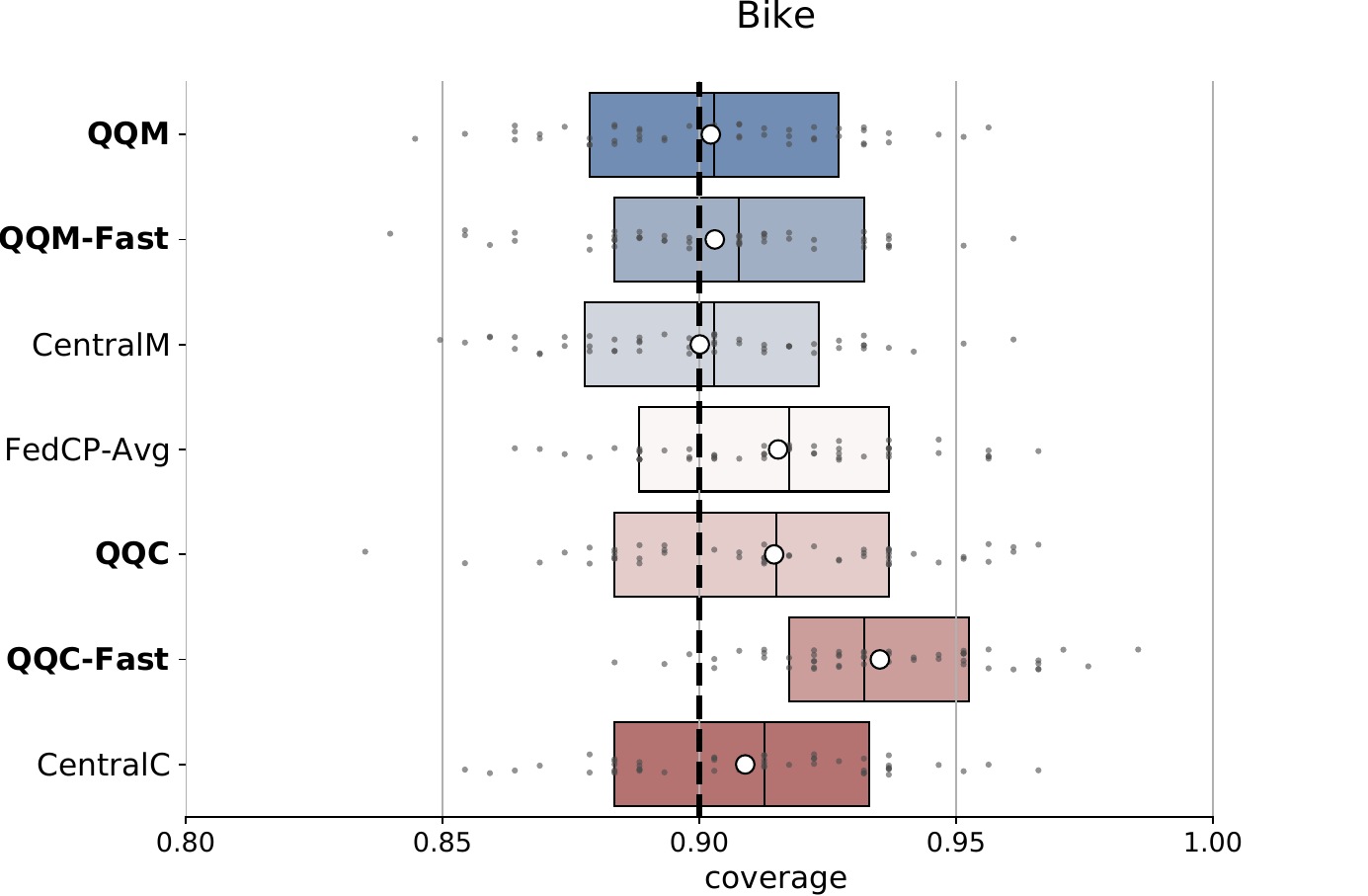}
	\includegraphics[width=0.49\linewidth]{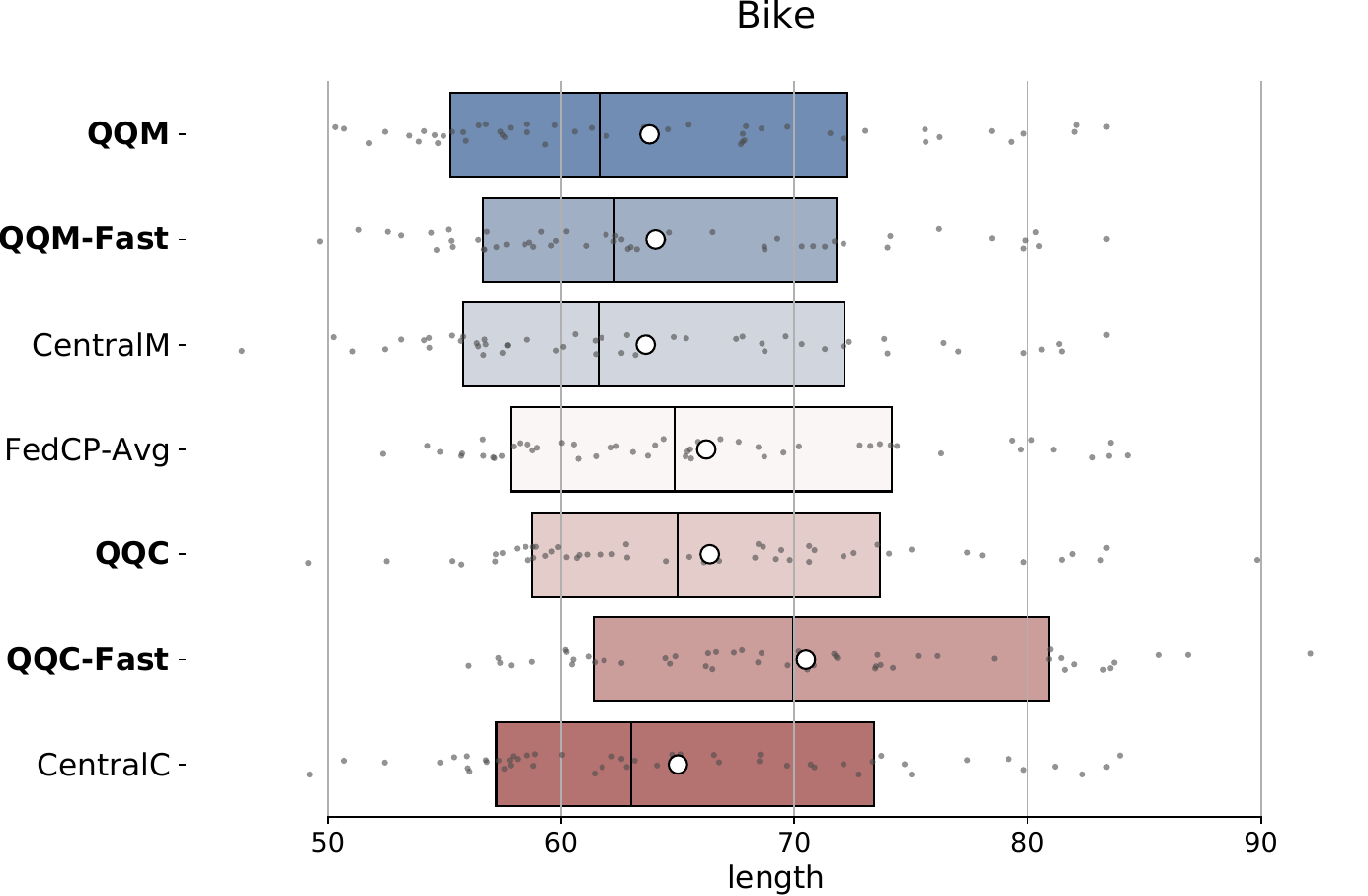}
\caption{Coverage (left) and average length (right) 
of prediction intervals for $50$ random learning-calibration-test 
splits of the bike data set. 
The miscoverage is $\alpha=0.1$, $\beta=0.2$, 
and the calibration set is split into $m=10$ disjoint subsets of equal size $n=40$. 
See the beginning of Section~\ref{sec:xps:real:results} for details. 
}
\label{fig:bike_low_m}
\end{figure}

\paragraph*{Marginally-valid algorithms.}
First, on average, 
\methodCentral, \method{} and \methodlow{} all return 
intervals with coverage greater than $1-\alpha = 0.9$ 
(the nominal coverage), without being too far from it. 
Importantly, our two one-shot FL methods 
return prediction sets with coverage and length 
close to those returned by split CP (\methodCentral). 
So, there is only a small loss when using our one-shot FL algorithms 
(whose coverage/length are slightly larger in terms of 
both expectation and dispersion). 
It is interesting to note that the results obtained 
with our two FL methods are quite similar, 
with slightly better results for \method{} notably 
when $m > n$ (left panel of Figure~\ref{fig:all_dataset_m}). 
As \methodlow{} is much faster, it is a 
good alternative to \method{} and can be preferred 
in real applications. 
Note finally that for marginal algorithms, 
the results for $m > n$ and $m < n$ 
are comparable.

\paragraph*{Conditionally-valid algorithms.} 
The comparison between the three training conditionally valid algorithms 
on Figures~\ref{fig:all_dataset_m}--\ref{fig:bike_low_m} 
is similar to the comparison made above between 
the three marginally-valid algorithms. 
The main difference is the loss of using our 
fast one-shot FL algorithm \methodcondFLk{} 
instead of \methodcondFL{} is larger than 
for our marginally-valid one-shot FL algorithms. 
Therefore, our advice here is to use \methodcondFL{} 
as long as it is computationally tractable, 
that is, when the total number $mn$ of calibration data points is not huge 
---since its complexity might be proportional to $mn$. 

Note that on Figures~\ref{fig:all_dataset_m}--\ref{fig:bike_low_m}, 
the empirical $\beta$-quantile of the coverage of 
some training-conditionally-valid algorithms 
is slightly smaller than $1-\alpha$. 
This apparent contradiction with the theoretical 
guarantees about these algorithms 
is due to the fact that the coverages reported in these experiments 
are only estimations obtained from finite test sets
(200 points for the bike dataset, for instance). 

Finally, as in Section~\ref{sec:synth_data}, 
we can notice that the (empirical) marginal coverages 
of training-conditional algorithms 
are higher than those of the marginal algorithms. 

\paragraph*{\methodAvg{}.}
First, in our experiments, \methodAvg{} appears to be 
marginally valid for each data set, 
and training-conditionally valid at a confidence level $\beta$ 
for all data sets except bike, community and star 
(each time when $m<n$). 
So, even without general theoretical guarantees 
(and even some counterexamples, see Appendix~\ref{app.avg_qq}), 
it here turns out to be valid in most cases 
(but not all, which is an issue). 
Second, \methodAvg{} is much more conservative 
than all competing algorithms in most data sets, 
and the few cases where it is less conservative 
than training-conditional algorithms 
precisely are the cases where \methodAvg{} is not 
training-conditionally valid. 
Therefore, \methodAvg{} appears to be clearly worse 
than our one-shot FL algorithms in all our experiments, 
a fact which is also supported by the few theoretical 
results presented in Appendix~\ref{app.avg_qq}. 

\paragraph*{Conclusion.}
Overall, these experiments support the fact that 
our one-shot FL methods are well-suited 
for building prediction sets in a federated setting, 
with only a mild loss compared to the centralized case.

\section{Discussion} \label{sec:conclusion}
% !TEX root = ../main.tex

All prediction sets proposed in this paper 
are of the form $\Chat_{\bar{\ell},\bar{k}}$ 
---as defined by Eq.~\eqref{CPQQ_set}--- 
with 
\[
(\bar{\ell},\bar{k}) 
\in \argmin_{(\ell,k) \in \calE} \bigl\{ \crit(\ell,k) \bigr\} 
\, , 
\]
where $\calE \subset \intset{n} \times \intset{m}$ 
is a set of marginally 
(resp. training-conditionally) valid pairs, 
and $\crit(\ell,k)$ is an upper bound 
on the expectation (resp. some quantile) 
of the coverage $1 - \alpha_{\ell,k}(\calD)$ 
when the scores cdf is continuous. 
Note that we adopt here Assumption~\ref{ass:same_n} 
to maintain notation simplicity, 
but the comments made in this paragraph also apply  
to Algorithms~\ref{algo.FCP-QQ.1_nj}--\ref{algo.FCP-QQ-cond.1_nj} 
in Section~\ref{sec:diff_n}. 
The selection criteria ($\crit$) used in Algorithms~\ref{alg:FedCPQQ}--\ref{algo.FCP-QQ-cond.1_nj} 
come from 
arguments presented in Section~\ref{sec.prelim.perf} 
(and Appendix~\ref{app.prelim.perf-compl}), 
and they are validated by the theoretical and numerical 
results obtained in 
Sections \ref{sec.uppbound_cov} and~\ref{sec:xps}. 
Yet, other choices are possible, 
leading to numerous variants of our algorithms. 
For instance, one could modify 
Algorithms~\ref{algo.FCP-QQ-cond.1}--\ref{algo.FCP-QQ-cond.2} 
by taking 
\[
\crit(\ell,k) 
= F_{U_{(\ell:n, k:m)}}^{-1} (1-\beta') 
\]
with any $\beta' \in (0,1)$ 
---instead of fixing $\beta'=\beta$---, 
leading to the algorithms respectively defined 
by Eq.~\eqref{eq.def.algo.FCP-QQ-cond.1.var} 
and~\eqref{eq.def.algo.FCP-QQ-cond.2.var} 
in Appendix~\ref{thm:cov_up_cond_proof}. 
Remarkably, theoretical coverage upper bounds 
similar to Theorem~\ref{thm:cov_up_cond} 
can also be obtained for these variants, 
as proved by 
Eq. \eqref{eq.cov-upper.algo.FCP-QQ-cond-1.var-min-upper-bound}  
and~\eqref{eq.cov-upper.algo.FCP-QQ-cond-2.var-min-upper-bound} 
in Appendix~\ref{thm:cov_up_cond_proof}. 
One could also use this criterion 
for selecting among marginally-valid pairs $(\ell,k)$, 
or use $\crit(\ell,k) = M_{\ell,k}$ 
for selecting among training-conditionally-valid pairs $(\ell,k)$. 
The function $\crit$ can also be chosen by the user 
---for instance, with a specific application in mind---, 
together with a set $\calE$ chosen 
among those used by our Algorithms~\ref{alg:FedCPQQ}--\ref{algo.FCP-QQ-cond.1_nj}, 
and the resulting prediction set would satisfy 
the corresponding distribution-free coverage 
lower bounds proved in this paper. 

% Limitation
This work brings many possible future research directions. 
It would be interesting to investigate 
how our quantile-of-quantiles estimators 
can be adapted to the heterogeneous case, 
that is, when agents have data following different distributions. 
This could potentially extend the approaches of \citep{lu2023federated} and \citep{pmlr-v202-plassier23a} to the one-shot setting. 
Another interesting line of research would be 
to propose and study differentially private versions of our algorithms. 
On a more theoretical aspect, 
it would be interesting to control the deviations 
of the coverage of marginal prediction sets  
built in Section~\ref{subsec:marg_valid}, 
and conversely to control the expected coverage of the 
training-conditional prediction sets of Section~\ref{sec.multi-order.algos}.  
Finally, our paper focuses on the calibration step, 
making it particularly suited for split-based conformal methods. 
It would be interesting to study 
how our FL approach could be extended to 
full conformal prediction, 
and to the more general framework of nested conformal prediction \citep{gupta2022nested}.

\bibliographystyle{imsart-number}
\bibliography{biblio}

\newpage
\appendix

\section{Prediction set performance measure: complements} \label{app.prelim.perf-compl}
% !TEX root = ../main.tex

Recall that we prove in Section~\ref{sec.prelim.perf} 
that for every $(\ell,k)$, 
the size of $\Chat_{\ell,k}(x)$ is 
a nondecreasing function of the coverage 
$1 - \alpha_{\ell, k}(\calD)$. 

As a consequence, for any $\beta \in (0,1)$, 
minimizing the $(1-\beta)$-quantile of the coverage $1 - \alpha_{\ell, k}(\calD)$ 
yields a minimizer of the $(1-\beta)$-quantile of the size of $\Chat_{\ell,k}(x)$, 
simultaneously for all $x \in \calX$. 
We therefore use this strategy for building the two 
training-conditionally-valid algorithms described in Section~\ref{sec.multi-order.algos}. 

The problem is a bit more difficult when one measures 
the performance of a predictor set by its size \emph{on average over $\calD$} 
(even at a single $x \in \calX$). 
As an illustration, let us consider regression with the fitted absolute residual score function. 
Then, 
\[ 
\forall x \in \calX \, , \qquad 
\Chat_{\ell,k} (x) 
= \bigl[ \fh(x) - S_{(\ell, k)} \, , \, \fh(x) + S_{(\ell, k)} \bigr] 
\]
is a prediction interval of length 
\[
2 S_{(\ell, k)} 
= 2 F_S^{-1} \bigl( 1 - \alpha_{\ell, k}(\calD) \bigr) 
\] 
almost surely. 
In this example, the size of $\Chat_{\ell,k}(x)$ 
(its length) 
is a nondecreasing function 
of the coverage $1 - \alpha_{\ell, k}(\calD)$, 
but the relationship between size and coverage is highly non-linear in general, 
hence the expected size $\E[ \Chat_{\ell,k}(x) ]$ 
cannot easily be linked to the expected coverage 
$\E [ 1 - \alpha_{\ell, k}(\calD) ]$. 
Nevertheless, following what is usually done for split CP 
(see Section~\ref{subsec:splitCP}), 
choosing $(\ell,k)$ by minimizing $\E [ 1 - \alpha_{\ell, k}(\calD) ]$ 
---or at least its value when $F_S$ is continuous, 
see Sections \ref{subsec:marg_valid} and~\ref{subsec:splitCP}--- 
is a natural choice, which is efficient according to numerical experiments. 

The latter argument might seem questionable, 
especially when $F_S$ is not guaranteed to be continuous. 
Yet, the strategy we propose is reasonable for the following reasons.  
$S_{(\ell,k)}$ is a nondecreasing function 
of $(\ell,k)$ (for the lexicographic order on $\intset{n} \times \intset{m}$), 
hence both the expected size and the expected coverage are nondecreasing functions of $(\ell,k)$. 
Therefore, minimizing the value $M_{\ell,k}$ 
of the expected coverage $\E [ 1 - \alpha_{\ell, k}(\calD) ]$ when $F_S$ is continuous 
(see Section~\ref{subsec:marg_valid}) 
over the set of 
(marginally or conditionally) valid pairs $(\ell,k)$ 
yields some $(\ell,k)$ minimal for the lexicographic order 
in the set of valid pairs. 
So, even if this does not choose the optimal valid pair $(\ell,k)$, 
it yields a performance close to its optimal value 
---at least numerically. 
We therefore follow this strategy when building 
marginally-valid algorithms in Section~\ref{subsec:marg_valid}.

\section{Key preliminary results on order statistics}\label{app.order-stat}
% !TEX root = ../main.tex

In this section, 
we provide some known and new important results about order statistics 
that play a key role in our proofs. 
We refer to Appendix~\ref{app.more-details-order-stat} 
for proofs and additional results, and to 
\cite{david2004order} for an in-depth presentation on this topic. 

\subsection{Order statistics} \label{app.order-stat.order}
Let us first introduce some notation. 
Given a real-valued random variable $Z$ 
with arbitrary cumulative distribution function (cdf)~$F_Z$, 
its quantile function is defined as the generalized inverse of~$F_Z$:
\begin{equation} 
\label{eq.def-gal-inverse}
\forall p \in (0,1) \, , \qquad 
F_Z^{-1}(p) 
:= \inf \{x \in \IR : F_Z(x) \geq p  \} 
\, .
\end{equation}
A key property of the generalized inverse 
---straightforward from Eq.~\eqref{eq.def-gal-inverse}--- 
is that 
\begin{equation} 
\label{eq.prop-gal-inverse}
\forall p \in [0,1] \, , \quad 
F_Z \circ F_Z^{-1} (p) \geq p 
\qquad \text{with equality if} \qquad 
p \in \Im(F_Z)
\, . 
\end{equation}
In particular, when $F_Z$ is continuous, 
Eq.~\eqref{eq.prop-gal-inverse} is an equality 
for every $p \in [0,1] = \Im(F_Z)$. 
Given a sample $Z_1, \ldots, Z_N$, 
we denote by $Z_{(1:N)} \leq \cdots \leq Z_{(N:N)}$ 
the corresponding order statistics, 
so that $Z_{(i:N)} = \Qh_{(i)} ( Z_1 , \ldots, Z_N )$ for every $i \in \intset{N}$, 
using the notation defined by Eq.~\eqref{def:Qk} 
in Section~\ref{subsec:splitCP}. 
When the sample size is clear from context, 
we note $Z_{(i)} := Z_{(i:N)}$. 

We first recall the well-known following link between general order statistics 
and uniform order statistics 
(see e.g. \citep{david2004order}).
\begin{lemma} \label{lemma:centra_unif}
For any $N \geq 1$, 
let $U_1, \ldots U_N$ be independent random variables 
with uniform distribution over $[0,1]$, 
and $Z_1, \ldots, Z_N$ be independent and 
identically distributed real-valued random variables with common cdf $F_Z$. 
For any $r \in \intset{N}$, 
let $U_{(r)} = U_{(r:N)}$ and $Z_{(r)} = Z_{(r:N)}$ respectively denote the corresponding 
$r$-th order statistics. 
Then, we have 
\begin{equation}
\label{eq.lemma:centra_unif.domin-stoch}
F_Z( Z_{(r)} ) 
\egalloi F_Z \circ F_Z^{-1} ( U_{(r)} ) 
\geq U_{(r)} 
\, , \qquad \text{hence} \qquad 
F_Z( Z_{(r)} ) \succeq U_{(r)}
\, , 
\end{equation}
that is, $F_Z( Z_{(r)} )$ stochastically dominates $U_{(r)}$. 
In particular, Eq.~\eqref{eq.lemma:centra_unif.domin-stoch} implies that 
\begin{equation}
\label{eq.lemma:centra_unif.domin-stoch.conseq}
\forall a \in \R \, , \quad 
\P\bigl( a \leq F_Z(Z_{(r)}) \bigr)
\geq 
\P\bigl( a \leq U_{(r)} \bigr) 
\quad \text{and} \quad 
\E \bigl[ F_Z(Z_{(r)}) \bigr]
\geq 
\E[U_{(r)}] 
\, . 
\end{equation}
Furthermore, when $F_Z$ is continuous, 
\begin{equation}
\label{eq.lemma:centra_unif.egal-loi}
F_Z( Z_{(r)} ) 
\egalloi U_{(r)} 
\end{equation}
both follow the 
$\Beta(r,N-r+1)$ distribution, 
and they have the same expectation and cdf. 
\end{lemma}
Lemma~\ref{lemma:centra_unif} is proved in Appendix~\ref{lemma:centra_unif_proof} for completeness. 
It shows that concentration/deviation inequalities for order statistics 
can be obtained in general from results on uniform order statistics. 
Therefore, any result about the Beta distribution can be used to bound $F_Z(Z_{(r)})$ 
---hence the bounds on the coverage of 
(centralized) split CP in Section~\ref{subsec:splitCP}, 
for instance. 

We now present non-asymptotic bounds on the 
quantile function of uniform order statistics, 
which rely on a concentration result for the Beta distribution 
\cite[Theorem~1]{marchal2017sub}. 
\begin{proposition} \label{pro.FUr}
For any integer $N \geq 1$ and $r \in \intset{N}$, 
the quantile function $F_{U_{(r:N)}}^{-1}$ of the $r$-th order statistics 
of a sample of $N$ independent standard uniform variables 
satisfies the following inequalities, 
for any $\delta>0$\textup{:} 
\begin{align}
\label{eq.pro.FUr.dev-gauche.inv}
F_{U_{(r:N)}}^{-1} (\delta) 
&\geq \frac{r}{N+1} - \sqrt{\frac{\log(1/\delta)}{2(N+2)}}
\\
\label{eq.pro.FUr.dev-droite.inv}
\text{and} \qquad 
F_{U_{(r:N)}}^{-1} (1-\delta) 
&\leq \frac{r}{N+1} + \sqrt{\frac{\log(1/\delta)}{2(N+2)}}
\, . 
\end{align}
\end{proposition}
Proposition~\ref{pro.FUr} is proved in Appendix~\ref{sec.pr.cor.FUr}.

Finally, our coverage upper bounds 
(Theorems~\ref{thm:cov_up_margin} and~\ref{thm:cov_up_cond}) 
rely on the following non-asymptotic upper bound on 
the increments of the quantile function 
$F^{-1}_{U_{(\ell:n)}}$ of the $\Beta(\ell,n-\ell+1)$ distribution, 
showing that it is locally $\mathcal{O}(n^{-1/2})$-Lipschitz. 
\begin{theorem} \label{thm:len_bound2}
For any $n \geq 1$, $\ell \in \intset{n}$, 
$\varepsilon \in (0 , 1)$, and $x \in (0, 1-\varepsilon)$, 
we have 
\begin{align}
\label{eq.thm:len_bound2}
0 
\leq F^{-1}_{U_{(\ell:n)}}\left(x + \varepsilon \right) - F^{-1}_{U_{(\ell:n)}}\left(x \right) 
&\leq 
\frac{\varepsilon}{\sqrt{n+2}} 
\, \frac{\sqrt{2}}{g( \min\{ x , 1-x-\varepsilon \})}
\, ,
\\
\notag 
\text{where} \quad 
\forall t \in \left( 0 \, , \, \frac{1}{2} \right] \, , \qquad 
g(t) 
&\egaldef \sup_{ \delta \in ( 0 , \min\{ t \, , \, 1-t\} ) } 
\frac{t-\delta}{\sqrt{\log(1/\delta)}}
\geq \frac{t}{2 \sqrt{\log(2/t)}}
\\
\notag 
\text{and} \qquad  
\forall t \in [1/3 , 1/2] \, , 
\qquad 
\frac{ \sqrt{2} }{g(t)} 
&\leq 9 
\, . 
\end{align}
\end{theorem}
Theorem~\ref{thm:len_bound2} is proved in Appendix~\ref{app.pr.thm:len_bound2}. 
To the best of our knowledge, 
such an upper bound has never been proved previously. 
It is a key result of the paper, 
used in the proof of our coverage upper bounds.
\begin{remark}
The upper bound in Eq.~\eqref{eq.thm:len_bound2} is always 
well defined since $g$ is defined on $(0,1/2]$ 
and $\min \{ x , 1 - x- \varepsilon \} \leq 1/2$ 
for every $\varepsilon \in (0,1)$ and $x \in (0,1-\varepsilon)$. 
\end{remark}

\subsection{Order statistics of order statistics}  \label{app.order-stat.order-of-order}
The federated learning methods of Section~\ref{sec:FCP-QQ} 
are all based on the quantile-of-quantiles family of estimators 
(Definition~\ref{def:QQ} in  Section~\ref{sec:QQ}), 
which are order statistics of order statistics, 
so we focus on them in this subsection. 

Given a collection $(Z_{i,j})_{1 \leq i \leq n , 1 \leq j \leq m}$ 
of $m$ samples of size $n$, for every $j \in \intset{m}$, 
we denote the order statistics of 
the $j$-th sample $(Z_{i,j})_{1 \leq i \leq n}$ by 
\[ 
Z_{(1:n),j} \leq Z_{(2:n),j} \leq \cdots \leq Z_{(n:n),j}
\, . 
\]
Then, for every $\ell \in \intset{n}$, 
we denote the order statistics of the sample of the 
$\ell$-th order statistics $( Z_{(\ell:n),j} )_{1 \leq j \leq m}$ by 
\[ 
Z_{(\ell:n),(1:m)} \leq Z_{(\ell:n),(2:m)} \leq \cdots \leq Z_{(\ell:n),(m:m)}
\, . 
\]
For every $\ell \in \intset{n}$ and $k \in \intset{m}$, 
$Z_{(\ell:n),(k:m)} \defegal Z_{(\ell:n, k:m)}$ is the $(\ell,k)$-th 
order statistics of order statistics of 
$(Z_{i,j})_{1 \leq i \leq n , 1 \leq j \leq m}$, 
which can also be written as follows with the 
notation defined by Eq.~\eqref{def:Qk} 
in Section~\ref{subsec:splitCP}: 
\begin{align*}
Z_{(\ell:n, k:m)} 
&= \Qh_{(k)} \Bigl( \bigl( \Qh_{(\ell)}(\mathcal{Z}_j) \bigr)_{j \in \intset{m}} \Bigr) 
\\
\text{where} \quad 
\forall j \in \intset{m} \, , \quad 
\mathcal{Z}_j 
&\egaldef ( Z_{i,j} )_{i \in \intset{n}}  
\qquad \text{so that} \qquad 
Z_{(\ell:n), j} 
= \Qh_{(\ell)}(\mathcal{Z}_j) 
\, .
\end{align*}
In the following, we write $Z_{(\ell),j}$ 
(resp. $Z_{(\ell,k)}$) instead of 
$Z_{(\ell:n),j}$ (resp. $Z_{(\ell:n, k:m)}$) 
when the sample sizes $n$ and $m$ are clear from context. 

Our first key result is a link with order statistics 
of uniform order statistics, 
similar to Lemma~\ref{lemma:centra_unif}. 
\begin{lemma} \label{lemma:centra_unif_kl}
For any $n,m \geq 1$, 
let $(U_{i,j})_{1 \leq i \leq n , 1 \leq j \leq m}$ 
be independent random variables with uniform distribution over $[0,1]$, 
and $(Z_{i,j})_{1 \leq i \leq n , 1 \leq j \leq m}$ 
be independent and identically distributed real-valued 
random variables with common cdf~$F_Z$. 
For every $\ell \in \intset{n}$ and $k \in \intset{m}$, 
let $U_{(\ell,k)}$ and $Z_{(\ell, k)}$ 
respectively denote the corresponding 
$(\ell,k)$-th order statistics of order statistics, 
as defined at the beginning of Appendix~\ref{app.order-stat.order-of-order}.  
Then, we have 
\begin{equation}
\label{eq.lemma:centra_unif_kl.domin-stoch}
F_Z( Z_{(\ell,k)} ) 
\egalloi F_Z \circ F_Z^{-1} ( U_{(\ell,k)} ) 
\geq U_{(\ell,k)} 
\, , \quad \text{hence} \quad 
F_Z( Z_{(\ell,k)} ) \succeq U_{(\ell,k)}
\, ,  
\end{equation}
that is, $F_Z( Z_{(\ell,k)} )$ stochastically dominates $U_{(\ell,k)}$. 
In particular, Eq.~\eqref{eq.lemma:centra_unif_kl.domin-stoch} 
implies that 
\begin{equation}
\label{eq.lemma:centra_unif_kl.domin-stoch.conseq}
\forall a \in \R \, , \quad 
\P \bigl( a \leq F_Z ( Z_{(\ell,k)} )\bigr)
\geq 
\P (a \leq U_{(\ell,k)} ) 
\quad \text{and} \quad 
\E \bigl[ F_Z( Z_{(\ell,k)} ) \bigr]
\geq \E [ U_{(\ell,k)} ] 
\, . 
\end{equation}
Furthermore, when $F_Z$ is continuous, 
\begin{equation}
\label{eq.lemma:centra_unif_kl.egal-loi}
F_Z( Z_{(\ell,k)} ) 
\egalloi U_{(\ell,k)} 
\end{equation}
hence they have the same expectation and cdf. 
\end{lemma}
Lemma~\ref{lemma:centra_unif_kl} is proved in Appendix~\ref{lemma:centra_unif_kl_proof}. 
It shows that concentration/deviation inequalities 
for order statistics of order statistics  
can be obtained in general from results on 
order statistics of uniform order statistics. 
Therefore, we now focus on the distribution of 
$U_{(\ell:n , k:m )}$. 

By definition, $U_{(\ell:n , k:m )}$ is the 
$k$-th order statistics of a sample of 
$m$ independent and identically distributed random variables 
following a $\Beta(\ell, n-\ell+1)$ distribution. 
This is called a Beta order statistics, 
and it follows a Beta-Beta distribution 
\citep{jones2004families, al2012composition, makgai2019beta, connor2022integrals}. 
The Beta-Beta distribution is difficult to analyze directly 
\citep{cordeiro2013simple, castellares2021note} 
---as one can see from the formula of its expectation 
$M_{\ell,k}$ provided by Theorem~\ref{them:main} 
in Section~\ref{subsec:marg_valid}---, 
so it is more convenient to first relate it to 
the well-known Beta distribution, as done by the next result.
\begin{lemma} 
\label{lemma:beta_beta}
Let $n, m \geq 1$ be two integers, 
$k \in \intset{m}$, $\ell \in \intset{n}$. 
Let $U_{(\ell:n)}$ and $U_{(k:m)}$ be defined as in Lemma~\ref{lemma:centra_unif},  
and $U_{(\ell:n, k:m)} = U_{(\ell,k)}$ as in Lemma~\ref{lemma:centra_unif_kl}. 
Then, their cdfs satisfy   
\begin{align}
\label{eq.lemma:beta_beta.cdf-U}
F_{ U_{(\ell:n, k:m)} } 
&= F_{U_{(k:m)}} \circ F_{U_{(\ell:n)}}
\, . 
\end{align}
\end{lemma}
Lemma~\ref{lemma:beta_beta} is proved in Appendix~\ref{lemma:beta_beta_proof}. 
Since the cdfs of $U_{(\ell:n)}$ and $U_{(k:m)}$ are known 
(see Lemma~\ref{lemma:centra_unif}), 
it shows that the cdf of $U_{(\ell:n, k:m)} $ 
is the composition of the cdf of a $\Beta(k, m-k+1)$ 
with the cdf of a $\Beta(\ell, n-\ell+1)$.

\section{Proofs of Section \ref{subsec:marg_valid} (marginal guarantees)}
% !TEX root = ../main.tex

\subsection{Proof of Theorem~\ref{them:main}}
\label{thm:main_proof}
Throughout the proof, we reason conditionally 
to $\calD^{lrn}$. 
Since $(X_{i,j}, Y_{i,j})_{1 \leq i \leq n , 1 \leq j \leq m}, (X, Y)$ are i.i.d., 
the associated scores $(S_{i,j})_{1 \leq i \leq n , 1 \leq j \leq m}, S$ are i.i.d. 
Let us denote by $F_S$ their cdf (given $\calD^{lrn}$). 

By definition of $\Chat_{\ell, k}(X)$, 
we have 
$\P(Y \in \Chat_{\ell, k}(X)) 
= \IP\left(S \leq S_{(\ell, k)}\right) 
= \IE[F_S(S_{(\ell, k)})]$. 
Furthermore, by Eq.~\eqref{eq.lemma:centra_unif_kl.domin-stoch.conseq} 
in Lemma~\ref{lemma:centra_unif_kl} 
with $Z_{i,j} = S_{i,j}$ (hence $F_Z = F_S$), 
we have 
\begin{equation}
\label{eq.thm:main_proof.1}
\E \bigl[ F_S( S_{(\ell, k)}) \bigr] 
\geq \E [ U_{(\ell, k)} ] 
=: M_{\ell, k} 
\end{equation}
in general, 
with equality when $F_S$ is continuous 
---or equivalently when the scores are a.s. distinct--- 
by Eq.~\eqref{eq.lemma:centra_unif_kl.egal-loi} 
in Lemma~\ref{lemma:centra_unif_kl}. 
It remains to prove the formula for $M_{\ell,k}$ 
announced in Theorem~\ref{them:main}. 

Recall that by definition, $U_{(\ell,k)} = U_{(\ell:n, k:m)}$ 
is the $k$-th order statistics of a sample 
$( U_{(\ell:n), j} )_{1 \leq j \leq m}$ of $m$ 
independent random variables with 
common distribution $\Beta (\ell, n - \ell + 1)$. 
Therefore, its pdf is given by 
Eq.~\eqref{eq.Xr.pdf} in Appendix~\ref{app.more-details-order-stat.exact} 
with $r=k$, $N=m$, $f_Z = f_{U_{(\ell:n)}}$ 
and $F_Z = F_{U_{(\ell:n)}}$. 
Using Eq.~\eqref{eq.FUl.pdf-cdf} with $r=\ell$ and $N=n$, 
we get that 
for every $t \in \R$, 
\begin{align*}
f_{U_{(\ell, k)}}(t)
&=\dfrac{m!}{(k-1)!(m-k)!} F_{U_{(\ell)}}(t)^{k-1} \bigl[1 - F_{U_{(\ell)}}(t) \bigr]^{m-k} f_{U_{(\ell)}}(t) \\
&=k \binom{m}{k} F_{U_{(\ell)}}(t)^{k-1} \bigl[1 - F_{U_{(\ell)}}(t) \bigr]^{m-k} f_{U_{(\ell)}}(t) \\
&=k \binom{m}{k} \left[\sum^n_{i=\ell} \binom{n}{i} t^i (1 - t)^{n-i}\right]^{k-1} \left[1 - \sum^n_{i=\ell} \binom{n}{i} t^i (1 - t)^{n-i}\right]^{m-k} f_{U_{(\ell)}}(t) \\
&=k \binom{m}{k} \left[\sum^n_{i=\ell} \binom{n}{i} t^i (1 - t)^{n-i}\right]^{k-1} \left[\sum^{\ell-1}_{i=0} \binom{n}{i} t^i (1 - t)^{n-i}\right]^{m-k} f_{U_{(\ell)}}(t) \\
\text{since} \qquad 
1 
&= \sum^n_{i=0} \binom{n}{i} t^i (1 - t)^{n-i} 
= \sum^{\ell-1}_{i=0} \binom{n}{i} t^i (1 - t)^{n-i}
+ \sum^n_{i=\ell} \binom{n}{i} t^i (1 - t)^{n-i} \, .
\end{align*}
Then, by developing the powers over the summations, we get that for every $t \in [0,1]$, 
\begin{align*} 
&f_{U_{(\ell, k)}}(t)\\
&= k \binom{m}{k} \sum_{i_1=\ell}^{n} \ldots \sum_{i_{k-1}=\ell}^{n} 
\sum_{i_{k+1}=0}^{\ell-1} \ldots \sum_{i_{m}=0}^{\ell-1} 
\binom{n}{i_1} \cdots \binom{n}{i_{k-1}} \binom{n}{i_{k+1}} \cdots \binom{n}{i_{m}} 
\\
& \qquad \qquad \qquad \qquad \qquad 
\times t^{i_1 + \cdots + i_{k-1} + i_{k+1} + \cdots + i_m} 
(1-t)^{n(m-1) - (i_1 + \cdots + i_{k-1} + i_{k+1} + \cdots + i_m)} f_{U_{(\ell)}}(t) 
\\
&= \dfrac{\displaystyle k \binom{m}{k}}{\mathrm{B}(\ell, n-\ell+1)} 
\cdot \sum_{i_1=\ell}^{n} \ldots \sum_{i_{k-1}=\ell}^{n} 
\sum_{i_{k+1}=0}^{\ell-1} \ldots \sum_{i_{m}=0}^{\ell-1}
\binom{n}{i_1} \cdots \binom{n}{i_{k-1}} \binom{n}{i_{k+1}} \cdots \binom{n}{i_{m}} 
\\
& \qquad \qquad \qquad \qquad \qquad 
\times t^{i_1 + \cdots + i_{k-1} + i_{k+1} + \cdots + i_m + \ell-1} 
(1-t)^{n(m-1) - (i_1 + \cdots + i_{k-1} + i_{k+1} + \cdots + i_m)+n-\ell} 
\, ,
\end{align*}
where in the last equality we used the definition of $f_{U_{(\ell)}}$. 
We now note that 
\begin{align*}
&\int_{0}^{1} t \cdot t^{i_1 + \cdots + i_{k-1} + i_{k+1} + \cdots + i_m + \ell-1} 
(1-t)^{n(m-1) - (i_1 + \cdots + i_{k-1} + i_{k+1} + \cdots + i_m)+n-\ell} \mathrm{d}t 
\\
&= \int_{0}^{1} t^{i_1 + \cdots + i_{k-1} + \ell + i_{k+1} + \cdots + i_m} (1-t)^{mn-(i_1 + \cdots + i_{k-1} + \ell +  i_{k+1} + \cdots + i_m)} \mathrm{d}t  
\\
&=\mathrm{B}(i_1 + \cdots + i_{k-1} + \ell + i_{k+1} + \cdots + i_m + 1, mn-(i_1 + \cdots + i_{k-1} + \ell + i_{k+1} + \cdots + i_m)+1) 
\\
&= \left[(mn+1)\binom{m n}{i_1 + \cdots + i_{k-1} + \ell + i_{k+1} + \ldots i_{m}}\right]^{-1} 
\, ,
\end{align*}
since $\mathrm{B}(b+1 , a-b+1) = [(a+1) \binom{a}{b}]^{-1}$ 
for every $a,b \in \IN$ such that $a \geq b$, 
here used with $a=mn$ and 
$b=(i_1 + \cdots + i_{k-1} + \ell + i_{k+1} + \cdots + i_m)$. 
We end the proof using the definition of the expectation 
(and the fact that $U_{(\ell,k)}$ is supported by $[0,1]$): 
\begin{align*}
\E[ U_{(\ell,k)} ] 
&= \int_{0}^{1} t \cdot f_{U_{(\ell, k)}}(t) \mathrm{d}t 
\\
&= 
\dfrac{ 
\displaystyle k \binom{m}{k} \sum^{n}_{i_1=\ell} \ldots \sum^{n}_{i_{k-1}=\ell} \sum^{n}_{i_{k+1}=0} \ldots \sum^{n}_{i_{m}=0} 
		\dfrac{\binom{n}{i_1} \cdots \binom{n}{i_{k-1}} \binom{n}{i_{k+1}} \cdots \binom{n}{i_{m}}}{\binom{m n}{i_1 + \cdots + i_{k-1} + \ell + i_{k+1} + \ldots + i_{m}}}
	}{(m n + 1)\mathrm{B}(\ell, n-\ell+1)} 
\, .
\end{align*}
\qed

\begin{remark}
The first part of the proof of Theorem~\ref{them:main} 
---proving Eq.~\eqref{eq:main_equa} in general, 
and showing that it is an equality when $F_S$ is continuous--- 
is identical to the first part of the proof of 
\cite[Theorem~3.2]{humbert2023one}. 
The second part of the proof is new\textup{:}  
we here compute $\E[U_{(\ell,k)}]$ from its pdf 
---contrary to \cite[Theorem~3.2]{humbert2023one} 
which uses its cdf---, 
which leads to a slightly simpler formula for $M_{\ell,k}$. 
\end{remark}

\subsection{Proof of Lemma~\ref{le.alg:FedCPQQ.non-trivial}} \label{le.alg:FedCPQQ.non-trivial_proof}
The argmin defining $(\ell^*, k^*)$ 
in Algorithm~\ref{alg:FedCPQQ} is non-empty 
if and only if 
\[
\max_{(\ell, k) \in \intset{n} \times \intset{m}} M_{\ell, k} 
\geq 1 - \alpha
\, . 
\]
Since $M_{\ell,k} = \E[U_{(\ell:n, k:m)}]$ 
---see Eq.~\eqref{eq.thm:main_proof.1} in Appendix~\ref{thm:main_proof}---, 
it is a nondecreasing function of $(\ell,k)$ 
(for the lexicographic order), 
hence  
\[
\max_{(\ell, k) \in \intset{n} \times \intset{m}} M_{\ell, k} 
= M_{n,m} 
= \E[U_{(n:n,m:m)}]
= \E \Bigl[ \max_{(i,j) \in \intset{n} \times \intset{m}} U_{i,j} \Bigr] 
\]
where the $(U_{i,j})_{(i,j) \in \intset{n} \times \intset{m}}$ 
are independent standard uniform random variables, 
hence  
\[
M_{n,m} 
= \E[U_{(mn : mn)}] 
= \frac{mn}{mn+1} 
\]
by Eq.~\eqref{eq.unif-order-stat.E-Var} 
with $r=N=mn$.  
Finally, we have proved that 
the argmin defining $(\ell^*, k^*)$ 
in Algorithm~\ref{alg:FedCPQQ} is non-empty 
if and only if
\[
\frac{mn}{mn+1}  = 1 - \frac{1}{mn+1} \geq 1 - \alpha
\, , 
\]
which is equivalent to 
$mn+1 \geq \alpha^{-1}$. 
\qed

\subsection{Proof of Proposition~\ref{cor:bound_Ekl}} \label{cor:bound_Ekl_proof}
Recall that by Eq.~\eqref{eq.thm:main_proof.1} 
in Appendix~\ref{thm:main_proof}, 
$M_{\ell,k} = \E [ U_{(\ell:n, k:m)} ]$ 
where $U_{(\ell:n, k:m)}$ is the 
$k$-th order statistics  of a sample 
$( U_{(\ell:n), j} )_{1 \leq j \leq m}$ of $m$ independent random variables 
with common distribution $\Beta (\ell, n - \ell + 1)$ 
and cdf $F_{U_{(\ell:n)}}$. 
Since $\ell \geq 1$ and $n-\ell+1 \geq 1$, 
the function $1-F_{U_{(\ell:n)}}$ is log-concave 
\cite[Section~6.3 and Theorem~3]{bagnoli2006log} 
and thus 
\citep[Equation (4.5.7)]{david2004order} 
shows that
\begin{align}
\notag 
F_{U_{(\ell:n)}} \bigl( \E [ U_{(\ell:n, k:m)} ]\bigr) 
&= F_{U_{(\ell;n)}} \Bigl( \E \bigl[ \Qh_{(k)}(U_{(\ell:n), 1}, \ldots, U_{(\ell:n), m}) \bigr]\Bigr)
\\
&\leq 1 - \exp\left( -\sum_{i=0}^{k-1} \frac{1}{m-i} \right)
< \frac{k}{m+1/2} 
\label{cor:bound_Ekl_upbound}
\, , 
\end{align}
hence the upper bound 
of Eq.~\eqref{eq:bound_marg}.

For the lower bound, we first remark that 
$\E[U_{(\ell, k)}] = 1 - \E[U_{(n-\ell+1, m-k+1)}]$.  
Indeed, defining $\tU_{i,j} = 1 - U_{i,j}$ 
for every $i,j \in \intset{n} \times \intset{m}$, 
on the one hand, we have $\tU_{(\ell,k)} = 1 - U_{(n-\ell+1,m-k+1)}$,  
and on the other hand, the fact that 
$(\tU_{i,j})_{ i,j \in \intset{n} \times \intset{m} } 
\egalloi (U_{i,j})_{ i,j \in \intset{n} \times \intset{m} }$ 
implies that $U_{(\ell,k)} \egalloi \tU_{(\ell,k)}$. 
So, we have proved that 
\begin{equation} \notag 
U_{(\ell:n, k:m)} 
\egalloi 1 - U_{((n-\ell+1):n , (m-k+1):m)} 
\end{equation}
hence the result by taking the expectation. 
Therefore, using the upper bound 
of Eq.~\eqref{eq:bound_marg}, 
we get that 
\begin{align*}
M_{\ell,k} 
= 1 - M_{n-\ell+1, m-k+1} 
> 1 - F_{U_{(n-\ell+1:n)}}^{-1} \left( \frac{m-k+1}{m+1/2} \right) 
%= F_{U_{(\ell:n)}}^{-1} \left( 1 - \frac{m-k+1}{m+1/2} \right) 
= F_{U_{(\ell:n)}}^{-1} \left( \frac{k-1/2}{m+1/2} \right) 
\end{align*}
since $F^{-1}_{U_{(n-\ell+1:n)}}(x) 
= 1 -  F^{-1}_{U_{(\ell:n)}}(1-x)$, 
again by a symmetry argument ($U_{(\ell:n)} \egalloi 1 - U_{(n-\ell+1:n)}$) 
\cite[see also][about this property of the incomplete Beta function]{temme1996special}. 
\qed

\subsection{Proof of Lemma~\ref{le.algo.FCP-QQ-marg.2.non-trivial}} \label{le.algo.FCP-QQ-marg.2.non-trivial_proof}
The argmin defining $\tilde{\ell}$ 
in Algorithm~\ref{algo.FCP-QQ-marg.2} is non-empty 
if and only if 
\[
\min_{\ell \in \intset{n}} \tilde{k}_{m, n}(\ell, \alpha) 
\leq m
\, . 
\]
Since $F_{U_{(\ell:n)}}(1-\alpha) = \P( U_{(\ell:n)} \leq 1-\alpha)$, 
it is a nonincreasing function of $\ell$, 
hence, by Eq.~\eqref{eq:k_low_bound}, 
\begin{align*}
\min_{\ell \in \intset{n}} \tilde{k}_{m, n}(\ell, \alpha) 
= \tilde{k}_{m, n}(n, \alpha) 
&= \left\lceil (m+1/2) \cdot F_{U_{(n:n)}}(1-\alpha) + 1/2 \right\rceil 
\\
&= \left\lceil (m+1/2) (1-\alpha)^n + 1/2 \right\rceil 
\, , 
\end{align*}
using Eq.~\eqref{eq.FUl.pdf-cdf} in Appendix~\ref{app.more-details-order-stat.exact}.  
This directly leads to the necessary and sufficient 
condition~\eqref{eq.le.algo.FCP-QQ-marg.2.non-trivial.CNS}. 

Note that Eq.~\eqref{eq.le.algo.FCP-QQ-marg.2.non-trivial.CNS} 
can be rewritten 
\[
\left( 1 + \frac{1}{m-1/2} \right)^{1/n} 
\leq \frac{1}{1-\alpha} 
\, . 
\]
The function $x \in [0,+\infty) \mapsto (1+x)^{1/n}$ 
being concave, 
it is below its first-order approximation at $x=0$,
so that 
\[
\left( 1 + \frac{1}{m-1/2} \right)^{1/n} 
\leq 1 + \frac{1}{n(m-1/2)} 
\]
and, if $n (m-1/2) \geq \alpha^{-1} - 1$, 
this upper bound is smaller than 
\[
1 + \frac{1}{\alpha^{-1} - 1} 
= \frac{\alpha^{-1}}{\alpha^{-1} - 1} 
= \frac{1}{1-\alpha} 
\, . 
\]
\qed

\section{Proofs of Section \ref{sec.multi-order.algos}  (training-conditional guarantees)}
% !TEX root = ../main.tex

\subsection{Proof of Theorem~\ref{them:cond_main}} \label{app:proof-them:cond_main}
As previously, we make this proof conditionally to $\calD^{lrn}$ 
without writing it explicitly in the following. 
Since $(X_{i,j}, Y_{i,j})_{1 \leq i \leq n , 1 \leq j \leq m}, (X, Y)$ are i.i.d., 
the associated scores $(S_{i,j})_{1 \leq i \leq n , 1 \leq j \leq m}, S$ are i.i.d. 
We denote by $F_S$ their cdf (given $\calD^{lrn}$). 
Now, notice that 
\begin{align}
\notag 
&1 - \alpha_{\ell, k}(\calD) 
:= \IP\bigl(Y \in \Chat_{\ell, k}(X) \,\big\vert\, \calD \bigr)
= \IP\bigl( S \leq S_{(\ell, k)} \,\big\vert\, (S_{i,j})_{1 \leq i \leq n , 1 \leq j \leq m} \bigr) 
= F_S(S_{(\ell, k)})
\, .
\end{align}
Therefore, applying Lemma~\ref{lemma:centra_unif_kl} 
with $F_Z = F_S$, 
Eq.~\eqref{eq.lemma:centra_unif_kl.domin-stoch.conseq} 
with $a = F_{U_{(\ell,k)}}^{-1} (\beta)$ 
shows that 
\begin{align}
\notag 
\P \bigl( 
	1 - \alpha_{\ell, k}(\calD) \geq F_{U_{(\ell,k)}}^{-1} (\beta) 
	\bigr) 
&= \P \bigl( 
	F_S(S_{(\ell, k)}) \geq F_{U_{(\ell,k)}}^{-1} (\beta) 
	\bigr)
\\
\label{eq.pr.them:cond_main.1}
&\geq \P \bigl( 
	U_{(\ell,k)} \geq F_{U_{(\ell,k)}}^{-1} (\beta) 
	\bigr)
= 1 - \beta 
\, ,
\end{align}
that is, Eq.~\eqref{eq:cond_main_equa} holds true. 
The formula for $F_{U_{(\ell,k)}}^{-1}$ 
follows from Lemma~\ref{lemma:beta_beta}. 
When the scores are a.s. distinct 
---or equivalently when $F_S$ is continuous---, 
Eq.~\eqref{eq.lemma:centra_unif_kl.egal-loi} 
in Lemma~\ref{lemma:centra_unif_kl} 
shows that 
$1 - \alpha_{\ell, k}(\calD) \egalloi U_{(\ell,k)}$ 
hence Eq.~\eqref{eq.pr.them:cond_main.1} 
is an equality and 
Eq.~\eqref{eq:cov_two_side} holds true. 
\qed 

\subsection{Proof of Lemma~\ref{le.algo.FCP-QQ-cond.1.non-trivial}}
\label{app.pr.le.algo.FCP-QQ-cond.1.non-trivial}
The $\argmin$ defining $(\ell^*_c,k^*_c)$ 
in Algorithm~\ref{algo.FCP-QQ-cond.1} is non-empty 
if and only if 
\[
\max_{(\ell, k) \in \intset{n} \times \intset{m}} 
	F_{U_{(\ell:n, k:m)}}^{-1} (\beta) 
\geq 1 - \alpha 
\, . 
\]
Since $U_{(\ell:n, k:m)}$ is almost surely a nondecreasing 
function of $(\ell,k)$ for the lexicographic order, 
$F_{U_{(\ell:n, k:m)}}^{-1} (\beta)$ 
is a nondecreasing function of $(\ell,k)$, 
hence 
\[
\max_{(\ell, k) \in \intset{n} \times \intset{m}} 
	F_{U_{(\ell:n, k:m)}}^{-1} (\beta) 
= F_{U_{(n:n, m:m)}}^{-1} (\beta)
= F_{U_{(nm : nm)}}^{-1} (\beta)
= \beta^{1/(mn)} 
\]
by Eq.~\eqref{eq.max-uniformes.cdf} in Appendix~\ref{app.more-details-order-stat.exact}.  
So, $(\ell^*_c,k^*_c)$ is well defined 
if and only if $\beta^{1/(mn)} \geq 1 - \alpha$. 
The conclusion follows. 
\qed

\subsection{Proof of Proposition~\ref{prop:low_bound_Fl}} \label{prop:low_bound_Fl_proof}
By Eq.~\eqref{eq.lemma:beta_beta.cdf-U} 
in Lemma~\ref{lemma:beta_beta}, 
$F^{-1}_{ U_{(\ell:n, k:m)} } (\beta) 
= F^{-1}_{U_{(\ell:n)}} \circ F^{-1}_{U_{(k:m)}} (\beta)$. 
Furthermore, Eq.~\eqref{eq.pro.FUr.dev-gauche.inv} in Proposition~\ref{pro.FUr} 
with $N=m$, $r=k$ and $\delta = \beta$ shows that 
\begin{equation} \notag 
	F_{U_{(k:m)}}^{-1} (\beta) \geq \frac{k}{m+1} - \sqrt{\dfrac{\log(1/\beta)}{2(m+2)}} \; .
\end{equation}
This proves Eq.~\eqref{eq.cond-cov-Xlk.lower.minor.1} 
since $F^{-1}_{U_{(\ell:n)}}$ is nondecreasing. 
\qed

\subsection{Proof of Lemma~\ref{le.algo.FCP-QQ-cond.2.non-trivial}}
\label{app.pr.le.algo.FCP-QQ-cond.2.non-trivial}
The $\argmin$ defining $\tilde{\ell}^{\text{cond}}$ 
in Algorithm~\ref{algo.FCP-QQ-cond.2} is non-empty 
if and only if 
\[
\min_{\ell \in \intset{n}} \tilde{k}^{\text{cond}}_{m,n} (\ell, \alpha, \beta) 
\leq m
\, . 
\]
Since $F_{U_{(\ell:n)}} (1-\alpha)$ 
is a nonincreasing function of $\ell$, 
using also Eq.~\eqref{eq.max-uniformes.cdf} in Appendix~\ref{app.more-details-order-stat.exact}, 
we have 
\[
\min_{\ell \in \intset{n}} \tilde{k}^{\text{cond}}_{m,n} (\ell, \alpha, \beta) 
= \tilde{k}^{\text{cond}}_{m,n} (n, \alpha, \beta)
= \left\lceil (m+1) \left( (1-\alpha)^n + \sqrt{\dfrac{\log(1/\beta)}{2(m+2)}} \, \right) \right \rceil 
\]
hence the $\argmin$ defining $\tilde{\ell}^{\text{cond}}$ 
in Algorithm~\ref{algo.FCP-QQ-cond.2} is non-empty 
if and only if
\[
(m+1) \left( (1-\alpha)^n + \sqrt{\dfrac{\log(1/\beta)}{2(m+2)}} \, \right) 
\leq m 
\, , 
\]
which is equivalent to 
Eq.~\eqref{eq.le.FCP-QQ-cond.2.non-trivial.CNS}. 

For the sufficient condition, 
remark that when $m \geq 2$ we have 
\[
\frac{m}{m+1} \geq \frac{2}{3}
\, . 
\]
If in addition 
\[
m \geq \frac{9}{2}  \log ( 1 / \beta) - 2 
\, , 
\]
then 
\[
\sqrt{\dfrac{\log(1/\beta)}{2(m+2)}} 
\leq \frac{1}{3}
\, , 
\qquad \text{hence} \qquad 
\frac{m}{m+1} - \sqrt{\dfrac{\log(1/\beta)}{2(m+2)}} 
\geq \frac{1}{3}
\, . 
\]
Finally, if we also have 
$n \geq \frac{\log(1/3)}{\log(1-\alpha)}$, 
then $(1-\alpha)^n \leq 1/3$ and we get that 
Eq.~\eqref{eq.le.FCP-QQ-cond.2.non-trivial.CNS} 
holds true, which proves 
Eq.~\eqref{eq.le.FCP-QQ-cond.2.non-trivial.CS}. 
\qed 

\section{Proofs of Section \ref{sec.uppbound_cov}  (upper bounds)} \label{sec.pr.upper-bounds}
% !TEX root = ../main.tex

\subsection{Proof of Theorem~\ref{thm:cov_up_margin}}\label{thm:cov_up_margin_proof}

Let $\alpha \in (0,1)$ be fixed. 
Let $(\ell_1,k_1) \egaldef (\ell^{*},k^{*})$ as defined by Algorithm~\ref{alg:FedCPQQ}, 
and $(\ell_2,k_2) \egaldef (\tilde{\ell},\tilde{k}_{m, n}(\tilde{\ell}, \alpha))$ 
as defined by Algorithm~\ref{algo.FCP-QQ-marg.2}. 
First, by definition of Algorithms~\ref{alg:FedCPQQ}--\ref{algo.FCP-QQ-marg.2} 
and by Proposition~\ref{cor:bound_Ekl}, we have 
\begin{align}
1 - \alpha 
\leq M_{\ell_1,k_1} 
&\egaldef \min_{(\ell,k)} \left\{M_{\ell, k} : M_{\ell, k} \geq 1-\alpha\right\} 
\leq M_{\ell_2,k_2} 
\leq F_{U_{(\ell_2:n)}}^{-1} \left(\dfrac{\tilde{k}_{m, n}(\ell_2, \alpha)}{m+1/2}\right)  
\, .
\label{eq.sec.cov-upper.marginal.Alg1}
\end{align}
Second, by definition of Algorithm~\ref{algo.FCP-QQ-marg.2}, 
for every $\ell \in \intset{n}$ such that 
$\tilde{k}_{m, n}(\ell, \alpha) \leq m$, we have 
\begin{align}
\notag 
F_{U_{(\ell_2:n)}}^{-1} \left(\dfrac{\tilde{k}_{m, n}(\ell_2, \alpha)}{m+1/2}\right)
&\leq F_{U_{(\ell:n)}}^{-1} \left(\dfrac{\tilde{k}_{m, n}(\ell, \alpha)}{m+1/2}\right) 
\\
&=  F_{U_{(\ell:n)}}^{-1} \left(\dfrac{\left\lceil (m+1/2) \cdot F_{U_{(\ell:n)}}(1-\alpha) + 1/2 \right\rceil}{m+1/2}\right) \notag
\\
&\leq
F_{U_{(\ell:n)}}^{-1} \left( F_{U_{(\ell:n)}}(1-\alpha) + \frac{3}{2m+1}\right)
\, . 
\label{eq.sec.cov-upper.marginal.Alg2.1}
\end{align}
By Theorem~\ref{thm:len_bound2} with 
$x = F_{U_{(\ell:n)}}(1-\alpha)$, 
hence $F^{-1}_{U_{(\ell:n)}}(x) = 1-\alpha$, 
and $\varepsilon = \frac{3}{2m+1}$, 
we get that if $0 < F_{U_{(\ell:n)}}(1-\alpha) < 1 - \frac{3}{2m+1}$, 
then 
\begin{align}
&\qquad 
F_{U_{(\ell:n)}}^{-1} \left( F_{U_{(\ell:n)}}(1-\alpha) + \frac{3}{2m+1}\right) \notag 
\\
&\leq 1 - \alpha + \frac{3}{(2m+1)\sqrt{n+2}} 
\, \frac{\sqrt{2}}{g( \min\{ F_{U_{(\ell:n)}}(1-\alpha) , 1-F_{U_{(\ell:n)}}(1-\alpha)-\frac{3}{2m+1} \})}
\, . 
\label{eq.sec.cov-upper.marginal.Alg2.2}
\end{align}
Third, note that for every $\alpha \in (0,1)$, 
by Eq.~\eqref{eq.uniform-order-stat.loi-asympt} 
in Appendix~\ref{app.more-details-order-stat.asymp} 
with $p = 1-\alpha$, 
\begin{align}
\notag 
\sqrt{n}\left(U_{(\lceil (1-\alpha) n \rceil)} - (1-\alpha)\right) 
&\xrightarrow[n \to + \infty]{\mathcal{L}} 
\mathcal{N} \bigl( 0, \alpha (1-\alpha) \bigr) 
\\
\text{hence} \qquad 
F_{U_{(\lceil (1-\alpha) n \rceil:n)}}(1-\alpha)
&\xrightarrow[n \to + \infty]{} 
\frac{1}{2} 
\label{eq.lim.FUnr(r)}
\, . 
\end{align}
As a consequence, some $n_0(\alpha)$ exists such that, for every $n \geq n_0(\alpha)$, 
we have 
$1/3 \leq F_{U_{(\lceil (1-\alpha) n \rceil:n)}}(1-\alpha) \leq 7/12$. 
Therefore, for every $n \geq n_0(\alpha)$ and $m \geq 18$, 
\[
\min\left\{ 
	F_{U_{(\lceil (1-\alpha) n \rceil:n)}}(1-\alpha) 
	\, , \, 
	1-F_{U_{(\lceil (1-\alpha) n \rceil:n)}}(1-\alpha)-\frac{3}{2m+1} 
	\right\}
\geq \frac{1}{3} 
\]
and by Eq.~\eqref{eq.sec.cov-upper.marginal.Alg1}-- 
\eqref{eq.sec.cov-upper.marginal.Alg2.2}, we get that 
\begin{align}
\notag 
1-\alpha 
\leq M_{\ell_1,k_1} 
\leq M_{\ell_2,k_2} 
\leq 1 - \alpha + \frac{27}{(2m+1)\sqrt{n+2}} 
\, .
\end{align}
Note that the choice of the constants $1/3$ and $7/12$ in the proof is arbitrary. 
Other choices would lead to the same result with other values for 
$C$, $m_0$ and $n_0(\alpha)$. 
For instance, $C$ can be made as close to $3\sqrt{2}/g(1/2) \leq 16.1$ as desired, 
at the price of enlarging $m_0$ and $n_0(\alpha)$. 
\qed 

\subsection{Proof of Theorem~\ref{thm:cov_up_cond}}\label{thm:cov_up_cond_proof}
The proof relies on the following lemma. 
\begin{lemma} \label{le.cov-upper.conditional}
Let $n,m \geq 1$, $\ell \in \intset{n}$, 
$\alpha,\beta,\beta' \in (0,1)$ 
and let $\tilde{k}^{\text{cond}}_{m,n} (\ell, \alpha, \beta)$ 
be defined by Eq.~\eqref{eq.algo.FCP-QQ-cond.k}. 
Then, some $n'_0(\alpha), m'_0 (\beta , \beta')  \geq 1$ exist such that, 
if $n \geq n'_0(\alpha)$ and $m \geq m'_0 (\beta , \beta') $, 
we have 
$\lceil n (1-\alpha) \rceil \in \intset{n}$, 
$\tilde{k}^{\text{cond}}_{m,n} ( \lceil n (1-\alpha) \rceil , \alpha, \beta) \in \intset{m}$ 
and
\begin{equation}
\label{eq.le.cov-upper.conditional}
F_{ U_{( \lceil n (1-\alpha) \rceil :n , 
	\tilde{k}^{\text{cond}}_{m,n} ( \lceil n (1-\alpha) \rceil , \alpha, \beta):m)} }^{-1} (1-\beta')
\leq 1-\alpha 
	+ \frac{12 \max\{\sqrt{\log(1/\beta)} + \sqrt{\log(1/\beta')} , 1 \}}{\sqrt{(m+2)(n+2)}} 
\, .
\end{equation}
\end{lemma}
\begin{proof}
The first condition holds true when $n \geq 1/\alpha$, 
hence it suffices to take $n'_0(\alpha) \geq n'_{0,a}(\alpha) = 1/\alpha$. 
For the second condition, which can be rewritten
\[
\tilde{k}^{\text{cond}}_{m,n} \bigl( \lceil n (1-\alpha) \rceil , \alpha, \beta \bigr)
= \left\lceil (m+1) \left( F_{ U_{( \lceil n (1-\alpha) \rceil :n)} } (1-\alpha) 
	+ \sqrt{\frac{\log(1/\beta)}{2(m+2)}} \, \right) \right \rceil
\leq m 
\, , 
\]
we first notice that by Eq.~\eqref{eq.lim.FUnr(r)}, 
some $n'_{0,b}(\alpha)$ exists such that, for every $n \geq n'_{0,b}(\alpha)$, 
\begin{equation}
\label{eq.pr.le.cov-upper.conditional.asympt}
\frac{1}{3} 
\leq F_{U_{(\lceil (1-\alpha) n \rceil:n)}}(1-\alpha) 
\leq \frac{2}{3} 
\, . 
\end{equation}
Therefore, some $m'_{0,a}(\beta)$ exists such that, 
if $m \geq m'_{0,a}(\beta)$ and $n \geq n'_{0,b}(\alpha)$, 
\[
\tilde{k}^{\text{cond}}_{m,n} \bigl( \lceil n (1-\alpha) \rceil , \alpha, \beta \bigr)
\leq \left\lceil (m+1) \left( \frac{2}{3} 
	+ \sqrt{\dfrac{\log(1/\beta)}{2(m+2)}} \, \right) \right \rceil
\leq m 
\, , 
\]
hence the second condition holds true if 
$m'_0 (\beta,\beta') \geq m'_{0,a}(\beta)$. 

It remains to prove Eq.~\eqref{eq.le.cov-upper.conditional}. 
For every $\ell \in \intset{n}$, 
if $\tilde{k}^{\text{cond}}_{m,n} ( \ell , \alpha, \beta) \in \intset{m}$, 
we have 
\begin{align}
\notag 
&
F_{ U_{( \ell :n, \tilde{k}^{\text{cond}}_{m,n} ( \ell , \alpha, \beta):m)} }^{-1} (1-\beta') 
\\
\notag 
&= F_{ U_{( \ell :n)} }^{-1} \circ 
	F_{U_{( \tilde{k}^{\text{cond}}_{m,n} ( \ell , \alpha, \beta):m)} }^{-1} (1-\beta') 
\qquad \text{by Eq.~\eqref{eq.lemma:beta_beta.cdf-U} in Lemma~\ref{lemma:beta_beta}}
\\
\notag 
&\leq F_{ U_{( \ell :n)} }^{-1} \left( 
\frac{\tilde{k}^{\text{cond}}_{m,n} ( \ell , \alpha, \beta)}{m+1} + \sqrt{\dfrac{\log(1/\beta')}{2(m+2)}} \, 
\right)
\qquad \text{by Eq.~\eqref{eq.pro.FUr.dev-droite.inv} in Proposition~\ref{pro.FUr}}
\\
\notag 
&\leq F_{ U_{( \ell :n)} }^{-1} \left( 
F_{ U_{( \ell :n)} }(1-\alpha) 
+ \frac{\sqrt{\log(1/\beta)} + \sqrt{\log(1/\beta')}}{\sqrt{2(m+2)}} + \frac{1}{m+1}
\right)
\quad \text{by Eq.~\eqref{eq.algo.FCP-QQ-cond.k}} \\
&\leq 1-\alpha + 
\frac{\varepsilon_m }{\sqrt{n+2}} 
\, \frac{\sqrt{2}}{g( \min\{ x_n , 1-x_n-\varepsilon_m \})}
\label{eq.pr.le.cov-upper.conditional.1}
\end{align}
by Theorem~\ref{thm:len_bound2} 
with 
\[ 
x = x_n \egaldef F_{ U_{( \ell :n)} }(1-\alpha) 
\qquad \text{and} \qquad   
\varepsilon = \varepsilon_m \egaldef 
\frac{\sqrt{\log(1/\beta)} + \sqrt{\log(1/\beta')}}{\sqrt{2(m+2)}} + \frac{1}{m+1} 
\, , 
\] 
assuming also that $x_n+\varepsilon_m < 1$. 
Now, taking $\ell = \lceil (1-\alpha) n \rceil$ 
in Eq.~\eqref{eq.pr.le.cov-upper.conditional.1}, 
when $m \geq m'_{0,a}(\beta)$ and 
$n \geq \max \{ n'_{0,a}(\alpha) , n'_{0,b}(\alpha) \}$, 
we have $\ell \in \intset{n}$, 
$\tilde{k}^{\text{cond}}_{m,n} ( 
	\lceil (1-\alpha) n \rceil , \alpha, \beta) \in \intset{m}$ 
and $x_n \in [1/3,2/3]$ by Eq.~\eqref{eq.pr.le.cov-upper.conditional.asympt}. 
Since some $m'_{0,b}(\beta,\beta')$ exists such that 
$\varepsilon_m \leq 1/12$ as soon as $m \geq m'_{0,b}(\beta,\beta')$, 
we get that for $m \geq \max\{ m'_{0,a}(\beta) , m'_{0,b}(\beta,\beta') \}$ 
and $n \geq \max \{ n'_{0,a}(\alpha) , n'_{0,b}(\alpha) \}$, 
we have $x_n + \varepsilon_m < 1$ and 
$\min\{ x_n , 1-x_n-\varepsilon_m \} \geq 1/4$,  
hence Eq.~\eqref{eq.pr.le.cov-upper.conditional.1} 
leads to 
\begin{align}
\notag 
&\qquad 
F_{ U_{( \lceil (1-\alpha) n \rceil :n, \tilde{k}^{\text{cond}}_{m,n} ( \lceil (1-\alpha) n \rceil , \alpha, \beta):m)} }^{-1} (1-\beta') 
\\
\notag 
&\leq 1-\alpha + 
\left( \frac{\sqrt{\log(1/\beta)} + \sqrt{\log(1/\beta')}}{\sqrt{2(m+2)}} 
	+ \frac{1}{m+1} \right) 
\frac{1}{\sqrt{n+2}} \, \frac{\sqrt{2}}{g( 1/4)}
\\
\notag 
&\stackrel{\text{(if $m \geq 12$)}}{\leq} 1-\alpha + 
\frac{12 \max\{\sqrt{\log(1/\beta)} + \sqrt{\log(1/\beta')} , 1 \}}{\sqrt{(m+2)(n+2)}} \, ,
\end{align}
where we assume at the last line that $m \geq 12$, 
hence $\sqrt{m+2}/(m+1) \leq 0.9949267 - 1/\sqrt{2}$, 
and we use that $g(1/4) \geq G(1/4 , 0.0316) \geq 0.1175$ 
hence $0.9949267 \times \frac{\sqrt{2}}{g(1/4)} 
\leq 11.9749 \leq 12$.
This ends the proof, by taking 
$m'_0 (\beta , \beta') = \max\{ m'_{0,a}(\beta) , m'_{0,b}(\beta,\beta') , 12 \}$ 
and $n'_0(\alpha) = \max\{ n'_{0,a}(\alpha) , n'_{0,b}(\alpha) \}$. 
\end{proof}

We can now prove Theorem~\ref{thm:cov_up_cond}.
We actually prove a more general result, about the following generalizations 
of Algorithms \ref{algo.FCP-QQ-cond.1} and ~\ref{algo.FCP-QQ-cond.2}, respectively, 
where $\beta'\in (0,1)$ is given:  
\begin{gather} 
\label{eq.def.algo.FCP-QQ-cond.1.var}
(\ell_A, k_A) 
\egaldef \argmin_{(\ell, k) \in \intset{n} \times \intset{m}} 
\left\{ F_{U_{(\ell:n, k:m)}}^{-1} (1-\beta') \, : \,F_{U_{(\ell:n, k:m)}}^{-1} (\beta) \geq 1 - \alpha 
	\right\} 
\\
\label{eq.def.algo.FCP-QQ-cond.2.var}
\ell_B 
\egaldef \argmin_{\ell \in \intset{n} \text{ s.t. } \tilde{k}^{\text{cond}}_{m,n} (\ell, \alpha, \beta) \leq m} 
	\left\{F^{-1}_{U_{(\ell:n, \tilde{k}^{\text{cond}}_{m,n} (\ell, \alpha, \beta):m)}} (1-\beta')\right\} 
\quad \text{and} \quad 
k_B 
\egaldef \tilde{k}^{\text{cond}}_{m,n} ( \ell_B , \alpha, \beta) 
\, . 
\end{gather}

\paragraph*{Algorithm~\ref{algo.FCP-QQ-cond.1} -- Proof of Eq.~\eqref{eq.thm:cov_up_cond.algo-cond-1}} 
Let us first consider the generalization of Algorithm~\ref{algo.FCP-QQ-cond.1}  
defined by Eq.~\eqref{eq.def.algo.FCP-QQ-cond.1.var}. 
Let $\ell_n \egaldef \lceil n(1-\alpha) \rceil$ 
and $k_n \egaldef \tilde{k}^{\text{cond}}_{m,n} ( \ell_n , \alpha, \beta)$. 
By Lemma~\ref{le.cov-upper.conditional}, 
when $n \geq n'_0(\alpha)$ and $m \geq m'_0 (\beta , \beta') $, 
we have 
$\ell_n \in \intset{n}$ and 
$k_n \in \intset{m}$, 
hence, 
$F_{ U_{( \ell_n :n, k_n :m)} }^{-1} (\beta) 
\geq 1-\alpha$ 
by Proposition~\ref{prop:low_bound_Fl} and Eq.~\eqref{eq.algo.FCP-QQ-cond.k}. 
Therefore, by definition 
\eqref{eq.def.algo.FCP-QQ-cond.1.var} of $(\ell_{A}, k_{A})$,  
\begin{align}
\notag 
F_{U_{(\ell_{A}:n, k_{A}:m)}}^{-1} (1-\beta')
&\leq 
F_{ U_{( \ell_n :n, k_n :m)} }^{-1} (1-\beta') 
\\
\notag 
&\leq 
1-\alpha + 
\frac{12 \max \bigl\{ \sqrt{\log(1/\beta)} + \sqrt{\log(1/\beta')} , 1 \bigr\}}{\sqrt{(m+2)(n+2)}} 
\end{align}
by Eq.~\eqref{eq.le.cov-upper.conditional} 
in Lemma~\ref{le.cov-upper.conditional}. 
Finally, using Eq.~\eqref{eq:cov_two_side} 
in Theorem~\ref{them:cond_main} 
---with $\beta=0$ and $(\ell,k) = (\ell_A,k_A)$---, 
we get that if the scores $S_{i,j},S$ are a.s. distinct, 
when $n \geq n'_0(\alpha)$ and $m \geq m'_0 (\beta , \beta') $, 
the following inequality holds true with probability  $1-\beta'$: 
\begin{align}
\label{eq.cov-upper.algo.FCP-QQ-cond-1.var-min-upper-bound}
1 - \alpha_{\ell_{A},k_{A}} (\calD) 
&\leq 1-\alpha 
+ \frac{12 \max \bigl\{\sqrt{\log(1/\beta)} + \sqrt{\log(1/\beta')} , 1 \bigr\}}{\sqrt{(m+2)(n+2)}} 
\, . 
\end{align} 
Eq.~\eqref{eq.thm:cov_up_cond.algo-cond-1} follows by taking 
$\beta'=\beta$ and defining $m'_0(\beta) \egaldef m'_0(\beta,\beta)$. 

\paragraph*{Algorithm~\ref{algo.FCP-QQ-cond.2} -- Proof of Eq.~\eqref{eq.thm:cov_up_cond.algo-cond-2}} 
Similarly, for the generalization of Algorithm~\ref{algo.FCP-QQ-cond.2} 
defined by Eq.~\eqref{eq.def.algo.FCP-QQ-cond.2.var},  
we get that when 
$n \geq n'_0(\alpha)$ and $m \geq m'_0 (\beta , \beta') $, 
the following inequality holds true with probability $1-\beta'$: 
\begin{align}
\label{eq.cov-upper.algo.FCP-QQ-cond-2.var-min-upper-bound}
1 - \alpha_{\ell_{B},k_{B}} (\calD) 
&\leq 1-\alpha 
+ \frac{12 \max \bigl\{\sqrt{\log(1/\beta)} + \sqrt{\log(1/\beta')} , 1 \bigr\}}{\sqrt{(m+2)(n+2)}} 
\, . 
\end{align}
Eq.~\eqref{eq.thm:cov_up_cond.algo-cond-2} follows by taking 
$\beta'=\beta$, 
since Eq.~\eqref{eq.pr.le.cov-upper.conditional.asympt} 
still holds true for such a sequence $\ell_n$ 
by \cite[Section~10.2]{david2004order}. 
\qed

\section{Proofs of Section \ref{sec:diff_n}  (different $n_j$)}
% !TEX root = ../main.tex

Before proving Theorems \ref{thm:main_nj}--\ref{them:cond_main_nj}, 
notice that the proof of Lemma~\ref{lemma:centra_unif_kl} 
straightforwardly generalizes to the case of the 
QQ estimator $S_{(\ell_1, \dots, \ell_m, k)}$ 
defined by Eq.~\eqref{eq:CP-QQ_nj}. 
So, if $(U_{i,j})_{i \in \intset{\ell_j} , j \in \intset{m}}$ 
denote independent standard uniform variables, we have 
\begin{equation}
\label{eq.nj-diff.lem-lien-unif}
F_S ( S_{(\ell_1, \dots, \ell_m, k)} ) 
\egalloi F_S \circ F_S^{-1} ( U_{(\ell_1, \dots, \ell_m, k)} )
\geq U_{(\ell_1, \dots, \ell_m, k)} 
\, , 
\end{equation}
with equality when $F_S$ is continuous. 	

\subsection{Proof of Theorem~\ref{thm:main_nj}}
\label{thm:main_nj_proof}
First, by Eq.~\eqref{eq.nj-diff.lem-lien-unif}, we have 
\begin{align*}
\P\left(Y \in \Chat_{\ell_1, \dots, \ell_m, k}(X)\right) 
&= \P\left(S \leq S_{(\ell_1, \dots, \ell_m, k)}\right) 
\\
&= \E \bigl[ F_S ( S_{(\ell_1, \dots, \ell_m, k)} ) \bigr] 
\geq \E \bigl[ U_{(\ell_1, \dots, \ell_m, k)} \bigr] 
=: M_{\ell_1, \dots, \ell_m, k} 
\, . 
\end{align*}
It remains to prove the formula for $M_{\ell_1, \dots, \ell_m, k}$. 
The difference with the proof of Theorem~\ref{them:main} 
is that the variables 
$(\Qh_{(\ell_j:n_j)}(\mathcal{S}_j))_{1 \leq j \leq m}$  
are not identically distributed, 
so $U_{(\ell_1, \dots, \ell_m, k)}$ is the $k$-th order statistics 
of a set of independent but \emph{not} identically distributed 
(inid) random variables. 
Its cdf can still be computed as follows. 
By \cite[Equation (16)]{balakrishnan2007permanents}, we have, 
for every $t \in [0,1]$, 
\begin{align*}
&F_{U_{(\ell_1, \dots, \ell_m, k)}}(t) \\
&=\sum^{m}_{j=k} \sum_{a \in \mathcal{P}_j} F_{U_{(\ell_{a_1}:n_1)}}(t) 
\cdots F_{U_{(\ell_{a_j}:n_j)}}(t) \cdot 
\left[1 - F_{U_{(\ell_{a_{j+1}}:n_{j+1})}}(t) \right] 
\cdots \left[1 - F_{U_{(\ell_{a_m}:n_m)}}(t) \right] 
\, ,
\end{align*}
where $\mathcal{P}_{j}$ denotes the set of 
permutations $(a_1, \ldots, a_m)$ of $\{ 1, \ldots, m \}$ 
for which $a_1 < a_2 < \ldots < a_{j}$ and 
$a_{j + 1 } < a_{j + 2 } < \ldots < a_m$.
Then, using Eq. \eqref{eq.FUl.pdf-cdf} and~\eqref{eq.FUl.pdf-cdf.bis} 
in Appendix~\ref{app.more-details-order-stat.exact}  
and rearranging the terms, we get that for every $t \in [0,1]$, 
\begin{align*} 
&F_{U_{(\ell_1, \dots, \ell_m, k)}}(t) 
\\
&= \sum^{m}_{j=k} \sum_{\mathcal{P}_j} \sum^{n_{a_1}}_{i_1=\ell_{a_1}} 
\cdots \sum^{n_{a_{j}}}_{i_{j}=\ell_{a_{j}}} 
\sum^{\ell_{a_{j+1}} - 1}_{i_{j+1}=0} \cdots 
\sum^{\ell_{a_m}-1}_{i_m=0} 
\binom{n_{a_1}}{i_1}\cdots\binom{n_{a_m}}{i_m} 
t^{i_1 + \cdots + i_m} (1-t)^{n_1+\cdots+ n_m-(i_1 + \cdots + i_m)} 
\, . 
\end{align*}
We finally obtain the result using that 
\begin{align*}
M_{\ell_1, \ldots, \ell_m, k} 
&= \E [ U_{(\ell_1, \ldots, \ell_m, k)} ] 
= \int_0^1 \bigl[ 1 - F_{U_{(\ell_1, \dots, \ell_m, k)}} (t) \bigr] \mathrm{d}t 
\\
\text{and} \qquad 
&\qquad 
\int_0^1 t^{i_1 + \cdots + i_m} (1-t)^{n_1+\cdots+ n_m-(i_1 + \cdots + i_m)} \mathrm{d}t 
\\
&= B ( i_1 + \cdots + i_m + 1 , n_1+\cdots+ n_m-(i_1 + \cdots + i_m) + 1) 
\\
&= ( n_1+\cdots+ n_m + 1 )^{-1} \binom{n_1+\cdots+ n_m}{i_1 + \cdots + i_m}^{-1}
\, , 
\end{align*}
using a property of the Beta function already used 
in Appendix~\ref{thm:main_proof}. 
\qed

\subsection{Proof of Theorem~\ref{them:cond_main_nj}}
\label{them:cond_main_nj_proof}
By the definition \eqref{eq:set_nj} of 
$\Chat_{\ell_1, \dots, \ell_m, k}$ and 
Eq.~\eqref{eq.nj-diff.lem-lien-unif}, we have 
\begin{align}
\notag  
\P \bigl( 
	1-\alpha_{\ell_1, \dots, \ell_m, k}(\calD) 
	\geq 1-\alpha\bigr) 
&= \P \bigl( F_S ( S_{(\ell_1, \dots, \ell_m, k)} ) \geq 1-\alpha \bigr) 
\\
&\geq \P \bigl( U_{(\ell_1, \dots, \ell_m, k)} \geq 1-\alpha\bigr) 
\label{eq:miscov_nj_unif}
\, ,
\end{align}
with equality when $F_S$ is continuous. 
By definition of $U_{(\ell_1, \dots, \ell_m, k)} 
= \Qh_{(\ell_j:n_j)} ( ( U_{(\ell_j:n_j)} )_{1 \leq j \leq m} )$, 
we have 
\begin{align*}
\P \bigl(U_{(\ell_1, \dots, \ell_m, k)} \geq 1-\alpha\bigr)
&= \P \left(
	\sum_{j=1}^{m} \1_{ \{ U_{(\ell_j:n_j)} \leq 1-\alpha \} } \leq k-1 
	\right) 
\, . 
\end{align*}
Furthermore,  
$\displaystyle\sum_{j=1}^{m} \1_{\{  U_{(\ell_j:n_j)} \leq 1-\alpha \} }$ 
is a sum of $m$ independent Bernoulli variables 
with respective parameters 
$p_j = \P (U_{(\ell_j:n_j)} \leq 1-\alpha) 
= F_{U_{(\ell_j:n_j)}}(1-\alpha)$. 
Therefore, it follows a Poisson-Binomial distribution 
of parameters $(p_1, \ldots, p_m)$. 
Using Eq.~\eqref{eq:miscov_nj_unif}, we obtain 
Eq.~\eqref{eq:cond_main_equa_nj}.  
\qed

\section{Classical results about order statistics distribution}
\label{app.more-details-order-stat}
% !TEX root = ../main.tex

We recall in this section some useful well-known results 
about order statistics \cite{david2004order}. 
Throughout this section, 
$U_1, \ldots, U_N, \ldots $ denote a sequence of 
independent standard uniform variables, 
and $Z_1, \ldots, Z_N, \ldots$ a sequence of 
i.i.d. random variables with common cdf~$F_Z$. 
For any integers $1 \leq r \leq N$, 
$U_{(r:N)} = \Qh_{(r)} (U_1, \ldots, U_N)$ 
and $Z_{(r:N)} = \Qh_{(r)} (Z_1, \ldots, Z_N)$ 
denote the corresponding $r$-th order statistics. 

\subsection{Exact distribution} \label{app.more-details-order-stat.exact}
For every $r \in \intset{N}$, the cdf 
of $Z_{(r:N)}$ is given by
\begin{align}
\label{eq.Xr.cdf}
\forall t \in \R \, , \qquad 
F_{Z_{(r:N)}}(t) 
&= \sum_{i=r}^N \binom{N}{i} F_Z(t)^i \bigl[ 1 - F_Z(t) \bigr]^{N-i} 
\, .
\end{align}
If we further assume that $F_Z$ is continuous 
with corresponding probability density function (pdf) $f_Z = F_Z'$, 
then, for every $r \in \intset{N}$, the pdf of $Z_{(r:N)}$ is given by
\begin{align}
\label{eq.Xr.pdf}
\forall t \in \R \, , \qquad 
f_{Z_{(r:N)}}(t) 
&= \frac{N!}{(r-1)!(N-r)!} F_Z(t)^{r-1} \bigl[ 1 - F_Z(t) \bigr]^{N-r} f_Z(t)
\, .
\end{align}

\medbreak

In the case of uniform order statistics, 
$U_{(r:N)}$ follows a $\Beta(r, N-r+1)$ distribution 
and its cdf $F_{U_{(r:N)}}$ and pdf $f_{U_{(r:N)}}$ 
are respectively given by 
\begin{gather} 
\label{eq.FUl.pdf-cdf} 
\forall t \in [0,1] \, , \
F_{U_{(r:N)}}(t) 
= \sum_{i=r}^N \binom{N}{i} t^i (1 - t)^{N-i} 
\  \text{ and } \ 
f_{U_{(r:N)}}(t) = \dfrac{t^{r-1} (1-t)^{N-r}}{\mathrm{B}(r, N-r+1)}
\\
\notag 
\text{where} \quad  
\mathrm{B} : (a,b) \in (0,+\infty)^2 \mapsto \int_0^1 t^{a-1} (1-t)^{b-1} \mathrm{d}t
\end{gather}
denotes the Beta function \citep{temme1996special}. 
In particular, for every $N \geq 1$, we have 
\begin{align}
\label{eq.max-uniformes.cdf}
\forall t \in [0,1] \, , \qquad 
F_{U_{(N:N)}} (t) &= t^N 
\qquad \text{and} \qquad 
F_{U_{(N:N)}}^{-1} (t) = t^{1/N}
\, . 
\end{align}
Since $\sum_{i=0}^N \binom{N}{i} t^i (1 - t)^{N-i}  = 1$, 
Eq.~\eqref{eq.FUl.pdf-cdf} also implies that 
\begin{align} 
\label{eq.FUl.pdf-cdf.bis} 
\forall t \in [0,1] \, , \quad 
1 - F_{U_{(r:N)}}(t) 
&= \sum_{i=0}^{r-1} \binom{N}{i} t^i (1 - t)^{N-i} 
\, , 
\\
\label{eq.min-uniformes.cdf}
\text{hence} \qquad 
F_{U_{(1:N)}} (t) &= 1 - (1-t)^N 
\, .
\end{align}

In addition, closed-form formulas are available for the  
expectation and variance of uniform order statistics: 
\begin{equation}
\label{eq.unif-order-stat.E-Var}
\E \bigl[ U_{(r:N)} \bigr] 
= \frac{r}{N+1} 
\qquad \text{and} \qquad 
\Var \bigl( U_{(r:N)} \bigr) 
= \frac{r (N-r+1)}{(N+1)^2 (N+2)} 
\, . 
\end{equation}

\subsection{Asymptotics}
\label{app.more-details-order-stat.asymp}
Central order statistics are known to be asymptotically normal. 
More precisely, following \cite{Gho:1971,Lah:1992}, we have the following result. 
\begin{theorem}[Asymptotic normality of central uniform order statistics \cite{Gho:1971,Lah:1992}]
\label{thm.asympt-unif-order-stat}
Let $p \in (0,1)$, $\gamma \in \R$ and 
$(\ell_N)_{N \geq 1}$ be a sequence such that 
$\ell_N \in \intset{N}$ for every $N \geq 1$ 
and $\ell_N = p N + \gamma \sqrt{N} + \mathrm{o}(\sqrt{N})$ 
when $N \to +\infty$. 
Then, we have 
\begin{equation}
\label{eq.thm.asympt-unif-order-stat}
\sqrt{N} \bigl( U_{(\ell_N:N)} - p \bigr) 
\xrightarrow[N \to +\infty]{\mathcal{L}} 
\mathcal{N} \bigl( \gamma , p (1-p) \bigr) 
\, . 
\end{equation}
\end{theorem}
In particular, Theorem~\ref{thm.asympt-unif-order-stat} 
implies that 
\begin{equation}
\label{eq.uniform-order-stat.loi-asympt}
\sqrt{N} \bigl( U_{(\lceil p N \rceil:N)} - p \bigr) 
\xrightarrow[N \to +\infty]{\mathcal{L}} 
\mathcal{N} \bigl( 0 , p (1-p) \bigr) 
\, . 
\end{equation}
By Lemma~\ref{lemma:centra_unif}, 
Eq.~\eqref{eq.thm.asympt-unif-order-stat}--\eqref{eq.uniform-order-stat.loi-asympt} 
imply that when $F_Z$ is continuous, 
under the assumptions of Theorem~\ref{thm.asympt-unif-order-stat}, 
\begin{align}
\label{eq.rk.order-stat.asympt.FZ}
\sqrt{N}\bigl( F_Z( Z_{(\ell_N:N)} ) - p \bigr) 
&\xrightarrow[N \to + \infty]{\mathcal{L}} 
\mathcal{N}\bigl( \gamma , p (1-p) \bigr) 
\\
\label{eq.rk.order-stat.asympt.FZrhoN}
\text{and} \qquad 
\sqrt{N}\bigl( F_Z( Z_{(\lceil p N \rceil : N)} ) - p \bigr) 
&\xrightarrow[N \to + \infty]{\mathcal{L}} 
\mathcal{N}\bigl( 0, p (1-p) \bigr) 
\, .
\end{align} 

\section{Proofs of Appendix~\ref{app.order-stat}}
\subsection{Proof of Lemma~\ref{lemma:centra_unif}} \label{lemma:centra_unif_proof}

First, for every $i \in \intset{N}$, let us define 
$\Xu_{i} \egaldef F_Z^{-1}(U_{i})$, 
so that $\Xu_1, \ldots, \Xu_N$ are independent and identically distributed, 
with cdf~$F_Z$. 
Therefore, $(Z_1, \ldots, Z_N) \egalloi (\Xu_1, \ldots, \Xu_N)$, 
hence 
\begin{equation}
\label{eq.pr.lemma:centra_unif_proof.1}
Z_{(r)} \egalloi \Xu_{(r)}
\, . 
\end{equation}
Second, remark that for any integers $1 \leq r \leq N$,  
any $\mathcal{S} \in \R^N$ and any nondecreasing function $h: \R \to \R$, 
we have 
\begin{equation}
\label{eq.pr.lemma:centra_unif_proof.2.Qh}
h \bigl( \Qh_{(r)} ( \mathcal{S} ) \bigr) 
= \Qh_{(r)} \bigl( h(\mathcal{S}) \bigr)
\qquad \text{where} \qquad 
h(\mathcal{S}) \egaldef \bigl( h(\mathcal{S}_i))_{i \in \intset{N}} \bigr)
\end{equation}
and $\Qh_{(r)}$ is defined by Eq.~\eqref{def:Qk} 
in Section~\ref{subsec:splitCP}. 

Indeed, let $\sigma \in $ be some permutation 
of $\intset{N}$ 
such that $\mathcal{S}_{\sigma(1)} \leq \cdots \leq \mathcal{S}_{\sigma(N)}$, 
hence $\Qh_{(r)} ( \mathcal{S} ) = \mathcal{S}_{\sigma(r)}$. 
Then, 
$h( \mathcal{S}_{\sigma(1)} ) \leq \cdots \leq h ( \mathcal{S}_{\sigma(N)} )$ 
since $h$ is nondecreasing, 
and we get that 
\[ 
\Qh_{(r)} \bigl( h(\mathcal{S}) \bigr) 
= h( \mathcal{S}_{\sigma(r)} ) 
= h \bigl( \Qh_{(r)} ( \mathcal{S} ) \bigr) 
\, . 
\]

From Eq. \eqref{eq.pr.lemma:centra_unif_proof.1} 
and~\eqref{eq.pr.lemma:centra_unif_proof.2.Qh}, 
since $F_Z^{-1}$ is nondecreasing, 
we obtain that  
\begin{align*} 
F_Z (Z_{(r)}) 
\egalloi F_Z(\Xu_{(r)}) 
= F_Z \Bigl( \Qh_{(r)} \bigl( F_Z^{-1} ( U_1, \ldots, U_N ) \bigr) \Bigr)
&= F_Z \circ F_Z^{-1} \Bigl( \Qh_{(r)} ( U_1, \ldots, U_N ) \Bigr) 
\\
&= F_Z \circ F_Z^{-1} ( U_{(r)} )
\, . 
\end{align*} 
Using Eq.~\eqref{eq.prop-gal-inverse}, 
this proves Eq.~\eqref{eq.lemma:centra_unif.domin-stoch} 
in the general case 
---Eq.~\eqref{eq.lemma:centra_unif.domin-stoch.conseq} 
follows directly---, 
and Eq.~\eqref{eq.lemma:centra_unif.egal-loi} 
when $F_Z$ is continuous. 
The distribution of the uniform order statistic $U_{(r)}$ 
is well known, see for instance \cite{david2004order} 
or Appendix~\ref{app.more-details-order-stat.exact}.

\qed 

\begin{remark}
\label{rk.lemma:centra_unif.bonus}
In the case of a general cdf $F_Z$, another consequence of Lemma~\ref{lemma:centra_unif} 
is that, 
for every $a,b \in \R$ such that $a \leq b$, 
\begin{align}
\label{eq.rk.lemma:centra_unif.bonus.dev-droite-ineq}
\IP\left(U_{(r)} \leq b \right) 
&\leq \IP\left(F_Z(Z_{(r)}) \leq F_Z \circ F_Z^{-1}(b) \right)
\\
\label{eq.rk.lemma:centra_unif.bonus.dev-ineq}
\text{and} \qquad 
\IP\left(a \leq U_{(r)} \leq b \right) 
&\leq \IP\left(a \leq F_Z(Z_{(r)}) \leq F_Z \circ F_Z^{-1}(b)\right) 
\, . 
\end{align}
\end{remark}
\begin{proof}
Since $F_Z \circ F_Z^{-1}$ is nondecreasing, 
\begin{equation}
\label{eq.pr.rk.lemma:centra_unif.bonus}
U_{(r)} \leq b  
\qquad \text{implies that} \qquad  
F_Z \circ F_Z^{-1} (U_{(r)}) \leq F_Z \circ F_Z^{-1}(b) 
\, , 
\end{equation}
hence Eq.~\eqref{eq.rk.lemma:centra_unif.bonus.dev-droite-ineq} 
thanks to Eq.~\eqref{eq.lemma:centra_unif.domin-stoch}. 
Since $F_Z \circ F_Z^{-1} (U_{(r)}) \geq U_{(r)}$ 
by Eq.~\eqref{eq.prop-gal-inverse}, 
Eq.~\eqref{eq.pr.rk.lemma:centra_unif.bonus} also shows that 
\begin{equation}
\notag
a \leq U_{(r)} \leq b  
\qquad \text{implies that} \qquad  
a \leq F_Z \circ F_Z^{-1} (U_{(r)}) \leq F_Z \circ F_Z^{-1}(b) 
\, , 
\end{equation}
hence Eq.~\eqref{eq.rk.lemma:centra_unif.bonus.dev-ineq} 
thanks to Eq.~\eqref{eq.lemma:centra_unif.domin-stoch}. 
\end{proof}

\subsection{Proof of Proposition~\ref{pro.FUr}}
\label{sec.pr.cor.FUr}
We start by recalling some concentration 
result about the Beta distribution 
\cite[Theorem~1]{marchal2017sub}. 
\begin{theorem}[taken from \cite{marchal2017sub}]
\label{them:subgaussbeta}
For any $a, b > 0$, the Beta distribution 
$\Beta(a, b)$ is $\omega$-sub-Gaussian, 
with $\omega = \frac{1}{4 (a + b + 1)}$. 
This implies that, if $Z \sim \Beta(a, b)$, 
for any $\varepsilon > 0$,
\begin{align*}
\max\left\{ 
\IP\left(Z \geq \frac{a}{a+b} + \varepsilon\right) 
\, , \, 
\IP\left(Z \leq \frac{a}{a+b} - \varepsilon\right) 
\right\} 
\leq e^{- 2 (a + b + 1) \varepsilon^2} 
\, .
\end{align*}
\end{theorem}
Note that Theorem~\ref{them:subgaussbeta} does not provide 
the ``best'' concentration inequality for the Beta distribution. 
Indeed, although the bound is sharp for the symmetric case $a = b$, 
when the Beta distribution is skewed, 
Theorem~\ref{them:subgaussbeta} has been improved 
with the Bernstein-type bound of \cite{skorski2023bernstein}. 
Nevertheless, it is sufficient for our needs 
and yields simpler formulas. 

We can now prove Proposition~\ref{pro.FUr}. 
Since $U_{(r:N)}$ follows a $\Beta(r,N-r+1)$ distribution, 
Theorem~\ref{them:subgaussbeta} 
with $a=r$, $b=N-r+1$ and $\varepsilon = \sqrt{\frac{\log(1/\delta)}{2(N+2)}}$ 
shows that  
\begin{equation}
\label{eq.pro.FUr.dev-gauche}
F_{U_{(r:N)}} \left(
	\frac{r}{N+1} - \sqrt{\dfrac{\log(1/\delta)}{2(N+2)}}
	\, \right)
= \P \left( 
	U_{(r:N)} \leq \frac{r}{N+1} - \sqrt{\dfrac{\log(1/\delta)}{2(N+2)}}
	\, \right)
\leq \delta 
\, , 
\end{equation}
hence Eq.~\eqref{eq.pro.FUr.dev-gauche.inv}. 
Similarly, 
\begin{equation}
\label{eq.pro.FUr.dev-droite}
F_{U_{(r:N)}} \left(
	\frac{r}{N+1} + \sqrt{\dfrac{\log(1/\delta)}{2(N+2)}}
	\right)
= 1 - \P \left( 
	U_{(r:N)} \geq \frac{r}{N+1} + \sqrt{\dfrac{\log(1/\delta)}{2(N+2)}}
	\right) 
\geq 1-\delta 
\end{equation}
by Theorem~\ref{them:subgaussbeta}, 
hence Eq.~\eqref{eq.pro.FUr.dev-droite.inv}. 
\qed 

\subsection{Deviation inequalities for general order statistics}
\label{app.order-stat-gal.dev-ineq}
The combination of Theorem~\ref{them:subgaussbeta} 
with Lemma~\ref{lemma:centra_unif} 
yields the following deviation inequalities, 
which are not directly useful in our study of 
one-shot FL CP algorithms, 
but can be of interest in other contexts 
such as the non-asymptotic analysis of 
(centralized) split CP 
---see Appendix~\ref{app.pr.split-CP}. 
\begin{corollary} \label{coro:bernstein_Sl}
For any $N \geq 1$, 
let $Z_1, \ldots, Z_N$ be independent and 
identically distributed real-valued 
random variables with common cdf~$F_Z$. 
Then, for any $\delta > 0$ and $r \in \intset{N}$, 
the following inequality holds true\textup{:}
\begin{align}
\label{eq.coro:bernstein_Sl.dev-gauche}
\IP\left(F_Z(Z_{(r)}) \geq \frac{r}{N+1} - \sqrt{\dfrac{\log(1/\delta)}{2(N+2)}} \, \right) 
&\geq 1 - \delta 
\; .
\end{align}
Moreover, if $F_Z$ is continuous, we also have 
\begin{align}
\label{eq.coro:bernstein_Sl.dev-droite}
\P\left(F_Z(Z_{(r)}) \leq \frac{r}{N+1} + \sqrt{\dfrac{\log(1/\delta)}{2(N+2)}} \, \right) 
&\geq 1 - \delta 
\\
\text{and} \qquad 
\P\left( \left\lvert F_Z(Z_{(r)}) - \frac{r}{N+1} \right\rvert 
\leq \sqrt{\frac{\log(2/\delta)}{2(N+2)}} \, \right) 
&\geq 1 - \delta 
\, .
\label{eq.coro:bernstein_Sl.dev-sym}
\end{align} 
\end{corollary}
\begin{proof}
By Eq.~\eqref{eq.lemma:centra_unif.domin-stoch.conseq} in Lemma~\ref{lemma:centra_unif} 
with 
$a=\frac{r}{N+1} - \sqrt{\frac{\log(1/\delta)}{2(N+2)}}$ 
(and using its notation), 
we have 
\begin{equation}
\label{eq.pr.coro:bernstein_Sl.1}
\begin{split}
& \P \left( 
	F_Z(Z_{(r)}) 
	\geq \frac{r}{N+1} - \sqrt{\frac{\log(1/\delta)}{2(N+2)}} 
	\, \right)
\\
& \qquad \geq 
\P \left( 
	U_{(r:N)} 
	\geq \frac{r}{N+1} - \sqrt{\frac{\log(1/\delta)}{2(N+2)}} 
	\, \right) 
\geq 1 - \delta 
\end{split}
\end{equation}
by Eq.~\eqref{eq.pro.FUr.dev-gauche}
in the proof of Proposition~\ref{pro.FUr}, 
hence Eq.~\eqref{eq.coro:bernstein_Sl.dev-gauche}. 
If we also assume that $F_Z$ is continuous, 
combining Eq.~\eqref{eq.lemma:centra_unif.egal-loi} 
in Lemma~\ref{lemma:centra_unif} 
and Eq.~\eqref{eq.pro.FUr.dev-droite}
in the proof of Proposition~\ref{pro.FUr} 
shows that 
\begin{align*}
\P \left( 
	F_Z (Z_{(r:N)}) \leq \frac{r}{N+1} + \sqrt{\dfrac{\log(1/\delta)}{2(N+2)}}
	\right) 
&= \P \left( 
	U_{(r:N)} \leq \frac{r}{N+1} + \sqrt{\dfrac{\log(1/\delta)}{2(N+2)}}
	\, \right) 
\\
&= F_{U_{(r:N)}} \left(
	\frac{r}{N+1} + \sqrt{\dfrac{\log(1/\delta)}{2(N+2)}}
	\, \right)
\geq 1 - \delta 
\, , 
\end{align*}
hence Eq.~\eqref{eq.coro:bernstein_Sl.dev-droite}. 
Finally, replacing $\delta$ by $\delta/2$ in 
Eq.~\eqref{eq.coro:bernstein_Sl.dev-gauche}--\eqref{eq.coro:bernstein_Sl.dev-droite} 
implies Eq.~\eqref{eq.coro:bernstein_Sl.dev-sym} 
by the union bound. 
\end{proof}
\begin{remark}
For a general cdf $F_Z$, instead of Eq.~\eqref{eq.coro:bernstein_Sl.dev-droite}, 
we have  
\begin{equation}
\label{eq.coro:bernstein_Sl.dev-droite.gal} 
\IP\left(F_Z(Z_{(r)}) \leq F_Z\circ F_Z^{-1} \left( \frac{r}{N+1} + \sqrt{\dfrac{\log(1/\delta)}{2(N+2)}} \, \right) \right) 
\geq 1 - \delta 
\, . 
\end{equation}
\end{remark}
\begin{proof}
Combine Eq.~\eqref{eq.rk.lemma:centra_unif.bonus.dev-droite-ineq} 
with $b = \frac{r}{N+1} + \sqrt{\frac{\log(1/\delta)}{2(N+2)}}$ 
and Eq.~\eqref{eq.pro.FUr.dev-droite}. 
\end{proof}
\begin{remark} 
Eq.~\eqref{eq.rk.order-stat.asympt.FZrhoN} in Section~\ref{app.more-details-order-stat.asymp} 
shows that for any $p \in (0,1)$, 
$F_Z( Z_{(\lceil p N \rceil)} )$ 
converges towards $p$ at the rate $\sqrt{N}$. 
Therefore, the order of magnitude $1/\sqrt{N}$ of 
the deviations in Corollary~\ref{coro:bernstein_Sl} is optimal 
\textup{(}at least, asymptotically\textup{)}. 
\end{remark}

\subsection{Proof of Theorem~\ref{thm:len_bound2}} 
\label{app.pr.thm:len_bound2}
We proceed in four steps. 
\paragraph*{Step 1: mean value theorem.}
Since $F_{ U_{(\ell:n)} }^{-1}$ is continuous on $[x,x+\varepsilon]$ 
and differentiable on $(x,x+\varepsilon)$ 
---it is even $\mathcal{C}^{\infty}$ on $(0,1)$---, 
by the mean-value theorem, 
some $c \in (x,x+\varepsilon)$ exists such that 
\begin{align}
\notag 
F_{ U_{(\ell:n)} }^{-1} (x+\varepsilon) 
- F_{ U_{(\ell:n)} }^{-1} (x) 
= \bigl( F_{ U_{(\ell:n)} }^{-1} \bigr)' (c) \times \varepsilon 
&= \frac{\varepsilon }{f_{ U_{(\ell:n)} } \bigl( F_{ U_{(\ell:n)} }^{-1}(c) \bigr) } 
\\
&\leq \frac{\varepsilon }{
	\inf_{ \bigl[ F_{ U_{(\ell:n)} }^{-1} (x) 
	\, , \, 
	F_{ U_{(\ell:n)} }^{-1} (x+\varepsilon) \bigr] } 
	f_{ U_{(\ell:n)} }}
\, . 
\label{eq.pr.thm:len_bound.IAF}
\end{align}
Therefore, it remains to get a uniform lower bound 
on the density $f_{ U_{(\ell:n)} }$. 

\paragraph*{Step 2: density lower bound with the inverse cdf $F_{ U_{(\ell:n)} }^{-1}$.} 
Since $f_{ U_{(\ell:n)} }$ is unimodal \citep{bagnoli1989log}, 
a straightforward proof by exhaustion (depending on the relative position of 
$\argmax f_{ U_{(\ell:n)} }$ and $[ F_{ U_{(\ell:n)} }^{-1} (x) \, , \, F_{ U_{(\ell:n)} }^{-1} (x+\varepsilon) ]$) shows that 
\begin{align}
\inf_{ \bigl[ F_{ U_{(\ell:n)} }^{-1} (x) 
	\, , \, 
	F_{ U_{(\ell:n)} }^{-1} (x+\varepsilon) \bigr] } 
	f_{ U_{(\ell:n)} }
&\geq 
\min\biggl\{ 
	\sup_{\bigl[ 0 \, , \, F_{ U_{(\ell:n)} }^{-1} (x) \bigr]} f_{ U_{(\ell:n)} }
	\, , \,
	\sup_{\bigl[ F_{ U_{(\ell:n)} }^{-1} (x+\varepsilon)  
		\, , \, 1 \bigr]} f_{ U_{(\ell:n)} }
	\biggr\}
\, . 
\label{eq.pr.thm:len_bound.min-dens.1}
\end{align}
Let us consider separately the two terms 
of the right-hand side of Eq.~\eqref{eq.pr.thm:len_bound.min-dens.1}. 
On the one hand, for every $\delta \in (0,x)$, 
\begin{align}
\notag 
\sup_{\bigl[ 0 \, , \, F_{ U_{(\ell:n)} }^{-1} (x) \bigr]} f_{ U_{(\ell:n)} }
&\geq \frac{1}{F_{ U_{(\ell:n)} }^{-1} (x) - F_{ U_{(\ell:n)} }^{-1} (\delta)} 
	\int_{ F_{ U_{(\ell:n)} }^{-1} (\delta) }^{F_{ U_{(\ell:n)} }^{-1} (x)}
		f_{ U_{(\ell:n)} } (t) \mathrm{d}t
\\
&= \frac{
	F_{ U_{(\ell:n)} } \circ F_{ U_{(\ell:n)} }^{-1} (x) 
		- F_{ U_{(\ell:n)} } \circ F_{ U_{(\ell:n)} }^{-1} (\delta)
	}{
	F_{ U_{(\ell:n)} }^{-1} (x) - F_{ U_{(\ell:n)} }^{-1} (\delta)
	}
\notag 
\\
&= \frac{x-\delta}{
	F_{ U_{(\ell:n)} }^{-1} (x) - F_{ U_{(\ell:n)} }^{-1} (\delta)
	}
\, . 
\label{eq.pr.thm:len_bound.min-dens.2a}
\end{align}
On the other hand, similarly, for every $\delta' \in (0,1-x-\varepsilon)$, 
\begin{align}
\notag 
\sup_{\bigl[ F_{ U_{(\ell:n)} }^{-1} (x+\varepsilon) \, , \, 1 \bigr]} f_{ U_{(\ell:n)} }
&\geq \frac{1}{F_{ U_{(\ell:n)} }^{-1} (1-\delta') - F_{ U_{(\ell:n)} }^{-1} (x+\varepsilon)} 
	\int_{ F_{ U_{(\ell:n)} }^{-1} (x+\varepsilon) }^{F_{ U_{(\ell:n)} }^{-1} (1-\delta')}
		f_{ U_{(\ell:n)} } (t) \mathrm{d}t 
\\
&= \frac{(1-\delta')-(x+\varepsilon)}{
	F_{ U_{(\ell:n)} }^{-1} (1-\delta') - F_{ U_{(\ell:n)} }^{-1} (x+\varepsilon)
	}
\, . 
\label{eq.pr.thm:len_bound.min-dens.2b}
\end{align}

\paragraph*{Step 3: use of concentration results.}
Given Eq.~\eqref{eq.pr.thm:len_bound.min-dens.2a} 
and~\eqref{eq.pr.thm:len_bound.min-dens.2b}, 
it remains to show upper bounds on 
$F_{ U_{(\ell:n)} }^{-1} (x) - F_{ U_{(\ell:n)} }^{-1} (\delta)$
and 
$F_{ U_{(\ell:n)} }^{-1} (1-\delta') - F_{ U_{(\ell:n)} }^{-1} (x+\varepsilon)$. 
On the one hand, 
by Eq.~\eqref{eq.pro.FUr.dev-gauche.inv}--\eqref{eq.pro.FUr.dev-droite.inv} 
in Proposition~\ref{pro.FUr} with $r=\ell$ and $N=n$, 
\begin{equation}
\notag
\forall \delta \in ( 0 , 1-x ) \, , 
\quad 
F_{U_{(\ell:n)}}^{-1} (x) - F_{ U_{(\ell:n)} }^{-1} (\delta) 
\leq F_{U_{(\ell:n)}}^{-1} (1-\delta) - F_{ U_{(\ell:n)} }^{-1} (\delta) 
\leq \sqrt{\dfrac{2 \log(1/\delta)}{n+2}}
\, . 
\end{equation}
Therefore, using Eq.~\eqref{eq.pr.thm:len_bound.min-dens.2a}, 
we obtain that 
\begin{align}
\sup_{[ 0 \, , \, F_{ U_{(\ell:n)} }^{-1} (x) ]} f_{ U_{(\ell:n)} }
&\geq \sup_{ \delta \in ( 0 , \min\{ x \, , \, 1-x\} ) }
	\left\{ 
		\frac{x-\delta}{
		F_{ U_{(\ell:n)} }^{-1} (x) - F_{ U_{(\ell:n)} }^{-1} (\delta)
		}
	\right\}
\geq g(x) \sqrt{\frac{n+2}{2}}
\label{eq.pr.thm:len_bound.min-dens.3a}
\\
\notag 
\text{with} \qquad 
g(x) &\egaldef \sup_{ \delta \in ( 0 , \min\{ x \, , \, 1-x\} ) }
	\left\{ 
	\frac{x-\delta}{\sqrt{\log(1/\delta)}}
	\right\}
\, . 
\end{align}

On the other hand, similarly, 
for every $\delta' \in (0 , x + \varepsilon)$, 
\begin{equation}
\notag
F_{U_{(\ell:n)}}^{-1} (1-\delta') - F_{ U_{(\ell:n)} }^{-1} (x+\varepsilon) 
\leq F_{U_{(\ell:n)}}^{-1} (1-\delta') - F_{ U_{(\ell:n)} }^{-1} (\delta') 
\leq \sqrt{\dfrac{2 \log(1/\delta')}{n+2}}
\, . 
\end{equation}
Therefore, using Eq.~\eqref{eq.pr.thm:len_bound.min-dens.2b}, 
we obtain that 
\begin{align}
\notag 
\sup_{[ F_{ U_{(\ell:n)} }^{-1} (x+\varepsilon)  \, , \, 1 ]} f_{ U_{(\ell:n)} }
&\geq 
\sup_{\delta' \in ( 0 , \min\{ x + \varepsilon  \, , \, 1-(x+\varepsilon) \} )} \left\{
	\frac{(1-\delta')-(x+\varepsilon)}{
		F_{ U_{(\ell:n)} }^{-1} (1-\delta') - F_{ U_{(\ell:n)} }^{-1} (x+\varepsilon)
		}
	\right\}
\\
&\geq \sup_{\delta' \in ( 0 , \min\{ x + \varepsilon  \, , \, 1-(x+\varepsilon) \} )} \left\{
	\frac{1-(x+\varepsilon)-\delta')}{\sqrt{\log(1/\delta')}}
	\right\}
	\times \sqrt{\frac{n+2}{2}}
\\&= g(1-x-\varepsilon) \sqrt{\frac{n+2}{2}}
\, . 
\label{eq.pr.thm:len_bound.min-dens.3b}
\end{align}

\paragraph*{Step 4: conclusion.}
Combining Eq.~\eqref{eq.pr.thm:len_bound.IAF}, 
\eqref{eq.pr.thm:len_bound.min-dens.1}, 
\eqref{eq.pr.thm:len_bound.min-dens.3a} 
and~\eqref{eq.pr.thm:len_bound.min-dens.3b}, 
we obtain that 
\begin{align}
\label{eq.pr.thm:len_bound.maj-abstraite}
F_{ U_{(\ell:n)} }^{-1} (x+\varepsilon) 
- F_{ U_{(\ell:n)} }^{-1} (x) 
&\leq \frac{\sqrt{2}}{\sqrt{n+2}} 
	\, \frac{\varepsilon}{\min\{ g(x) \, , \, g(1-x-\varepsilon) \}}
\, . 
\end{align}

Now, let us rewrite the definition of $g$ as 
\begin{align*}
\forall t \in (0,1) \, , 
\quad 
g(t) 
\egaldef \sup_{ \delta \in ( 0 , \min\{ t \, , \, 1-t\} ) } 
	G(t,\delta) 
\quad 
\text{where} \quad 
\forall \delta > 0 \, , \, 
\qquad 
G(t,\delta) 
\egaldef \frac{t-\delta}{\sqrt{\log(1/\delta)}}
\, . 
\end{align*}
Note that $t \mapsto G(t,\delta)$ is increasing 
whatever $\delta>0$, 
so that $g$ is nondecreasing on $(0,1/2]$. 
In addition, for every $t \in (0,1/2]$, 
using that $t \leq 1-t$, 
we have 
$G(t,\delta) \leq G(1-t,\delta)$ for every $\delta>0$,  
hence $g(t) \leq g(1-t)$. 
This proves that 
\[
\forall t \in (0,1) \, , 
\qquad 
g\bigl( \min\{ t , 1-t \} \bigr) 
= \min\bigl\{ g(t) , g(1-t) \bigr\}
\, , 
\]
hence 
\begin{align*}
\min \bigl\{ 
	g(x) \, , \, g(1-x-\varepsilon) 
	\bigr\}
&\geq 
\min \bigl\{ 
	g(x) \, , \, g(1-x) \, , \, g(x+\varepsilon) \, , \, g(1-x-\varepsilon) 
	\bigr\}
\\
&= \min \Bigl\{ 
g\bigl( \min\{ x , 1-x \} \bigr) 
\, , \, 
g\bigl( \min\{ x+\varepsilon , 1-x-\varepsilon \} \bigr) 
\Bigr\}
\\
&= g\bigl( \min\{ x , 1 - x , x + \varepsilon , 1-x-\varepsilon \} \bigr) 
\qquad \text{since $g$ is nondecreasing on $(0,1/2]$} 
\\
&= g\bigl( \min\{ x , 1-x-\varepsilon \} \bigr) 
\, , 
\end{align*}
By Eq.~\eqref{eq.pr.thm:len_bound.maj-abstraite}, 
we obtain that 
Eq.~\eqref{eq.thm:len_bound2} holds true. 

Let us finally prove the inequalities stated about the function $g$. 
First, for every $t \in (0,1/2]$, we have 
$t/2 \leq \min\{t,1-t\}$, hence 
\[
g(t) \geq G(t,t/2) = \frac{t}{2 \sqrt{\log(2/t)}}
\, . 
\]
Second, since $g$ is nondecreasing over $(0,1/2]$, 
for every $t \in [1/3,1/2]$, 
\[
g(t) \geq g \left( \frac{1}{3} \right) 
\geq G\left( \frac{1}{3} \, , \, 0.047 \right) 
\geq 
0.163
\geq \frac{\sqrt{2}}{9}
\, . 
\]
\qed 

\subsection{Proof of Lemma~\ref{lemma:centra_unif_kl}}
\label{lemma:centra_unif_kl_proof}
We proceed similarly to the proof of Lemma~\ref{lemma:centra_unif}. 
First, let us define $\Xu_{i,j} = F_Z^{-1}( U_{i,j} )$ 
for every $i \in \intset{n}$ 
and $j \in \intset{m}$, 
so that $(\Xu_{i,j})_{1 \leq i \leq n , 1 \leq j \leq m}$ 
has the same distribution as $(Z_{i,j})_{1 \leq i \leq n , 1 \leq j \leq m}$. 
Therefore, the corresponding order statistics of order statistics have the same distributions: 
\begin{equation}
\label{eq.pr.lemma:centra_unif_kl.Z(ij)-egalloi-Xu(ij)}
Z_{(\ell,k)} \egalloi \Xu_{(\ell,k)} 
\, . 
\end{equation}
Second, using twice Eq.~\eqref{eq.pr.lemma:centra_unif_proof.2.Qh} 
in the proof of Lemma~\ref{lemma:centra_unif}, 
since $F_Z^{-1}$ is nondecreasing, we have 
\begin{align*}
\Xu_{(\ell,k)} 
&= \Qh_{(k)} \biggl( \Bigl( \Qh_{(\ell)} \bigl( F_Z^{-1} ( U_{1,j} ) , \ldots, F_Z^{-1} (U_{n,j} ) \bigr) \Bigr)_{j \in \intset{m}} \biggr)  
\\
&= \Qh_{(k)} \Bigl( F_Z^{-1} \bigl( \Qh_{(\ell)} ( U_{1,j} , \ldots, U_{n,j} ) \bigr)_{j \in \intset{m}} \Bigr)  
\\
&= F_Z^{-1} \biggl( \Qh_{(k)} \Bigl(  \bigl( \Qh_{(\ell)} ( U_{1,j} , \ldots, U_{n,j} ) \bigr)_{j \in \intset{m}} \Bigr) \biggr) 
= F_Z^{-1} ( U_{(\ell,k)} ) 
\, , 
\end{align*}
hence Eq.~\eqref{eq.pr.lemma:centra_unif_kl.Z(ij)-egalloi-Xu(ij)} 
implies that 
\[
F_Z ( Z_{(\ell,k)} ) 
\egalloi F_Z ( \Xu_{(\ell,k)} ) 
= F_Z \circ F_Z^{-1} ( U_{(\ell,k)} )
\, . 
\]
Using Eq.~\eqref{eq.prop-gal-inverse}, this proves 
Eq.~\eqref{eq.lemma:centra_unif_kl.domin-stoch} 
in the general case, 
and Eq.~\eqref{eq.lemma:centra_unif_kl.egal-loi}
when $F_Z$ is continuous. 
\qed

\begin{remark}
\label{rk.lemma:centra_unif_kl.bonus}
In the case of a general cdf $F_Z$, another consequence of Lemma~\ref{lemma:centra_unif_kl} 
is that, 
for every $a,b \in \R$ such that $a \leq b$, 
\begin{align}
\label{eq.lemma:centra_unif_kl.dev-droite-ineq}
\IP\left(U_{(\ell,k)} \leq b \right) 
&\leq \IP\left(F_Z(Z_{(\ell,k)}) \leq F_Z \circ F_Z^{-1}(b) \right)
\\
\label{eq.lemma:centra_unif_kl.dev-ineq}
\text{and} \qquad 
\IP\left(a \leq U_{(\ell,k)} \leq b \right) 
&\leq \IP\left(a \leq F_Z(Z_{(\ell,k)}) \leq F_Z \circ F_Z^{-1}(b)\right) 
\, . 
\end{align}
\end{remark}
\begin{proof}
Since $F_Z \circ F_Z^{-1}$ is nondecreasing, 
\begin{equation}
\label{eq.pr.rk.lemma:centra_unif_kl.bonus}
U_{(\ell,k)} \leq b  
\qquad \text{implies that} \qquad  
F_Z \circ F_Z^{-1} (U_{(\ell,k)}) \leq F_Z \circ F_Z^{-1}(b) 
\, , 
\end{equation}
hence Eq.~\eqref{eq.lemma:centra_unif_kl.dev-droite-ineq} 
thanks to Eq.~\eqref{eq.lemma:centra_unif_kl.domin-stoch}. 
Since $F_Z \circ F_Z^{-1} (U_{(\ell,k)}) \geq U_{(\ell,k)}$ 
by Eq.~\eqref{eq.prop-gal-inverse}, 
Eq.~\eqref{eq.pr.rk.lemma:centra_unif_kl.bonus} also shows that 
\begin{equation}
\notag 
a \leq U_{(\ell,k)} \leq b  
\qquad \text{implies that} \qquad  
a \leq F_Z \circ F_Z^{-1} (U_{(\ell,k)}) \leq F_Z \circ F_Z^{-1}(b) 
\, , 
\end{equation}
hence Eq.~\eqref{eq.lemma:centra_unif_kl.dev-ineq} 
thanks to Eq.~\eqref{eq.lemma:centra_unif_kl.domin-stoch}. 
\end{proof}

\subsection{Proof of Lemma~\ref{lemma:beta_beta}} \label{lemma:beta_beta_proof}
By Eq.~\eqref{eq.lemma:centra_unif.egal-loi} 
in Lemma~\ref{lemma:centra_unif} 
with $F_Z = F_{U_{(\ell:n)}}$ (which is continuous), 
$r=k$ and $N=m$, we have 
\[ 
F_{U_{(\ell:n)}} ( Z_{(k:m)} ) 
\egalloi U_{(k:m)} 
\]
where $Z_{(k:m)}$ is the $k$-th order statistics 
of an i.i.d. sample of $m$ random variables 
distributed as $U_{(\ell:n)}$. 
In other words, with the notation of Lemma~\ref{lemma:centra_unif_kl}, 
we have $Z_{(k:m)} \egalloi U_{(\ell:n, k:m)}$, 
hence 
\[ 
F_{U_{(\ell:n)}} ( U_{(\ell:n, k:m)} ) 
\egalloi U_{(k:m)} 
\egalloi F_{ U_{(k:m)} }^{-1} (U)
\]
since $F_{ U_{(k:m)} }$ is continuous 
(where $U$ follows the uniform distribution over $[0,1]$). 
So, 
\[
U_{(\ell:n, k:m)}
\egalloi F_{U_{(\ell:n)}}^{-1} \circ F_{ U_{(k:m)} }^{-1} (U)
= ( F_{ U_{(k:m)} } \circ F_{U_{(\ell:n)}} )^{-1} (U)
\]
has cdf $F_{ U_{(k:m)} } \circ F_{U_{(\ell:n)}} $. 
\qed

\section{Previous and additional results on related works}
% !TEX root = ../main.tex

\subsection{Proofs of Section~\ref{subsec:splitCP} (theoretical analysis of split CP)}
\label{app.pr.split-CP}

In this section, we prove the results stated 
about split conformal prediction (split CP) 
in Section~\ref{subsec:splitCP}. 
Although these results are mostly classical, 
we think useful to prove them here for completeness, 
since these proofs are particularly simple with the notation 
introduced in this paper and the results of 
Appendix~\ref{app.order-stat.order}. 

\subsubsection{Marginal guarantees}
\label{app.pr.split-CP.marginal}
By Lemma~\ref{lemma:centra_unif}, 
the expected coverage of $\Chat_{r}$ is 
\begin{align*}
\P\bigl( Y \in \Chat_{r}(X) \bigr) 
= \E\bigl[ 1 - \alpha_r (\calD) \bigr] 
= \E\bigl[ F_S (S_{(r)}) \bigr] 
\geq \E[ U_{(r:n_c)} ]
\end{align*}
with equality when $F_S$ is continuous 
---or, equivalently, when the scores are 
almost surely distinct. 
Using Eq.~\eqref{eq.unif-order-stat.E-Var} 
in Appendix~\ref{app.more-details-order-stat.exact}, 
we get that 
\begin{align*}
\P\bigl( Y \in \Chat_{r}(X) \bigr) 
\geq \frac{r}{n_c+1}
\end{align*}
with equality when $F_S$ is continuous. 
This proves that $\Chat_{r}$ is marginally valid 
when $r = \lceil (1-\alpha) (n_c+1) \rceil$, 
and Eq.~\eqref{eq.splitCP.marg-cov-upp} 
when $F_S$ is continuous. 

\subsubsection{Training-conditional guarantees}
\label{app.pr.split-CP.cond}
By Lemma~\ref{lemma:centra_unif}, 
\begin{align}
\P\bigl( Y \in \Chat_{r}(X) \,\big\vert\, \calD \bigr) 
= 1 - \alpha_r (\calD) 
= F_S (S_{(r)}) 
\succeq U_{(r:n_c)} 
\label{eq.app.pr.split-CP.cond.domin-stoch}
\end{align}
which follows the $\Beta(r,n_c-r+1)$ distribution, 
hence 
\begin{align*}
\P\bigl( 1-\alpha_{r}(\calD) \geq F_{U_{(r:n_c)}}^{-1} (\beta) \bigr) 
\geq \P\bigl( U_{(r:n_c)} \geq F_{U_{(r:n_c)}}^{-1} (\beta) \bigr) 
= 1 - F_{U_{(r:n_c)}} \circ F_{U_{(r:n_c)}}^{-1} (\beta)
= 1-\beta 
\, ,
\end{align*}
which proves Eq.~\eqref{eq:miscov_vovk}. 
Note that Eq.~\eqref{eq:miscov_vovk} is similar to 
\cite[Proposition~2b]{vovk2012conditional}, 
even if its statement and proof are quite different. 
Lemma~\ref{lemma:centra_unif} also proves that 
Eq.~\eqref{eq:miscov_vovk} is an equality 
when $F_S$ is continuous, 
that is, 
$1 - \alpha_r (\calD) \egalloi \Beta(r,n_c-r+1)$. 

Similarly to Algorithm~\ref{algo.FCP-QQ-cond.1} 
and its generalization given by 
Eq.~\eqref{eq.def.algo.FCP-QQ-cond.1.var} 
in Appendix~\ref{thm:cov_up_cond_proof}, 
in order to have the best possible 
training-conditionally valid prediction set 
$\Chat_{r}$ with split CP, 
Eq.~\eqref{eq:miscov_vovk} suggests the choice 
\begin{align}
\notag
r^*_c 
&\egaldef \argmin_{r \in \intset{n_c}} 
	\bigl\{ F_{U_{(r:n_c)}}^{-1} (1-\beta') \, : \, 
	 F_{U_{(r:n_c)}}^{-1} (\beta) \geq 1-\alpha \bigr\}
\\
\label{eq.splitCP.algo-cond}
&= \min \bigl\{ r \in \intset{n_c} \, : \, 
	 F_{U_{(r:n_c)}}^{-1} (\beta) \geq 1-\alpha \bigr\}
\, . 
\end{align}
Equality~\eqref{eq.splitCP.algo-cond} above 
comes from the fact that 
$F_{U_{(r:n_c)}}^{-1} (1-\beta')$ is an 
increasing function of $r \in \intset{n_c}$ 
for every $\beta' \in (0,1)$, 
hence minimizing $F_{U_{(r:n_c)}}^{-1} (1-\beta')$ 
is equivalent to minimizing~$r$, 
whatever $\beta' \in (0,1)$. 

By Theorem~\ref{thm.asympt-unif-order-stat}  
in Appendix~\ref{app.more-details-order-stat.asymp}, 
for any $\gamma \in \R$, 
when $r = r_{n_c} = (1-\alpha) n_c + \gamma \sqrt{n_c} + \mathrm{o}(n_c)$ 
as $n_c \to +\infty$, we have 
\begin{align}
\P \bigl( 1 - \alpha_{r_{n_c}} (\calD) \geq 1 - \alpha \bigr) 
&\geq \P ( U_{(r:n_c)} \geq 1 - \alpha ) 
\qquad \text{by Eq.~\eqref{eq.app.pr.split-CP.cond.domin-stoch}}
\label{eq.app.pr.split-CP.cond.1}
\\
\notag 
&= \P \Bigl( \sqrt{n_c} \bigl[ U_{(r:n_c)} - (1 - \alpha) \bigr] \geq 0 \Bigr)
\\
\label{eq.app.pr.split-CP.cond.2}
&\xrightarrow[n_c \to +\infty]{} 
\Phi \left( \frac{\gamma}{\sqrt{\alpha(1-\alpha)}} \right)
\end{align}
where $\Phi$ denotes the standard normal cdf. 
Taking $\gamma = \Phi^{-1}(1-\beta) \sqrt{\alpha(1-\alpha)}$, 
this proves that any  
\begin{equation}
\label{eq.app.pr.split-CP.cond.algo-asympt}
r = r_{n_c} = (1-\alpha) n_c + \Phi^{-1}(1-\beta) \sqrt{\alpha(1-\alpha)} \sqrt{n_c} 
	+ \mathrm{o}(\sqrt{n_c})
\in \intset{n_c}
\end{equation}
is (asymptotically) training-conditionally valid, 
that is, satisfies Eq.~\eqref{eq:cond_cov}. 
Furthermore, the choice \eqref{eq.app.pr.split-CP.cond.algo-asympt} 
is not improvable in general since 
Eq.~\eqref{eq.app.pr.split-CP.cond.1} is an equality 
when $F_S$ is continuous. 
Then, for any $r$ satisfying 
Eq.~\eqref{eq.app.pr.split-CP.cond.algo-asympt} 
and any $\beta' \in (0,1)$, 
when $F_S$ is continuous, 
by Theorem~\ref{thm.asympt-unif-order-stat}, 
the $(1-\beta')$-quantile of the distribution 
$\Beta( r_{n_c} , n_c - r_{n_c} + 1)$ 
of the coverage $1 - \alpha_{r_{n_c}}(\calD)$ 
of $\Chat_{r_{n_c}}$ satisfies 
\begin{equation}
\label{eq.app.pr.split-CP.cond.algo-asympt.coverage}
F_{U_{(r_{n_c}:n_c)}}^{-1} (1-\beta') 
= 1 - \alpha 
+ \frac{\sqrt{\alpha (1-\alpha)} \bigl[\Phi^{-1} (1-\beta') + \Phi^{-1} (1-\beta) \bigr]}{\sqrt{n_c}} 
+ \mathrm{o} \left( \frac{1}{\sqrt{n_c}} \right)
\, . 
\end{equation}

In addition, combining Lemma~\ref{lemma:centra_unif} 
and Eq.~\eqref{eq.coro:bernstein_Sl.dev-gauche} 
shows that 
\begin{equation}
\label{eq.app.pr.split-CP.cond.algo-non-asympt}
r = \left\lceil (n_c+1) \left( 1 - \alpha + \sqrt{\frac{\log(1/\beta)}{2 (n_c+2)}} \right) \right\rceil
\end{equation}
yields a training-conditionally valid prediction set $\Chat_{r}$ 
whatever $n_c \geq 1$, 
provided $r \in \intset{n_c}$. 
For instance, when $n_c \geq \max\{ 2 (\alpha^{-1} - 1) , (8/3) \alpha^{-2} \log(1/\beta) \}$, 
we have $(n_c+1)(1-\alpha) \leq (1-\alpha/2) n_c$ 
and $n_c \sqrt{n_c+2} / (n_c+1) \geq \sqrt{n_c} \sqrt{3/4} \geq \sqrt{2 \log(1/\beta)} / \alpha$, 
hence $(n_c+1) \sqrt{\frac{\log(1/\beta)}{2 (n_c+2)}} \leq n_c \alpha/2$, 
so that $r \in \intset{n_c}$. 
Note that the choice \eqref{eq.app.pr.split-CP.cond.algo-non-asympt} 
is similar to what can be deduced from 
\cite[Proposition~2a]{vovk2012conditional}, 
up to minor differences, and completely different 
notation and proof; 
see also \cite[Theorem~1]{bian2023training} for a 
rewriting of this result with notation closer to ours.
If $r$ satisfies Eq.~\eqref{eq.app.pr.split-CP.cond.algo-non-asympt}, 
when $F_S$ is continuous, 
Eq. \eqref{eq.lemma:centra_unif.egal-loi} 
and~\eqref{eq.pro.FUr.dev-droite.inv}  
show that for every $\beta' \in (0,1)$, 
with probability at least $1-\beta'$, 
the coverage $1 - \alpha_r (\calD)$ of $\Chat_r$ 
is smaller than 
\begin{equation}
\label{eq.app.pr.split-CP.cond.algo-non-asympt.coverage}
F_{U_{(r:n_c)}}^{-1} (1-\beta') 
\leq 1 - \alpha + \frac{ \sqrt{\log(1/\beta)} + \sqrt{\log(1/\beta')} }{\sqrt{2(n_c+2)}} + \frac{1}{n_c+1}
\, . 
\end{equation}
Notice that this upper bound is of order $1 - \alpha + \mathcal{O} (1/\sqrt{n_c})$ 
as Eq.~\eqref{eq.app.pr.split-CP.cond.algo-asympt.coverage}, 
with only a slightly worse dependence on $\alpha$ 
and $\beta,\beta'$.
% !TEX root = ../main.tex

\subsection{Additional results: average of quantiles}
\label{app.avg_qq}

In the one-shot FL setting, 
an alternative approach to quantile-of-quantiles is to use 
an average of quantiles, 
as proposed by \cite{lu2021distribution}. 
The idea is that each agent returns its 
$\lceil (n_j+1)(1-\alpha) \rceil$-th smallest score, 
and then the central server averages these scores. 
Using the notation of Section~\ref{sec:QQ}, 
the associated prediction set is 
\begin{equation}
\label{eq.def.Ch-Avg}
\Chat^{\mathrm{Avg}}_{\ell_1 , \ldots , \ell_m} (x) 
:= \bigg\{y \in \calY : s(x, y) \leq 
\frac{1}{m} \sum_{j=1}^m \Qh_{( \ell_j )}(\mathcal{S}_j)
\bigg\}
\end{equation}
with $\ell_j = \lceil (n_j+1) (1-\alpha) \rceil$ 
for every $j \in \intset{m}$. 
To the best of our knowledge, 
no coverage guarantee has ever been proved for 
$\Chat^{\mathrm{Avg}}_{\ell_1 , \ldots , \ell_m}$. 
We study this prediction set in this section, 
showing that its marginal coverage 
strongly depends on the scores distribution, 
so that one cannot always find some $\ell_j \in \intset{n}$ 
such that the marginal coverage is larger than a given $1-\alpha$. 

\subsubsection{First remark}
\label{app.avg_qq.1st-rk}
A first simple remark is that when each agent 
$j \in \intset{m}$ has $n_j=1$ calibration point, 
then one must take $\ell_j=1$ and the 
marginal coverage of $\Chat^{\mathrm{Avg}}$ 
cannot be close to~$1$ in general since 
it is equal to 
\begin{align*}
\E \left[ F_S \left( \frac{1}{m} \sum_{j=1}^m \Qh_{( 1 )}(\mathcal{S}_j) \right) \right]
= \E \left[ F_S \left( \frac{1}{m} \sum_{j=1}^m S_{1,j} \right) \right]
\approx F_S \bigl( \E [S] \bigr)
\end{align*}
(at least when $m$ is large) by the law of large numbers. 
Therefore, $\Chat^{\mathrm{Avg}}$ cannot be 
marginally valid for all distributions 
without assuming that $n_j \geq n_0(\alpha)$  
for some large enough $n_0(\alpha)$, 
at least for a large part of the agents $j \in \intset{m}$. 
The next result shows that such an assumption is not sufficient. 

\subsubsection{Negative result for Bernoulli scores}
\label{app.avg_qq.Bernoulli}
We now prove that when $n_j=n$ for every $j \in \intset{m}$, 
for every $\alpha \in (0,1)$, 
some distribution of the scores exists such that 
the marginal coverage of $\Chat^{\mathrm{Avg}}_{\ell_1 , \ldots , \ell_m}$ 
is smaller than $1-\alpha$ whatever $\ell_j \in \intset{m}$, 
for arbitrary large $m$ and~$n$. 
Therefore, contrary to our quantile-of-quantiles prediction sets 
of Section~\ref{sec:FCP-QQ}, 
the $\ell_j$ cannot be chosen in such a way that 
$\Chat^{\mathrm{Avg}}_{\ell_1 , \ldots , \ell_m}$ 
is a marginally-valid distribution-free prediction set, 
even when the data set is large enough. 

\begin{proposition}
\label{pro.Avg-Bernoulli}
In the setting of Section~\ref{sec:QQ}, with Assumption~\ref{ass:same_n}, 
for any $\ell_1, \ldots , \ell_m \in \intset{m}$, 
the prediction set $\Chat^{\mathrm{Avg}}_{\ell_1 , \ldots , \ell_m}$ 
defined by Eq.~\eqref{eq.def.Ch-Avg} satisfies 
\begin{equation}
\label{eq.pro.Avg-Bernoulli.max-cov} 
\P \bigl( Y \in \Chat^{\mathrm{Avg}}_{\ell_1 , \ldots , \ell_m} (X) \bigr) 
\leq \P \bigl( Y \in \Chat^{\mathrm{Avg}}_{n , \ldots , n} (X) \bigr) 
\, .
\end{equation}
Assuming in addition that the scores $S_{i,j},S$ are i.i.d. 
Bernoulli$(p)$ random variables for some $p \in [0,1]$, 
for every $\ell \in \intset{n}$, we have 
\begin{equation}
\label{eq.pro.Avg-Bernoulli.formule-cov} 
\P \bigl( Y \in \Chat^{\mathrm{Avg}}_{\ell , \ldots , \ell} (X) \bigr)
= (1-p) + p \cdot  F_{U_{(n-\ell+1:n)}} (p)^m 
\end{equation}
where $F_{U_{(n-\ell+1:n)}}$ is the cdf of 
the $\Beta(n-\ell+1,\ell)$ distribution. 
Therefore, for every $c \in (0,1)$, 
if $m = m_n$ and $p = p_n = 1 - [ \log(1/c) / m_n ]^{1/n}$ 
are such that $\log(m_n)/n \to +\infty$ as $n \to +\infty$, 
we have 
\begin{equation}
\label{eq.pro.Avg-Bernoulli.contre-ex} 
\P \bigl( Y \in \Chat^{\mathrm{Avg}}_{n , \ldots , n} (X) \bigr)
= (1-p_n) + p_n \bigl[ 1 - (1-p_n)^n \bigr]^{m_n} 
\xrightarrow[n \to +\infty]{} c
\, . 
\end{equation}
\end{proposition}
Eq.~\eqref{eq.pro.Avg-Bernoulli.max-cov}--\eqref{eq.pro.Avg-Bernoulli.contre-ex} 
show that for any $\alpha \in (0,1)$, 
when the scores follow a Bernoulli$(p_n)$ distribution with 
\[ 
p_n 
= 1 - \left[ \frac{\log[ 2/(1-\alpha)] }{m_n}\right]^{1/n}
\qquad \text{and} \qquad 
\log(m_n) \gg n \to +\infty
\]
---for instance, when $m_n \geq n^n$---, 
the marginal coverage of $\Chat^{\mathrm{Avg}}_{\ell_1 , \ldots , \ell_m}$ 
is strictly smaller than $1-\alpha$ 
when $n$ is large enough, 
whatever the $\ell_j$, $j \in \intset{m}$. 
So, $\Chat^{\mathrm{Avg}}$ cannot be used 
as a distribution-free marginally-valid prediction set. 

Note that the counterexample provided by 
Proposition~\ref{pro.Avg-Bernoulli} 
relies on the non-robustness of the empirical mean 
of the $\Qh_{(\ell_j)} (\mathcal{S}_j)$ 
defining $\Chat^{\mathrm{Avg}}_{\ell_1 , \ldots , \ell_m}$: 
when the scores follow a Bernoulli distribution, 
it suffices to have one $\Qh_{(\ell_j)} (\mathcal{S}_j) \neq 1$ 
to make the expected coverage strictly smaller than $1-p$,  
and such an ``outlier'' event occurs with 
a large probability if $m$ is large enough, 
even when $p$ is close to~$1$. 
This phenomenon enlightens the interest of our 
quantile-of-quantiles approach, 
which does not have such a drawback thanks to 
the robustness of the $k$-th empirical quantile. 

\begin{proof}[Proof of Proposition~\ref{pro.Avg-Bernoulli}] 
First, for every $j \in \intset{m}$, 
$\Qh_{( \ell_j )}(\mathcal{S}_j)$ is a nondecreasing function 
of $\ell_j$, hence $\Chat^{\mathrm{Avg}}_{\ell_1 , \ldots , \ell_m} (X)$ 
is a nondecreasing function of $\ell_j$, 
which implies Eq.~\eqref{eq.pro.Avg-Bernoulli.max-cov}. 

Second, when the scores are i.i.d. and 
follow a Bernoulli$(p)$ distribution, 
for any $i \in \intset{m}$, 
$S_{(\ell:n), i}$ follows a Bernoulli distribution with parameter 
\[ 
\P( S_{(\ell:n), i} = 1) 
= \sum_{k = n-\ell+1}^n p^k (1-p)^{n-k} \binom{n}{k} 
= F_{U_{(n-\ell+1 : n)}} (p)
\]
by Eq.~\eqref{eq.FUl.pdf-cdf}. 
Therefore, we have 
\begin{align*}
\P \bigl( Y \in \Chat^{\mathrm{Avg}}_{\ell , \ldots , \ell} (X) \bigr)
&= \P \left( S \leq \dfrac{1}{m} \sum_{j=1}^m S_{(\ell:n), j}\right)
\\
&= \P(S=0) + \P( S=1 \text{ and } S_{(\ell:n), j} = 1 
\text{ for every } j \in \intset{m} )
\\
&= (1 - p) + p \cdot \P(S_{(\ell:n), 1} = 1)^m 
= (1 - p) + p \cdot F_{U_{(n-\ell+1 : n)}} (p)^m 
\, ,
\end{align*}
which proves Eq.~\eqref{eq.pro.Avg-Bernoulli.formule-cov}. 

Third, when $m=m_n$ is such that 
$\log(m_n)/n \to +\infty$ as $n \to +\infty$, 
we have $p_n = 1 - [ \log(1/c) / m_n ]^{1/n} \in (0,1)$ 
for $n$ large enough, 
and $p_n \to 1$ and $m_n \to +\infty$ when $n \to +\infty$.  
So, Eq.~\eqref{eq.pro.Avg-Bernoulli.formule-cov} 
together with Eq.~\eqref{eq.min-uniformes.cdf} shows that 
\begin{align*}
\P \bigl( Y \in \Chat^{\mathrm{Avg}}_{n , \ldots , n} (X) \bigr)
&= (1-p_n) + p_n \bigl[ 1 - (1-p_n)^n \bigr]^{m_n} 
\\
&= \mathrm{o}(1) + \bigl[ 1 + \mathrm{o}(1) \bigr] 
\left( 1 - \frac{\log(1/c)}{m_n} \right)^{m_n}
\\
&\xrightarrow[n \to +\infty]{} \exp \bigl( - \log(1/c) \bigr) 
= c
\, . 
\end{align*}
\end{proof}

\subsubsection{Uniform scores}
\label{app.avg_qq.uniform}
We now study the case of scores following 
a uniform distribution. 
\begin{proposition}
\label{pro.Avg-uniform}
Suppose that Assumption~\ref{ass:same_n} holds true. 
When the scores $S_{i,j} , S$ are independent 
with a uniform distribution over $[a,b]$ for some $a<b$, 
for any $\ell \in \intset{n}$, we have 
\begin{equation}
\label{eq.pro.Avg-uniform}
\P \bigl( Y \in \Chat^{\mathrm{Avg}}_{\ell , \ldots , \ell} (X) \bigr) 
= \frac{\ell}{n+1} 
\, . 
\end{equation}
\end{proposition}
Proposition~\ref{pro.Avg-uniform} shows that 
when the scores are uniform (and using this knowledge), 
$\Chat^{\mathrm{Avg}}_{\ell , \ldots , \ell}$ is 
marginally valid as soon as $\ell \geq (n+1)(1-\alpha)$, 
which suggests to take $\ell = \lceil (n+1)(1-\alpha) \rceil$. 
Nevertheless, the difference between the expected coverage 
and the nominal coverage is then equal to 
\[
\frac{\lceil (n+1)(1-\alpha) \rceil}{n+1} - (1-\alpha) 
\in \left[ 0 , \frac{1}{n+1} \right) 
\, , 
\]
so its worst-case value $1/(n+1)$ is independent from $m$, 
which is not a desirable property. 
\begin{proof}
By definition of $\Chat^{\mathrm{Avg}}$, 
we have 
\begin{align*}
\P \bigl( Y \in \Chat^{\mathrm{Avg}}_{\ell , \ldots , \ell} (X) \bigr) 
= \P \left(S \leq \frac{1}{m} \sum^m_{j=1} S_{(\ell:n), j} \right) 
&= \E \left[ F_S \left( 
	\frac{1}{m} \sum^m_{j=1} S_{(\ell:n), j} 
	\right) \right]
\\
&= \E \left[ \frac{1}{m} \sum^m_{j=1} F_S ( S_{(\ell:n), j} ) \right]
\end{align*}
since $F_S(t) = (t-a)/(b-a)$ for every $t \in [a,b]$. 
Therefore, we get that 
\begin{align*}
\P \bigl( Y \in \Chat^{\mathrm{Avg}}_{\ell , \ldots , \ell} (X) \bigr) 
= \E \bigl[ F_S ( S_{(\ell:n), 1} ) \bigr]
= \E [ U_{(\ell:n)} ]
= \frac{\ell}{n+1} 
\end{align*}
by Eq.~\eqref{eq.lemma:centra_unif.egal-loi} 
in Appendix~\ref{app.order-stat.order} 
and Eq.~\eqref{eq.unif-order-stat.E-Var} 
in Appendix~\ref{app.more-details-order-stat.exact}. 
\end{proof}

\subsubsection{Exponential scores}
\label{app.avg_qq.exp}
We now study the case of scores following 
an exponential distribution. 
\begin{proposition}
\label{pro.Avg-exp}
Suppose that Assumption~\ref{ass:same_n} holds true. 
When the scores $S_{i,j} , S$ are independent 
with an exponential distribution with parameter $\lambda>0$, 
for any $\ell \in \intset{n}$, we have 
\begin{equation}
\label{eq.pro.Avg-exp}
\P \bigl( Y \in \Chat^{\mathrm{Avg}}_{\ell , \ldots , \ell} (X) \bigr) 
= 1 - \prod_{j=1}^{\ell} \left( 1 + \frac{1}{m(n-j+1)} \right)^{-m}
\, . 
\end{equation}
\end{proposition}
Proposition~\ref{pro.Avg-exp} shows that 
when the scores are exponential (and using this knowledge), 
it is possible to choose $\ell$ such that the expected 
coverage of $\Chat^{\mathrm{Avg}}_{\ell , \ldots , \ell}$ 
is larger than $1-\alpha$, 
provided that $m$ and $n$ are large enough. 
Indeed, by Eq.~\eqref{eq.pro.Avg-exp}, 
\begin{align*}
&\qquad \max_{\ell \in \intset{n}} 
\P \bigl( Y \in \Chat^{\mathrm{Avg}}_{\ell , \ldots , \ell} (X) \bigr) 
\\
&= \P \bigl( Y \in \Chat^{\mathrm{Avg}}_{n , \ldots , n} (X) \bigr) 
\\
&= 1 - \prod_{j=1}^{n} \left( 1 + \frac{1}{m(n-j+1)} \right)^{-m}
\xrightarrow[m \to +\infty]{} 1 - \exp \left( - \sum_{j=1}^n \frac{1}{j} \right) 
\xrightarrow[n \to +\infty]{} 1
\, , 
\end{align*}
where the first limit is taken for any fixed $n \geq 1$. 
\begin{proof}[Proof of Proposition~\ref{pro.Avg-exp}] 
A classical result about the order statistics of 
a sample of exponential random variables \citep{renyi1953theory} 
shows that for any $i \in \intset{m}$, 
\[ 
S_{(\ell:n), i} 
\egalloi \frac{1}{\lambda} \sum_{j=1}^{\ell} \frac{E_{j, i}}{n-j+1} 
\]
where the $(E_{j,i})_{1 \leq j \leq n}$ 
are independent standard exponential variables. 
Therefore, we have 
\begin{align*}
\P \bigl( Y \in \Chat^{\mathrm{Avg}}_{\ell , \ldots , \ell} (X) \bigr) 
&= \E \left[ F_S \left( 
	\frac{1}{m} \sum^m_{j=1} S_{(\ell:n), j} 
	\right) \right]
\\
&= 1 - \E \left[ \exp\left( -\frac{\lambda}{m} 
	\sum^m_{i=1} S_{(\ell:n), i} 
	\right)\right] 
\\
&=  1 - \E \left[ \exp\left( -\frac{1}{m} 
	\sum^m_{i=1} \sum_{j=1}^{\ell} \frac{E_{j,i}}{n-j+1} 
	\right)\right] 
\\
&= 1 - \E \left[ \exp\left( -\sum_{j=1}^{\ell} 
	\sum^m_{i=1} \frac{E_{j,i}}{m(n-j+1)} 
	\right)\right] 
\\
&= 1 - \prod_{j=1}^{\ell} \E \left[ 
	\exp\left( -\sum^m_{i=1} \frac{E_{j,i}}{m(n-j+1)} \right)
	\right] 
= 1 - \prod_{j=1}^{\ell} \left( 1 + \frac{1}{m(n-j+1)} \right)^{-m}
\, . 
\end{align*}
For the last equality, since $1/(m(n-j+1))$ is positive, 
$E_{j,i} / [m(n-j+1)]$ follows an exponential distribution 
with parameter $m(n-j+1)$. 
Therefore, $\sum^m_{i=1} \frac{E_{j,i}}{m(n-j+1)}$ 
follows a Gamma distribution with parameter 
$( m, m(n-j+1) )$ 
and the value at $-1$ of its moment generating function 
is equal to $(1 + \frac{1}{m(n-j+1)} )^{-m}$ since $-1 < m(n-j+1)$. 
\end{proof}

\section{Additional experimental results}
% !TEX root = ../main.tex

\subsection{Generic comparison, equal $n_j$}
\label{app:rate.bound}

In this section, we provide additional results 
in the generic setting of Section~\ref{sec:synth_data}. 

\begin{figure}[H]
	\centering
	\includegraphics[width=0.47\linewidth]{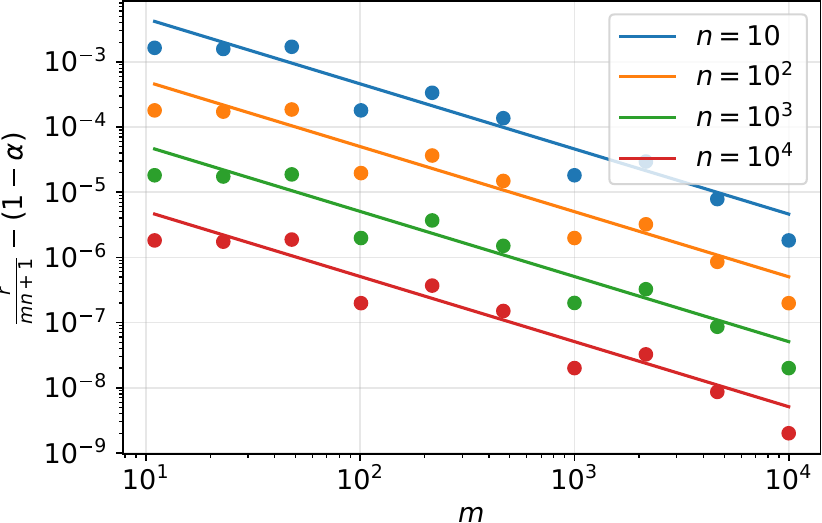}\hspace{2em}
	\includegraphics[width=0.47\linewidth]{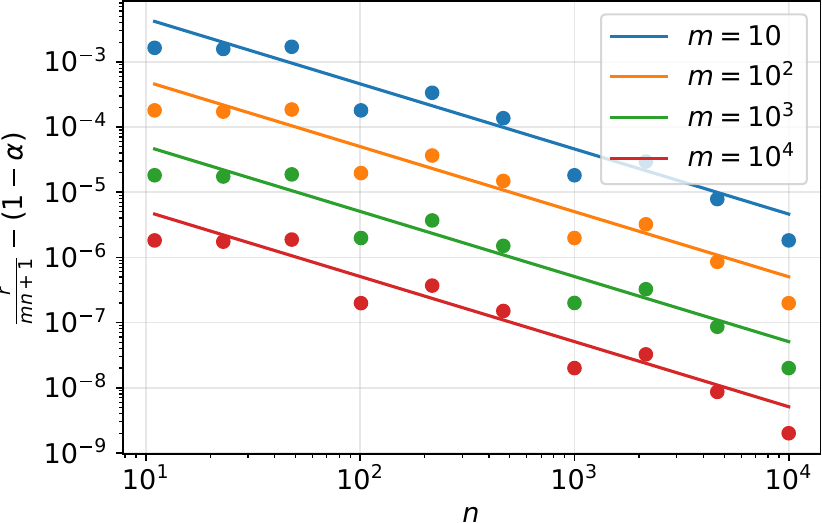}
\caption{%
\methodCentral\textup{:}  
log-log plot of $\Delta\E$ as a function of 
$m$ \textup{(}left\textup{)} or $n$ \textup{(}right\textup{)}. 
Lines show the approximation 
$\log{\Delta \E} \approx \log(c_1) - \gamma_1 \log(m) - \delta_1 \log(n)$ 
with $c_1,\gamma_1,\delta_1$ given by Table~\ref{tab:res_coeff_lin_all}.
\label{fig.DeltaE-Central}
}
\end{figure}

\begin{figure}[H]
	\centering
	\includegraphics[width=0.47\linewidth]{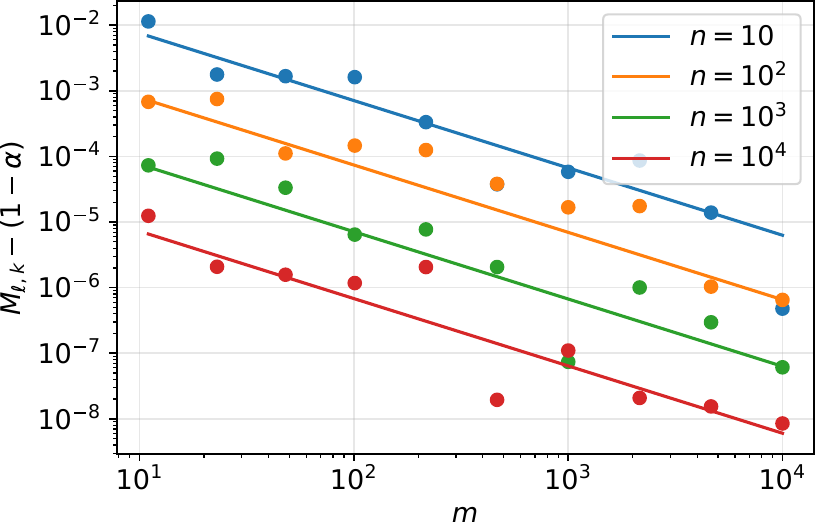}\hspace{2em}
	\includegraphics[width=0.47\linewidth]{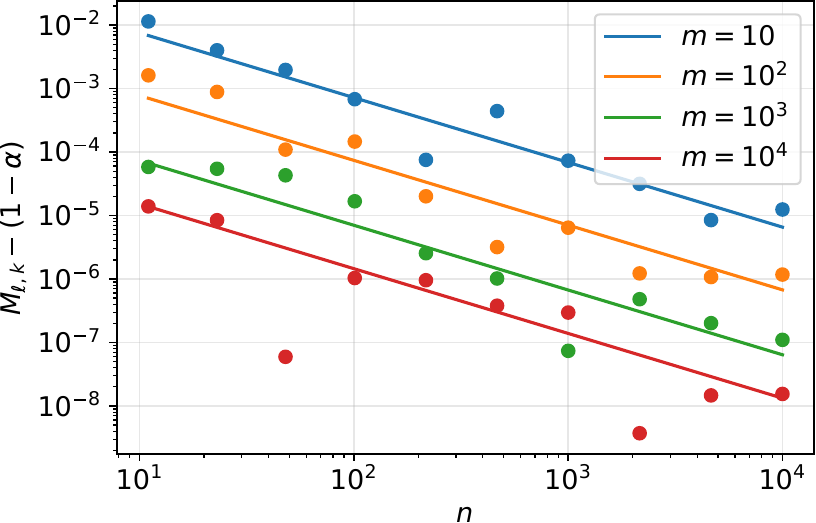}
\caption{%
\method\textup{:}  
log-log plot of $\Delta\E$ as a function of 
$m$ \textup{(}left\textup{)} or $n$ \textup{(}right\textup{)}. 
Lines show the approximation 
$\log{\Delta \E} \approx \log(c_1) - \gamma_1 \log(m) - \delta_1 \log(n)$ 
with $c_1,\gamma_1,\delta_1$ given by Table~\ref{tab:res_coeff_lin_all}. 
\label{fig.DeltaE-Algo1}
}
\end{figure}

\begin{figure}[H]
	\centering
	\includegraphics[width=0.47\linewidth]{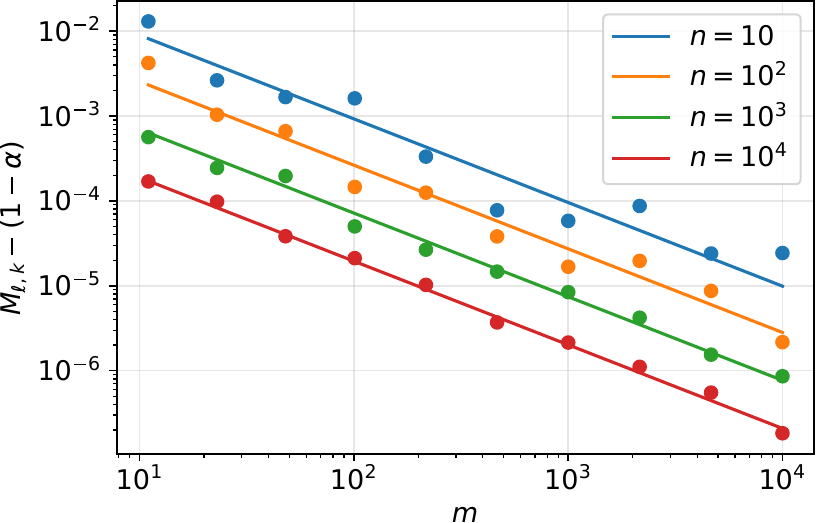}\hspace{2em}
	\includegraphics[width=0.47\linewidth]{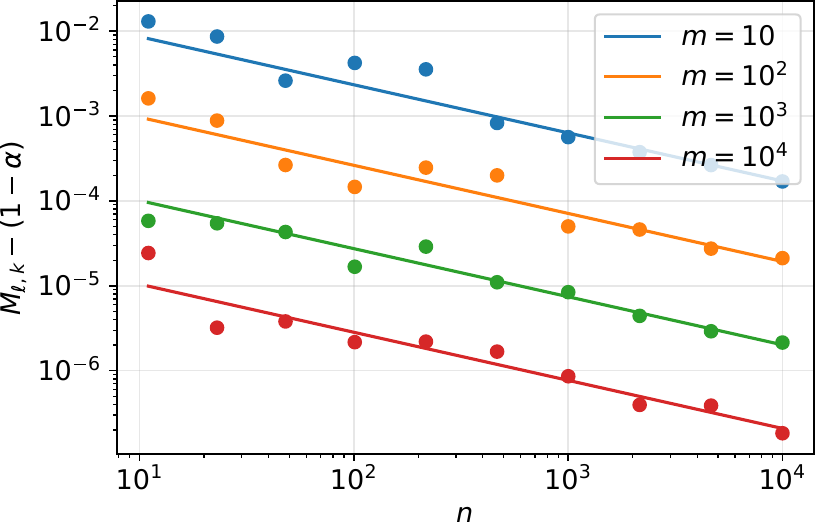}
\caption{%
\methodlow\textup{:}  
log-log plot of $\Delta\E$ as a function of 
$m$ \textup{(}left\textup{)} or $n$ \textup{(}right\textup{)}. 
Lines show the approximation 
$\log{\Delta \E} \approx \log(c_1) - \gamma_1 \log(m) - \delta_1 \log(n)$ 
with $c_1,\gamma_1,\delta_1$ given by Table~\ref{tab:res_coeff_lin_all}.
\label{fig.DeltaE-Algo2}
}
\end{figure}

\begin{figure}[H]
	\centering
	\includegraphics[width=0.47\linewidth]{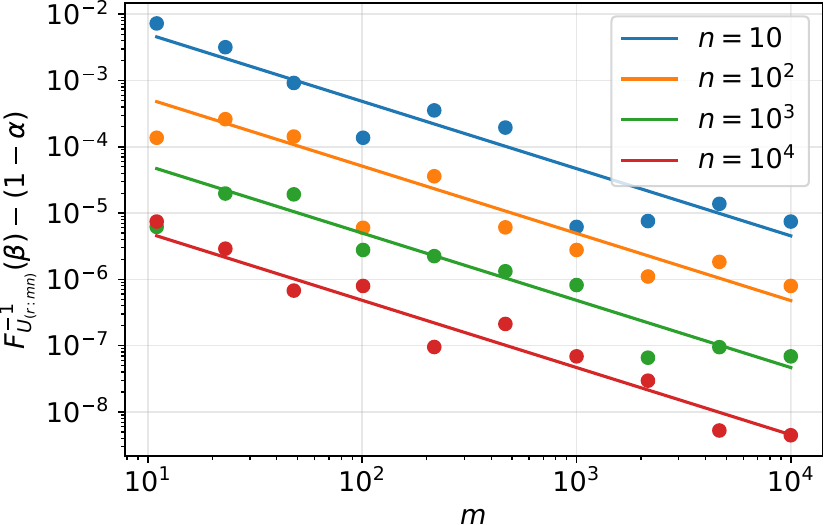}\hspace{2em}
	\includegraphics[width=0.47\linewidth]{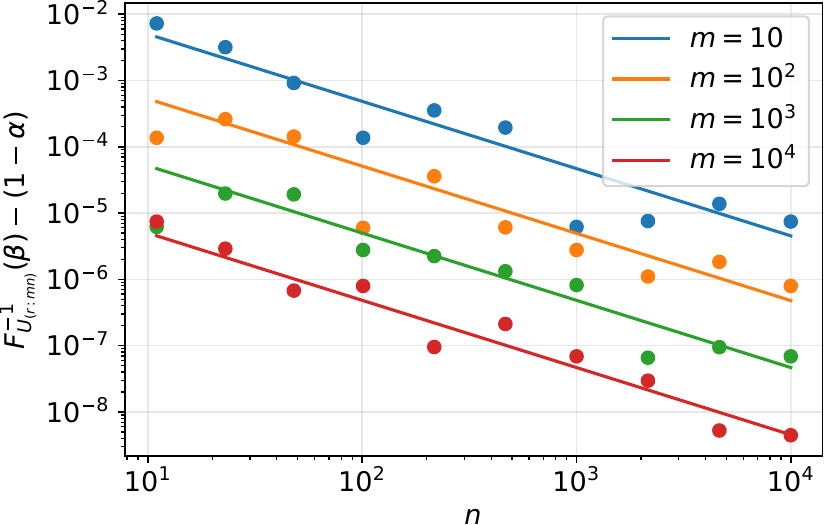}
\caption{%
\methodCentralC\textup{:}  
log-log plot of $\Delta q_{\beta}$ as a function of 
$m$ \textup{(}left\textup{)} or $n$ \textup{(}right\textup{)}. 
Lines show the approximation 
$\log{\Delta q_{\beta}} \approx \log(c_2) - \gamma_2 \log(m) - \delta_2 \log(n)$ 
with $c_2,\gamma_2,\delta_2$ given by Table~\ref{tab:res_coeff_lin_all}.
\label{fig.qbeta-Cond-Central}
}
\end{figure}

\begin{figure}[H]
	\centering
	\includegraphics[width=0.47\linewidth]{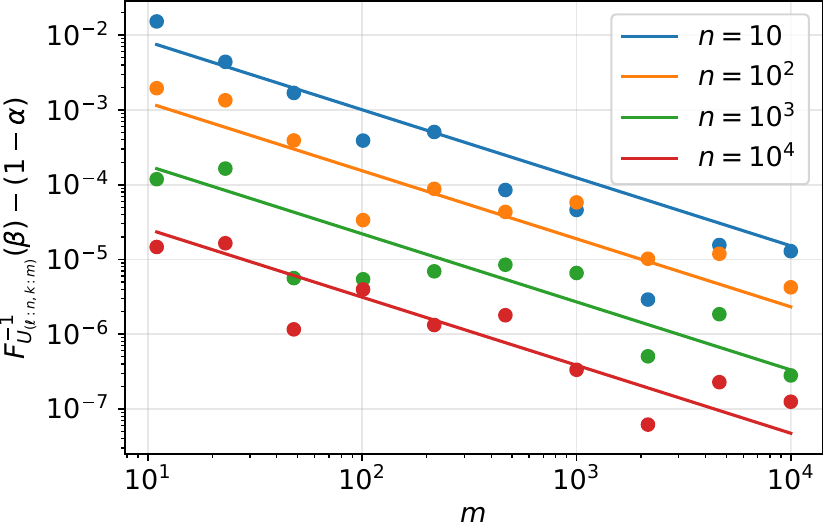}\hspace{2em}
	\includegraphics[width=0.47\linewidth]{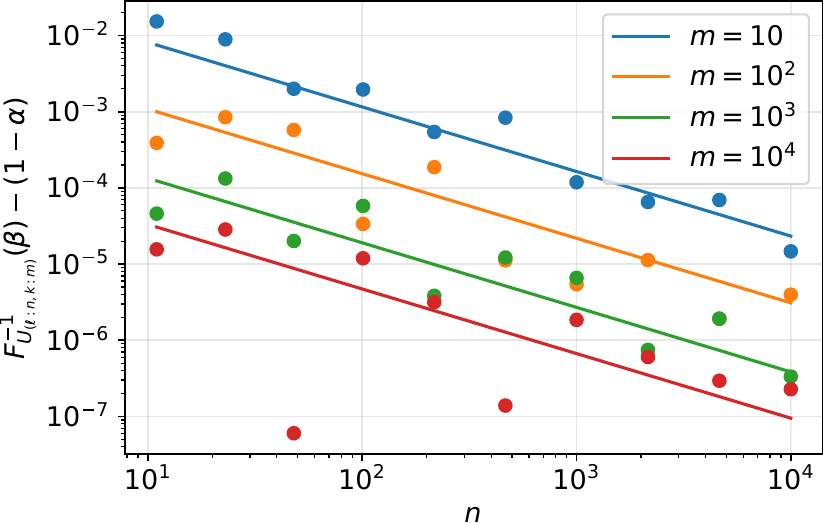}
\caption{%
\methodcondFL\textup{:}  
log-log plot of $\Delta q_{\beta}$ as a function of 
$m$ \textup{(}left\textup{)} or $n$ \textup{(}right\textup{)}. 
Lines show the approximation 
$\log{\Delta q_{\beta}} \approx \log(c_2) - \gamma_2 \log(m) - \delta_2 \log(n)$ 
with $c_2,\gamma_2,\delta_2$ given by Table~\ref{tab:res_coeff_lin_all}.
\label{fig.qbeta-Algo3}
}
\end{figure}

\begin{figure}[H]
	\centering
	\includegraphics[width=0.47\linewidth]{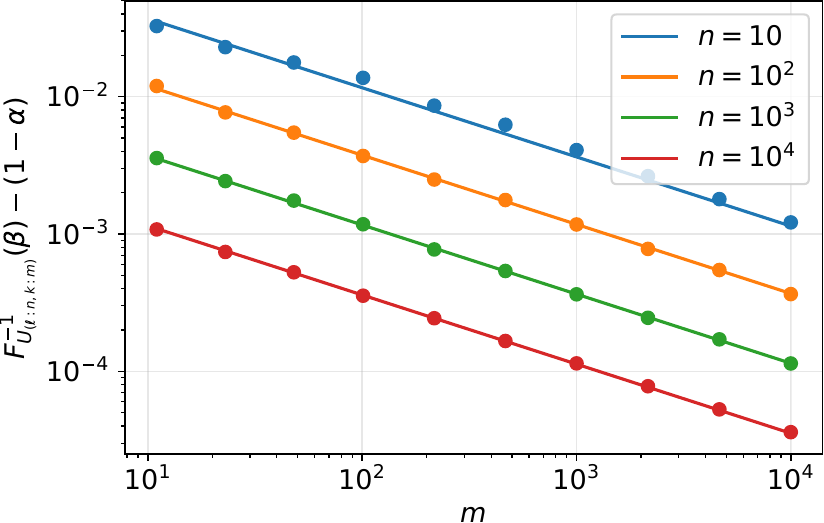}\hspace{2em}
	\includegraphics[width=0.47\linewidth]{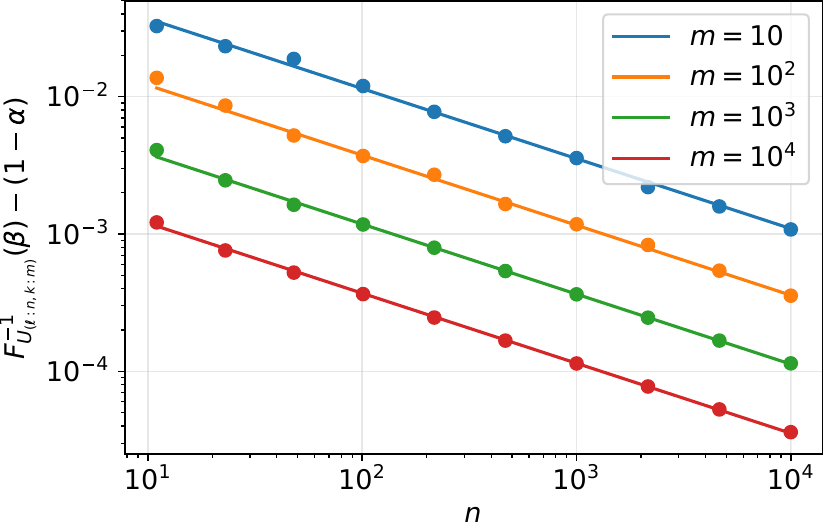}
\caption{%
\methodcondFLk\textup{:}  
log-log plot of $\Delta q_{\beta}$ as a function of 
$m$ \textup{(}left\textup{)} or $n$ \textup{(}right\textup{)}. 
Lines show the approximation 
$\log{\Delta q_{\beta}} \approx \log(c_2) - \gamma_2 \log(m) - \delta_2 \log(n)$ 
with $c_2,\gamma_2,\delta_2$ given by Table~\ref{tab:res_coeff_lin_all}.
\label{fig.qbeta-Algo4}
}
\end{figure}

\begin{figure}[H]
	\centering
	\includegraphics[width=0.47\linewidth]{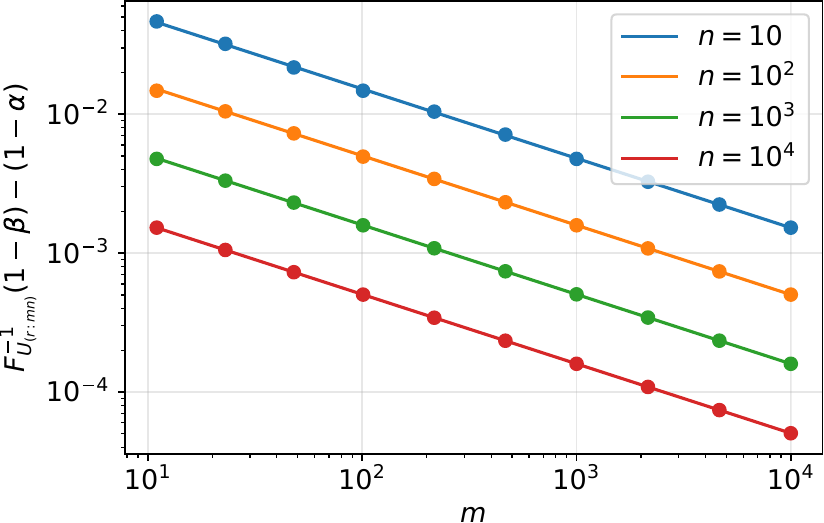}\hspace{2em}
	\includegraphics[width=0.47\linewidth]{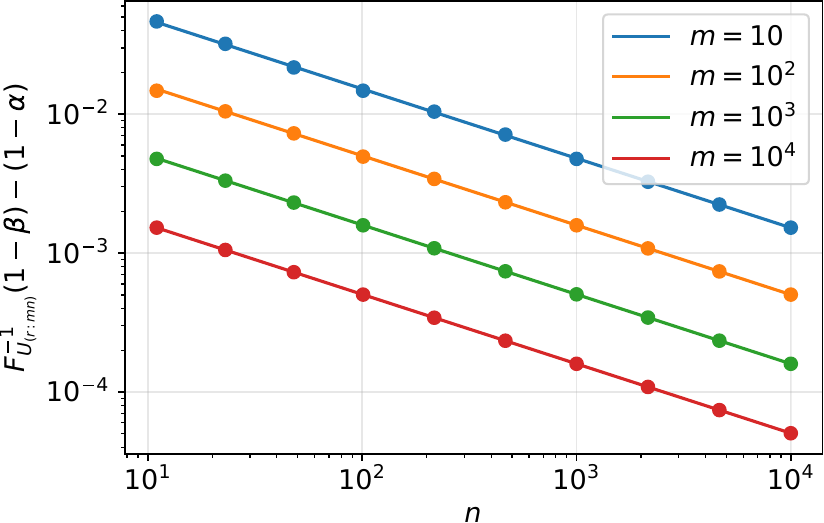}
\caption{%
\methodCentralC\textup{:}  
log-log plot of $\Delta q_{1-\beta}$ as a function of 
$m$ \textup{(}left\textup{)} or $n$ \textup{(}right\textup{)}. 
Lines show the approximation 
$\log{\Delta q_{1-\beta}} \approx \log(c_3) - \gamma_3 \log(m) - \delta_3 \log(n)$ 
with $c_3,\gamma_3,\delta_3$ given by Table~\ref{tab:res_coeff_lin_all}.
\label{fig.1_qbeta-Cond-Central}
}
\end{figure}

\begin{figure}[H]
	\centering
	\includegraphics[width=0.47\linewidth]{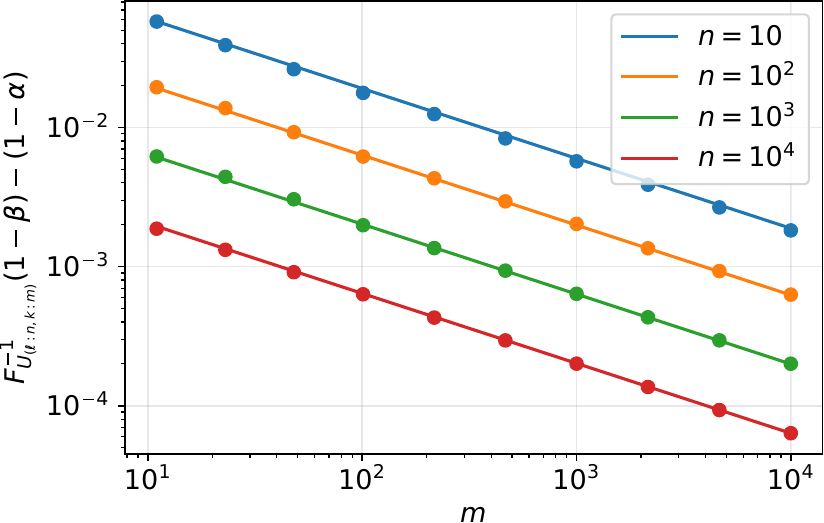}\hspace{2em}
	\includegraphics[width=0.47\linewidth]{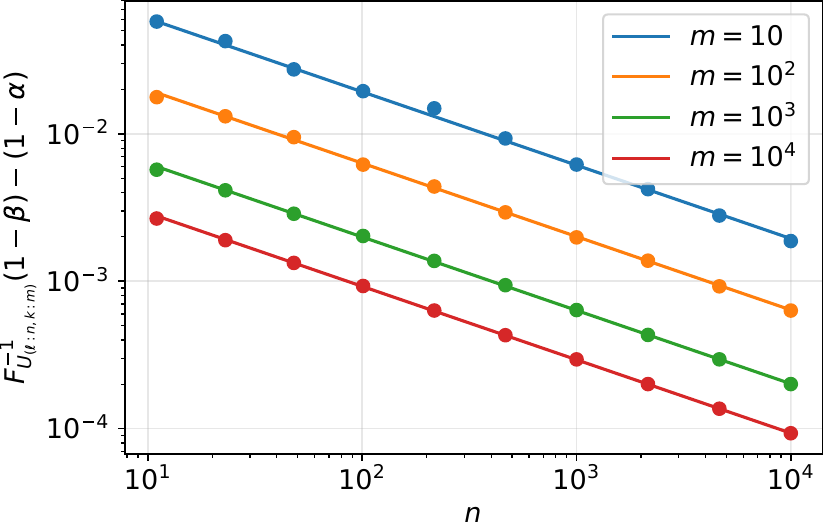}
\caption{%
\methodcondFL\textup{:}  
log-log plot of $\Delta q_{1-\beta}$ as a function of 
$m$ \textup{(}left\textup{)} or $n$ \textup{(}right\textup{)}. 
Lines show the approximation 
$\log{\Delta q_{1-\beta}} \approx \log(c_3) - \gamma_3 \log(m) - \delta_3 \log(n)$ 
with $c_3,\gamma_3,\delta_3$ given by Table~\ref{tab:res_coeff_lin_all}.
\label{fig.1_qbeta-Algo3}
}
\end{figure}

\begin{figure}[H]
	\centering
	\includegraphics[width=0.47\linewidth]{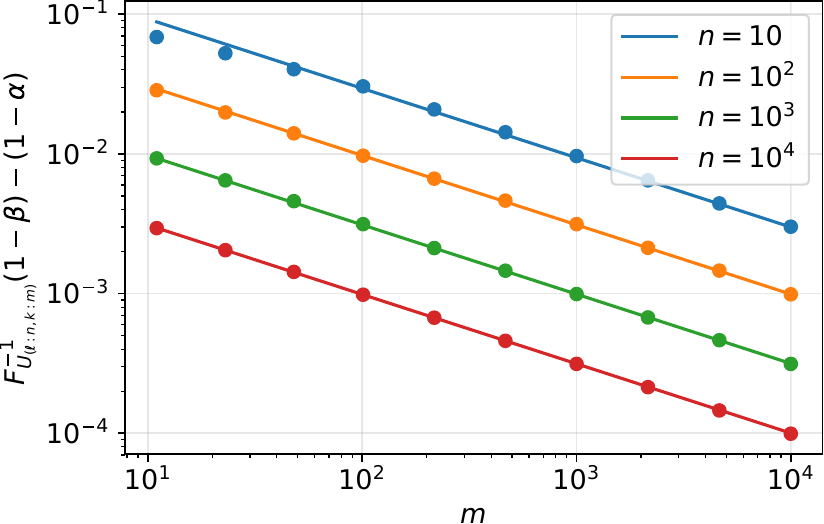}\hspace{2em}
	\includegraphics[width=0.47\linewidth]{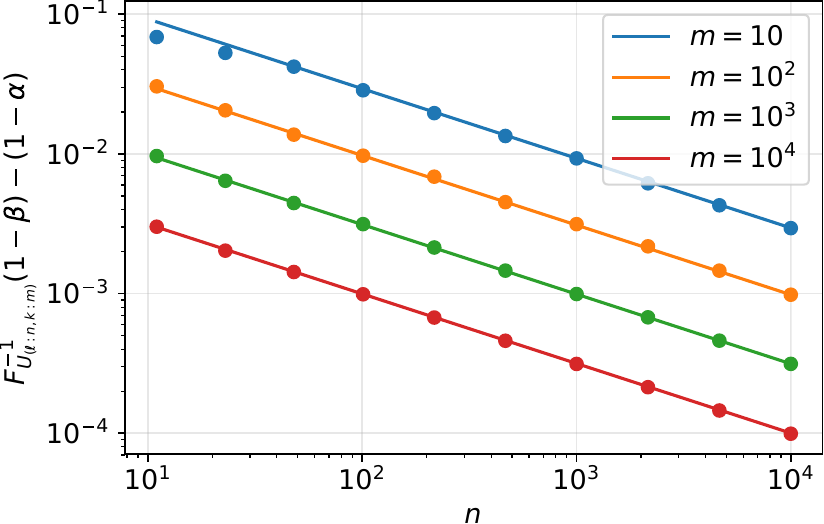}
\caption{%
\methodcondFLk\textup{:}  
log-log plot of $\Delta q_{1-\beta}$ as a function of 
$m$ \textup{(}left\textup{)} or $n$ \textup{(}right\textup{)}. 
Lines show the approximation 
$\log{\Delta q_{\beta}} \approx \log(c_3) - \gamma_3 \log(m) - \delta_3 \log(n)$ 
with $c_3,\gamma_3,\delta_3$ given by Table~\ref{tab:res_coeff_lin_all}.
\label{fig.1_qbeta-Algo4}
}
\end{figure}

\begin{figure}[H]
	\centering
	\includegraphics[width=0.47\linewidth]{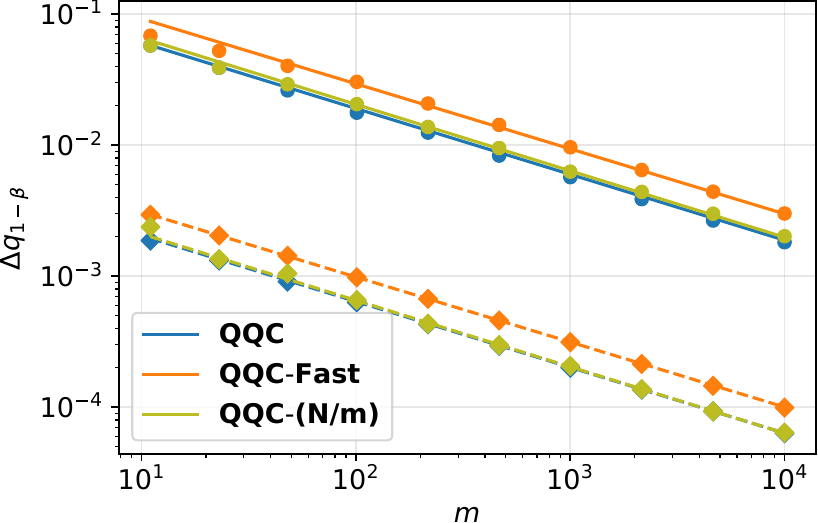}\hspace{2em}
	\includegraphics[width=0.47\linewidth]{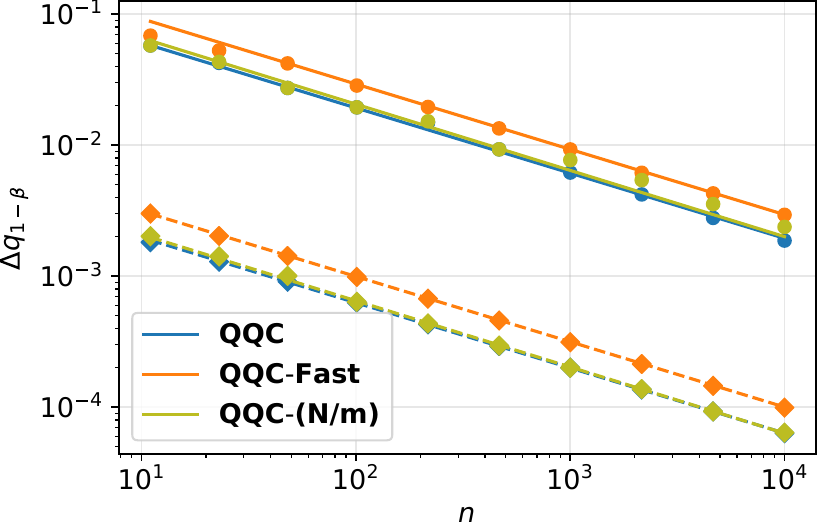}
\caption{%
Conditionally-valid algorithms\textup{:}  
log-log plot of $\Delta q_{1-\beta}$ as a function of 
$m$ \textup{(}left\textup{)} or $n$ \textup{(}right\textup{)}. 
Lines show the approximation 
$\log{\Delta q_{1-\beta}} \approx \log(c_3) - \gamma_3 \log(m) - \delta_3 \log(n)$ 
with $c_3,\gamma_3,\delta_3$ given by Table~\ref{tab:res_coeff_lin_all}. 
Plain lines and dots correspond to 
$n=10$ \textup{(}left\textup{)} or 
$m=10$ \textup{(}right\textup{)}. 
Dashed lines and diamonds correspond to 
$n=10^4$ \textup{(}left\textup{)} or 
$m=10^4$ \textup{(}right\textup{)}. 
\label{fig.q1_beta-Algos-cond}
}
\end{figure}

\subsection{Generic comparison, different  $n_j$}
This subsection provides experimental results 
that complement the ones of Section~\ref{sec:synth_data.diffnj}. 
\subsubsection{With random $n_j$}
\label{app:rate.bound_diffnj}
We start with experiments using $n_j$ (random) that are not equal. 
Figure~\ref{fig.val_nj.25-160} gives the values of 
the $n_j$ for $m \in \{4, 25\}$ in the experiments 
of Section~\ref{sec:synth_data.diffnj}. 
Figure~\ref{fig.diffnj.m-variable.q1-b} provides 
an equivalent of Figure~\ref{fig.diffnj.m-variable.E-qb} right 
(performance of conditionally-valid algorithms as a function of $m$) 
with $\Delta q_{1-\beta}$ instead of $\Delta q_{\beta}$. 

\begin{figure}[H]
	\centering
	\includegraphics[width=0.47\linewidth]{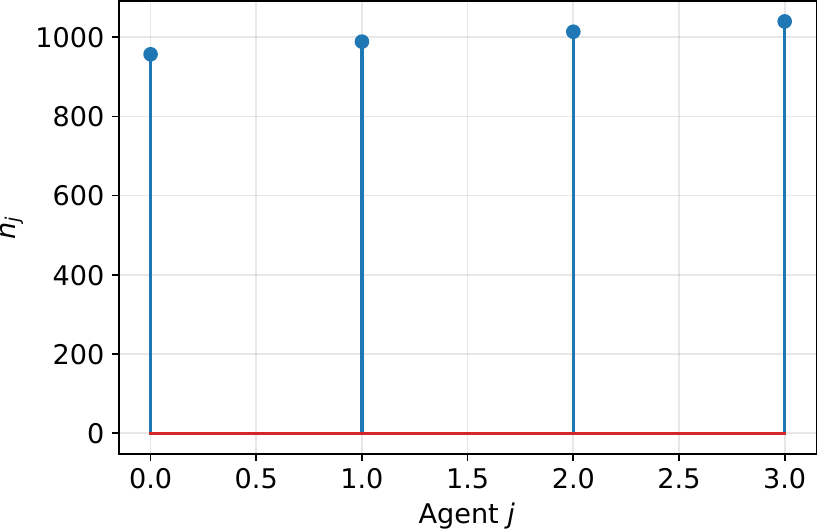}
	\includegraphics[width=0.47\linewidth]{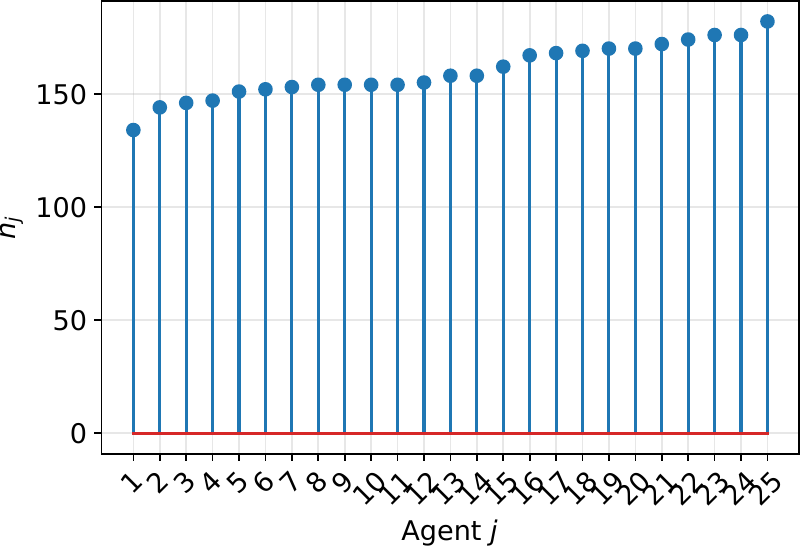}\hspace{2em}
\caption{%
Experiment of Section~\ref{sec:synth_data.diffnj}\textup{:} 
values of $(n_j)_{j \in \intset{m}}$ 
in increasing order. 
Left\textup{:} $m=4$. 
Right\textup{:} $m=25$. 
\label{fig.val_nj.25-160}
}
\end{figure}

\begin{figure}[H]
	\centering
	\includegraphics[width=0.49\linewidth]{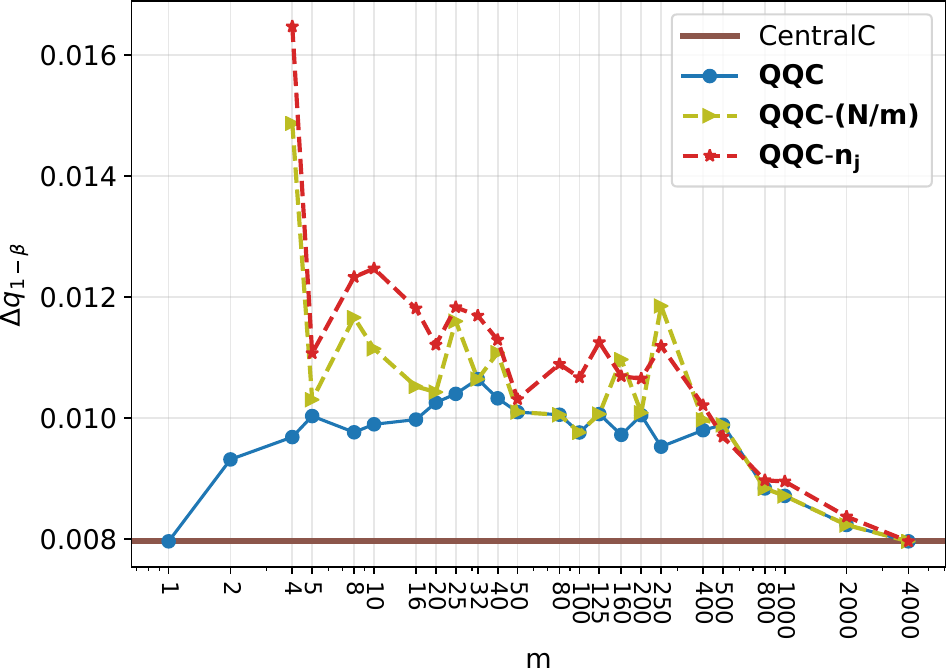}
	\caption{% 
Different $n_j$\textup{:}  
log-log plot of the performance 
\textup{(}$\Delta q_{1-\beta}$\textup{)} 
of the conditionally-valid algorithms 
considered in Section~\ref{sec:synth_data.diffnj} 
as a function of the number of agents~$m$.  
The total number of data points $N=\sum_{j=1}^{m} n_j = 4000$ 
is fixed and $m$ describes the set of divisors of~$N$. 
The values when $m=1, 2$ for \methodcondFL-$(N/m)$ and \methodcondFL-$n_j$ are missing because there is no $k$ such that the resulting sets are conditionally valid.
}
\label{fig.diffnj.m-variable.q1-b}
\end{figure}

\subsubsection{Algorithms~\ref{algo.FCP-QQ.1_nj} and \ref{algo.FCP-QQ-cond.1_nj} with $n_j = N/m$}
\label{app:rate.bound_lfix}
As in Section~\ref{sec:synth_data.diffnj}, we now consider 
Algorithms~\ref{algo.FCP-QQ.1_nj} and \ref{algo.FCP-QQ-cond.1_nj} with $n_j = N/m$ for every $h \in \intset{m}$, 
that we call \method-$(N/m)$ and  \methodcondFL-$(N/m)$. 
Table~\ref{tab:res_coeff_lin_N_M} gives the values of the 
coefficients of the robust log-linear regression obtained 
for \method-$(N/m)$ and \methodcondFL-$(N/m)$ (following the exact same procedure 
as for \method{}, \methodlow{}, \methodcondFL{}, and \methodcondFLk{},
as described in Section~\ref{sec:synth_data.equalnj}). 
In Figure~\ref{fig.DeltaE-Algos-margin_lfix}, we compare 
the performances (measured by $\Delta \E$) 
of the marginally-valid algorithms 
\method, \methodlow, and \method-$(N/m)$, 
as a function of either $m$ or $n \egaldef N/m$.
In Figure~\ref{fig.DeltaE-Algos-cond_lfix}, we compare 
the performances (measured by $\Delta q_{\beta}$) 
of the conditionally valid algorithms 
\methodcondFL, \methodcondFLk, and \methodcondFL-$(N/m)$, 
as a function of either $m$ or $n = N/m$. 

\begin{table}[H]
	\footnotesize
	{\centering
		\ra{1.}
		\begin{adjustbox}{max width=\textwidth}
			\begin{tabular}{@{}llll|lll|lll@{}}
				\toprule
				& 
				\multicolumn{3}{c}{
					$\boldsymbol{\Delta\E \approx c_1 m^{-\gamma_1} n^{-\delta_1}}$
				}
				& 
				\multicolumn{3}{c}{
					$\boldsymbol{\Delta q_{\beta} \approx c_2 m^{-\gamma_2} n^{-\delta_2}}$
				}
				& 
				\multicolumn{3}{c}{
					$\boldsymbol{\Delta q_{1-\beta} \approx c_3 m^{-\gamma_3} n^{-\delta_3}}$
				}
				\\ 
				\textbf{Method} & $\boldsymbol{c_1}$ & $\boldsymbol{\gamma_1}$ & $\boldsymbol{\delta_1}$ & $\boldsymbol{c_2}$ & $\boldsymbol{\gamma_2}$ & $\boldsymbol{\delta_2}$ & $\boldsymbol{c_3}$ & $\boldsymbol{\gamma_3}$ & $\boldsymbol{\delta_3}$ \\
				\midrule
				\method-$(N/m)$ & $0.248$ & $0.952$ & $0.521$ & $0.259$ & $0.467$ & $0.509$ & $0.417$ & $0.520$ & $0.511$ \\
				\methodcondFL-$(N/m)$ & $0.408$ & $0.521$ & $0.507$ & $0.462$ & $0.992$ & $0.552$ & $0.713$ & $0.507$ & $0.505$ \\
				\bottomrule
			\end{tabular}
		\end{adjustbox}
	}
\caption{%
Estimated parameters of the log-linear model 
$\log{y} = \log(c_i) - \gamma_i \log(m) - \delta_i \log(n)$ 
where $y$ is either $\Delta\E$, $\Delta q_{\beta}$ 
or $\Delta q_{1-\beta}$, 
for \method-$(N/m)$ and \methodcondFL-$(N/m)$, 
introduced in Section~\ref{sec:synth_data.diffnj}\textup{;} 
See also text of Section~\ref{sec:synth_data.equalnj} for details.
\label{tab:res_coeff_lin_N_M}
}
\end{table}

\begin{figure}[H]
	\centering
	\includegraphics[width=0.47\linewidth]{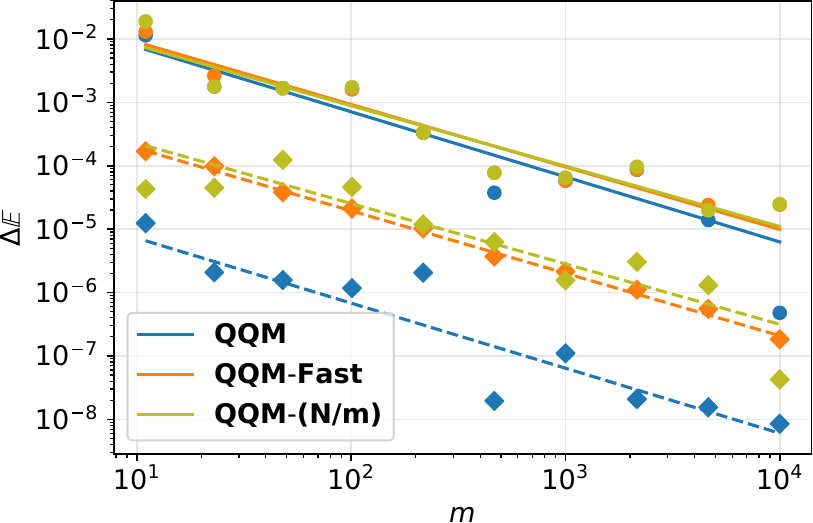}\hspace{2em}
	\includegraphics[width=0.47\linewidth]{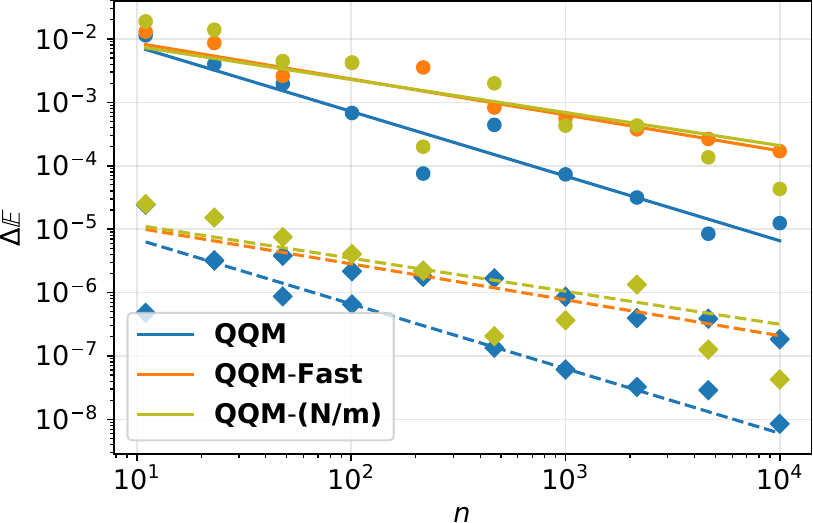}
\caption{%
Marginally-valid algorithms\textup{:}  
log-log plot of $\Delta\E$ as a function of 
$m$ \textup{(}left\textup{)} or $n$ \textup{(}right\textup{)}. 
Lines show the approximation 
$\log{\Delta \E} \approx \log(c) - \gamma \log(m) - \delta \log(n)$. 
Plain lines and dots correspond to 
$n=10$ \textup{(}left\textup{)} or 
$m=10$ \textup{(}right\textup{)}. 
Dashed lines and diamonds correspond to 
$n=10^4$ \textup{(}left\textup{)} or 
$m=10^4$ \textup{(}right\textup{)}. 
\label{fig.DeltaE-Algos-margin_lfix}
}
\end{figure}

\begin{figure}[H]
	\centering
	\includegraphics[width=0.47\linewidth]{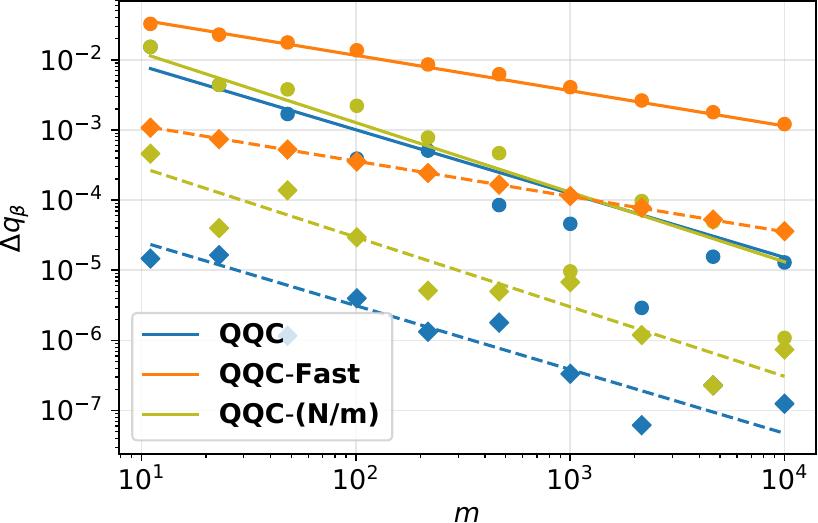}\hspace{2em}
	\includegraphics[width=0.47\linewidth]{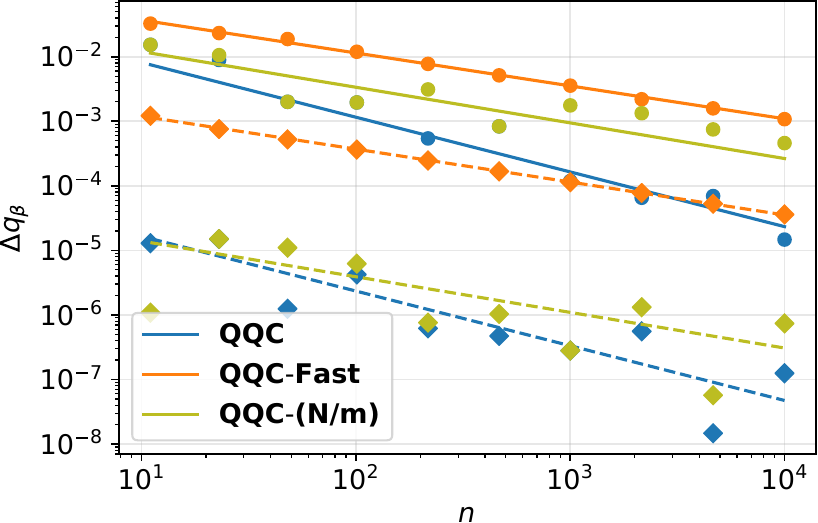}
	\caption{%
		Conditionally-valid algorithms\textup{:}  
		log-log plot of $\Delta q_{\beta}$ as a function of 
		$m$ \textup{(}left\textup{)} or $n$ \textup{(}right\textup{)}. 
		Lines show the approximation 
		$\log{\Delta q_{\beta}} \approx \log(c) - \gamma \log(m) - \delta \log(n)$. 
		Plain lines and dots correspond to 
		$n=10$ \textup{(}left\textup{)} or 
		$m=10$ \textup{(}right\textup{)}. 
		Dashed lines and diamonds correspond to 
		$n=10^4$ \textup{(}left\textup{)} or 
		$m=10^4$ \textup{(}right\textup{)}. 
		\label{fig.DeltaE-Algos-cond_lfix}
	}
\end{figure}

% !TEX root = ../main.tex

\subsection{Real data: additional results on individual data sets} \label{app:add-xp}
In this section, we present in Figures \ref{fig:bio} to \ref{fig:concrete_low_m} the results of the experiments of Section~\ref{sec:xps:real} on individual data sets.

% =============================================== Bio
\begin{figure}[!htb]
	\centering
	\includegraphics[width=0.47\linewidth]{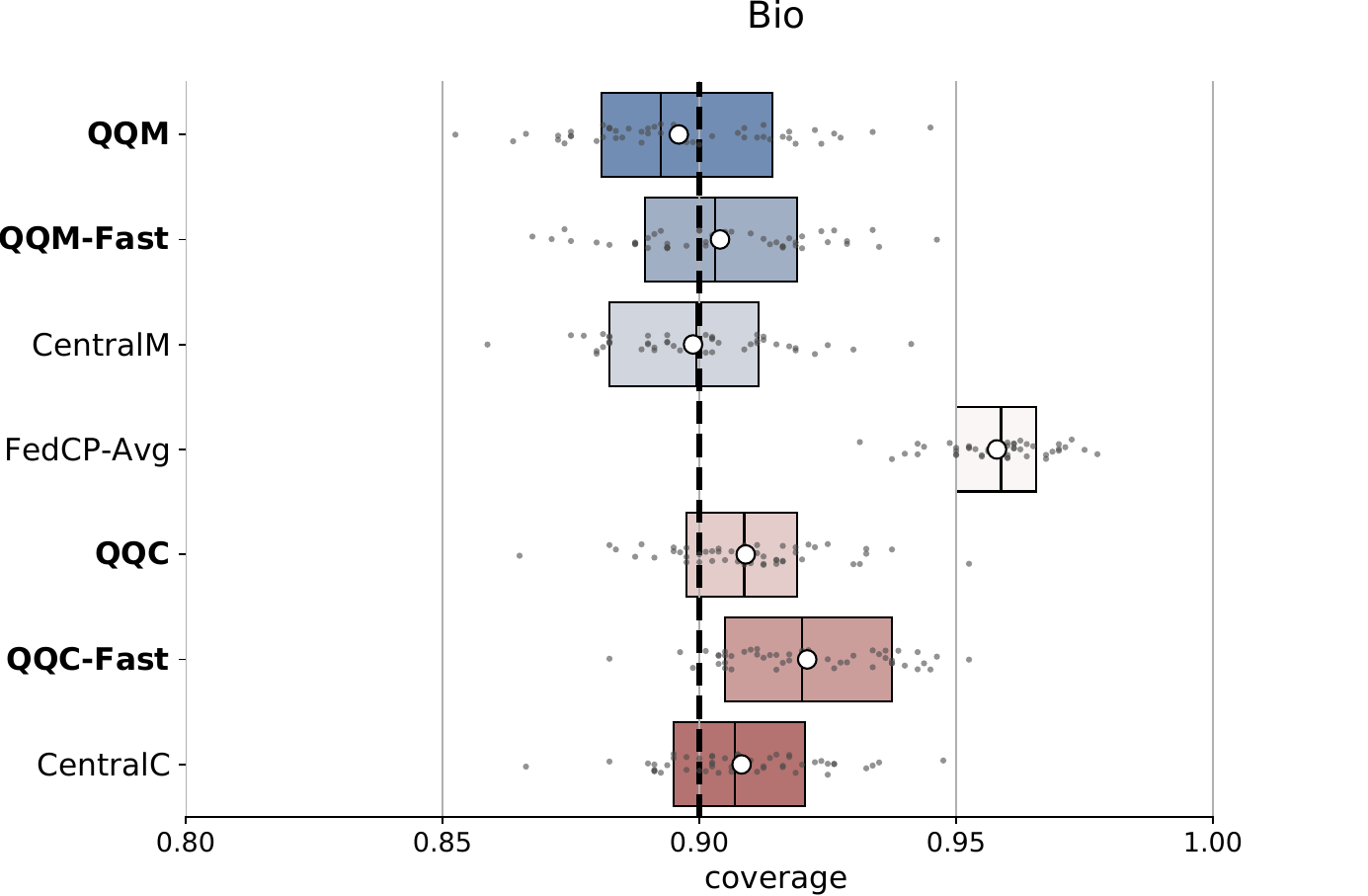}
	\includegraphics[width=0.47\linewidth]{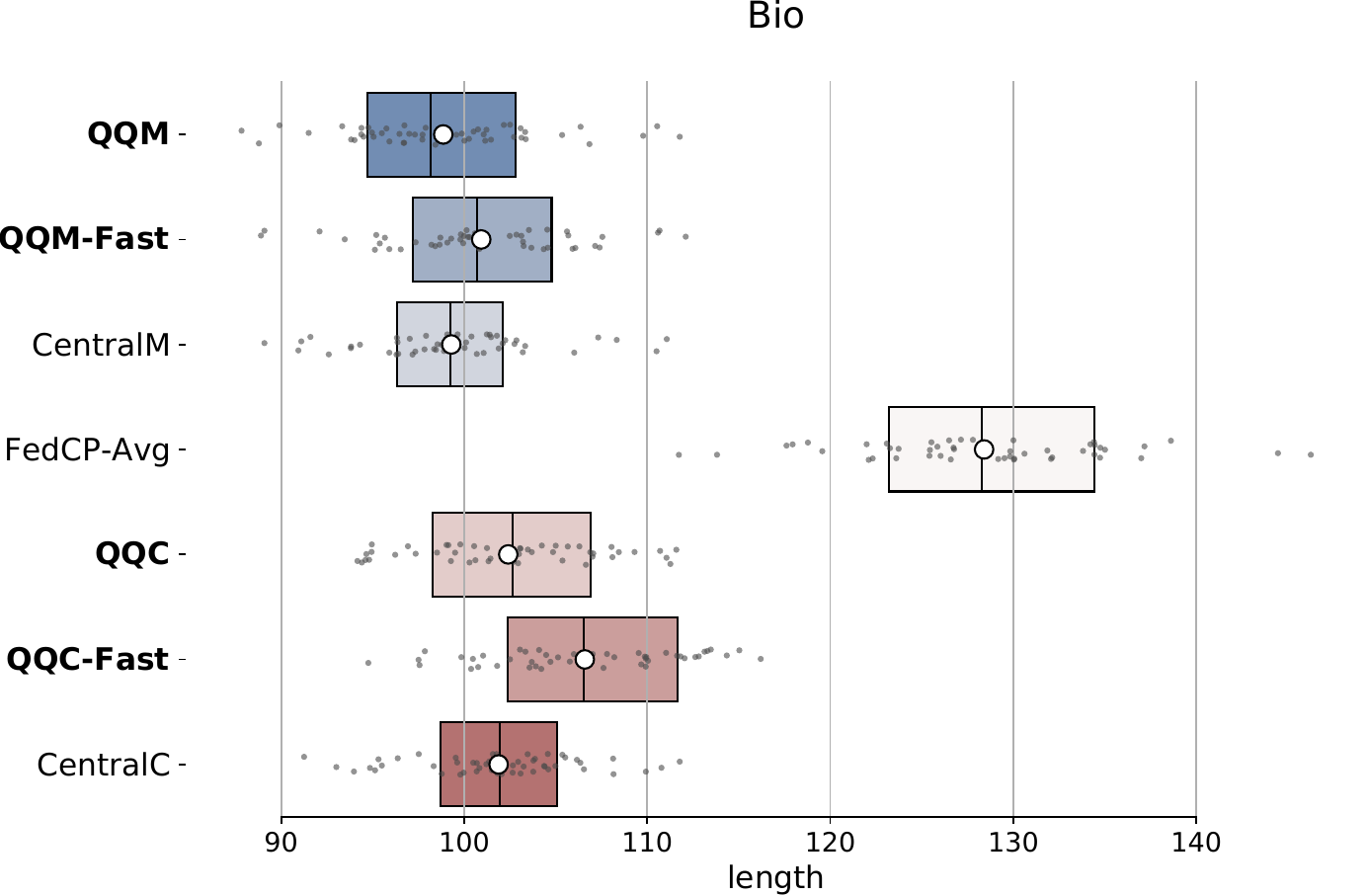}
	\caption{Coverage (left) and average length (right) 
of prediction intervals for $50$ random learning-calibration-test splits. 
The miscoverage is $\alpha=0.1$, $\beta=0.2$, 
and the calibration set is split into $m=80$ disjoint 
subsets of equal size $n=10$. 
The white circle represents the mean and the name of 
the data set is located at the top of each plot.}
	\label{fig:bio}
\end{figure}

\begin{figure}[!htb]
	\centering
	\includegraphics[width=0.47\linewidth]{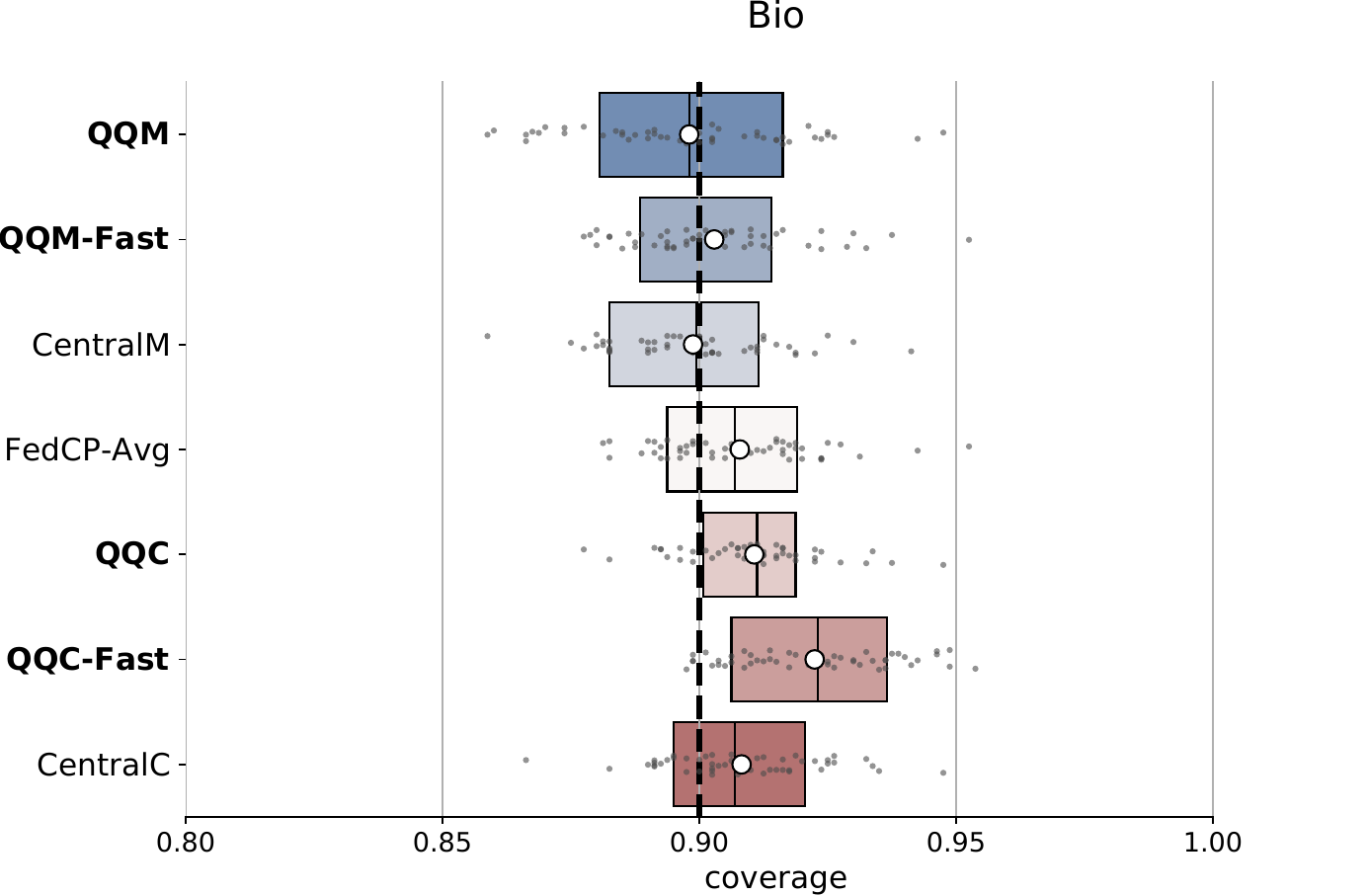}
	\includegraphics[width=0.47\linewidth]{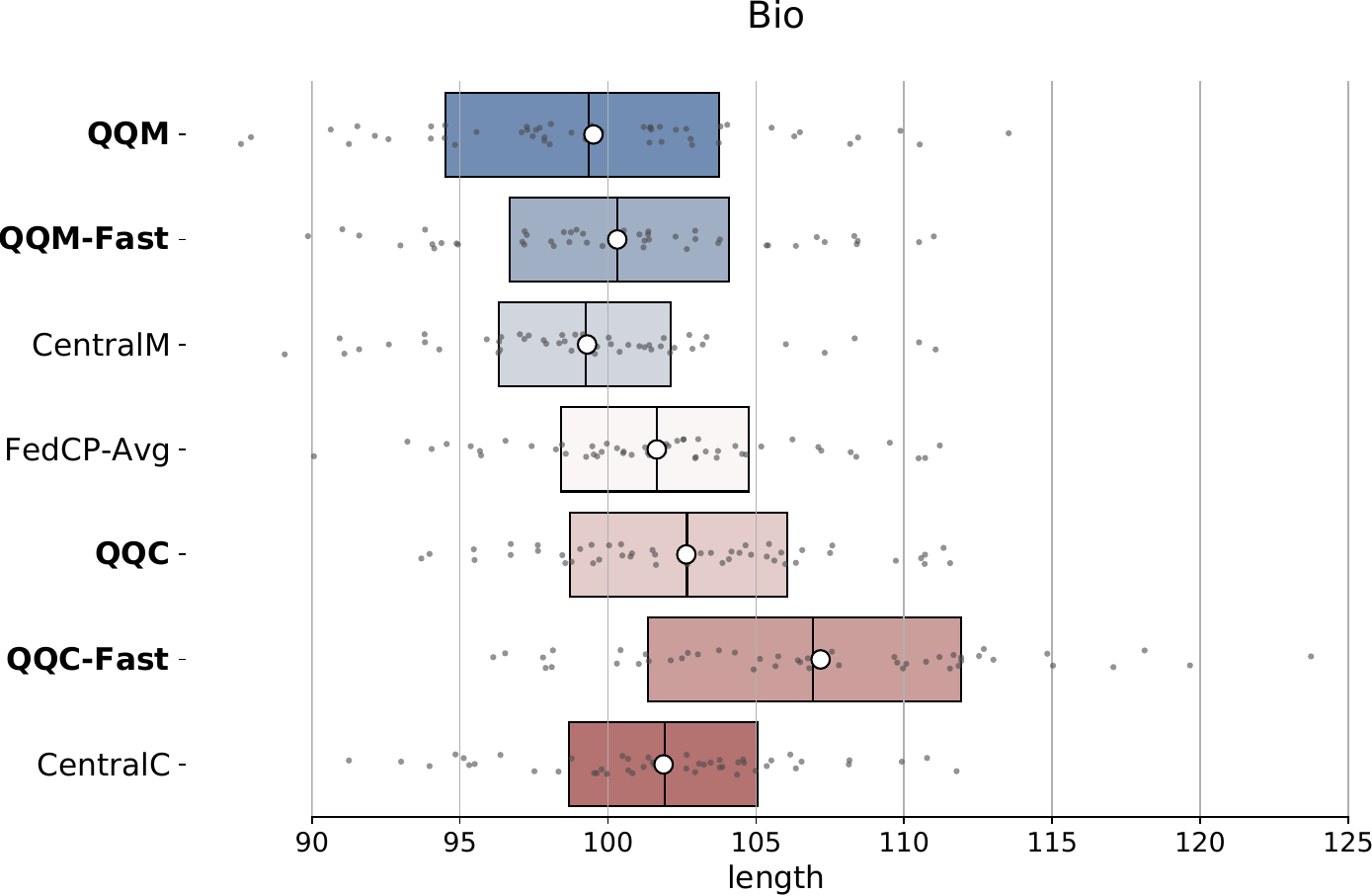}
	\caption{Coverage (left) and average length (right) 
of prediction intervals for $50$ random learning-calibration-test splits. 
The miscoverage is $\alpha=0.1$, $\beta=0.2$, 
and the calibration set is split into $m=10$ disjoint 
subsets of equal size $n=80$. 
The white circle represents the mean and the name of 
the data set is located at the top of each plot.}
	\label{fig:bio_low_m}
\end{figure}

% =============================================== Community
\begin{figure}[!htb]
	\centering
	\includegraphics[width=0.47\linewidth]{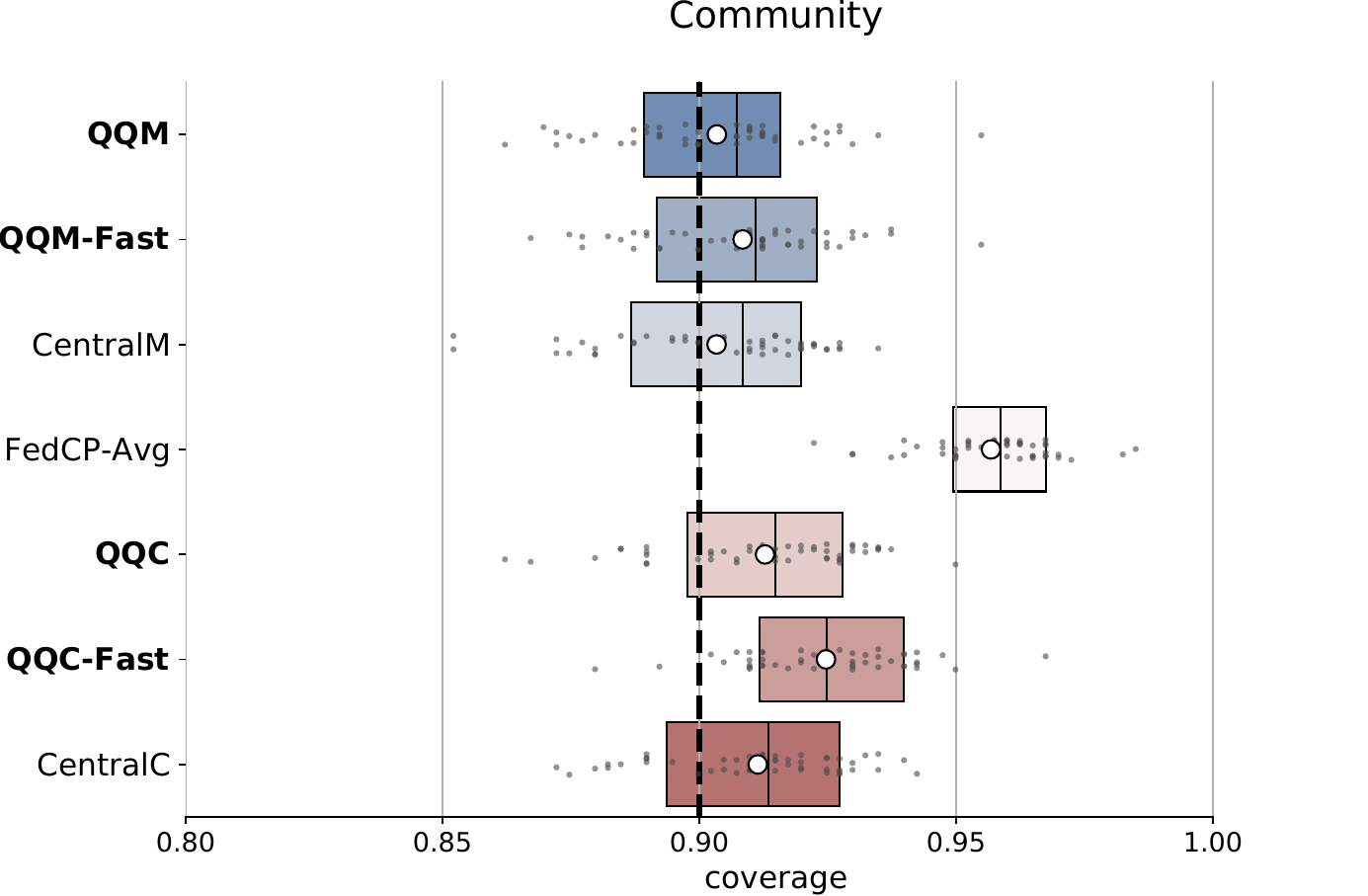}
	\includegraphics[width=0.47\linewidth]{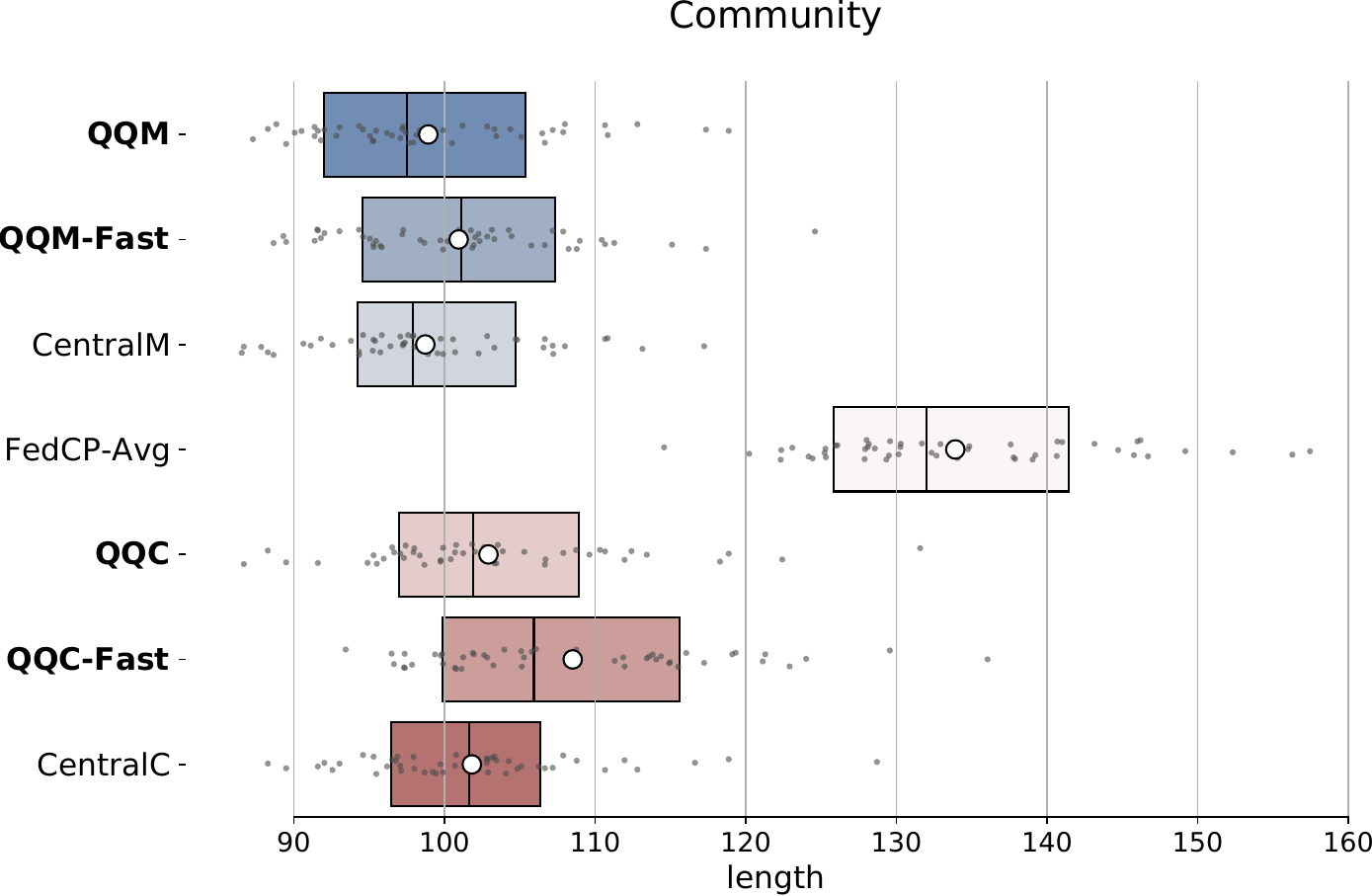}
	\caption{Same as Figure \ref{fig:bio} (see its caption) with $m=80$ and $n=10$.}
	\label{fig:commu}
\end{figure}

\begin{figure}[!htb]
	\centering
	\includegraphics[width=0.47\linewidth]{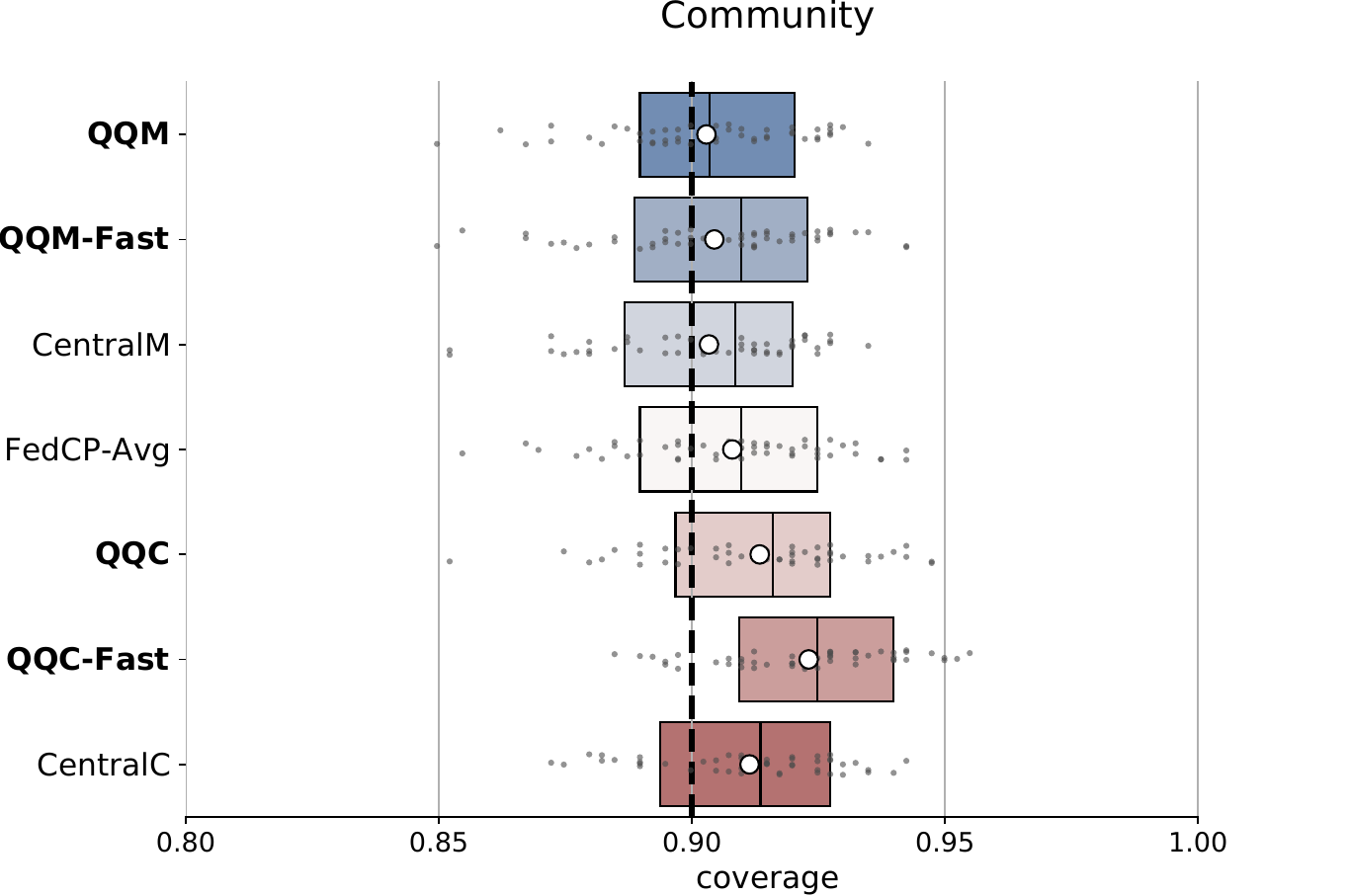}
	\includegraphics[width=0.47\linewidth]{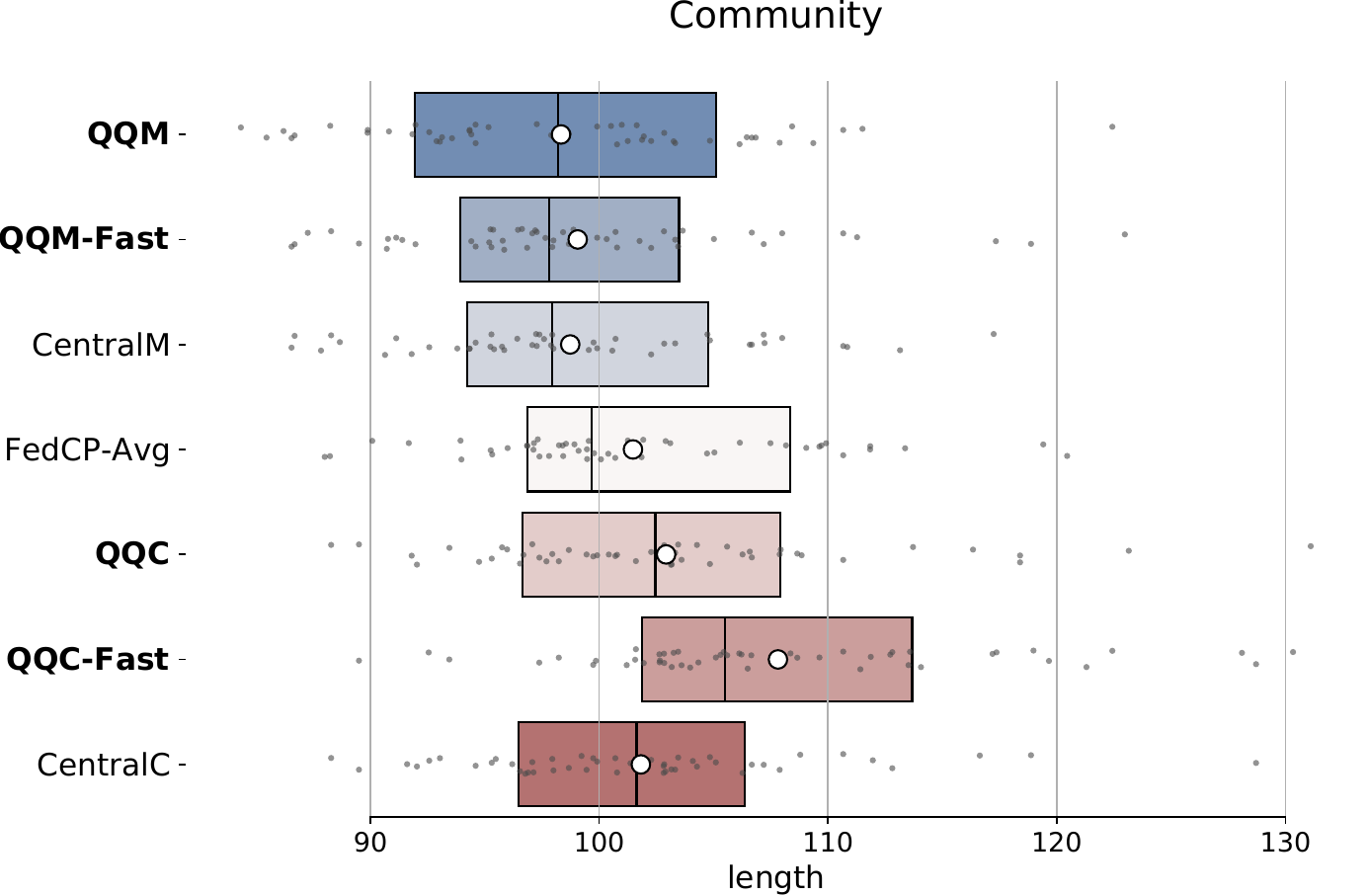}
	\caption{Same as Figure \ref{fig:bio_low_m} (see its caption) with $m=10$ and $n=80$.}
	\label{fig:commu_low_m}
\end{figure}

% =============================================== Star
\begin{figure}[!htb]
	\centering
	\includegraphics[width=0.47\linewidth]{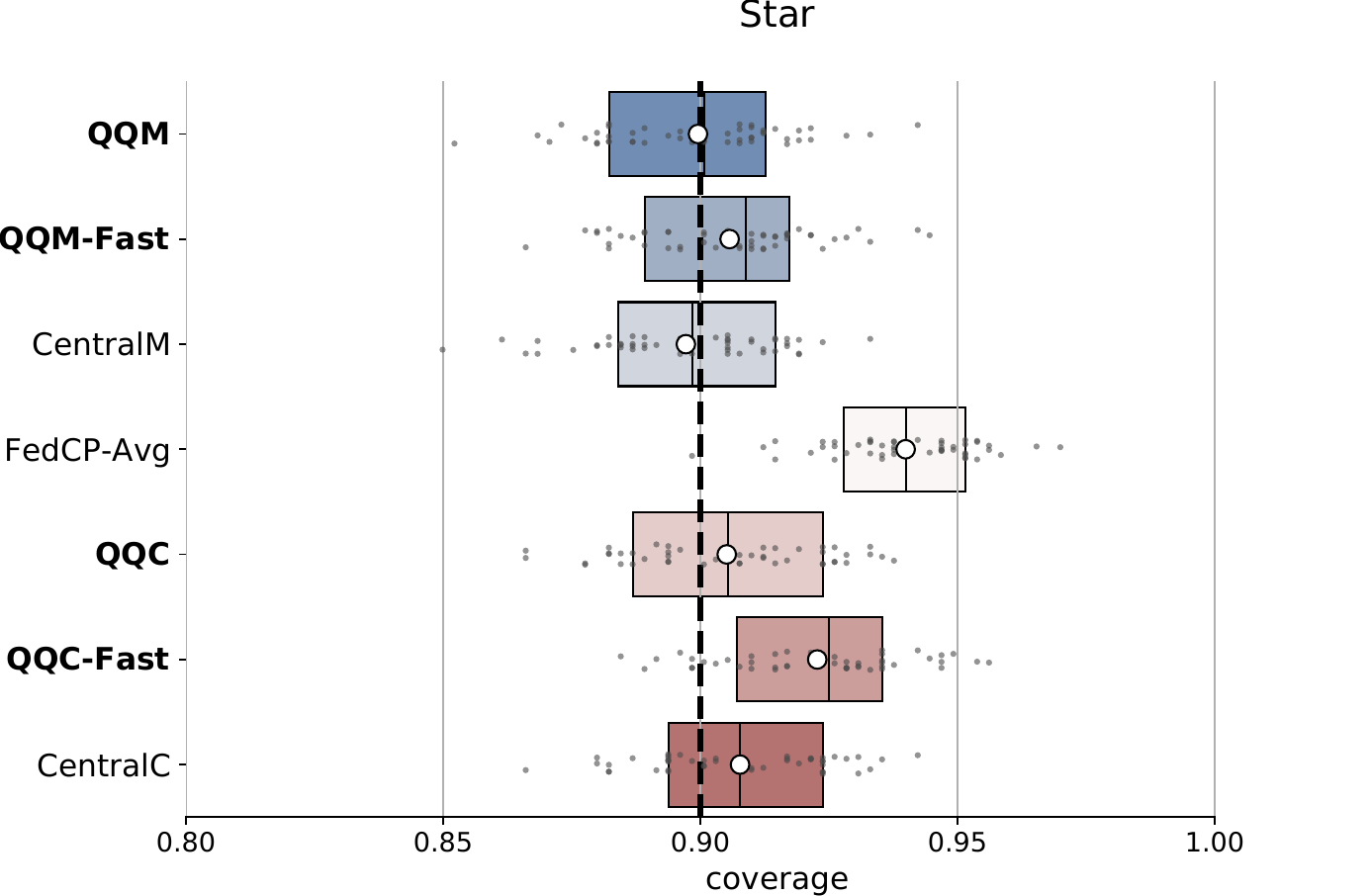}
	\includegraphics[width=0.47\linewidth]{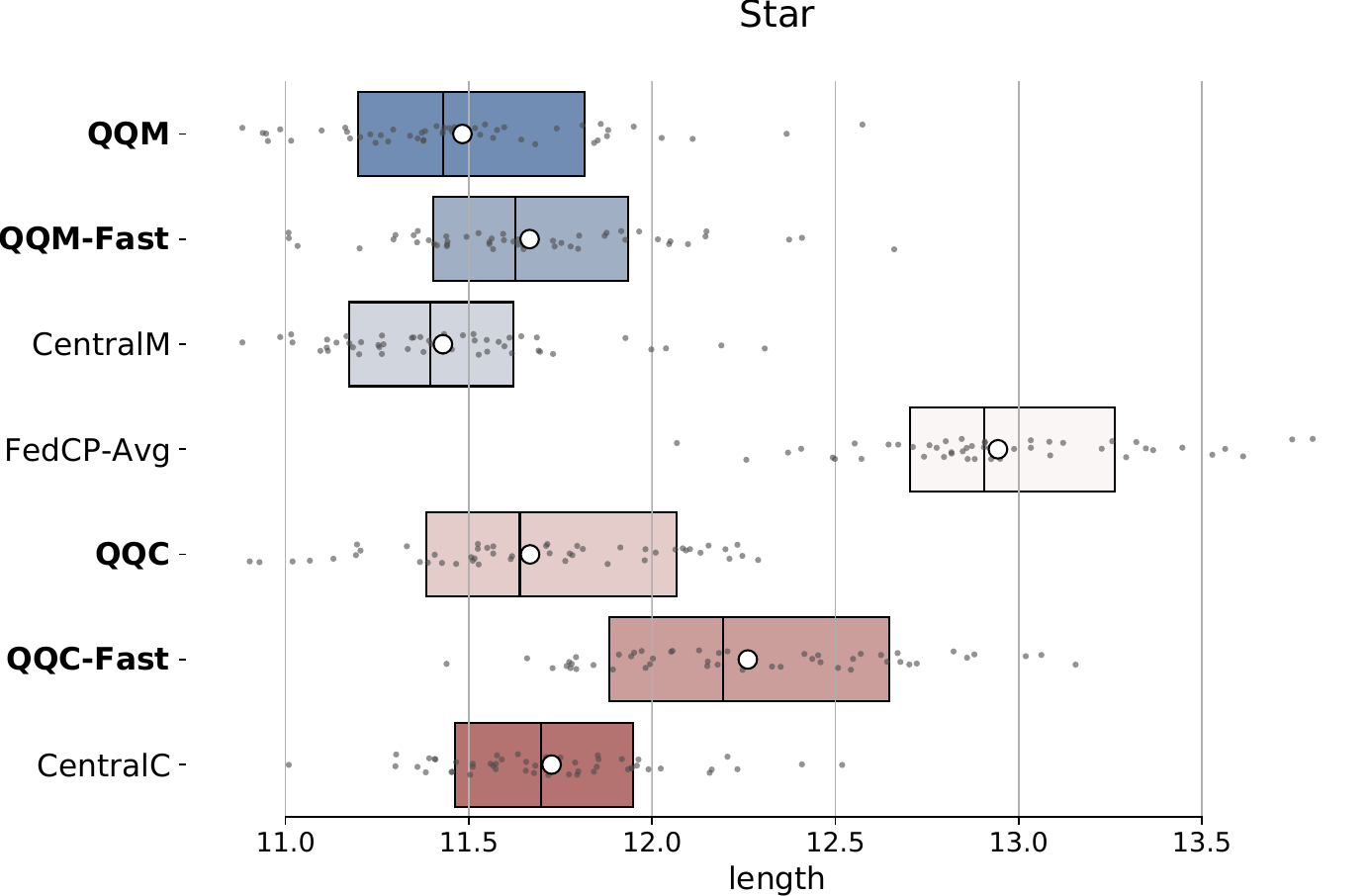}
	\caption{Same as Figure \ref{fig:bio} (see its caption) with $m=80$ and $n=10$.}
	\label{fig:star}
\end{figure}

\begin{figure}[!htb]
	\centering
	\includegraphics[width=0.47\linewidth]{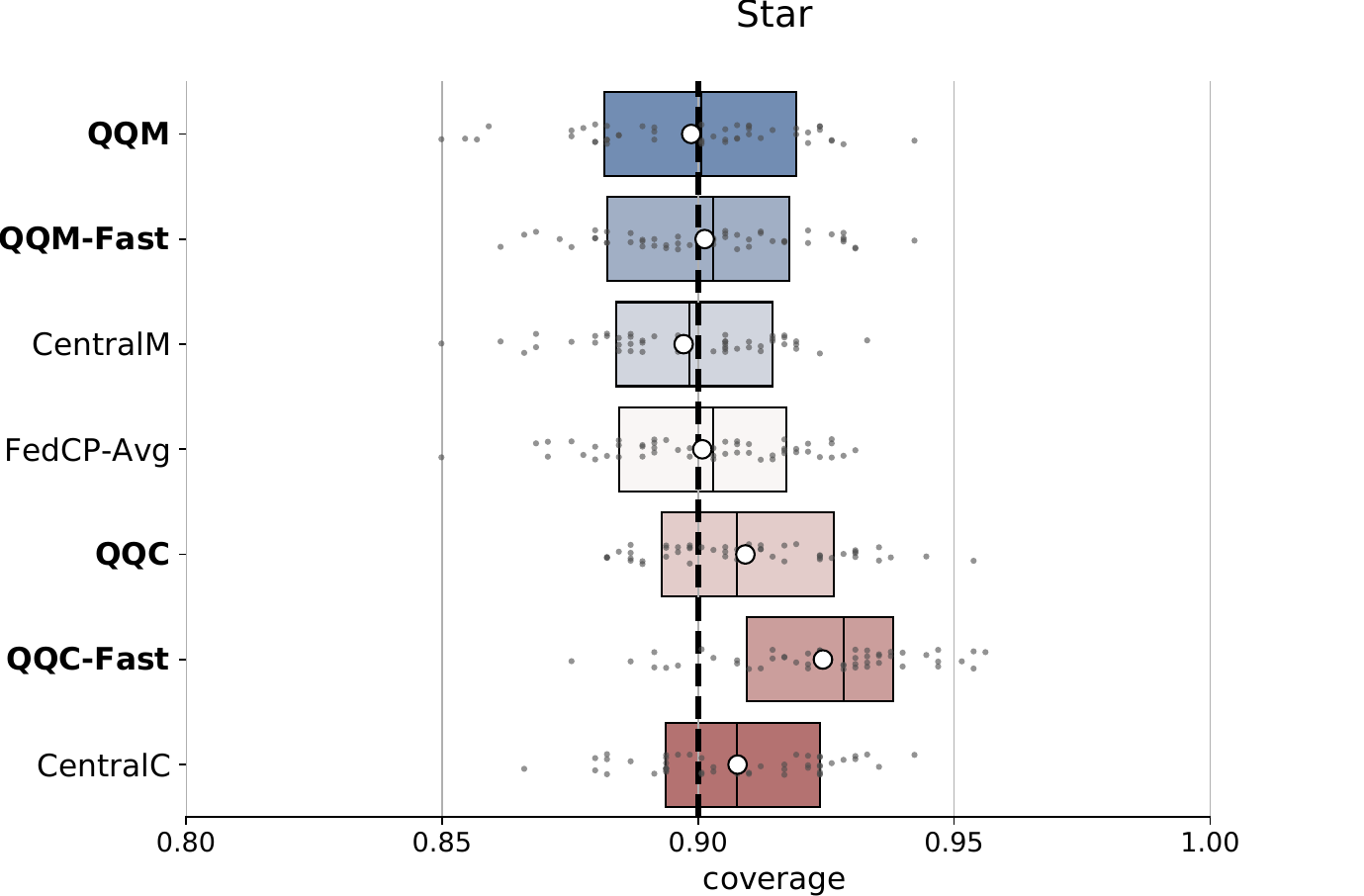}
	\includegraphics[width=0.47\linewidth]{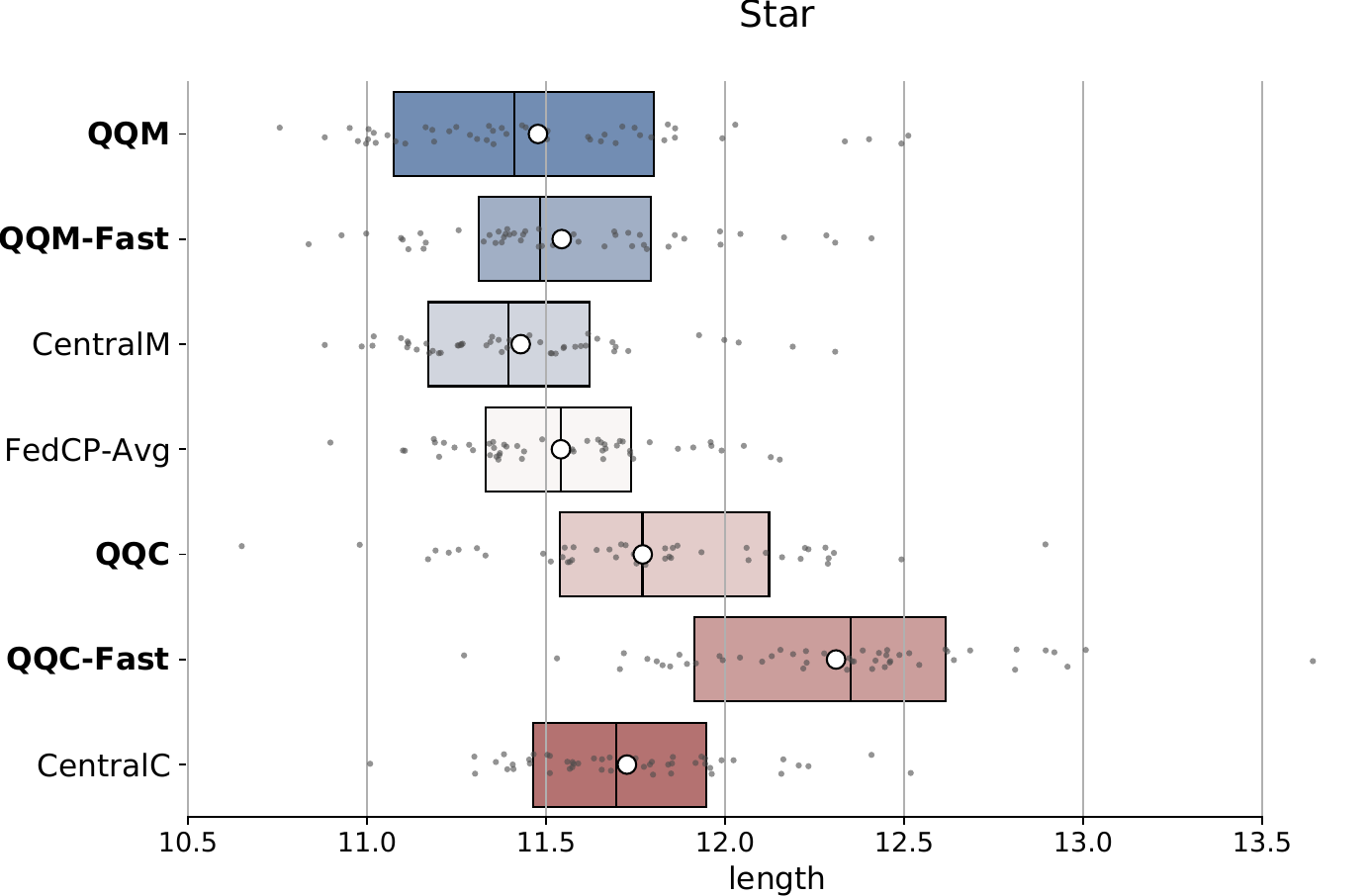}
	\caption{Same as Figure \ref{fig:bio_low_m} (see its caption) with $m=10$ and $n=80$.}
	\label{fig:star_low_m}
\end{figure}

% =============================================== Concrete
\begin{figure}[!htb]
	\centering
	\includegraphics[width=0.47\linewidth]{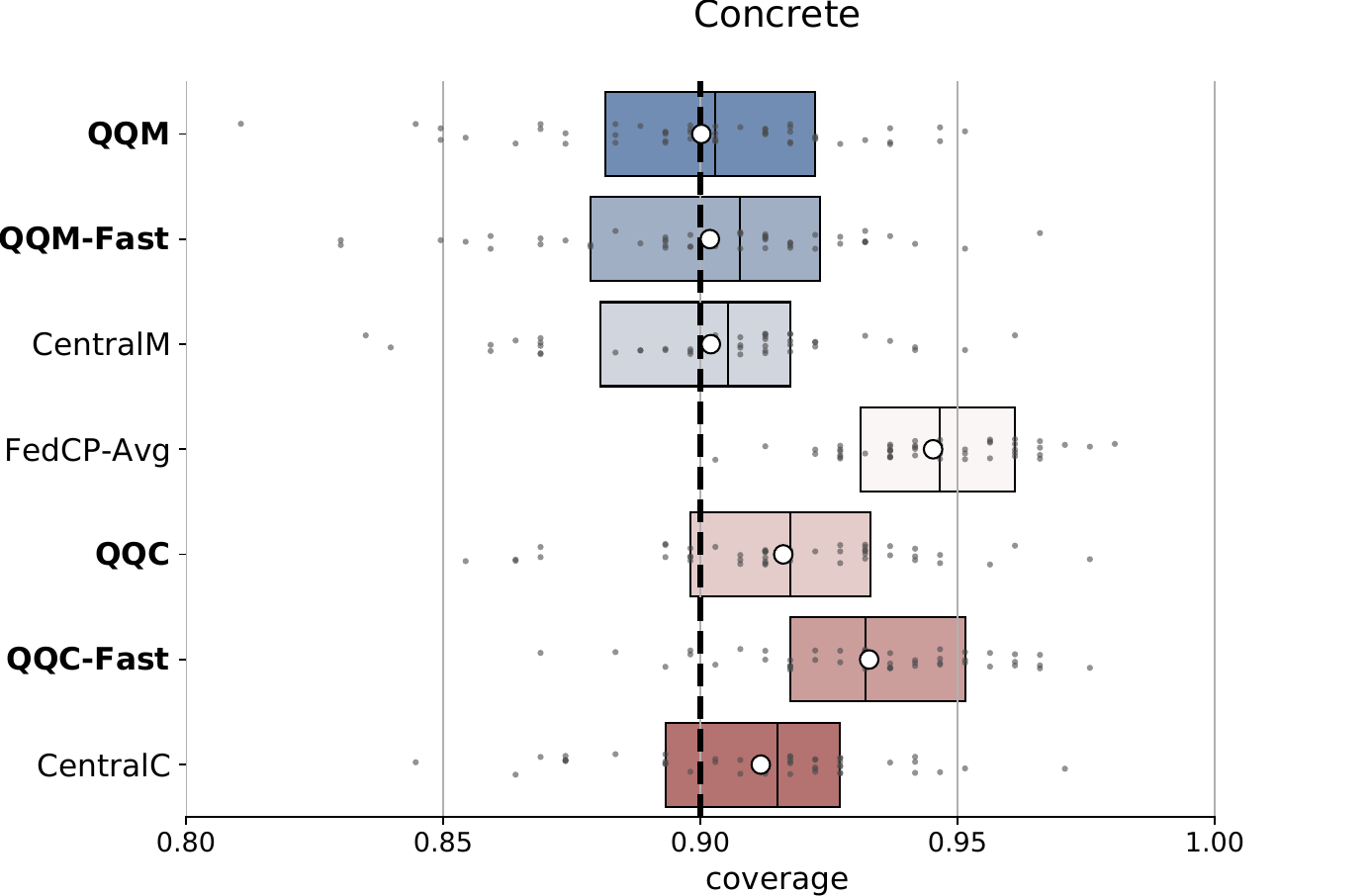}
	\includegraphics[width=0.47\linewidth]{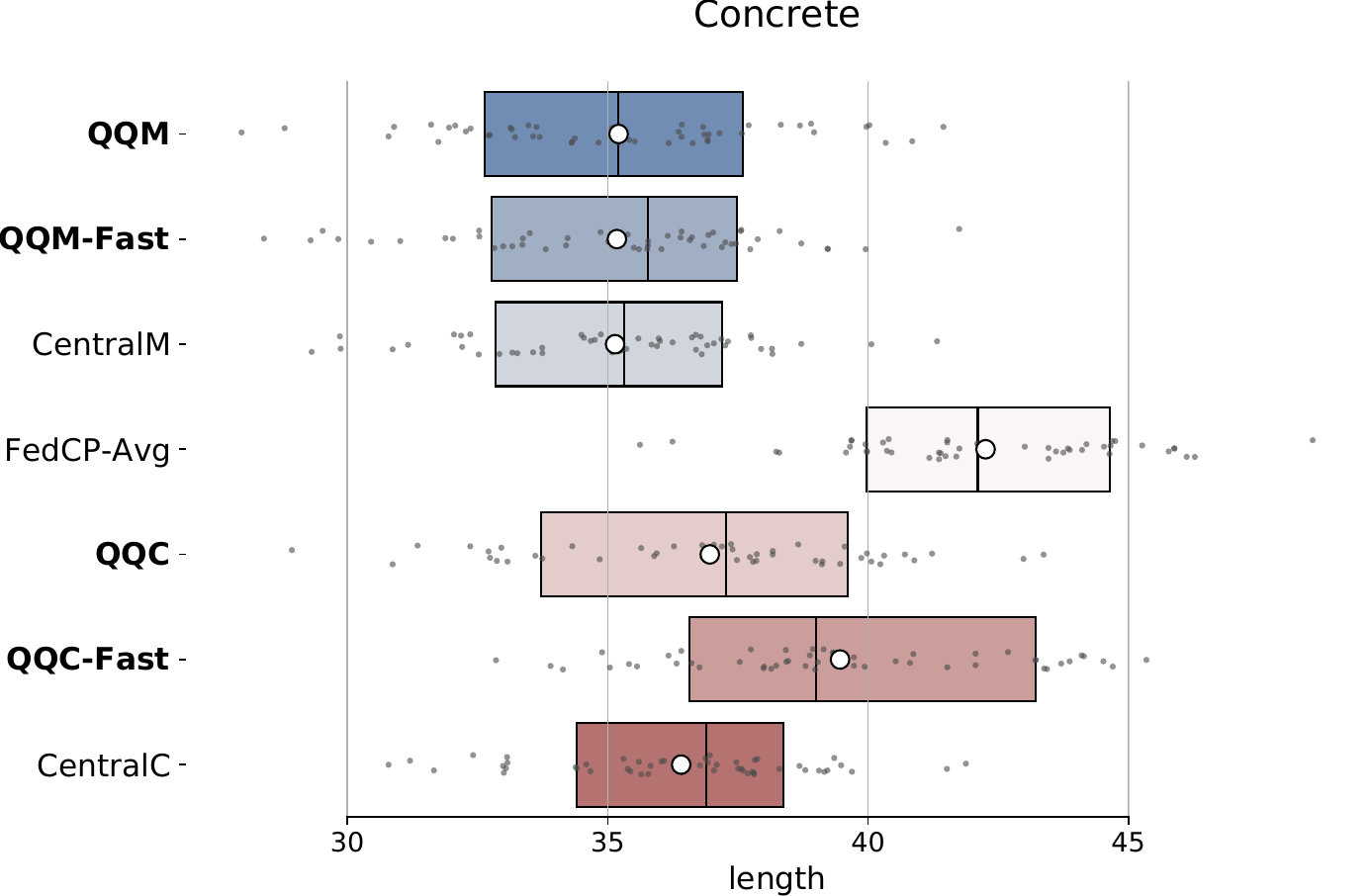}
	\caption{Same as Figure \ref{fig:bio} (see its caption) with $m=40$ and $n=10$.}
	\label{fig:concrete}
\end{figure}

\begin{figure}[!htb]
	\centering
	\includegraphics[width=0.47\linewidth]{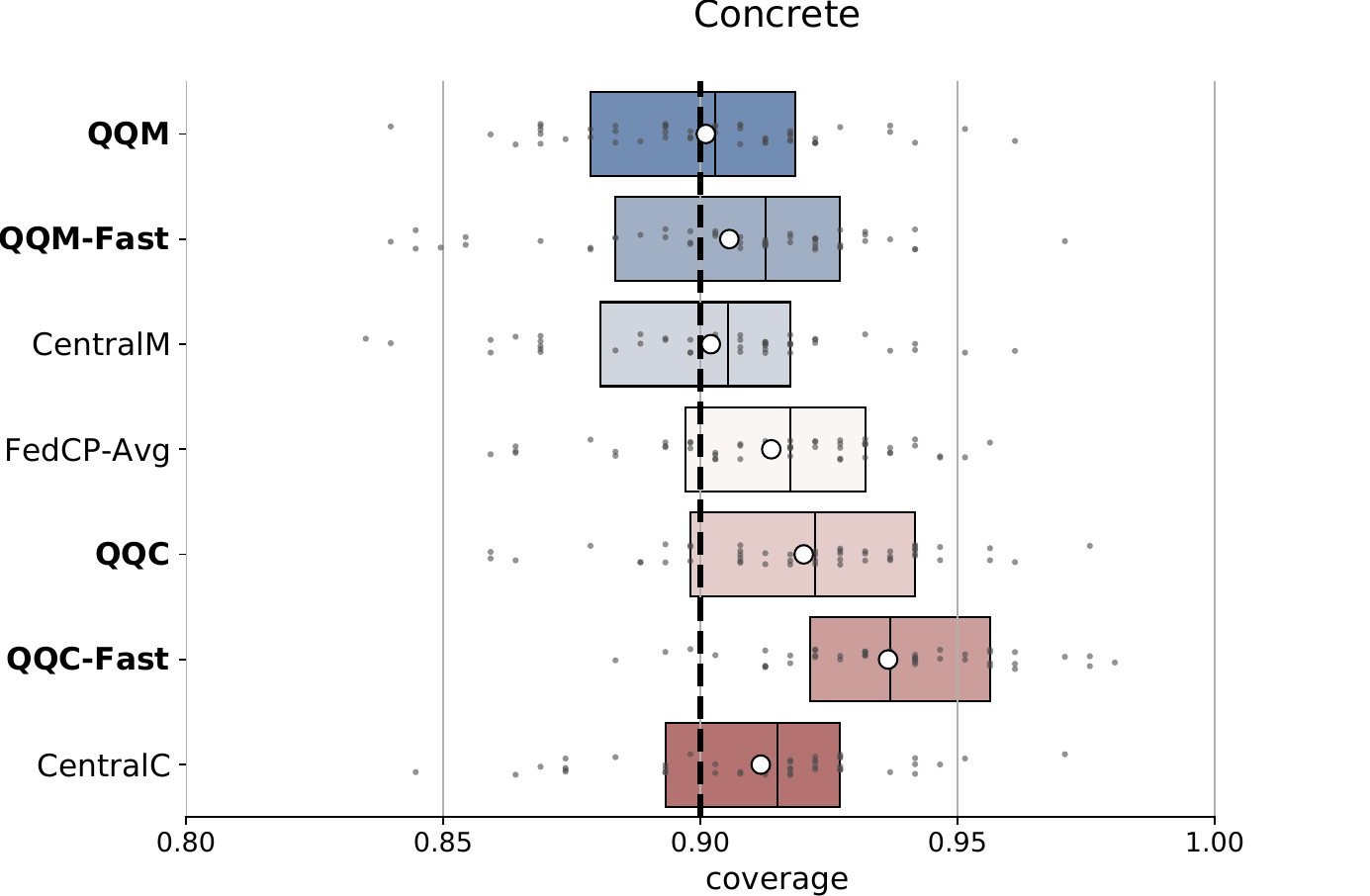}
	\includegraphics[width=0.47\linewidth]{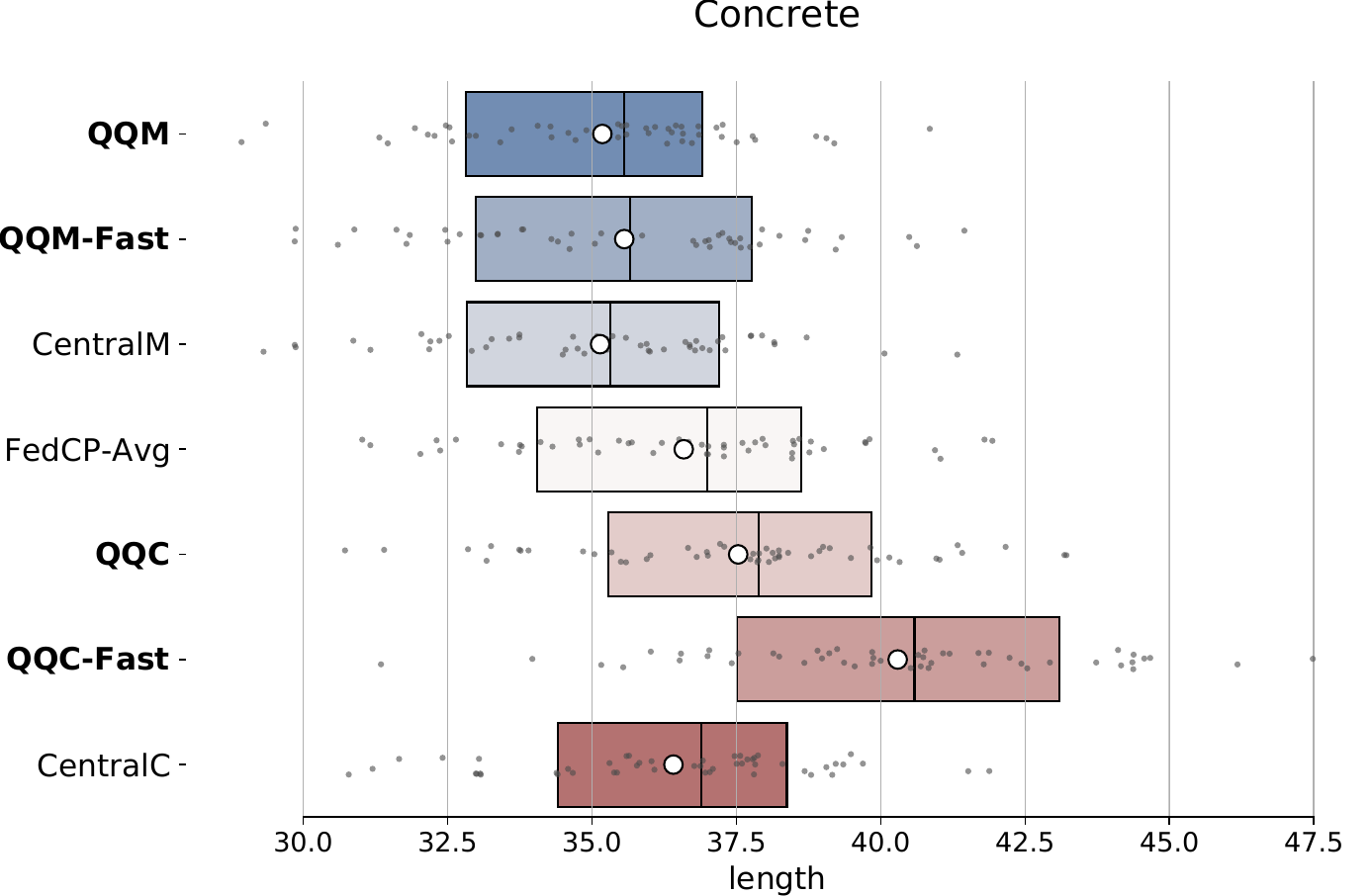}
	\caption{Same as Figure \ref{fig:bio_low_m} (see its caption) with $m=10$ and $n=40$.}
	\label{fig:concrete_low_m}
\end{figure}

%%%%%%%%%%%%%%%%%%%%%%%%%%%%%%%%%%

\end{document}